\documentclass[a4paper,11pt,
 leqno
]{article}

\usepackage{times}
\usepackage{comment}
\usepackage[utf8]{inputenc}
\usepackage[T1]{fontenc}
\usepackage{amsthm}
\usepackage{amssymb}
\usepackage{mathrsfs}
\usepackage{amsmath}
\usepackage{amsfonts}
\usepackage{mathtools}
\usepackage{graphicx}
\usepackage{float}
\usepackage{dsfont}
\usepackage{enumerate}
\usepackage{enumitem}
\usepackage[a4paper]{geometry}
\usepackage{csquotes}
\usepackage{IEEEtrantools}
\usepackage{appendix}
\usepackage{constants}
\usepackage[obeyDraft]{todonotes}
\usepackage{tikz-cd}
\usepackage{enumerate}
\usepackage{mathabx}
\usepackage
{hyperref}

\newtheorem{The}{Theorem}[section]
\newtheorem{Lem}[The]{Lemma}
\newtheorem{Prop}[The]{Proposition}
\newtheorem{Cor}[The]{Corollary}

\theoremstyle{definition}

\newtheorem{Ex}[The]{Example}
\theoremstyle{remark}

\newtheorem{Rk}[The]{Remark}

\numberwithin{equation}{section}

\usepackage{caption}
\title{
\normalsize
\textbf{{
SHARP THRESHOLD FOR THE BALLISTICITY OF THE RANDOM WALK ON THE EXCLUSION PROCESS
}}}
\author{}
\date{}
\makeatletter
\newcommand{\cE}{\mathcal{E}}

\newcommand{\Q}{\mathbb{Q}}
\newcommand{\R}{\mathbb{R}}
\newcommand{\Z}{\mathbb{Z}}
\newcommand{\N}{\mathbb{N}}
\newcommand{\bL}{\mathbb{L}}

\renewcommand{\P}{\mathbb{P}}
\newcommand{\bP}{\mathbf{P}}

\newcommand{\1}{\mathds{1}}

\renewcommand{\phi}{\varphi}
\renewcommand{\tilde}{\widetilde}
\renewcommand{\hat}{\widehat}

\newconstantfamily{c}{symbol=c}

\usepackage{titlesec}
\titleformat{\subsection}[runin]{\normalfont\bfseries}{\thesubsection.}{.5em}{}[.]\titlespacing{\subsection}{0pt}{2ex plus .1ex minus .2ex}{.8em}
\titleformat{\subsubsection}[runin]{\normalfont\bfseries}{\thesubsubsection.}{.5em}{}[.]
\titlespacing{\subsubsection}{0pt}{2ex plus .1ex minus .2ex}{.8em}

\usepackage{color}
\definecolor{Red}{rgb}{1,0,0}
\definecolor{Blue}{rgb}{0,0,1}
\definecolor{Olive}{rgb}{0.41,0.55,0.13}
\definecolor{Yarok}{rgb}{0,0.5,0}
\definecolor{Green}{rgb}{0,1,0}
\definecolor{MGreen}{rgb}{0,0.8,0}
\definecolor{DGreen}{rgb}{0,0.55,0}
\definecolor{Yellow}{rgb}{1,1,0}
\definecolor{Cyan}{rgb}{0,1,1}
\definecolor{Magenta}{rgb}{1,0,1}
\definecolor{Orange}{rgb}{1,.5,0}
\definecolor{Violet}{rgb}{.5,0,.5}
\definecolor{Purple}{rgb}{.75,0,.25}
\definecolor{Brown}{rgb}{.75,.5,.25}
\definecolor{Grey}{rgb}{.7,.7,.7}
\definecolor{Black}{rgb}{0,0,0}
\def\red{\color{Red}}

\def\black{\color{Black}}

\usepackage{hyperref}
\hypersetup{linktocpage,
  colorlinks   = true,
  urlcolor     = blue,
  linkcolor    = blue,
  citecolor   = black
}

\usepackage{titletoc}

\dottedcontents{section}[4em]{}{2.9em}{0.7pc}
\dottedcontents{subsection}[0em]{}{3.3em}{1pc}

\usepackage{multicol}
\usepackage{parcolumns}

\usepackage{tabu}
\usepackage{caption}
\begin{document}
\thispagestyle{empty}
\maketitle
\vspace{0.1cm}
\begin{center}
\vspace{-1.9cm}
Guillaume Conchon-Kerjan$^1$, Daniel Kious$^2$ and Pierre-Fran\c cois Rodriguez$^3$

\end{center}
\vspace{0.1cm}
\begin{abstract}
We study a non-reversible random walk advected by the symmetric simple exclusion process, so that the walk has a local drift of opposite sign when sitting atop an occupied or an empty site. We prove that the back-tracking probability of the walk exhibits a sharp transition as the density $\rho$ of particles in the underlying exclusion process varies across a critical density $\rho_c$. Our results imply that the speed $v=v(\rho)$ of the walk is a strictly monotone function and that the zero-speed regime is either absent or collapses to a single point, $\rho_c$, thus solving a conjecture of~\cite{HKT20}.  The proof proceeds by exhibiting a quantitative monotonicity result for the speed of a truncated model, in which the environment is renewed after a finite time horizon $L$. The truncation parameter $L$ is subsequently pitted against the density $\rho$ to carry estimates over to the full model. Our strategy is somewhat reminiscent of certain techniques recently used to prove sharpness results in percolation problems. A key instrument is a combination of renormalisation arguments with refined couplings of environments at slightly different densities, which we develop in this article. Our results hold in fact in greater generality and apply to a class of environments with possibly egregious features, outside perturbative regimes. 
\end{abstract}

\vspace{5cm}
\begin{flushleft}

\noindent\rule{5cm}{0.4pt} \hfill September 2024 \\
\bigskip
\begin{multicols}{2}

$^1$King's College London\\
Department of Mathematics\\
London WC2R 2LS\\
United Kingdom\\
\url{guillaume.conchon-kerjan@kcl.ac.uk}\\

\bigskip

$^2$University of Bath\\
Department of Mathematical Sciences \\
Bath BA2 7AY \\
United Kingdom\\
\url{daniel.kious@bath.ac.uk}\\

\columnbreak
\thispagestyle{empty}
\bigskip
\medskip
\hfill$^3$Imperial College London\\
\hfill Department of Mathematics\\
\hfill London SW7 2AZ \\
\hfill United Kingdom\\
\hfill \url{p.rodriguez@imperial.ac.uk} 
\end{multicols}
\end{flushleft}

\newpage

\section{Introduction}\label{sec:intro}

Transport in random media has been an active field of research for over fifty years but many basic questions remain mathematically very challenging unless the medium satisfies specific (and often rather restrictive) structural assumptions; see for instance \cite{zbMATH06028501, SznitmanConj,zbMATH02070282} and references therein. In this article we consider the problem of ballistic behavior in a benchmark setting that is both i) \emph{non-reversible}, and ii) \emph{non-perturbative} (in the parameters of the model). To get a sense of the difficulty these combined features entail, general results in case i) are already hard to come by in perturbative regimes, see, e.g.,~\cite{zbMATH04215135,zbMATH05033664} regarding the problem of diffusive behavior.

Our focus is on a certain random walk in dynamic random environment, which lies outside of the well-studied `classes' and has attracted increasing attention in the last decade, see~for instance~\cite{HKT20,HS15,JaraMenezes} and references~below. The interest in this model stems in no small part from the nature of the environment, which is driven by a particle system (e.g.~the exclusion process)  that is typically conservative and exhibits slow mixing. Even in the (1+1)-dimensional case, unless one restricts its parameters to perturbative regimes, the features of the model preclude the use of virtually all 
classical techniques: owing to the dynamics of the environment, the model is genuinely \emph{non-reversible}, and its properties make the search for an invariant measure of the environment as seen by the walker (in the spirit of \cite{zbMATH03951720,zbMATH06458602,zbMATH03944984}) inaccessible by current methods; see Section~\ref{sec:discussion} for a more thorough discussion of these and related matters.  

The model we study depends on a parameter $\rho$ governing the density of particles in the environment, which in turn affects the walk, whose transition probabilities depend on whether the walk sits on top of a particle (advection occurs) or not. This may (or not) induce ballistic behaviour. Our aim in this article is to prove that ballisticity is a property of the walk undergoing a `sharp transition' as $\rho$ varies, which is an inherently \emph{non-perturbative} result. Our results answer a number of conjectures of the past fifteen years, and stand in stark contrast with (i.i.d.) static environments~\cite{KKS, Sinai}, where the zero-speed regime is typically extended. The sharpness terminology is borrowed from critical phenomena, and the analogy runs deep, as will become apparent. Drawing inspiration from it is one of the cornerstones of the present work.

\subsection{Main results}
\label{sec:mainresults}

We present our main results with minimal formalism in the model case where the environment is the (simple) symmetric exclusion process, abbreviated (S)SEP in the sequel, and refer to Sections~\ref{sec:prelims} and~\ref{sec:SEP} for full definitions. We will in fact prove more general versions of these results, see Theorem~\ref{T:generic}, which hold under rather broad assumptions on the environment, satisfied for instance by the SEP. Another environment of interest that fulfils these conditions is discussed in  Appendix~\ref{sec:PCRW}, and is built using a Poisson system of particles performing independent simple random walks (PCRW, short for Poisson Cloud of Random Walks).

The SEP is the continuous-time Markov process $\eta=(\eta_t)_{t\geq 0}$ taking values in $\{0,1\}^\Z$ and describing the simultaneous evolution of continuous-time simple random walks subject to the exclusion rule. That is, for a given realization of $\eta_0$, one puts a particle at those sites $x$ such that $\eta_0(x)=1$. Each particle then attempts to move at a given rate $\nu >0$, independently of its previous moves and of the other particles, to one of its two neighbours chosen with equal probability. The move happens if and only if the target site is currently unoccupied; see Section~\ref{subsec:SEP-basic} for precise definitions. For the purposes of this article, one could simply set $\nu=1$. We have kept the dependence on $\nu$ explicit in anticipation of future applications, for which one may wish to slow down/speed up the environment over time.

On top of the SEP $\eta$, a random walk $X=(X_n)_{n \geq 0}$ starting~from $X_0=0$ moves randomly as follows. Fixing two parameters $p_{\bullet},p_\circ\in (0,1)$ such that $p_\bullet >p_\circ$, the process $X$ moves at integer times to a neighbouring site, the right site being chosen with probability $p_\bullet$, resp.~$p_\circ$, depending on whether $X$ is currently located on an occupied or empty site. Formally, given a realization $\eta$ of the SEP, and for all $n\geq 1$, 
\begin{equation}
\label{eq:rwinformal}
X_{n+1}-X_n= \begin{cases}
+1, & \text{ with prob.~$p_{\bullet}\mathbf{1}_{\eta_n(X_n)=1}+p_{\circ}\mathbf{1}_{\eta_n(X_n)=0}$}\\
-1, & \text{ with prob.~$(1-p_{\bullet})\mathbf{1}_{\eta_n(X_n)=1}+(1-p_{\circ})\mathbf{1}_{\eta_n(X_n)=1}$}
\end{cases}
\end{equation}
(notice that the transition probabilities depend on $\eta$ since $p_{\circ} \neq p_{\bullet}$ by assumption). For $\rho \in (0,1)$ we denote by $\P^\rho$ the \emph{annealed} (i.e.~averaged over all sources of randomness; see Section~\ref{subsec:rw} for details) law of $(\eta,X)$
whereby $\eta_0$ is sampled according to a product Bernoulli measure $((1-\rho)\delta_0 + \rho \delta_1)^{\otimes \Z}$, with $\delta_x$ a Dirac measure at $x$, which is an invariant measure for the SEP, see Lemma~\ref{L:SEP-P}. Incidentally, one can check that except for the trivial cases $\rho\in \{0,1\}$, none of the invariant measures for the SEP (cf.~\cite[Chap.~III, Cor.~1.11]{MR0776231}) is invariant for the environment viewed from the walker (to see this, one can for instance consider the probability that the origin and one of its neighbours are both occupied). \color{black}  In writing $\P^\rho$, we leave the dependence on $\nu,p_\bullet$ and $p_\circ$ implicit, which are regarded as fixed, and we will focus on the dependence of quantities on $\rho$, the main parameter of interest. For a mental picture, the reader is invited to think of $X$ as evolving in a half-plane, with space being horizontal and time running upwards. All figures below will follow this convention.

We now describe our main results, using a language that will make analogies to critical phenomena apparent. Our first main result is simplest to formulate in terms of the \emph{fast-tracking} probability, defined for $n\ge0$ and $\rho \in (0,1)$ as
\begin{equation}\label{eq:Bn}
\theta_n(\rho)= \P^{\rho}( H_n < H_{-1} ),
\end{equation}
where $H_k= \inf\{ n \geq 0: X_n =k \}$. In words $\theta_n(\rho)$ is the probability that $X$ visits $n$ before $-1$. Analogous results can be proved for a corresponding \emph{back-tracking} probability, involving the event $\{H_{-n}< H_1 \}$ instead; see the end of Section~\ref{sec:mainresults} for more on this. The events in \eqref{eq:Bn} being decreasing in $n$, the following limit
\begin{equation}\label{eq:order-parameter}
\theta(\rho)\stackrel{\text{def.}}{=} \lim_n \theta_n(\rho)
\end{equation}
 is well-defined and constitutes an \emph{order parameter} (in the parlance of statistical physics) for the
model. Indeed, it is not difficult to see from \eqref{eq:rwinformal}
and monotonicity properties of the environment (see  \eqref{pe:monotonicity} in Section~\ref{subsec:re} and Lemma~\ref{L:monotone}), that the functions $\theta_n(\cdot)$, $n \geq 0$, and hence $\theta(\cdot)$ are non-decreasing, i.e.~that $\theta(\rho) \leq \theta(\rho')$ whenever $\rho \leq \rho'$. One thus naturally associates to the function $\theta(\cdot)$ a corresponding critical threshold
\begin{equation}\label{eq:rho_c}
\rho_c \stackrel{\text{def.}}{=} \sup\{ \rho \in [0,1]: \theta(\rho)=0 \}.
\end{equation}
The transition from the subcritical ($\rho < \rho_c$) to the supercritical ($\rho > \rho_c$) regime corresponds to the onset of a phase where the walk has a positive chance to escape to the right. In fact it will do so at linear speed, as will follow from our second main result. To begin with, in view of \eqref{eq:rho_c} one naturally aims to quantify the behavior of $\theta_n(\rho)$ for large $n$. Our first theorem exhibits a sharp transition for its decay.

\begin{The}\label{sharpnesscor}
For all $\rho\in(0,1)$, there exist constants $c_1,c_2 \in (0,\infty)$ depending on $\rho$ such that, with $\rho_c$ as in \eqref{eq:rho_c} and for all $n\ge1$,
\begin{enumerate}
\item[(i)] $\theta_n(\rho) \le c_1 \exp(-(\log n)^{3/2})$, if $\rho<\rho_c$;
\item[(ii)] $\theta_n(\rho) \ge c_2$, if $\rho>\rho_c$.
\end{enumerate}
\end{The}

The proof of Theorem~\ref{sharpnesscor} appears at the end of Section~\ref{subsec:main}.
Since $\theta = \inf_n \theta_n$, item \textit{(ii)} is an immediate consequence of \eqref{eq:rho_c}. The crux of Theorem~\ref{sharpnesscor} is therefore to exhibit the rapid decay of item \textit{(i)} in the \emph{full} subcritical regime, and not just perturbatively in $\rho \ll 1$, which is the status quo, cf.~Section~\ref{sec:discussion}. In the context of critical phenomena, this is the famed question of (subcritical) sharpness, see for instance \cite{Me86,AB87,DumTas15, Va23} for sample results of this kind in the context of Ising and Bernoulli percolation models. Plausibly, the true order of decay in item~\textit{(i)} is in fact exponential in $n$ (for instance, a more intricate renormalisation following \cite{HKT20} should already provide a stretched-exponential bound in $n$). We will not delve further into this question in the present article.

Our approach to proving Theorem~\ref{sharpnesscor} is loosely inspired by the recent sharpness results \cite{SharpnessGFF} and \cite{RI-I, RI-II,RI-III} concerning percolation of the Gaussian free field and the vacant set of random interlacements, respectively, which both exhibit long-range dependence (somewhat akin to the SEP). Our model is nonetheless very different, and our proof strategy, outlined below in Section~\ref{sec:proofsketch}, vastly differs from these works. In particular, it does not rely on differential formulas (in $\rho$), nor does it involve the OSSS inequality or sharp threshold techniques, which have all proved useful in the context of percolation, see e.g.~\cite{DumRaoTas17b, zbMATH06406371, BefDum12, zbMATH05064645}. One similarity with \cite{RI-I}, and to some extent also with \cite{Va23} (in the present context though, stochastic domination results may well be too much to ask for), is our extensive use of couplings. Developing these couplings represents one of the most challenging technical aspects of our work; we return to this in Section~\ref{sec:proofsketch}. It is also interesting to note that, as with statistical physics models, the regime $\rho \approx \rho_c$ near and at the critical density encompasses a host of very natural (and mostly open) questions that seem difficult to answer and point towards interesting phenomena; see Section \ref{sec:discussion} and Corollary~\ref{Cor:cont} for more on this.

\bigskip

Our second main result concerns the asymptotic speed of the random walk, and addresses an open problem of \cite{HKT20} which inspired our work. In \cite{HKT20}, the authors prove the existence of a deterministic non-decreasing function $v:(0,1) \to \R$ such that, for all $\rho \in (0,1)\setminus \{\rho_-, \rho_+ \}$,
 \begin{equation}
 \label{eq:intro-LLN}
\P^{\rho}\text{-a.s.,~}\frac{X_n}{n} \to v(\rho) \text{ as } n \to \infty,
 \end{equation}
where 
\begin{equation}
\label{eq:rho-pm}
\begin{split}
&\rho_{-}\stackrel{\text{def.}}{=} \sup \{ \rho : v(\rho) < 0 \},\\
&\rho_{+}\stackrel{\text{def.}}{=} \inf \{ \rho : v(\rho) > 0 \},
\end{split}
\end{equation}
leaving open whether $\rho_+>\rho_-$ or not, and with it the possible existence of an extended (critical) zero-speed regime. Our second main result yields the strict monotonicity of $v(\cdot)$ and provides the answer to this question. Recall that our results hold for any choice of the parameters $0<p_\circ <p_\bullet <1$ and $\nu>0$.

\begin{The} \label{T:main1} 
With $\rho_c$ as defined in \eqref{eq:rho_c}, one has
\begin{equation}
\label{eq:equality}
\rho_- = \rho_+ = \rho_c.
\end{equation}
Moreover, for every $\rho,\rho'\in (0,1)$ such that $\rho >\rho'$, one has
\begin{equation} \label{eq:monotonic-intro}
v(\rho)>v(\rho').
\end{equation}
\end{The}
The proof of Theorem~\ref{T:main1} is given in Section~\ref{subsec:main} below. It will follow from a more general result, Theorem~\ref{T:generic}, which applies to a class of environments $\eta$ satisfying certain natural conditions. These will be shown to hold when $\eta$ is the SEP. 

Let us now briefly relate Theorems~\ref{sharpnesscor} and~\ref{T:main1}. Loosely speaking, Theorem~\ref{sharpnesscor} indicates that $v(\rho) >0$ whenever $\rho> \rho_c$, whence $\rho_+ \leq \rho_c$ in view of \eqref{eq:rho-pm}, which is morally half of Theorem~\ref{T:main1}. One can in fact derive a result akin to Theorem~\ref{sharpnesscor}, but concerning the order parameter $\hat{\theta}(\rho)= \lim_n \P^{\rho}( H_{-n} < H_{1} )$ associated to the \emph{back-tracking} probability instead. Defining $\hat{\rho}_c $ in the same way as \eqref{eq:rho_c} but with $\hat{\theta}$ in place of $\theta$, our results imply that $\hat{\rho}_c=\rho_c$ and that the transition for the back-tracking probability has similar sharpness features as in Theorem~\ref{sharpnesscor}, but in opposite directions --
 the sub-critical regime is now for $\rho>\hat{\rho}_c(=\rho_c)$. Intuitively, this corresponds to the other inequality $\rho_- \geq \rho_c$, which together with $\rho_+ \leq \rho_c$, yields \eqref{eq:equality}. 
 
Finally, Theorem~\ref{T:main1} implies that $v(\rho)\neq 0$ whenever $\rho\neq \rho_c$. Determining whether a law of large number holds at $\rho=\rho_c$ holds or not, let alone whether the limiting speed $v(\rho_c)$ vanishes when $0< \rho_c< 1$, is in general a difficult question, to which we hope to return elsewhere; we discuss this and related matters in more detail below in Section~\ref{sec:discussion}. 

One noticeable exception occurs in the presence of additional symmetry, as we now explain. In \cite{HKT20} the authors could prove that, at the `self-dual' point $p_{\bullet}=1-p_\circ$ (for any given value of $p_{\bullet} \in (1/2,1)$), one has $v(1/2)=0$ and the law of large numbers holds with limit speed $0$, but they could not prove that $v$ was non-zero for $\rho\neq 1/2$, or equivalently that $\rho_+=\rho_-=1/2$. Notice that the value $p_{\bullet}=1-p_\circ$ is special, since in this case (cf.~\eqref{eq:rwinformal}) $X$ has the same law under $\P^{\rho}$ as $-X$ under $\P^{1-\rho}$, and in particular $X\stackrel{\text{law}}{=}-X$ when $\rho=\frac12$. In the result below, which is an easy consequence of Theorem \ref{T:main1} and the results of \cite{HKT20}, we summarize the situation in the symmetric case. We include the (short) proof here. The following statement is of course reminiscent of a celebrated result of Kesten concerning percolation on the square lattice \cite{zbMATH03688370}; see also \cite{zbMATH05064645,BefDum12,10.1214/19-AIHP1006} in related contexts.

\begin{Cor}\label{Cor:cont}
If $p_\circ=1-p_\bullet$, then 
\begin{equation}\label{eq:cor-rho12}
\rho_c=\tfrac12 \, \text{ and } \, \theta(\tfrac12)=0.
\end{equation}
Moreover, under $\P^{1/2}$, $X$ is recurrent, and the law of large numbers \eqref{eq:intro-LLN} holds with vanishing limiting speed $v(1/2)=0$.
\end{Cor}

\begin{proof}
From the proof of \cite[Theorem 2.2]{HKT20} (and display (3.26) therein), one knows that $v(1/2)=0$ and that the law of large numbers holds at $\rho=\frac12$. By Theorem~\ref{T:main1}, see \eqref{eq:rho-pm} and \eqref{eq:equality}, $v$ is negative on $(0,\rho_c)$ and positive on $(\rho_c,1)$, which implies that $\rho_c=1/2$. As for the recurrence of $X$ under $\P^{1/2}$, it follows from the ergodic argument of~\cite{OrenshteindosSantos} (Corollary 2.2, see also Theorem~3.2 therein; these results are stated for a random walk jumping at exponential times but are easily adapted to our case). Since $\theta(\rho) \leq \P^{\rho}(\liminf X_n \geq 0)$ in view of \eqref{eq:order-parameter}, recurrence implies that $\theta(1/2)=0$, and \eqref{eq:cor-rho12} follows. 
\end{proof}

\subsection{Discussion}\label{sec:discussion}

We now place the above results in broader context and contrast our findings with existing results. We then discuss a few open questions in relation with Theorems~\ref{sharpnesscor} and~\ref{T:main1}.

\subsubsection{Related works}

Random walks in dynamic random environments have attracted increasing attention in the past two decades, both in statistical physics - as a way to model a particle advected by a fluid~\cite{PhysiqueKardar, PhysiqueBarmaetc, PhysiqueBasuMaes,PhysiqueHuveneers}, and in probability theory - where they provide a counterpart to the more classically studied random walks in static environments~\cite{AvenaHollanderRedigLLNcone,BlondelHilarioTeixeira,HKS-RWhighdensity,HKT20,JaraMenezes,RedigVollering}.

On $\mathbb{Z}$, the static setting already features a remarkably rich phenomenology, such as transience with zero speed~\cite{Solomon}, owing to traps that delay the random walk, and anomalous fluctuations~\cite{KKS} - which can be as small as polylogarithmic in the recurrent case~\cite{Sinai}. On $\mathbb{Z}^d$ for $d\geq 2$, some fundamental questions are still open in spite of decades of efforts, for instance the conjectures around Sznitman's Condition~(T) and related effective ballisticity criteria~\cite{SznitmanConj,zbMATH06370210}, and the possibility that, for a uniformly elliptic environment, directional transience is equivalent to ballisticity. If true, these conjectures would have profound structural consequences, in essentially ascertaining that the above trapping phenomena can only be witnessed in dimension one. As with the model studied in this article, a circumstance that seriously hinders progress is the truly non-reversible character of these problems. This severely limits the tools available to tackle them.

New challenges arise when the environment is dynamic, requiring new techniques to handle the fact that correlations between transition probabilities are affected by time. In particular, the trapping mechanisms identified in the static case do not hold anymore and it is difficult to understand if they simply dissolve or if they are replaced by different, possibly more complex, trapping mechanisms. As of today, how the static and dynamic worlds relate is still far from  understood. Under some specific conditions however, laws of large numbers (and sometimes central limit theorems) have been proved in dynamic contexts: when the environment has sufficiently good mixing properties~\cite{AvenaHollanderRedigLLNcone, BlondelHilarioTeixeira, RedigVollering}, a spectral gap~\cite{AvenaBlondelFaggionato}, or when one can show the existence of an invariant measure for the environment as seen by the walk (\cite{BethuelsenVollering}, again under some specific mixing conditions). One notable instance of such an environment is the supercritical contact process (see \cite{HollanderSantos, MountfordVares}, or~\cite{Allasia} for a recent generalization).
We also refer to \cite{zbMATH06865124, zbMATH06507481, zbMATH06824408, zbMATH07334604,deuschel2023gradientestimatesheatkernel} for recent results in the time-dependent reversible case, for which the assumptions on the environment can be substantially weakened.

The environments we consider are archetypal examples that do not satisfy any of the above conditions. In the model case of the environment driven by the SEP, the mixing time over a closed segment or a circle is  super-linear (even slightly super-quadratic~\cite{MorrisSSEPmixing}), creating a number of difficulties, e.g.~barring the option of directly building a renewal structure without having to make further assumptions (see~\cite{AvenadosSantosVollering} in the simpler non-nestling case). The fact that the systems may be conservative (as is the case for the SEP) further hampers their mixing properties. As such, they have been the subject of much attention in the past decade, see~\cite{Avenaal3TransientSameBias,Avenaal1HydroLim, Avenaal2LDP,BlondelHilarioTeixeira, dosSantosBoundsSpeed,HKT20,HHSST15,HS15,JaraMenezes, MenezesPetersonXie} among many others. All of these results are of two types: either they require particular assumptions, or they apply to some perturbative regime of parameters, e.g.~high density of particles~\cite{RWonRWshighdim, HKS-RWhighdensity}, strong drift, or high/low activity rate of the environment~\cite{HS15,SalviSimenhaus}.

Recently, building on this series of work, a relatively comprehensive result on the Law of Large Numbers (LLN) in dimension 1+1 has been proved in~\cite{HKT20}, see \eqref{eq:intro-LLN} above, opening an avenue to some fundamental questions that had previously remained out of reach, such as the possible existence of transient regime with zero speed. Indeed, cf.~\eqref{eq:intro-LLN} and \eqref{eq:rho-pm}, the results of \cite{HKT20} leave open the possibility of having an extended `critical' interval of $\rho$'s for which $v(\rho)=0$, which is now precluded as part of our main results, see~\eqref{eq:equality}. We seek the opportunity to stress that our strategy, outlined below in Section~\ref{sec:proofsketch}, is completely new and  rather robust, and we believe a similar approach will lead to progress on related questions for other models.

\subsubsection{Open questions}$\quad$

\bigskip

\textbf{1) Existence and value of $v(\rho_c)$.} Returning to Theorems~\ref{sharpnesscor} and~\ref{T:main1}, let us start by mentioning that $\rho_c$, defined by \eqref{eq:rho_c} and equivalently characterized via \eqref{eq:rho-pm}-\eqref{eq:equality}, may in fact be degenerate (i.e.~equal to $0$ or $1$) if $v(\cdot)$ stays of constant sign. It is plausible that $\rho_c\in(0,1)$ if and only if
$p_\circ<1/2<p_\bullet$, which corresponds to the (more challenging) \emph{nestling} case, in which 
the walk has a drift of opposite signs depending on whether it sits on top of a particle or an empty site. Moreover, the function $v(\cdot)$, as given by \cite[Theorem 3.4]{HKT20}, is well-defined for all values of $\rho\in(0,1)$ (this includes a candidate velocity at $\rho_c$), but whether a LLN holds at $\rho_c$ with speed $v(\rho_c)$ or the value of the latter is unknown in general, except in cases where one knows that $v(\rho_c)=0$ by other means (e.g.~symmetry), in which case a LLN can be proved, see \cite[Theorem 3.5]{HKT20}. When $\rho_c$ is non-degenerate, given Theorem~\ref{T:main1}, it is natural to expect that $v(\rho_c)=0$, but this is not obvious as we do not presently know if $v$ is continuous at $\rho_c$. It is relatively easy to believe that continuity on $(0,1)\setminus\{\rho_c\}$, when the speed is non-zero, can be obtained through an adaptation of the regeneration structure defined in \cite{HS15}, but continuity at $\rho_c$ (even in the symmetric case) seems to be more challenging.

Suppose now that one can prove that a law of large numbers holds with vanishing speed at $\rho=\rho_c$, then one can ask whether $X$ is recurrent or transient at $\rho_c$. A case in point where one knows the answer is the `self-dual' point $p_\circ=1-p_\bullet$ where the critical density equals $1/2$ and the walk is recurrent. In particular, this result and Theorem \ref{T:main1} imply that $v(\rho)$ is zero if and only if $\rho=1/2$, and thus there exists no transient regime with zero-speed in the symmetric case. This also answers positively the conjecture in~\cite{AvenaHollanderRedig} (end of Section 1.4), which states that in the symmetric case, the only density with zero speed is in fact recurrent. \\

\textbf{2) Regularity of $v(\cdot)$ near $\rho_c$ and fluctuations at $\rho_c$.}
In cases where $v(\rho_c)=0$ is proved, one may further wonder about the regularity of $\rho \mapsto v(\rho)$ around $\rho_c$ (and similarly of $\theta(\rho)$ as $\rho \downarrow \rho_c$). Is it continuous, and if yes, is it H\"{o}lder-continous, or even differentiable? This could be linked to the fluctuations of the random walk when $\rho=\rho_c$, in the spirit of Einstein's relation, see for instance \cite{zbMATH06786088} in the context of reversible dynamics. Whether the fluctuations of $X$ under $\P^{\rho_c}$ are actually diffusive, super-diffusive or sub-diffusive is a particularly difficult question. It has so far been the subject of various conjectures, both for the RWdRE (e.g.~Conjecture 3.5 in~\cite{AvenaThomann}) and very closely related models in statistical physics (e.g.~\cite{PhysiqueHuveneers}, \cite{Gopalakrishnan}, \cite{PhysiqueNagarAl}). One aspect making predictions especially difficult is that the answer may well depend on all parameters involved, i.e.~$\rho$, $p_\circ$ and $p_\bullet$. At present, any rigorous upper or lower bound on the fluctuations at $\rho=\rho_c$ would be a significant advance.\\

\textbf{3) Comparison with the static setting.}
In the past decade, there have been many questions as to which features of a static environment (when $\nu=0$, so that the particles do not move) are common to the dynamic environment, and which are different. In particular, it is well-known that in the static case, there exists a non-trivial interval of densities for which the random walker has zero speed, due to mesoscopic traps that the random walker has to cross on its way to infinity (see for instance~\cite[Example 1]{Peterson}, and above references). It was conjectured that if $\nu>0$ is small enough, there could be such an interval in the dynamic set-up, see~\cite[(3.8)]{AvenaThomann}. Theorem~\ref{T:main1} thus disproves this conjecture. 

Combined with the CLT from~\cite[Theorem 2.1]{HKT20}, which is valid outside of $(\rho_-,\rho_+)$, hence for all $\rho\neq \rho_c$ by Theorem~\ref{T:main1}, this also rules out the possibility of non-diffusive fluctuations for any $\rho\neq\rho_c$, which were conjectured in~\cite[Section 3.5]{AvenaThomann}.  In light of this, one may naturally wonder how much one has to slow down the environment with time (one would have to choose $\nu=\nu_t$ a suitably decreasing function of the time $t$) in order to start seeing effects from the static world. We plan to investigate this in future work.

\subsection{Overview of the proof} \label{sec:proofsketch}

We give here a relatively thorough overview of the proof, intended to help the reader navigate the upcoming sections. For concreteness, we focus on Theorem~\ref{T:main1} (see below its statement as to how Theorem~\ref{sharpnesscor} relates to it), and specifically on the equalities \eqref{eq:equality}, in the case of SEP (although our results are more general; see Section~\ref{sec:proofshort}). The assertion \eqref{eq:equality} is somewhat reminiscent of the chain of equalities $\bar{u}=u_*=u_{**}$ associated to the phase transition of random interlacements that was recently proved in \cite{RI-II,RI-III, RI-I}, and draws loose inspiration from the interpolation technique of~\cite{AG91}, see also \cite{PhasetransGFF}. Similarly to \cite{RI-III} in the context of interlacements, and as is seemingly often the case in situations lacking key structural features (in the present case, a notion of reversibility or more generally self-adjointness, cf.~\cite{zbMATH04215135,zbMATH05033664}), our methods rely extensively on the use of \emph{couplings}. 

\bigskip

The proof essentially consists of two parts, which we detail individually below. In the first part, we compare our model with a \textit{finite-range} version, in which we fully re-sample the environment at times multiple of some large integer parameter $L$ (see e.g.~\cite{RI-III} for a similar truncation to tame the long-range dependence). One obtains easily that for any $L$ and any density $\rho \in (0,1)$, the random walk in this environment satisfies a strong Law of Large Numbers with some speed $v_L(\rho)$ (Lemma~\ref{Prop:vLexists}), and we show that in fact, 
\begin{equation}\label{eq:introvLtov}
v_L(\rho-\varepsilon_L)-\delta_L\le v(\rho)\le v_L(\rho+\varepsilon_L) +\delta_L,
\end{equation}
 for some $\delta_L, \varepsilon_L $ that are quantitative and satisfy $\delta_L, \varepsilon_L  =o_L(1)$ as $L \to \infty$ (see Proposition~\ref{Prop:vLapprox}).

In the second part, we show that for any fixed $\rho,\rho+\varepsilon \in (0,1)$, we have 
\begin{equation}\label{eq:introvLquant}
v_L(\rho+\varepsilon) > v_L(\rho)+3\delta_L,
\end{equation}
 for $L$ large enough (Proposition~\ref{Prop:initialspeed}). Together, \eqref{eq:introvLtov} and \eqref{eq:introvLquant} readily imply that $v(\rho+\varepsilon)>v(\rho)$, and \eqref{eq:equality} follows. Let us now give more details. 
\\
\\
\textbf{First part: from infinite to finite range.} For a density $\rho\in (0,1)$ the environment $\eta$ of the range-$L$ model is the SEP starting from $\eta_0\sim \text{Ber}(\rho)^{\otimes \Z}$ during the time interval $[0,L)$. Then for every integer $k \geq1$, at time $kL$ we sample $\eta_{kL}$ again as $\text{Ber}(\rho)^{\otimes \Z}$, independently from the past, and let it evolve as the SEP during the interval $[kL,(k+1)L)$. The random walk $X$ on this environment is still defined formally as in \eqref{eq:rwinformal}, and we denote $\P^{\rho,L}$ the associated annealed probability measure. Clearly, the increments $(X_{kL}-X_{(k-1)L})_{k\geq 1}$ are i.i.d.~(and bounded), so that by the strong LLN, $X_n/n\rightarrow v_L(\rho) :=\mathbb{E}^\rho[X_L/L]$, a.s.~as $n\rightarrow \infty$. 

Now, to relate this to our original 'infinite-range' ($L=\infty$) model, the main idea is to compare the range-$L$ with the range-$2L$, and more generally the range-$2^kL$ with the range-$2^{k+1}L$ model for all $k \geq 0$ via successive couplings. The key is to prove a chain of inequalities of the kind
\begin{equation}\label{eq:chainsprinkling}
v_{L}(\rho)\leq v_{2L}(\rho+\varepsilon_{1,L})+\delta_{1,L} \leq \ldots \leq  v_{2^kL}(\rho+\varepsilon_{k,L})+\delta_{k,L}\leq \ldots,
\end{equation}
where the sequences $(\varepsilon_{k,L})_{k\geq 1}$ of and $(\delta_{k,L})_{k\geq 1}$ are increasing, with $\varepsilon_L:=\lim_{k\rightarrow \infty}\varepsilon_{k,L}=o_L(1)$ and $\delta_L:=\lim_{k\rightarrow \infty}\delta_{k,L}=o_L(1)$. In other words, we manage to pass from a scale $2^kL$ to a larger scale $2^{k+1}L$, at the expense of losing a bit of speed $(\delta_{k+1,L}-\delta_{k,L})$, and using a little sprinkling in the density $(\varepsilon_{k,L}-\varepsilon_{k-1,L})$, with $\delta_{0,L}= \varepsilon_{0,L}=0$. From~\eqref{eq:chainsprinkling} (and a converse inequality proved in a similar way), standard arguments allow us to deduce~\eqref{eq:introvLtov}. 
Let us mention that such a renormalization scheme, which trades scaling against sprinkling in one or several parameters, is by now a standard tool in percolation theory, see for instance~\cite{Sznitmaninterlac, MR3053773, PhasetransGFF}.

We now describe how we couple the range-$L$ and the range-$2L$ models in order to obtain the first inequality, the other couplings being identical up to a scaling factor. We do this in detail in Lemma~\ref{Lem:L2L}, where for technical purposes, we need in fact more refined estimates than only the first moment, but they follow from this same coupling. The coupling works roughly as follows (cf.~Fig.~\ref{f:L2L_intro}). We consider two random walks $X^{(1)}\stackrel{\text{law}}{=} \P^{\rho,L}$ and $X^{(2)}\stackrel{\text{law}}{=}  \P^{\rho+\varepsilon_L,2L}$ coupled via their respective environments $\eta^{(1)}$ and $\eta^{(2)}$ such that we retain good control (in a sense explained below) on the relative positions of their respective particles during the time interval $[0,2L]$, except during a short time interval after time $L$ after renewing the particles of $\eta^{(1)}$. In some sense, we want to dominate $\eta^{(1)}$ by $\eta^{(2)}$ `as much as possible', and use it in combination with the following monotonicity property of the walk:  since $p_\bullet >p_\circ$, $X^{(2)}$ cannot be overtaken by $X^{(1)}$ if $\eta^{(2)}$ covers $\eta^{(1)}$ (cf.~\eqref{eq:rwinformal}; see also Lemma~\ref{L:monotone} for a precise formulation). Roughly speaking, we will use this strategy from time $0$ to $L$, then lose control for short time before recovering a domination and using the same argument but with a lateral shift (in space), as illustrated in Figure~\ref{f:L2L_intro}.
The key for recovering a domination in as little time as possible is a property of the following flavour. 
\begin{equation}
\label{eq:coupling-intro}
\begin{cases}
&\text{\begin{minipage}{0.8\textwidth}
Let $L,t,\ell \geq 1$ be such that $L\gg t^2\gg\ell^4$, and let $\eta_0, \eta'_0 \in\{0,1\}^\mathbb{Z}$ be such that over $[-3L,3L]$, $\eta_0$ (resp.~$\eta'_0$) has empirical density $\rho+\varepsilon$ (resp.~$\rho$) on segments of length $\ell$. Then one can couple the time evolutions of $\eta,\eta'$ as SEPs during $[0,t]$ such that 
\[
\P\big(\eta_t(x)\geq \eta_t'(x),\, \forall x \in [-L+ct,L-ct]\big) \geq 1-c'tL\exp(-c''\varepsilon^2 t^{1/4}),  
\]
for some constants $c,c',c''>0$. 
\end{minipage}}
\end{cases}
\end{equation}
Roughly, for $t$ poly-logarithmic in $L$ the right-hand side of~\eqref{eq:coupling-intro} will be as close to 1 as we need. 
We do not give precise requirements on the scales $\ell,t,L$ (nor a definition of empirical density), except that they must satisfy some minimal power ratio due to the diffusivity of the particles. We will formalize this statement in Section~\ref{subsec:C}, see in particular condition~\ref{pe:couplings}, with exponents and constants that are likely not optimal but sufficient for our purposes. Even if the environment were to mix in super-quadratic (but still polynomial) time, our methods would still apply, up to changing the exponents in our conditions~\ref{pe:densitychange}-\ref{pe:nacelle}.

The only statistic we control is the empirical density, i.e.~the number of particles over intervals of a given length. The relevant control is formalised in Condition~\ref{pe:densitychange}. In particular, our environment couplings have to hold with a quenched initial data, and previous annealed couplings in the literature (see e.g.~\cite{BaldassoTeixeira18}) are not sufficient for this purpose (even after applying the usual annealed-to-quenched tricks), essentially due to the difficulty of controlling the environment around the walker in the second part of the proof, as we explain in Remark~\ref{Rk:quenchedconditions}; cf.~also \cite{PopovTeixeira, zbMATH06932732,PRS23} for related issues in other contexts.
\begin{figure}[]
  \center
\includegraphics[scale=0.8]{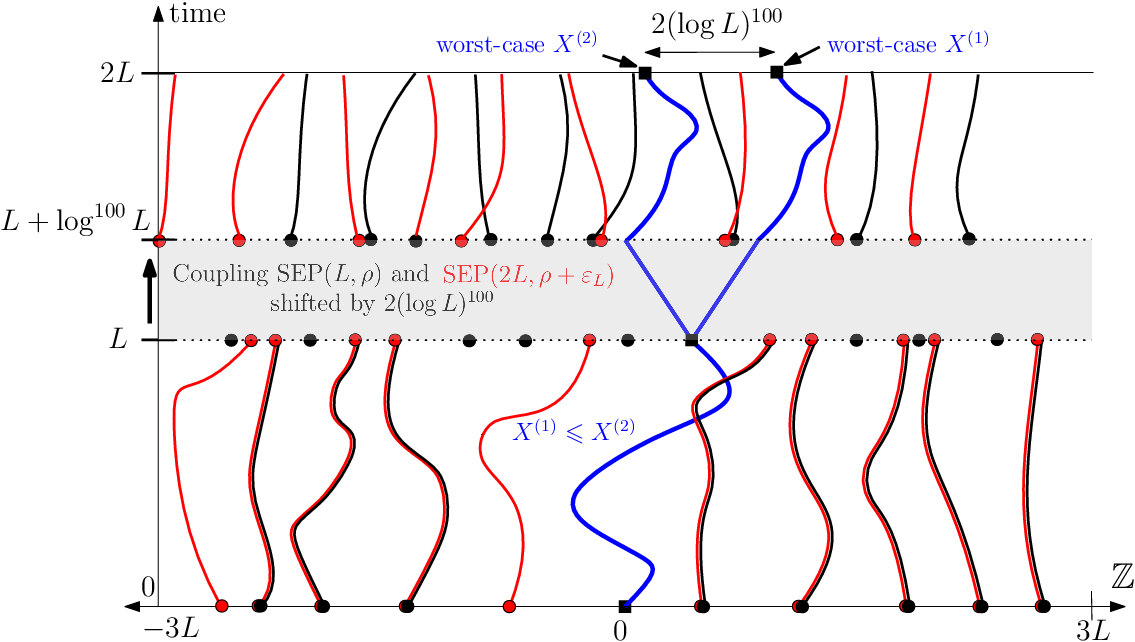}
  \caption{Coupling of $X^{(1)}\stackrel{\text{law}}{=}  \P^{\rho,L}$ and $X^{(2)} \stackrel{\text{law}}{=} \P^{\rho+\varepsilon_L,2L}$. The particles of $\eta^{(1)}$ and their trajectories are pictured in black, those of $\eta^{(2)}$ in red. In blue, from time $0$ to $L$, is the trajectory of $X^{(1)}$, which is a lower bound for that of $X^{(2)}$. From time $L$ to $2L$, we split it between the rightmost (resp.~leftmost) possible realization of $X^{(1)}$ (resp.~$X^{(2)}$). These correspond to worst-case scenarios. Owing to the efficiency of the coupling (grey band), the discrepancy incurred at time $2L$ remains controlled with high probability, but there is a barrier to how small this gap can be made, which is super-linear in $\log L$.
  }
  \label{f:L2L_intro}
\end{figure}

Let us now track the coupling a bit more precisely to witness the quantitative speed loss $\delta_L$ incurred by \eqref{eq:coupling-intro}. We start by sampling $\eta^{(1)}_0$ and $\eta^{(2)}_0$ such that a.s., $\eta^{(1)}_0(x)\leq \eta^{(2)}_0(x)$ for all $x\in \mathbb{Z}$, which is possible by stochastic domination. In plain terms, wherever there is a particle of $\eta^{(1)}$, there is one of $\eta^{(2)}$. Then, we can let $\eta^{(1)}$ and $\eta^{(2)}$ evolve during $[0,L]$ such that this domination is deterministically preserved over time (this feature is common to numerous particle systems,  including both SSEP and the PCRW). As mentioned above, conditionally on such a realization of the environments, one can then couple $X^{(1)}$ and $X^{(2)}$ such that $X^{(2)}_s\geq X^{(1)}_s$ for all $s\in [0,L]$. The intuition for this is that 'at worst', $X^{(1)}$ and $X^{(2)}$ sit on the same spot, where $X^{(2)}$ might see a particle and $X^{(1)}$ an empty site (hence giving a chance for $X^{(2)}$ to jump to the right, and for $X^{(1)}$ to the left), but not the other way around (recall to this effect that $p_\bullet >p_\circ$).

At time $L$, we have to renew the particles of $\eta^{(1)}$, hence losing track of the domination of $\eta^{(1)}$ by $\eta^{(2)}$. We intend to recover this domination within a time $t:=(\log L)^{100}$ by coupling the particles of $\eta^{(1)}_L$ with those of $\eta^{(2)}_L$ using \eqref{eq:coupling-intro} (at least over a space interval $[-3L,3L]$, that neither $X^{(1)}$ nor $X^{(2)}$ can leave during $[0,2L]$). Since during that time, $X^{(1)}$ could at worst drift of $t$ steps to the right (and $X^{(2)}$ make $t$ steps to the left), we couple de facto the shifted configurations $\eta^{(1)}(\cdot)$ with $\eta^{(2)}(\cdot -2t)$. Hence, we obtain that $\eta^{(1)}_{L+t}(x)\leq \eta^{(2)}_{L+t}(x -2t)$ for all $x\in [-3L+ct,3L+ct]$ with high probability as given by~\eqref{eq:coupling-intro}. As a result, all the particles of $\eta^{(1)}_{L+t}$ are covered by particles of $\eta^{(2)}_{L+t}$, when observed from the worst-case positions of $X^{(1)}_{L+t}$ and $X^{(2)}_{L+t},$ respectively.

Finally, during $[L+t,2L]$ we proceed similarly as we did during $[0,L]$, coupling $\eta^{(1)}$ with $\eta^{(2)}(\cdot -2t)$ so as to preserve the domination, and two random walks starting respectively from the rightmost possible position for $X^{(1)}_{L+t}$, and the leftmost possible one for $X^{(2)}_{L+t}$. The gap between $X^{(1)}$ and $X^{(2)}$ thus cannot increase, and we get that $X^{(1)}_{2L}\leq X^{(2)}_{2L}+2t$, except on a set of very small probability where~\eqref{eq:coupling-intro} fails. Dividing by $2L$ and taking expectations, we get 
\begin{equation}
v_L(\rho)\leq v_{2L}(\rho +\varepsilon)+\delta \text{, where } \delta=O\big(\tfrac{(\log L)^{100}}{L}\big),
\end{equation}
for suitable $\varepsilon > 0$. The exponent $100$ is somewhat arbitrary, but importantly, owing to the coupling time of the two processes (and the diffusivity of the SEP particles), this method could not work with a value of $\delta$ smaller than $(\log^2L)/L$. In particular, the exponent of the logarithm must be larger than one, and this prevents potentially simpler solutions for the second part. 
\\
\\
\textbf{Second part: quantitative speed increase at finite range.} We now sketch the proof of \eqref{eq:introvLquant}, which is a quantitative estimate on the monotonicity of $v_L$, equivalently stating that for $\rho\in (0,1)$ and $\varepsilon \in (0,1-\rho)$, 
\begin{equation}\label{eq:intro-expected-gain}
\mathbb{E}^{\rho+\varepsilon}[X_L]>\mathbb{E}^{\rho}[X_L]+3L\delta_L,
\end{equation}
with $\delta_L$ explicit and supplied by the first part. This is the most difficult part of the proof, as we have to show that a denser environment actually yields a \textit{positive gain} for the displacement of the random walk, and not just to limit the loss as in the first part. The main issue when adding an extra density $\varepsilon$ of particles is that whenever $X$ is on top of one such particle, and supposedly makes a step to the right instead of one to the left, it could soon after be on top of an empty site that would drive it to the left, and cancel its previous gain.

We design a strategy to get around this and preserve gaps, illustrated in Figure~\ref{f:couplingmaster_intro} (cf.~also Figure~\ref{f:One_step_gain} for the full picture, which is more involved). When coupling two walks $X^\rho\stackrel{\text{law}}{=} \P^\rho$, $X^{\rho+\varepsilon}\stackrel{\text{law}}{=}  \P^{\rho+\varepsilon}$ with respective environments $\eta^\rho, \eta^{\rho+\varepsilon}$, our strategy is to spot times when \textit{i)} $X^{\rho+\varepsilon}$ sees an extra particle (that we call \textit{sprinkler}) and jumps to the right, while $X^\rho$ sees an empty site and jumps to the left, and \textit{ii)} this gap is extended with $X^\rho$ then drifting to the left and $X^{\rho+\varepsilon}$ drifting to the right, for some time $t$ that must necessarily satisfy $t  \ll \log L$ -- so that this has a chance of happening often during the time interval $[0,L]$.

Once such a gap is created, on a time interval say $[s_1,s_2]$ with $s_2-s_1=t$, we attempt to recouple the environments $\eta^\rho, \eta^{\rho+\varepsilon}$ around the respective positions of their walker, so that the local environment seen from $X^{\rho+\varepsilon}$ again dominates the environment seen from $X^\rho$, 
which then allows us to preserve the gap previously created, up to some time $T$ which is a polynomial in $\log L$. As we will explain shortly, this gap arises with not too small a probability, so that repeating the same procedure between times $iT$ and $(i+1)T$ (for $i\leq L/T-1$) will provide us with the discrepancy  $3\delta_L$ needed.

The re-coupling of the environments mentioned in the previous paragraph must happen in time less than $O(\log L)$ (in fact less than $t$), in order to preserve the gap previously created. This leads to two difficulties:
\begin{itemize}
\item in such a short time, we cannot possibly recouple the two environments over a space interval of size comparable to $L$ ($\approx$ the space horizon the walk can explore during $[0,L]$), and
\item the coupling will have a probability $\gg L^{-1}$ to fail, hence with high probability, this will actually happen during $[0,L]$! In this case, we lose track of the domination of the environments completely, hence it could even be possible that $X^\rho$ overtakes $X^{\rho+\varepsilon}$. 
\end{itemize}
We handle the first of these difficulties with a two-step \textit{surgery coupling}: 

\begin{enumerate}
\item[1)] \textit{(Small coupling).} We perform a first coupling of $\eta^\rho$ and $\eta^{\rho+\varepsilon}$ on an interval of stretched exponential width (think $\exp(t^{1/2})$ for instance) during a time $s_3-s_2=t/2$ (cf.~the orange region in Fig.~\ref{f:couplingmaster_intro}), so as to preserve at least half of the gap created. We call this coupling ``small'' because $t$ is small (compared to $L$). If that coupling is successful, $X^{\rho+\varepsilon} $ now sees an environment that strictly covers the one seen by $X^\rho$, on a spatial interval $I$ of width $\asymp \exp(t^{1/2}) \gg t$. Due to the length of $I$, this domination extends in time, on an interval $[s_3,T]$ of duration say $t^{200}$ (it could even be extended to timescales $\asymp \exp(t^{1/2})$): it could only be broken by SEP particles of $\eta^\rho_{s_3}$ lying outside of $I$, who travel all the distance  to meet the two walkers, but there is a large deviation control on the drift of SEP particles. The space-time zone in which the environment as seen from $X^{\rho+\varepsilon} $ dominates that from $X^{\rho} $ is the green trapezoid depicted in Figure~\ref{f:couplingmaster_intro}. 

\item[2)] \textit{(Surgery coupling).} When the coupling in 1) is successful, we use the time interval $[s_3,T]$ to recouple the environments $\eta^\rho$ and $\eta^{\rho+\varepsilon}$ on the two outer sides of the trapezoid, on a width of length $10L$ say, so that $\eta^{\rho+\varepsilon}_T(\cdot +t/2)$ dominates $\eta^{\rho}_T(\cdot -t/2)$ on the entire width, all the \textit{while} preserving the fact that $\eta^{\rho+\varepsilon}(\cdot +t/2)$ dominates $\eta^{\rho}(\cdot -t/2)$ inside the trapezoid at all times (for simplicity and by monotonicity, we can suppose assume that $X^{\rho+\varepsilon}_{s_3}-X^{\rho}_{s_3}$ is exactly equal to $t$, as pictured in Figure~\ref{f:couplingmaster_intro}). This step is quite technical, and formulated precisely as Condition~\ref{pe:compatible} in Section~\ref{subsec:C}. It requires that different couplings  can be performed on disjoint contiguous intervals and glued in a `coherent' fashion, whence the name surgery coupling. Again, we will prove this specifically for the SEP in a dedicated section, see Lemma~\ref{Lem:compatible}. Since this second coupling operates on a much longer time interval, we can now ensure its success with overwhelming probability, say $1-O(L^{-100})$. 
\end{enumerate}
\begin{figure}[]
  \center
\includegraphics[scale=0.8]{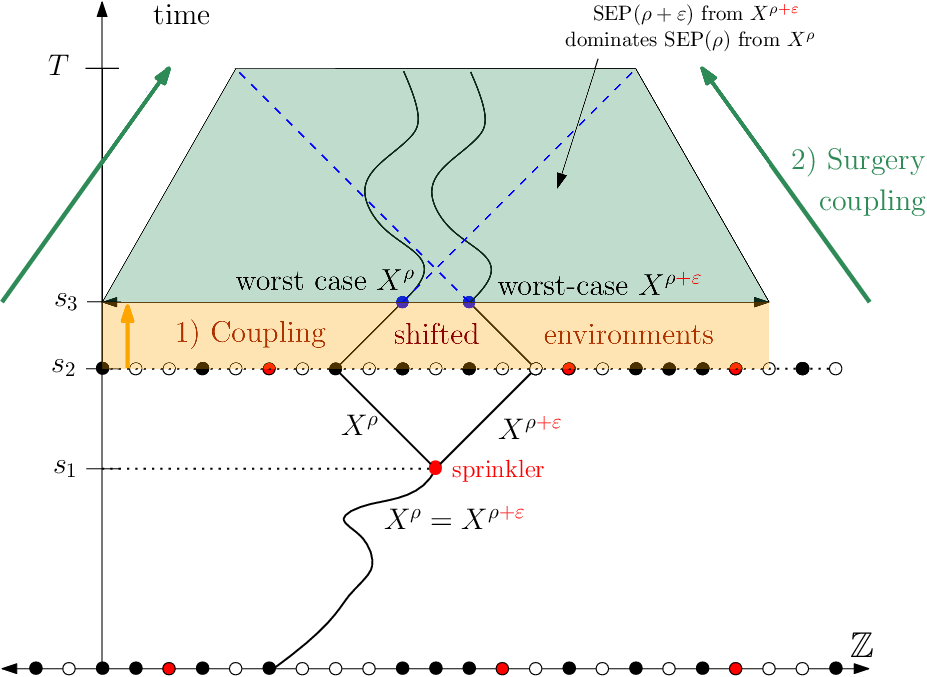}
  \caption{Coupling of $X^{\rho}\stackrel{\text{law}}{=} \P^{\rho}$ and $X^{\rho+\varepsilon}\stackrel{\text{law}}{=} \P^{\rho+\varepsilon}$ during $[0,T]$. The black curve(s) represent the minimal gap created between $X^{\rho}$ and $X^{\rho+\varepsilon}$ when all couplings succeed. The dashed blue lines represent the worst-case trajectories of the two walks when the first coupling (between $s_2$ and $s_3$) fails. In this case item 3), the parachute coupling (not depicted) comes into effect.}
  \label{f:couplingmaster_intro}
\end{figure}

We still need to address the difficulty described in the second bullet point above, which corresponds to the situation where the small coupling described at 1) fails. As argued, will happen a number of times within [0,L] because $t$ is small. If this occurs, we lose control of the domination of the environments seen by the two walkers. This leads to:

\begin{enumerate}
\item[3)] \textit{(Parachute coupling).} If the coupling in 1) is unsuccessful, instead of the surgery coupling, we perform a (parachute) coupling of $\eta^{\rho} $ and $ \eta^{\rho+\varepsilon}$  in $[s_3,T]$ so that $ \eta^{\rho+\varepsilon}(\cdot -(T-s_3))$ covers $ \eta^{\rho}(\cdot +(T-s_3))$, as we did in the first part of the proof. This again happens with extremely good probability $1-O(L^{-100})$. The price to pay is that on this event, the two walks may have drifted linearly in the wrong direction during $[s_3,T]$ (see the dashed blue trajectories in Figure~\ref{f:couplingmaster_intro}), but we ensure at least that $X^{\rho+\varepsilon}_{T}-X^{\rho}_{T}\geq -2(T-s_3)$, and we have now re-coupled the environments seen from the walkers, hence we are ready to start a new such step during the interval $[T,2T]$, etc. This last point is crucial, as we need to iterate to ensure an expected gain $\mathbb{E}^{\rho+\varepsilon}[X_{L}]-\mathbb{E}^{\rho}[X_{L}]$ that is large enough, cf.~\eqref{eq:intro-expected-gain}. 
\end{enumerate}

We now sketch a back-of-the-envelope calculation to argue that this scheme indeed generates the necessary discrepancy between $\mathbb{E}^{\rho+\varepsilon}[X_{L}]$ and $\mathbb{E}^{\rho}[X_{L}]$. During an interval of the form $[iT,(i+1)T]$ for $i=0, 1,\ldots , \lfloor L/T\rfloor -1$ (we take $i=0$ in the subsequent discussion for simplicity), with probability $\simeq 1-e^{-ct}$, we do not have the first separating event (around time $s_1$) between $X^{\rho+\varepsilon}$ and $X^{\rho}$, and we simply end up with $X^{\rho+\varepsilon}_{T} - X^{\rho}_{T}\geq 0$, which is a nonnegative expected gain. 

Else, with probability $\simeq e^{-ct}$ we have the first separation happening during $[s_1,s_2]$. On this event, the coupling in 1) (and the subsequent impermeability of the green trapezoid) succeeds with probability $\geq 1- e^{-t^{1/100}}$, and we end up with $X^{\rho+\varepsilon}_{T} - X^{\rho}_{T}\geq t$, resulting in an expected gain of order $\simeq  t$. If the first coupling fails, and we resort to the parachute coupling, we have a (negative) expected gain $\simeq -2T e^{-t^{1/100}}$. Multiplying by the probability of the first separation, we obtain a net expected gain $\simeq e^{-ct}(t-2T e^{-t^{1/100}})$. 

Finally, we evaluate the possibility that the surgery coupling (in step 2) above) or the parachute coupling 3) fails, preventing us to restore the domination of $\eta^{\rho}(\cdot+X^{\rho}) $ by $\eta^{\rho+\varepsilon}(\cdot+X^{\rho+\varepsilon}) $ at time $T$, and we do not control anymore the coupling of the walks. This has probability $O(L^{-100})$, and in the worst case we have deterministically $X^{\rho+\varepsilon}_L-X^{\rho}_L =-2L$, so the expected loss is $O(L^{-99})$.

Putting it all together, we get, for $L$ large enough by choosing e.g.~$t=\sqrt{\log L}$ and $T=(\log L)^{100}$ (it turns out that our couplings work with this choice of scales), that
\begin{equation}
\mathbb{E}^{\rho+\varepsilon}[X^{\rho+\varepsilon}_{T}]-\mathbb{E}^{\rho}[X^{\rho}_{T}]\geq e^{-ct}(t-2T e^{-t^{1/100}})+O(L^{-99})\geq L^{-o(1)}
\end{equation}
As long as the surgery coupling or the parachute succeed, we can repeat this coupling for up to $\simeq L/T$ times during $[0,L]$, thus obtaining 
\begin{equation}\label{eq:sizeablemargin}
\mathbb{E}^{\rho+\varepsilon}[X^{\rho+\varepsilon}_{L}]-\mathbb{E}^{\rho}[X^{\rho}_{L}]\geq L^{1-o(1)}/T\geq L^{1-o(1)},
\end{equation}
which establishes~\eqref{eq:intro-expected-gain}, with a sizeable margin.

\subsection{Organization of the paper}
In Section~\ref{sec:prelims}, we give rigorous definitions of general environments with a minimal set of properties~\eqref{pe:markov}-\eqref{pe:density}, and of the random walk $X$. 
In Section~\ref{sec:proofshort}, we impose the mild coupling conditions~\ref{pe:densitychange}-\ref{pe:nacelle} on the environments, state our general result, Theorem~\ref{T:generic}, and deduce Theorems~\ref{sharpnesscor} and~\ref{T:main1} from it. We then define the finite-range models and give a skeleton of the proof. 
In Section~\ref{sec:approx}, we proceed to the first part of the proof, and in Section~\ref{sec:initialspeed}, we proceed to the second part (cf.~Section~\ref{sec:proofsketch} reagrding the two parts). In Section~\ref{sec:SEP}, we define rigorously the SEP and show that it satisfies all the conditions mentioned above. The (short) Appendix~\ref{sec:appendix2} contains a few tail estimates in use throughout this article.
In Appendix~\ref{sec:PCRW}, we introduce the PCRW environment and show that it equally satisfies the above conditions, which entails that our results also apply to this environment. Throughout this article all quantities may implicitly depend on the two parameters  $p_{\bullet}, p_{\circ} \in (0,1)$ that are assumed to satisfy $p_\bullet >p_\circ$, cf.~above \eqref{eq:rwinformal} (and also Section~\ref{subsec:rw}).

\section{Setup and useful facts}\label{sec:prelims}

In this section, after introducing a small amount of notation (Section~\ref{subsec:not}), we proceed to define in Section~\ref{subsec:re} a class of (dynamic) random environments of interest, driven by a Markov process $\eta$ characterized by properties~\eqref{pe:markov}--\eqref{pe:density} below. We will primarily be interested in environments driven by the exclusion process, which is introduced in Section~\ref{sec:SEP} and shown to satisfy these properties. The framework developed in Section~\ref{subsec:re} and further in Section~\ref{sec:proofshort} will allow our results to apply directly to a second environment of interest, considered in Appendix~\ref{sec:PCRW}. We conclude by introducing in Section~\ref{subsec:rw} the relevant walk in random environment along with its associated quenched and annealed laws and collect a few basic features of this setup.

\subsection{Notation} \label{subsec:not} 

We write $\mathbb{Z}_+$, resp.~$\mathbb{R}_+$, for the set of nonnegative integers, resp.~real numbers. We use the letter $z \in \R \times \R_+$ exclusively to denote space-time points $z=(x,t)$, and typically $x,y, x', y' \dots$ for spatial coordinates and $s,t,s',t',\dots$ for time coordinates. We usually use $m,n, \dots$ for non-negative integers. 
 With a slight abuse of notation, for two integers $a\leq b$, we will denote by $[a,b]$ the set of integers $\{a,\ldots, b\}$ and declare that the length of $[a,b]$ is $\vert \{a,\ldots, b\}\vert =b-a+1$. Throughout, $c,c',\dots$ and $C,C',\dots$ denote generic constants in $(0,\infty)$ which are purely numerical and can change from place to place. Numbered constants are fixed upon first appearance.

\subsection{A class of dynamic random environments} \label{subsec:re} 
We will consider stationary Markov (jump) processes with values in $ \Sigma = (\mathbb{Z}_+)^{\mathbb{Z}}$. The state space $\Sigma$ carries a natural partial order: for two configurations $\eta, \eta' \in \Sigma$, we write $\eta\preccurlyeq \eta'$ (or $\eta'\succcurlyeq \eta$) if, for all $x\in\mathbb{Z}$, $\eta(x)\le \eta'(x)$. More generally, for $I \subset \mathbb{Z}$, $\eta\vert_I\preccurlyeq \eta'\vert_I$ (or $\eta'\vert_I\succcurlyeq \eta\vert_I$) means that $\eta(x)\le \eta'(x)$ for all $x \in I$.
For a finite subset $I\subset\mathbb{Z}$, we use the notation $\eta(I)=\sum_{x\in I}\eta(x)$. We will denote $\geq_{\text{st.}}$ and $\le_{\text{st.}}$ the usual stochastic dominations for probability measures. Let $J$ be a (fixed) non-empty open interval of $\mathbb{R}^+$. An \textit{environment} is specified in terms of two families of probability measures $(\mathbf{P}^{\eta_0}: \eta_0 \in \Sigma)$ governing the process $(\eta_t)_{t \geq 0}$ and $({\mu_\rho} : \rho \in J)$, where $\mu_\rho$ is a measure on $\Sigma$, which are required satisfy the following conditions: 
\begin{align}
&\tag{P.1}\label{pe:markov} \text{\parbox{14cm}{\textit{(Markov property and invariance).} For every $\eta_0\in \Sigma$, the process $(\eta_t)_{t \geq 0}$ defined under $\mathbf{P}^{\eta_0}$ is a time-homogeneous Markov process, such that for all $s\geq 0$ and $x\in \mathbb{Z}$, the process $(\eta_{s+t}(x+\cdot))_{t\geq 0}$ has law $\mathbf{P}^{\eta_s(x+\cdot)}$; moreover, the process $(\eta_t)_{t \geq 0}$ exhibits `axial symmetry' in the sense that $(\eta_{t}(x-\cdot))_{t\geq 0}$ has law $\mathbf{P}^{\eta_0(x-\cdot)}$. 
}}
\end{align}
\begin{align}
&\tag{P.2} \label{pe:stationary} \text{\parbox{14cm}{\textit{(Stationary measure).} The initial distribution ${\mu}_\rho$ is a stationary distribution for the Markov process $\mathbf{P}^{\eta_0}$. More precisely, letting $${\mathbf{P}}^{\rho}= \int {\mu}_{\rho}(d\eta_0) \mathbf{P}^{\eta_0}, \quad \rho\in J,$$ 
the ${\mathbf{P}}^\rho$-law of $(\eta_{t+s})_{t\ge0}$ is identical to the ${\mathbf{P}}^\rho$-law of $(\eta_{t})_{t\ge0}$ for all $s\geq 0$ and $\rho \in J$. In particular, the marginal law of $\eta_t$ under ${\mathbf{P}}^\rho$ is ${\mu}_\rho$ for all $t \geq 0$. 
}}\\[0.5em]
&\tag{P.3} \label{pe:monotonicity} \text{\parbox{14cm}{\textit{(Monotonicity).}  The following stochastic dominations hold:
\begin{itemize}
\item[i)] (quenched) For all $\eta'_0 \preccurlyeq \eta_0$, one has that $\mathbf{P}^{\eta'_0} \leq_{\text{st.}}\mathbf{P}^{\eta_0}$ , i.e.~there exists a coupling of $(\eta'_t)_{t \geq0}$ and $(\eta_t)_{t \geq 0}$ such that  $\eta'_t\preccurlyeq  \eta_t$ for all $t\geq 0$. 
 \item[ii)] (annealed) For all $\rho'\leq \rho$, one has that ${\mu}_{\rho'} \leq_{\text{st.}}{\mu}_{\rho}$. Together with i), this implies that ${\mathbf{P}}^{\rho'} \leq_{\text{st.}}{\mathbf{P}}^{\rho}$.
\end{itemize}
}}\\[0.5em]
&\tag{P.4} \label{pe:density} \text{\parbox{14cm}{\textit{(Density at stationarity).} There exists a constant $\Cl[c]{densitydev} >0$  such that for all $\rho\in J$, for all $\varepsilon \in (0,1)$ and every positive integer $\ell$, 
$$
\mathbf{P}^\rho(\{\eta_0\in \Sigma: \,  \vert\eta_0([0,\ell-1 ])-\rho\ell\vert \geq \varepsilon \ell \})\leq 2\exp(-\Cr{densitydev}\varepsilon^2\ell).
$$}}
\end{align}

A prime example satisfying the above conditions is the simple exclusion process, introduced and discussed further in Section~\ref{sec:SEP}; see in particular Lemma~\ref{L:SEP-P} regarding the validity of properties \eqref{pe:markov}--\eqref{pe:density}. We refer to 
Appendix~\ref{sec:PCRW} for another example.

\subsection{Random walk} \label{subsec:rw}

We now introduce the random walk in dynamic environment (RWdRE) that will be the main object of interest in this article.
To this effect, we fix two constants $p_{\bullet},p_{\circ}\in (0,1)$ such that $p_{\bullet}>p_{\circ}$ and an environment configuration $\eta=(\eta_t)_{t\geq 0}$ with $\eta_t \in \Sigma$ ($= \mathbb{Z}_+^{\mathbb{Z}}$, see Section~\ref{subsec:re}). 
Given this data, the random walk evolving on top of the environment $\eta$ is conveniently defined in terms of a family $(U_n)_{n\geq 0}$ of i.i.d.~uniform random variables on  $[0,1]$, as follows. 
For an initial space-time position $z=(x,m)\in \mathbb{Z}\times \mathbb{Z}_+$, let $P^{\eta}_z$ be the law of the discrete-time Markov chain $X=(X_n)_{n\geq 0}$ such that $X_0=x$ and for all integer $n\ge0$,
\begin{equation}\label{e:def-RWDRE}
X_{n+1}=X_n+2\times{\1}\big\{ U_{n} \leq (p_{\circ}-p_{\bullet})1\{ \eta_{n +m
}(X_n)=0 \} + p_{\bullet}\big\}-1.
\end{equation}
Let us call $x\in \mathbb{Z}$ an occupied site of $\eta_t \in \Sigma$ if $\eta_t(x)>0$, and empty otherwise.
With this terminology, \eqref{e:def-RWDRE} implies for instance that when $X_n$, started at a point $z=(x,t=0)$, is on an occupied (resp.~empty) site of $\eta_n$,  it jumps to its right neighbour with probability $p_{\bullet}$ (resp.~$p_{\circ}$) and to its left neighbour with probability $1-p_{\bullet}$ (resp.~$1-p_{\circ})$. We call $P^{\eta}_z$ the quenched law of the walk started at $z$ and abbreviate $P^{\eta}=P^{\eta}_{(0,0)}$; here \textit{quenched} refers to the fact that the environment $\eta$ is deterministic.

We now discuss annealed measures, i.e.~including averages over the dynamics of the environment $\eta$. We assume from here on that $\eta=(\eta_t)_{t\geq 0}$ satisfies the assumptions of Section~\ref{subsec:re}. Recall
that $\mathbf{P}^{\eta_0}$ denotes the law of $\eta$ starting from the configuration $\eta_0 \in \Sigma$ and that $\mathbf{P}^{\rho}$ is declared in \eqref{pe:stationary}. Correspondingly,  one introduces the following two annealed measures for the walk
\begin{align}
& \label{eq:RW_ann} \mathbb{P}^{\rho}_{z}[\, \cdot \, ]=\int \mathbf{P}^{\rho}(d\eta)P^{\eta}_z[\, \cdot \, ], \quad \rho \in J, \\
& \label{eq:RW_an}  \mathbb{P}^{\eta_0}_{z}[\, \cdot \, ]=\int \mathbf{P}^{\eta_0}(d\eta)P^{\eta}_z[\, \cdot \, ],\quad \eta_0 \in \Sigma,
\end{align} 
for arbitrary $z \in  \mathbb{Z}\times \mathbb{Z}_+$. Whereas the latter 
averages over the dynamics of $\eta$ for a fixed initial configuration $\eta_0$, the former includes $\eta_0$, which is sampled from the stationary distribution $\mu_{\rho}$ for $\eta$. Observe that
$\mathbb{P}^{\rho}_{z} =  \int {\mu}_{\rho}(d\eta_0) \mathbb{P}^{\eta_0}_z$.  For $\star\in\{\eta,\rho\}$, write $\P^{\star}_x$ for $\P^{\star}_{(x,0)}$, for all $x\in \mathbb{Z}$, and write simply $\P^\star$ for $\P^{\star}_0$.

We introduce a joint construction for the walk $X$ when started at different space-time points. This involves a graphical representation using arrows similar to that used in \cite[Section 3]{HKT20}, but simpler, which will be sufficient for our purposes. For a point $w=(x,n)\in\mathbb{Z} \times \mathbb{Z}_+$, we let $\pi_1(w)=x$ and $\pi_2(w)=n$ denote the projection onto the first (spatial) and second (temporal) coordinate. We consider the discrete lattice
\begin{equation}
\label{e:def_space_time}
\mathbb{L}= (2\mathbb{Z} \times 2\mathbb{Z}_+) \cup  \big( (1,1)+(2\mathbb{Z}  \times 2\mathbb{Z}_+)\big).
\end{equation}
Note that the process $(X_n,n)_{n \geq 0}$ evolves on the lattice $\mathbb{L} \subset (\mathbb{Z} \times \mathbb{Z}_+)$ defined by \eqref{e:def_space_time} when $X=(X_n)_{n \geq 0}$ is started at $z=(0,0)$ under any of $P_z^{\eta}$ and the measures in \eqref{eq:RW_ann}-\eqref{eq:RW_an}.

We proceed to define a family of processes $(X^w=(X^w_n)_{n \geq 0} : w\in\mathbb{L})$, such that $X^w_0=\pi_1(w)$ almost surely and $(X^{(0,0)}_n)_{n \geq 0}$ has the same law as $X$ under $P_{(0,0)}^{\eta}$. Furthermore, $X^{w'}$ and $X^w$ have the property that they coalesce whenever they intersect, that is if $X^{w'}_{m}=X^w_n$  for some $w,w'\in{\mathbb{L}}$ and $n,m\ge0$, then $X^{w'}_{m+k}=X^w_{n+k}$ for all $k\ge0$.

Let $U= \left(U_w\right)_{w\in\mathbb{L}}$ be a collection of i.i.d.~uniform random variables on $[0,1]$. Given the environment $\eta$, we define a field $A=(A_w)_{w\in \mathbb{L}} \in \{-1,1\}^{\mathbb{L}}$ of {\it arrows} (see Figure~\ref{f:arrows}), measurably in $(\eta,U )$ as follows:
\begin{align}\label{def:A}
A_w = A(\eta_n(x), U_w) =2\times{\1}\big\{ U_{w} \leq (p_{\circ}-p_{\bullet})1\{ \eta_n(x)= 0 \} + p_{\bullet}\big\}-1, \quad w=(x,n)\in\mathbb{L}.
\end{align}
For any  $w=(x,n)\in \mathbb{L}$, we then set $X^w_0=x$ and, for all integer $ k \ge0$, we define recursively
\begin{align} \label{eq:defX}
X^w_{k+1} = X^w_k+A_{(X^w_k, \,  n+k)}.
\end{align}

\begin{figure}[]
  \center
\includegraphics[scale=1.1]{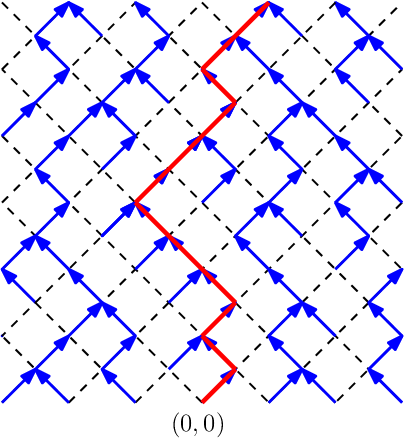}
  \caption{A representation of $\mathbb{L}$ (in dashed black) with the arrows in blue, the space $\mathbb{Z}$ being horizontal and time going upwards. The trajectory of $X^{(0,0)}$ is in red. Trajectories simply follow the arrows. Notice that two coalescing trajectories stay merged forever.}
  \label{f:arrows}
\end{figure}

This defines the coupled family $((X^w_k)_{k \geq 0}: w\in\mathbb{L})$, and we note that trajectories the process $(X_k^w , k+n)_{k \geq 0}$, where $\pi_2(w)=n$, are embedded in (i.e.~subsets of) $\mathbb{L}$. In view of \eqref{e:def-RWDRE} and \eqref{def:A}-\eqref{eq:defX}, it follows plainly that the law of $X^{(0,0)}= (X^{(0,0)}_n)_{n \geq 0}$ when averaging over $U$ while keeping $\eta$ fixed is the same as that of $X$ under $P^{\eta}_{(0,0)}= P^\eta$.

We now discuss a variation of the construction specified around \eqref{def:A}-\eqref{eq:defX}, which will be practical in Section~\ref{sec:Q}. 

\begin{Lem}\label{lem:Xlawredef}
Let $\eta_0\in \Sigma$ and suppose that $(\eta,U)$ are coupled under the probability measure $Q$, with $\eta$ having marginal law $\mathbf{P}^{\eta_0}$ and $\left(U_w\right)_{w\in\mathbb{L}}$ a collection of i.i.d.~uniform random variables on $[0,1]$, in such a way that under $Q$, for all integers $t \geq 0$,
\begin{equation}
\label{eq:arrows-variation}
\text{$(U_w: \pi_2(w)=t)$ is independent from $\mathcal{F}_t$},
\end{equation}
where
\begin{equation}\label{eq:UindepFsection2}
\mathcal{F}_t:=\sigma ((\eta_s)_{0 \leq s \leq t}, (U_w: \pi_2(w)\leq t-1)).
\end{equation}
Then, defining $X^w= X^w(\eta,U)$ as in \eqref{def:A}-\eqref{eq:defX}, it follows that $X^{(0,0)}$ has law  $\mathbb{P}^{\eta_0}$ (cf.~\eqref{eq:RW_an}) under $Q$. 
\end{Lem}

\begin{proof}
We show by induction over $k \geq 1$ integer that for all such $k$, the pair $((\eta_t)_{0\leq t\leq k-1},(X_t)_{0\leq t\leq k})$ has the same law under both $Q$ and $\mathbb{P}^{\eta_0}$ (which, for the duration of the proof, we extend here to denote the joint distribution of $(X, \eta)$, by a slight abuse of notation). The initialisation at $k=1$ is immediate, since $Q(X_1=1)=p_\bullet {\1}\big\{\eta_0(0)=1\big\}+p_\circ {\1}\big\{\eta_0(0)=0\big\}$ and $\eta_0$ is deterministic. 

As for the induction step, assume the induction hypothesis for some $k\geq 1$, and condition on the $\sigma$- algebra $\Sigma_k:=\sigma((\eta_t)_{0\leq t\leq k-1},(X_t)_{0\leq t\leq k})$.  We first let $\eta$ evolve between times $k-1$ and $k$. Under both $Q$ and $\mathbb{P}^{\eta_0}$ and conditionally on $\Sigma_k$, we have that $(\eta_t)_{k-1\leq t\leq k}$ is distributed as $(\eta_t)_{0\leq t\leq 1}$ under $ \mathbf{P}^{\eta_{k-1}}$, given the marginal distribution of $\eta$ and by the Markov property~\eqref{pe:markov}. Therefore, $((\eta_t)_{0\leq t\leq k},(X_t)_{0\leq t\leq k})$ has the same law under both $Q$ and $\mathbb{P}^{\eta_0}$. Now, conditioning on $\mathcal{F}_k$ as defined in \eqref{eq:UindepFsection2}, and noticing that $(\eta_t)_{0\leq t\leq k}$, $(X_t)_{0\leq t\leq k})$ (hence also $\eta_k(X_k)$) are all $\mathcal{F}_k$-measurable (for $(X_t)_{0\leq t\leq k}$ this follows from \eqref{def:A}-\eqref{eq:defX}), we obtain that
\begin{multline}\label{eq:increment-intermediate}
Q(X_{k+1}-X_k=1\,\vert \,\mathcal{F}_k) \\\stackrel{\eqref{def:A},\eqref{eq:defX}}{=}\Q(U_{(X_k,k)}\leq p_\bullet\,\vert \,\mathcal{F}_k) \1\big\{ \eta_k(X_k)=1\big\}+\Q(U_{(X_k,k)}\leq p_\circ\,\vert \,\mathcal{F}_k) \1\big\{ \eta_k(X_k)=0\big\}.
\end{multline}
Now, by~\eqref{eq:UindepFsection2}, and using that $X_k$ is $\mathcal{F}_k$-measurable, we can evaluate the conditional probabilities in \eqref{eq:increment-intermediate} to find that
\begin{multline} \label{eq:increment-compa}
Q(X_{k+1}-X_k=1\,\vert \, \mathcal{F}_k) = p_\bullet \1\big\{ \eta_k(X_k)=1\big\}+p_\circ\1\big\{ \eta_k(X_k)=0\big\}
\\=\mathbb{P}^{\eta_0}(X_{k+1}-X_k=1\,\vert \, \mathcal{F}_k),
\end{multline}
where the second equality comes from the independence of $\mathcal{F}_k$ and $U_k$ under $\mathbb{P}^{\eta_0}$, see~\eqref{e:def-RWDRE} and~\eqref{eq:RW_an}. Integrating both sides of \eqref{eq:increment-compa} under $Q$ and $\mathbb{P}^{\eta_0}$, respectively, against a suitable ($\mathcal{F}_k$-measurable) test function of $((\eta_t)_{0\leq t\leq k},(X_t)_{0\leq t\leq k})$, it follows that $((\eta_t)_{0\leq t\leq k},(X_t)_{0\leq t\leq k+1})$ have the same distribution under $Q$ and $\mathbb{P}^{\eta_0}$. This concludes the proof of the induction step. 
\end{proof}
\black

We conclude this section by collecting a useful monotonicity property for the collection of random walks defined above, similar to \cite[Proposition 3.1]{HKT20}. 
Its proof hinges on the fact that the trajectories considered cannot cross without first meeting at a vertex, after which they merge. In the sequel, for a given environment configuration $\widetilde{\eta} =(\widetilde{\eta}_t)_{t \geq 0}$, we refer to $\widetilde{X}^w = (\widetilde{X}_n^w)_{n \geq 0}$ the process defined as in \eqref{eq:defX} but with $\widetilde{\eta}$ in place of $\eta$ entering the definition of the arrows in \eqref{def:A}. We will be interested in the case where $\eta$ and $\widetilde{\eta}$ are such that
\begin{align}\label{domineta}
\eta_n(x)\le \widetilde{\eta}_n(x)\text{, for all } (x,n)\in\mathbb{L}\cap K,
\end{align}
for some $K\subseteq \mathbb{Z}\times\mathbb{Z}_+$.
The following result is already interesting for the choice $\eta=\widetilde{\eta}$, in which case $X^w= \widetilde{X}^w$ below.

\begin{Lem} \label{L:monotone}
If $\eta,\widetilde{\eta} \in \Sigma$ and $K\subseteq \mathbb{Z}\times\mathbb{Z}_+$ are such that \eqref{domineta} holds, then for every $w,w'\in K$ with $\pi_1(w')\le \pi_1(w)$ and $\pi_2(w)= \pi_2(w')$,and for every $n \geq 0$ such that $[\pi_1(w)-k,  \pi_1(w')+k]\times [\pi_2(w), \pi_2(w)+k] \subseteq K $ for all $0 \leq k \leq n$\color{black}, one has that
\begin{equation}\label{e:monotone1}
  X^{w'}_{n}\le \widetilde{X}^{w}_{n}.
\end{equation}  
\end{Lem}

\begin{proof}  We only treat the case $K= \mathbb{Z}\times\mathbb{Z}_+$ to lighten the argument. The adaptation to a general $K$ is straightforward, as all possible trajectories considered for $X$ and $\widetilde{X}$ lie in $K$ by assumption. The proof proceeds by a straightforward induction argument. Indeed, since $X_0^{w'}= \pi_1(w')$ and $\widetilde{X}_0^w= \pi_1(w)$, one has that
$\widetilde{X}^{w}_{0}- X^{w}_{0} \geq 0$ by assumption. To carry out the induction step, one notes that $\widetilde{X}^{w}_{n}-X^{w'}_{n}$ is even for any $n \geq 0$ with increments ranging in $-2$, $0$ or $+2$, and combines this with the following observation: if $\widetilde{X}^{w}_{n}-X^{w'}_{n}=0$, then
\begin{multline*}
\widetilde{X}^{w}_{n+1}-X^{w'}_{n+1} \stackrel{\eqref{eq:defX}}{=} A_{(\widetilde{X}_n^w, \pi_2(w)+n)} -  A_{({X}_n^{w'}, \pi_2(w')+n)} \stackrel{\eqref{def:A}}{=}  A(\widetilde{\eta}_n(X^{w}_{n}), U_w)- A({\eta}_n(X^{w}_{n}), U_w)\ge0,
\end{multline*}
where the second equality follows using that $\pi_2(w)=\pi_2(w')$ (along with $\widetilde{X}^{w}_{n}=X^{w'}_{n}$) and the inequality is due to \eqref{domineta} and the fact that $A(\cdot, \xi)$ is increasing for any $\xi \in [0,1]$, which is straightforward from \eqref{def:A}.
\end{proof}

For later reference, we also record that for any $w\in\mathbb{L}$ and $n\ge m\ge 0$, 
\begin{equation}\label{e:lipschitz}
  \left|X^w_n - X^w_m\right| \le n - m,
\end{equation}
as follows clearly from \eqref{eq:defX}.

\section{Main results}\label{sec:proofshort}

In this section, we start by formulating in Section~\ref{subsec:C}  precise coupling conditions that we require from the environments, and which are of independent interest. These are given by conditions~\ref{pe:densitychange}-\ref{pe:nacelle} below (a flavor of the second of these was given in \eqref{eq:coupling-intro} in the introduction). Our main result, Theorem \ref{T:generic}, appears in Section~\ref{subsec:main}. It concerns the generic random walk in random environment defined in Section~\ref{subsec:re}-\ref{subsec:rw}, subject to the conditions of Section~\ref{subsec:C}. Our standing assumptions will thus be that all of properties~\eqref{pe:markov}-\eqref{pe:density} and the conditions~\ref{pe:densitychange}-\ref{pe:nacelle} hold. We will verify separately in Section~\ref{sec:SEP} that all of these conditions hold for SEP. From Theorem \ref{T:generic} we then readily deduce Theorems~\ref{sharpnesscor}and~\ref{T:main1} at the end of Section~\ref{subsec:main}.

Towards the proof of Theorem~\ref{T:generic}, and following the outline of Section~\ref{sec:proofsketch}, we proceed to introduce in Section~\ref{subsec:finiterange} a finite-range approximation of the model, in which the environment is renewed after $L$ time steps and gather its essential features that will be useful for us. We then state two 
key intermediate results, Propositions~\ref{Prop:vLapprox} and \ref{Prop:initialspeed}, which correspond to the first and second parts from the discussion in~Section~\ref{sec:proofsketch}; cf.~also \eqref{eq:introvLtov} and \eqref{eq:introvLquant}. From these, we deduce Theorem~\ref{T:generic} in Section~\ref{subsec:proofskeleton}. The proofs of the two propositions appear in forthcoming sections.

\subsection{Coupling conditions on the environment}\label{subsec:C}

Recall the framework of Section~\ref{subsec:re}-\ref{subsec:rw}, which we now amend with three further conditions. These involve an additional parameter $\nu>0$ which quantifies the activity of the environment (in practice it will  correspond to the rate parameter appearing in \eqref{eq:gen-SEP} in the case of SEP). We keep the dependence on $\nu$ explicit in the following conditions for possible future applications, for which one may wish to tamper with the speed of the environment (for instance, by slowing it down). 
\\
For the purposes of the present article however, one could simply set $\nu=1$ in what follows.

The first two conditions \ref{pe:densitychange} and \ref{pe:compatible} below regard the environment $(\eta_t)_{t \geq 0}$ alone, which, following the setup of Section~\ref{subsec:re}, is assumed to be specified in terms of the measures $(\mathbf{P}^{\eta_0}: \eta_0 \in \Sigma)$ and $({\mu_\rho} : \rho \in J)$ that satisfy \eqref{pe:markov}-\eqref{pe:density}. The first condition, \ref{pe:densitychange}, concerns the empirical density of the environment. Roughly speaking, it gives a quantitative control on how  \eqref{pe:density} is conserved over time. The second condition, \ref{pe:compatible}, is more technical.
In a nutshell, it states that if one environment $\eta$ covers another environment $\eta'$ on a finite interval $I$ at a given time, and if $\eta$ has a larger empirical density than $\eta'$ outside $I$ at the same time, the evolutions of $\eta$ and $\eta'$ can be coupled in a way that $\eta$ covers $\eta'$ on a larger interval after some time. 

\medskip
We proceed to formalize these two properties:
\begin{enumerate}[label=(C.\arabic*) ]
\item \label{pe:densitychange} 
\textit{(Conservation of density).} There exists constants $\Cl{densitystable},\Cl[c]{densitystableexpo} \in (0,\infty)$ such that for all $\rho \in J$ and $\varepsilon \in (0,1)$ (with $\rho+\varepsilon \in J$), and for all 
$\ell, \ell', H,t\geq 1$ satisfying $H>4{\nu} t > \Cr{densitystable}\ell^2\varepsilon^{-2}(1+{ \vert\log^3 (\nu t)}\vert)$ and $\ell'\leq \sqrt{t}$, the following two inequalities hold. Let $\eta_0$ be such that on every interval $I$ of length $\ell$ included $[-H,H]$, one has $ \eta_0(I)\leq (\rho+\varepsilon)\ell$ (resp.~$ \eta_0(I)\geq (\rho-\varepsilon)\ell$).  Then
$$
{\mathbf{P}}^{\eta_0} \left(\begin{array}{c} \text{for all intervals $I'$ of length $\ell'$ included in} \\ \text{$[-H+2{\nu} t,H-2{\nu} t]$: $\eta_t(I') \leq (\rho+3\varepsilon) \ell'$}\\ \text{(resp.~$\eta_t(I')  
 \geq (\rho-3\varepsilon) \ell'$)} \end{array} \right) \geq 1 - 4  H \exp(-\Cr{densitystableexpo}\varepsilon^2\ell').
$$
\item\label{pe:compatible} \textit{(Couplings).} There exists $\Cl{compatible}, \Cl{SEPcoupling2}\in (0,\infty)$ such that the following holds. Let $\rho \in J$, $\varepsilon \in (0,1)$ (with $\rho+\varepsilon \in J$), and $H_1, H_2,t,\ell\geq 1$ be integers { such that $\min\{H_1, H_2-H_1-1\}> 10{\nu}t>4\nu\ell^{100}>\Cr{compatible}$, $\nu^{} \ell>\Cr{compatible}\varepsilon^{-2}(1+\vert\log^3(\nu\ell^4)\vert)$ and $\ell>80\nu\varepsilon^{-1}+\nu^{-2}$}. Let $\eta_0,\eta'_0\in \Sigma$ be such that $\eta_0\vert_{[-H_1, H_1]}\succcurlyeq \eta_{\red 0 \black}'\vert_{[-H_1, H_1]}$ and such that for every interval $I\subseteq [-H_2,H_2]$ of length $\lfloor \ell/2 \rfloor \leq |I| \leq \ell$ , we have $\eta_0(I)\geq (\rho+3\varepsilon/4)\vert I \vert$ and $\eta'_0(I)\leq (\rho+\varepsilon/4)\vert I \vert$. Then there exists a coupling $\Q$ of two environments $\eta,\eta'$ with respective marginals $\mathbf{P}^{\eta_0},\mathbf{P}^{\eta'_0}$ such that 
\begin{equation}\label{eq:compatible1}
\Q\big(\forall s\in [0,t], \, \eta_s\vert_{[-H_1+4{\nu}t, H_1-4{\nu}t]}\succcurlyeq \eta'_s\vert_{[-H_1+4{\nu}t, H_1-4{\nu}t]} \big)\geq 1 - 20t\exp (-{\nu}t/4)
\end{equation}
and 
\begin{equation}\label{eq:compatible2}
\Q\big(\eta_t \vert_{[-H_2+6{\nu}t, H_2-6{\nu}t]} \succcurlyeq \eta'_t\vert_{[-H_2+6{\nu}t, H_2-6{\nu}t]} \big)\geq 1 -5\Cr{SEPcoupling2} \ell^4 H_2\exp\left(- \Cr{SEPcoupling2}^{-1}\textstyle\frac{\nu}{\nu+1}\varepsilon^2\ell\right).
\end{equation}
\end{enumerate}

For later reference, we record the following two particular instances of~\ref{pe:compatible}, which correspond to the cases where $H_2=0$ and $H_1=0$, respectively. For convenience, we state them with better constants and exponents than in~\ref{pe:compatible}. In fact, when verifying condition~\ref{pe:compatible} for the SEP in Section~\ref{subsec:C-SEP}, we will first prove that these two conditions hold, and use them to prove~\ref{pe:compatible}.

\begin{enumerate}[label=(C.2.\arabic*)]
\item\label{pe:drift} \textit{(No particle drifting in from the side).} 
Let $t,H \geq 0$ and $k\geq 1$ be integers, and let $\eta_0,\eta'_0\in \Sigma$ be such that $\eta_0 \vert_{[-H, H]}\succcurlyeq \eta_0' \vert_{[-H, H]}$.   
There exists a coupling $\Q$ of environments $\eta,\eta'$ with respective marginals $\mathbf{P}^{\eta_0}$ and $\mathbf{P}^{\eta'_0}$ such that
\begin{equation}\label{eq:SEPdriftdeviations}
\mathbb{Q}\big(\forall s\in [0,t], \, \eta_s\vert_{[-H+2{\nu}kt, H-2{\nu}kt]}\succcurlyeq\eta'_s\vert_{[-H+2{\nu}kt, H-2{\nu}kt]} \big)\geq 1- 20\exp(-k{\nu} t/4).
\end{equation}
Informally, with high probability no particle of $\eta'$ outside of $[-H,H]$ can drift into $ [-H+2{\nu}t, H-2{\nu}t]$ (when $k=1$ for instance) before time $t$ to perturb the domination of $\eta'$ by $\eta$.

\item\label{pe:couplings} \textit{(Covering $\eta'$ by $\eta$).}
There exists $\Cl{SEPcoupling} >0$ such that for all $\rho\in J$ and $\varepsilon\in (0,1)$ (with $\rho+\varepsilon \in J$), the following holds. If  $H,t\geq 1$ satisfy {$H>4{\nu} t$, $\nu^8t>1$ and $\nu t^{1/4} >\Cr{SEPcoupling}\varepsilon^{-2}(1+\vert \log^{3}(\nu t) \vert)$}, then for all $\eta_0,\eta'_0\in \Sigma$ such that on each interval $I \subset [-H,H]$ of length $\lfloor \ell/2 \rfloor \leq |I| \leq \ell$, where $\ell:=\lfloor t^{1/4}\rfloor$, $\eta_0(I)\geq (\rho+3\varepsilon/4)|I|$ and $ \eta'_0(I) \leq (\rho+\varepsilon/4)|I|$, there exists a coupling $\mathbb{Q}$ of $\eta,\eta'$  with marginals $\mathbf{P}^{\eta_0}, \mathbf{P}^{\eta'_0}$ so that \begin{equation}\label{eq:doublecoupleSEP}
\mathbb{Q}( \eta_t \vert _{[-H+4{\nu} t, H-4{\nu} t]}\succcurlyeq \eta_t' \vert _{[-H+4{\nu} t, H- 4{\nu} t]})\geq 1- \Cr{SEPcoupling2} tH\exp\big(-(\Cr{SEPcoupling2}(1+\nu^{-1}))^{-1}\varepsilon^2t^{1/4} \big).
\end{equation}
In words, the coupling $\mathbb{Q}$ achieves order between $\eta$ and $\eta'$ at time $t$ under suitable regularity assumptions on the empirical density of their initial configurations $\eta_0$ and $\eta_0'$.

\end{enumerate}

The third and last property ensures that if an environment $\eta$ covers another environment $\eta'$ and has at least one extra particle at distance $\ell$ of the origin (which typically happens if $\eta$ has higher density than $\eta'$), then with probability at least exponentially small in $\ell$, by time $\ell$, this particle can reach the origin which will be empty for $\eta'$, while preserving the domination of $\eta'$ by $\eta$. 
We will combine this property with the uniform ellipticity of the random walk on top of the environment to show that the walker has at least an exponentially small probability to reach a position where $\eta$ has a particle but not $\eta'$, which in turn yields a probability bounded away from zero that a walker on $\eta$ steps to the right while a walker on $\eta'$ steps to the left (under the coupling mentioned in Section \ref{sec:proofsketch}), hence creating the desired initial gap that we will then exploit in our constructions.

\begin{enumerate}[label=(C.\arabic*) ]
\setcounter{enumi}{2}
\item\label{pe:nacelle} \textit{(Sprinkler).} 
For all $\rho \in J$, for all integers $H, \ell,k\geq 1$ with $H\geq 2\nu\ell k$ and $k \geq 48\nu^{-1}(\nu +\log(40)-\log(p_\circ(1-p_{\bullet})\nu/2))$, the following holds. If $\eta_0,\eta'_0 \in \Sigma$ satisfy $\eta_0(x)\geq \eta'_0(x)$ for all $x\in [-H,H]$, $ \eta_0([0,\ell])\geq \eta'_0([0,\ell])+1 $ and $\eta'_0([-3\ell+1, 3\ell])\leq 6(\rho+1)\ell$,  then there is a coupling $\mathbb{Q}$ of  $\eta'$ under ${\mathbf{P}}^{\eta_0'}$ and $\eta$ under ${\mathbf{P}}^{\eta_0}$ such that, with 
$
\delta =  (
\nu/2e^{{\nu}} )^{6(\rho+1)\ell}, 
$
\begin{equation}\label{eq:penacelleNEW}
\begin{split}
&\mathbb{Q}(\eta_{\ell}(x )>0  , \,    \eta'_{\ell}( x )=0)\geq 2\delta
\end{split}
\end{equation}
 for $x \in \{0,1\},$ and with $\delta'= \delta (p_{\circ}(1-p_{\bullet}))^{6(\rho+1)\ell}$,
\begin{equation}\label{eq:penacelleNEW2}
\begin{split}
\mathbb{Q}(\{\forall s\in [0, \ell],  \, \eta_s\vert_{[-H+2{\nu}k\ell, H-2{\nu}k\ell]}\succcurlyeq\eta'_s\vert_{[-H+2{\nu}k\ell, H-2{\nu}k\ell]}\}^c)&\leq 20e^{-k\nu \ell/4}
 \leq \delta'.
\end{split}
\end{equation}
\end{enumerate}

\subsection{Main result}\label{subsec:main} Following is our main theorem.

\begin{The}[Sharpness of $v(\cdot)$] \label{T:generic} Let $\eta$ be an environment as in Section~\ref{subsec:re} satisfying~\ref{pe:densitychange}-\ref{pe:nacelle}. Assume that for all $\rho\in J$, there exists $v(\rho)$ such that
\begin{equation}\label{eq:LLNagain}
\mathbb{P}^\rho\text{-a.s., } \lim_n  n^{-1}{X_n}= v(\rho) .
\end{equation}
Then, for all $\rho, \rho'\in J$ such that $\rho>\rho'$, one has  that
\begin{align}
&\label{eq:monotonic} v(\rho)>v(\rho').
\end{align}
\end{The}

With Theorem~\ref{T:generic} at hand, we first give the proofs of Theorems~\ref{sharpnesscor}~and~\ref{T:main1}. We start with the latter.

\begin{proof}[Proof of Theorem \ref{T:main1}] Theorem~\ref{T:main1} concerns the particular case where $\mathbb{P}^\rho$ refers to the walk of Section~\ref{subsec:rw} evolving on top of the exclusion process $\eta$ started from product Bernoulli$(\rho)$ distribution, for $\rho \in J \subset (0,1)$. The properties \eqref{pe:markov}-\eqref{pe:density} and~\ref{pe:densitychange}-\ref{pe:nacelle} are indeed all satisfied in
this case, as is proved separately in Section~\ref{sec:SEP} below, see Lemma~\ref{L:SEP-P} and Proposition~\ref{P:SEP-C}. 

In the notation of \eqref{eq:rho-pm}, we now separately consider the intervals $J\in \{J_-, J_0, J_+\}$, where $J_-=(0,\rho_-)$, $J_0=(\rho_-,\rho_+)$ and $J_+=(\rho_+,1)$. The fact that the law of large numbers \eqref{eq:LLNagain} holds for any choice of $J$ is the content of \cite[Theorem~1.1]{HKT20}. Thus Theorem~\ref{T:generic} is in force, and $v$ is strictly increasing on $J$ for any  $J\in \{J_-, J_0, J_+\}$. Since by direct application of~\eqref{pe:monotonicity} and Lemma~\ref{L:monotone}, $v$ is already non-decreasing on $J_-\cup J_0\cup J_+$, it must be (strictly) increasing on $J_-\cup J_0\cup J_+$ altogether, and \eqref{eq:monotonic-intro} follows. In view of \eqref{eq:rho-pm}, the first equality in \eqref{eq:equality} is an immediate consequence of \eqref{eq:monotonic} with $J=J_0$.

As to the second equality in \eqref{eq:equality}, recalling the definition of $\rho_c$ from \eqref{eq:rho_c}, we first argue that $\rho_- \leq\rho_c$. If $\rho < \rho_-$, then by \eqref{eq:intro-LLN}-\eqref{eq:rho-pm}, $X_n/n \to v(\rho)<0$ $\P^{\rho}$-a.s.~and thus in particular $\P^{\rho}(\limsup X_n <0)=1$ (in fact it equals $-\infty$ but we will not need this). On the other hand, $\{ H_n < H_{-1} \}\subset \{ H_{-1}>n\} \subset \{X_k \geq 0, \, \forall k \leq n\}$, and the latter has probability tending to $0$ as $n \to \infty$ under $\P^{\rho}$. It follows that $\theta(\rho)= 0$ in view of \eqref{eq:order-parameter}, whence $\rho \leq \rho_c$, and thus $\rho_- \leq \rho_c$ upon letting $\rho \uparrow \rho_-$. 

Since $\rho_-=\rho_+$, in order to complete the proof it is enough to show that $\rho_+ \geq \rho_c$. Let $\rho> \rho_+$. We aim to show that $\theta(\rho) > 0$. Using the fact that $v(\rho)>0$ and \eqref{eq:intro-LLN}, one first picks $n_0=n_0(\rho) \geq 1$ such that $\P^{\rho}_z(X_n >0, \, \forall n \geq n_0) \geq 1/2$ for any $z=(m,m)$ with $m \geq 0$ (the worst case is $m=0$, the other cases follow from the case $m=0$ using invariance under suitable space-time translations). Now, observe that for all $n \geq 0$, under $\P^{\rho}$,
\begin{multline}\label{eq:supercrit-1}
\{ H_n < H_{-1} \} \supset\big( \{ X_k-X_{k-1}= +1, \, \forall 1 \leq k \leq n_0\} \\
\cap \{ X_{2n_0+k'} >0, \, \forall k' \geq n_0\}\cap \{\limsup_{n\rightarrow\infty}X_n=+\infty\} \big);
\end{multline}
for, on the event on the right-hand side, one has that $X_{n_0}= n_0$ and the walk can in the worst case travel $n_0$ steps to the left during the time interval $(n_0, 2n_0]$, whence in fact $X_k \geq 0$ for all $k \geq 0$, and $H_n<\infty$ since $\limsup_{n\rightarrow\infty}X_n=+\infty$. Combining \eqref{eq:supercrit-1}, the fact that the last event on the RHS has  full probability due to \eqref{eq:intro-LLN}-\eqref{eq:rho-pm}, the fact that $P^{\eta}_{(0,0)}(X_k-X_{k-1}= +1, \, \forall 1 \leq k \leq n_0) \geq (p_{\bullet} \wedge p_{\circ})^{n_0}$ on account of \eqref{def:A}-\eqref{eq:defX} and the Markov property of the quenched law at time $n_0$, one finds that 
$$
\theta_n(\rho) \stackrel{\eqref{eq:Bn}}{\geq}  (p_{\bullet} \wedge p_{\circ})^{n_0} \cdot \P_{(n_0, n_0)}^{\rho}(X_{n_0+k'} >0, \, \forall k' \geq n_0 ) \geq 2^{-1} (p_{\bullet} \wedge p_{\circ})^{n_0} >0,
$$
where the second inequality follows by choice of $n_0$. Thus, $\theta(\rho)>0$ (see \eqref{eq:order-parameter}), i.e.~$\rho \geq \rho_c$. Letting $\rho \downarrow \rho_+$ one deduces that $\rho_+ \geq \rho_c$, and this completes the verification of \eqref{eq:equality}.
\end{proof}

\begin{proof}[Proof of Theorem~\ref{sharpnesscor}]
Only item~\textit{(i)} requires an explanation. This is an easy consequence of \cite[Proposition 3.6]{HKT20}, e.g.~with the choice $\varepsilon=(\rho_c-\rho)/4$ for a given $\rho<\rho_c$, and Theorem \ref{T:main1} (see \eqref{eq:equality}), as we now explain. Indeed one has that $\{ H_n < H_{-1} \}\subset \{X_n\ge v(\rho_c-\varepsilon)n\}$ under $\P^{\rho}$ as soon as $n$ is large enough (depending on $\rho$); to see the inclusion of events recall that $\rho_c=\rho_-$ on account of \eqref{eq:equality}, which has already been proved, and therefore $v(\rho_c-\varepsilon)<0$. The conclusion of item~\textit{(i)} now readily follows using the second estimate in \cite[(3.27)]{HKT20}.
\end{proof}

\subsection{The finite-range model $\mathbf{P}^{\rho,L}$}\label{subsec:finiterange}

Following the strategy outlined in Section~\ref{sec:proofsketch}, we will aim at comparing the random walk in dynamic random environment, which has infinite-range correlations, to a finite-range model, which enjoys regeneration properties, and which we now introduce.

For a density $\rho \in J$ (recall that $J$ is an open interval of $\mathbb{R}_+$) and an integer $L\geq1$, we define a finite-range version of the environment, that is, a probability measure $\mathbf{P}^{\rho,L}$ on $\Sigma \ni \eta=(\eta _t(x):\, x \in \mathbb{Z}, \, t\in\mathbb{R}_+) $ (see Section~\ref{subsec:re} for notation) such that the following holds. At every time $t$ multiple of $L$, $\eta_t$ is sampled under $\mathbf{P}^{\rho,L}$ according to ${\mu}_\rho$ (recall \eqref{pe:stationary}), independently of $(\eta _s)_{0\leq s<t}$ and, given $\eta_t $, the process $(\eta_{t+s})_{0\le s<L}$ has the same distribution under $\mathbf{P}^{\rho,L}$ as $(\eta_s)_{0\le s<L}$ under ${\bf P}^{\eta_t }$, cf.~Section~\ref{subsec:re} regarding the latter. We denote $\mathbf{E}^{\rho,L}$ the expectation corresponding to $\mathbf{P}^{\rho,L}$. It readily follows that $\eta$ is a homogenous Markov process under $\mathbf{P}^{\rho,L}$, and that $\mathbf{P}^{\rho,L}$ inherits all of Properties \eqref{pe:markov}-\eqref{pe:density} from $\mathbf{P}^{\rho}$. In particular ${\mu}_\rho$ is still an invariant measure for the time-evolution of this environment. Note that $\mathbf{P}^{\rho,\infty}$ is well-defined and $\mathbf{P}^{\rho,\infty}= \mathbf{P}^{\rho}$.

Recalling the quenched law $P_z^{\eta}$ of the walk $X$ in environment $\eta$ started at $z\in \mathbb{Z}$ from Section~\ref{subsec:rw}, we extend the annealed law of the walk from \eqref{eq:RW_ann} by setting $\P_{z}^{\rho,L}[\cdot] = \int \mathbf{P}^{\rho,L}(d\eta) P_z^{\eta}[\cdot]$ so that $\P^{\rho,\infty}_z=\P^\rho_z$ corresponds to the annealed law defined in~\eqref{eq:RW_ann}. We also abbreviate $\P^{\rho,L}=\P^{\rho,L}_0$.

We now collect the key properties of finite-range models that will be used in the sequel. A straightforward consequence of the above definitions is that
\begin{equation}\label{eq:regenstructure}
\begin{array}{l}
\text{under }\P^{\rho,L}, \left\{(X_{kL+s}-X_{kL})_{0\le s\le L}: k\in\mathbb{N}\right\}\text{ is an i.i.d.~family,}\\
\text{with common distribution identical to the } \mathbb{P}^{\rho}\text{-law of } (X_{s})_{0\le s\le L}.
\end{array}
\end{equation}
Moreover, by direct inspection one sees that $\mathbf{P}^{\rho,L}$ inherits the properties listed in \eqref{subsec:C} from $\mathbf{P}^{\rho}$; that is, whenever $\mathbf{P}^{\rho}$ does,
\begin{equation}
\label{eq:C-for-finiterange}
\text{$\mathbf{P}^{\rho,L}$ satisfies \ref{pe:densitychange},~\ref{pe:compatible},~\ref{pe:drift},~\ref{pe:couplings}
 (all for $L\geq t$) and~\ref{pe:nacelle} (for $L\geq \ell$).}
\end{equation}
(more precisely, all of these conditions hold with $\mathbf{P}^{\eta,L}$ in place of $\mathbf{P}^{\eta}$ everywhere, where $\mathbf{P}^{\eta,L}$ refers to the evolution under $\mathbf{P}^{\rho,L}$ with initial condition $\eta_0=\eta$).
The next result provides a well-defined monotonic speed $v_L(\cdot)$ for the finite-range model. This is an easy fact to check. A much more refined quantitative monotonicity result will follow shortly in Proposition \ref{Prop:vLapprox} below (implying in particular strict monotonicity of $v_L(\cdot)$).

\begin{Lem}[Existence of the finite-range speed $v_L$]\label{Prop:vLexists}
For $\rho \in J$ and an integer $L\geq 1$, let
\begin{equation}\label{eq:defvL}
v_L(\rho)\stackrel{\textnormal{def.}}{=}\mathbb{E}^{\rho,L}\left[{X_L}/{L}\right]=\mathbb{E}^{\rho}\left[{X_L}/{L}\right].
\end{equation}
Then
\begin{equation}\label{eq:LLNvL}
 \mathbb{P}^{\rho,L}\textnormal{-a.s.}
\lim_{n\rightarrow +\infty} n^{-1}{X_n}=v_L(\rho).
\end{equation}
Moreover, for any fixed $L$, we have that
\begin{equation}\label{eq:vLmonotonic}
\text{the map $\rho\mapsto v_L(\rho)$ is non-decreasing on $J$.}
\end{equation}
\end{Lem}

\begin{proof}
The second equality in \eqref{eq:defvL} is justified by \eqref{eq:regenstructure}.  The limit in \eqref{eq:LLNvL} is an easy consequence of \eqref{eq:regenstructure}, \eqref{e:lipschitz}, the definition \eqref{eq:defvL} of $v_L$ and the law of large numbers.
The monotonicity \eqref{eq:vLmonotonic} is obtained  by combining \eqref{pe:monotonicity} and Lemma \ref{L:monotone}.
\end{proof}

\subsection{Key propositions and proof of Theorem~\ref{T:generic}}\label{subsec:proofskeleton}

In this section, we provide two key intermediate results, stated as Propositions~\ref{Prop:vLapprox} and~\ref{Prop:initialspeed}, which roughly correspond to the two parts of the proof outline in Section~\ref{sec:proofsketch}.
Theorem~\ref{T:generic} will readily follow from these results, and the proof appears at the end of this section. As explained in the introduction, the general strategy draws inspiration from recent interpolation techniques used in the context of sharpness results for percolation models with slow correlation decay, see in particular \cite{SharpnessGFF,RI-I}, but the proofs of the two results stated below are vastly different.

Our first result is proved in Section \ref{sec:approx} and allows us to compare the limiting speed of the full-range model to the speed of the easier finite-range model with a slightly different density. Note that, even though both $v(\cdot)$ and $v_L(\cdot)$ are monotonic, there is a priori no clear link between $v$ and $v_L$.

\begin{Prop}[Approximation of $v$ by $v_L$]\label{Prop:vLapprox}

Under the assumptions of Theorem \ref{T:generic}, there exists $\Cl{C:approx}=\Cr{C:approx}(\nu) \in (0,\infty)$ such that for all $L \geq 3$ and $\rho$ such that $[\rho -\frac1{\log L},\rho +\frac1{\log L}]\subseteq J$, 
\begin{equation}
\label{eq:vLapprox}
v_{L}\Big(\rho -\frac1{ \log L}\Big)-\frac{\Cr{C:approx}(\log L)^{100}}{L}\leq v(\rho)\leq v_L\Big(\rho +\frac1{ \log L}\Big) +\frac{\Cr{C:approx}(\log L)^{100}}{L}.
\end{equation}
\end{Prop}

With Proposition \ref{Prop:vLapprox} above, we now have a chance to deduce the strict monotonicity of $v(\cdot)$ from that of $v_L(\cdot)$. Proposition~\ref{Prop:initialspeed} below, proved in Section \ref{sec:initialspeed}, provides a quantitative strict monotonicity for the finite-range speed $v_L(\cdot)$. Note however that it is not easy to obtain such a statement, even for the finite-range model: indeed, from the definition of $v_L(\cdot)$ in \eqref{eq:defvL}, one can see that trying to directly compute the expectation in \eqref{eq:defvL} for \text{any} $L$ boils down to working with the difficult full range model. One of the main difficulties is that the environment mixes slowly and creates strong space-time correlations. Nonetheless, using sprinkling methods, it turns out that one can speed-up the mixing dramatically by increasing the density of the environment. This is the main tool we use in order to obtain the result below.

\begin{Prop}[Quantitative monotonicity of $v_L$]\label{Prop:initialspeed}
 Assume \eqref{eq:C-for-finiterange}  
 holds and let $\rho\in J$. For all $\epsilon>0$ such that $\rho+\epsilon\in J$, there exists $L_1=L_1(\rho,\epsilon, \red \nu \black)\geq 1$ such that, for all $L\ge L_1$,
\begin{equation}\label{eq:vLincrease}
v_{L}(\rho+\epsilon) - v_{L}(\rho) \geq \frac{3 
\Cr{C:approx} \color{black} (\log L)^{100}}{L}.
\end{equation}
\end{Prop}

Propositions~\ref{Prop:vLapprox} and~\ref{Prop:initialspeed} imply Theorem \ref{T:generic}, as we now show.
\begin{proof}[Proof of Theorem \ref{T:generic}.]  
Let $\rho, \rho'\in J$ with $\rho>\rho'$, and define $\varepsilon=\rho-\rho'$. 
Consider $L\geq 3 \vee L_1(\rho+\varepsilon/3, \varepsilon/3)$, so that the conclusions of Propositions~\ref{Prop:vLapprox} and~\ref{Prop:initialspeed} both hold. By choosing $L$ sufficiently large in a manner depending on $\rho$ and $\rho'$, we can further ensure that $(\log L)^{-1}<\varepsilon/3$ and $[\rho-(\log L)^{-1}, \rho+(\log L)^{-1}]\cup [\rho'-(\log L)^{-1}, \rho'+(\log L)^{-1}]\subseteq J$. Abbreviating $\alpha_L = L^{-1}{(\log L)^{100}}$, it follows that 
\begin{multline*}
v(\rho)\stackrel{\eqref{eq:vLapprox}}{\geq} v_{L}(\rho -(\log L)^{-1})-\Cr{C:approx} \alpha_L \stackrel{\eqref{eq:vLmonotonic}}{\geq}  v_{L}(\rho -\textstyle \frac{\varepsilon}{3})-\Cr{C:approx} \alpha_L \stackrel{\eqref{eq:vLmonotonic}}{\geq}
 v_L(\rho'+\frac{2\varepsilon}{3}) - \Cr{C:approx} \alpha_L
\\
\stackrel{\eqref{eq:vLincrease}}{\geq} \textstyle v_L(\rho'+\frac{\varepsilon}{3}) +2 \Cr{C:approx} \alpha_L
\stackrel{\eqref{eq:vLmonotonic}}{\geq} v_L(\rho'+(\log L)^{-1}) +2\Cr{C:approx} \alpha_L \stackrel{\eqref{eq:vLapprox}}{\geq} v(\rho')+\Cr{C:approx} \alpha_L
>v(\rho'),
\end{multline*}
yielding~\eqref{eq:monotonic}.
\end{proof}

\section{Finite-range approximation of $\P^{\rho}$}\label{sec:approx}
In this section, we prove Proposition~\ref{Prop:vLapprox}, which allows us to compare the speed $v(\cdot)$ of the full-range model to the speed $v_L(\cdot)$ (see \eqref{eq:defvL}) of the finite-range model introduced in \S\ref{subsec:finiterange}. The proof uses a dyadic renormalisation scheme. By virtue of the law(s) of large numbers, see \eqref{eq:LLNagain} and \eqref{eq:LLNvL}, and for $L_0$ and $\rho$ fixed, the speed $v(\rho)$ ought to be close to the speed $v_{2^KL_0}(\rho)$, for some large integer $K$. Hence, if one manages to control the discrepancies between $v_{2^{k+1}L_0}(\rho)$ and $v_{2^kL_0}(\rho)$ for all $0\le k\le K-1$ and prove that their sum is small, then the desired proximity between $v(\rho)$ and $v_{L_0}(\rho)$ follows. This is roughly the strategy we follow except that at each step, we slightly increase or decrease the density $\rho$ (depending on which bound we want to prove), in order to weaken the (strong) correlations in the model. This is why we only compare $v(\cdot)$ and $v_L(\cdot)$ for slightly different densities at the end. This decorrelation method is usually referred to as {\it sprinkling}; see e.g.~\cite{MR3053773,MR2891880} for similar ideas in other contexts.

As a first step towards proving Proposition \ref{Prop:vLapprox}, we establish in the following lemma a one-step version of the renormalization, with a flexible scaling of the sprinkling ($f(L)$ below) in anticipation of possible future applications. For $x\in\mathbb{R}$, let $x_-=\max(-x,0)$ denote the negative part of $x$. Throughout the remainder of this section, we are always tacitly working under the assumptions of Theorem~\ref{T:generic}.

\begin{Lem}\label{Lem:L2L} 
There exists $L_0=L_0(\nu) \geq 1$ such that for all $L \geq L_0$, all $f(L)\in [(\log L)^{90}, L^{1/10}]$ and all $\rho$ such that $(\rho-f(L)^{-1/40},\rho+f(L)^{-1/40})\subseteq J$, the following holds: there exists a coupling $\Q_{L}$ of $(X^{(i)}_{s})_{ 0 \leq s \leq 2L}$, $i=1,2$, such that $X^{(1)}\sim  \P^{\rho,L}$, $X^{(2)}\sim \P^{\rho + \varepsilon,2L}$ with $\varepsilon= f(L)^{-1/40}$,  and
\begin{equation}\label{eq:L2L-renormalization-minvalue}
\Q_L\Big(\min_{0\le s\le 2L}\big(X^{(2)}_{s}-X^{(1)}_{s}\big)\leq -f(L)\Big)\leq  e^{-f(L)^{1/40}}.
\end{equation}
Consequently,
\begin{equation}\label{eq:L2L-renormalization-expectation}
\mathbb{E}^{\Q_L}\Big[\max_{0\le s\le 2L}\big(X^{(2)}_{s}-X^{(1)}_{s}\big)_-\Big]\leq 2f(L),
\end{equation}
and
\begin{equation}\label{eq:L2L-renormalization-variance}
\emph{Var}^{\Q_L}\Big(\max_{0\le s\le 2L}\big(X^{(2)}_{s}-X^{(1)}_{s}\big)_-\Big)\leq 2f(L)^2.  
\end{equation}
The same conclusions hold with marginals $X^{(1)}\sim  \P^{\rho,2L}$ and $X^{(2)}\sim \P^{\rho + \varepsilon,L}$ instead.
\end{Lem}

\begin{proof} Towards showing \eqref{eq:L2L-renormalization-minvalue}, let us first assume that
\begin{equation}
\label{eq:coup11}
\begin{split}
&\text{there exists $L_0>3$ such that for all $L\ge L_0$, $(\rho+\log^{-2}L) \in J$ and there exists a}\\
&\text{coupling $\mathbb{Q}_L$ of $\eta^{(1)}\sim \mathbf{P}^{\rho,L}$ and $\eta^{(2)}\sim \mathbf{P}^{\rho+\varepsilon,2L}$ s.t. }
 \Q_L[G] \geq 1 - \exp(-f(L)^{1/40}),
 \end{split}
\end{equation}
where, setting $t=\lfloor \frac{f(L)}{2}\rfloor$, the `good' event $G$ is defined as
\begin{multline}
\label{eq:coup10}
 G=\big\{ \eta^{(1)}_s(x) \leq  \eta^{(2)}_s(x), \ \forall (x,s)\in [-3L,3L] \times [0,L) \big\}
 \\
 \cap\big\{ \eta^{(1)}_s(x) \leq  \eta^{(2)}_s(x-2t), \ \forall (x,s)\in [-3L, 3L]\times [L+t, 2L)\big\}.
\end{multline}
Given the above, we now extend the coupling $\mathbb{Q}_L$ to the random walks $X^{(1)}\sim P^{\eta^{(1)}}$ and $X^{(2)}\sim P^{\eta^{(2)}}$, defined as in Section \ref{subsec:rw}, up to time $2L$. To do that, we only need to specify how we couple the collections of independent uniform random variables $(U^{(1)}_w)_{w\in \mathbb{L}}$ and $(U^{(2)}_w)_{w\in \mathbb{L}}$ (recall that $\mathbb{L}$ denotes space-time, see \eqref{e:def_space_time}) used to determine the steps of each random walk; the walks $X^{(i)}$, $i=1,2$, up to time $2L$ are then specified in terms of $(U^{(i)}, \eta^{(i)})$ as in \eqref{def:A}-\eqref{eq:defX}. Under $\Q_L$, we let $(U^{(1)}_w)_{w\in \mathbb{L}}$ be i.i.d.~uniform random variable on $[0,1]$ and define, for $(x,s)\in\mathbb{L}$,
\begin{align*}
U^{(2)}_{(x,s)}=\begin{cases}
      U^{(1)}_{(x,s)}, & \text{if } 0\le s<L \\
      U^{(1)}_{(x-2t,s)}, & \text{if } L\le s<2L
    \end{cases}
\end{align*}
(for definiteness let $U^{(2)}_{(x,s)}=U^{(1)}_{(x,s)}$ when $s \geq 2L$). Clearly $(U^{(2)}_w)_{w\in \mathbb{L}}$ are i.i.d.~uniform variables, hence $X^{(2)}$ also has the desired marginal law.
Now, we will explain why, on this coupling, \eqref{eq:L2L-renormalization-minvalue} holds, and refer to Figure~\ref{f:L2L_detail} for illustration. We will only consider what happens on the event $G$ defined in \eqref{eq:coup10}.
\begin{figure}[]
  \center
\includegraphics[scale=0.8]{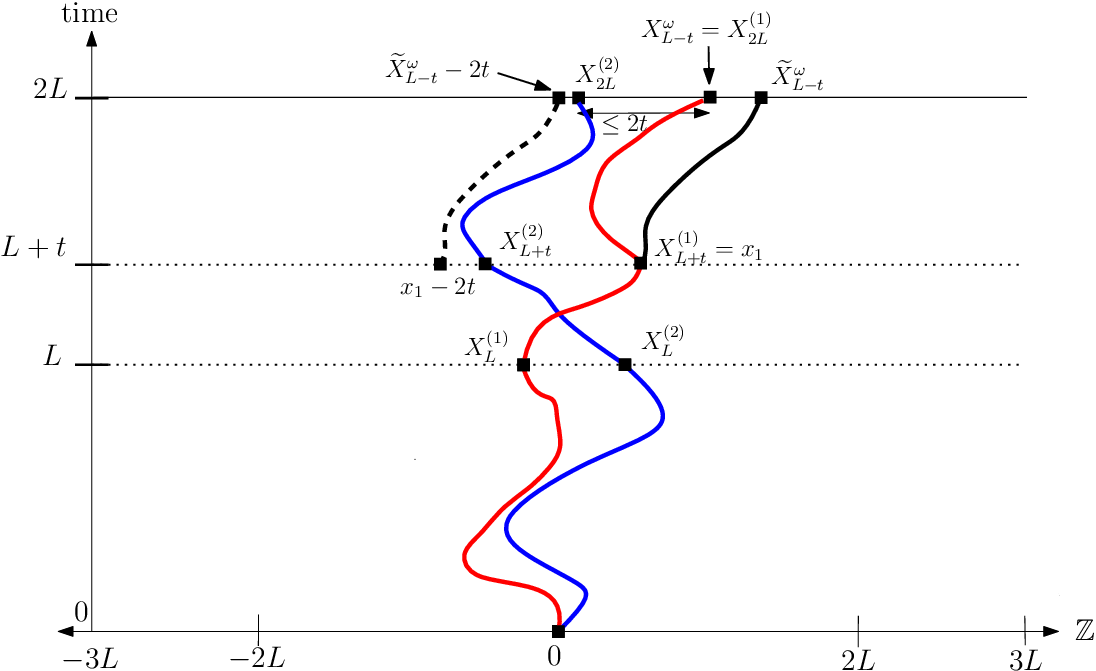}
  \caption{The trajectory of $X^{(1)}$ (resp. $X^{(2)}$) is pictured in blue (resp. red). The trajectory of $X^\omega$ coincides with that of $X^{(1)}$ during $[L+t,2L]$. The trajectory of $\widetilde{X}^\omega$ is in black (and shifted by $-2t$ in dashed black). 
On the event $G$, the points $X_{2L}^{(i)}$ for $i=1,2$ both lie in the interval $[\widetilde{X}^\omega_{L-t}-2t, \widetilde{X}^\omega_{L-t}]$, where $w=(x_1, L+t )$.}
  \label{f:L2L_detail}
\end{figure}
On $G$, by Lemma \ref{L:monotone} with $K=[-3L,3L]\times [0,L]$, we have that $X^{(2)}_s\ge X^{(1)}_s$ for all $0\le s\le L$. This implies that $X^{(2)}_{L+t}\ge X^{(1)}_{L+t}-2t$ owing to \eqref{e:lipschitz}. Let $\eta=\eta^{(1)}$, define $\widetilde{\eta}$  by $\widetilde{\eta}_n(x)=\eta^{(2)}_n(x-2t)$ and $w=(X^{(1)}_{L+t}, L+t)$. Let $X^w$ and $\widetilde{X}^w$ be random walks evolving on top of $\eta$ and  $\widetilde{\eta}$ respectively, and both using the collection of uniform random variables $U^{(1)}$. Then, conditionally on $\eta^{(1)}_{L+t}$, $\eta^{(2)}_{L+t}$ and $X^{(1)}_{L+t}$, and on the event $\{X^{(1)}_{L+t}= x_1\}$, $X^{w}_{L-t}$ has the law of $X^{(1)}_{2L}$ and  $\widetilde{X}^w_{L-t}-2t$ has the law of the position at time $2L$ of a random walk started at $(x_1-2t,L+t)$ and evolving on top of $\eta^{(2)}$, using the collection $U^{(2)}$. In particular, it evolves on the same environment as $X^{(2)}$ and starts on the left of $X^{(2)}_{L+t}$, so that $\widetilde{X}^w_{L-t}\leq X^{(2)}_{2L}$ (by Lemma \ref{L:monotone} applied with $K=\mathbb{Z}\times \mathbb{Z}^+$ and $\eta=\widetilde{\eta}$). 
On $G$, using the second event in the intersection on the right-hand side of \eqref{eq:coup10}, we can again apply Lemma \ref{L:monotone}, now at time $L+t$, with $X$ and $\widetilde{X}$ as described and $K=[-3L+2t,3L-2t]\times [L+t,2L]$. We obtain that $X^{w}_{s-t}\le \widetilde{X}^w_{s-t}$ for all $s\in[t,L]$. Thus, subtracting $2t$ on both sides and using the previous facts, we have on $G$ that $X^{(1)}_{s}-2t\le X^{(2)}_{s}$, for all $0\le s\le 2L$. All in all, we have proved that
\begin{align*}
G \subseteq \big\{\max_{0\le s\le 2L} \big(X^{(2)}_{s}-X^{(1)}_{s}\big)_-\leq 2t \big\},
\end{align*}
and therefore \eqref{eq:L2L-renormalization-minvalue} follows from our assumption \eqref{eq:coup11}. The fact that \eqref{eq:L2L-renormalization-expectation} and~\eqref{eq:L2L-renormalization-variance} hold is a simple consequence of \eqref{eq:L2L-renormalization-minvalue} and a straightforward computation using the deterministic inequality $\big(X^{(2)}_{s}-X^{(1)}_{s}\big)_-\le 4L$ valid for all $0\le s\le 2L$ and the fact that $L\ge L_0> 3$ (note that for $L$ large enough, we have $(2L)^2e^{-f(L)^{1/40}} \leq f(L)$ for all $f(L)\in [\log^{90}L,L^{1/10}]$).

It remains to show that \eqref{eq:coup11} holds, which brings into play several of the properties gathered in \S\ref{subsec:C} and which hold by assumption. Let $L_0>3$ be such that for every $L\geq L_0$ and any choice of $f(L)\geq \log^{90}L$, with $\varepsilon=f(L)^{-1/40}$ we have that $\rho+\varepsilon\in J$ and
\begin{equation}\label{eq:L2LdefL0-2}
L> t>10^{100}\nu^{-1}+\nu^{-8}+\nu^{-7}\Cr{SEPcoupling}^8\varepsilon^{-16},
\end{equation}
where $t=\lfloor f(L)/2\rfloor$ as defined above \eqref{eq:coup10}.

We will provide a step-by-step construction of $\Q_L$, in four steps. First, note that by~\eqref{pe:monotonicity}-ii), there exists a coupling $\Q_L$ of environments $(\eta^{(1)}_t(x) : x\in \mathbb{Z},t\in [0,L))$ and $(\eta^{(2)}_t(x) : x\in \mathbb{Z},t\in [0,L])$ such that under $\Q_L$, $\eta^{(1)}\sim \mathbf{P}^{\rho,L}$ on the time interval $[0,L)$, $\eta^{(2)}\sim \mathbf{P}^{\rho+\varepsilon,2L}$ on the time interval $[0,L]$ and such that $\Q_L$-a.s., $\eta^{(1)}_t(x)\leq \eta^{(2)}_t(x)$ for all $x\in \mathbb{Z}$ and $t\in [0,L)$. 
We then extend $\Q_L$ to time $L$ for $\eta^{(1)}$ by sampling $\eta^{(1)}_L\sim {\mu}_\rho$, independently of $(\eta^{(1)}_t)_{0\leq t<L}$ and $(\eta^{(2)}_t)_{0\leq t\le L}$.
In particular the above already yields that the inequalities required as part of the event $G$ in \eqref{eq:coup10} which involve $(\eta^{(i)}_t)_{0\leq t<L}$ actually hold $\mathbb{Q}_L$-a.s.

Second, letting $\ell\stackrel{\text{def.}}{=}\lfloor t^{1/4} \rfloor$, we define $G_1$ to be the event that for every interval $I\subseteq [-10L(1+{\nu}), 10L(1+{\nu})]$ of length $\lfloor \ell/2 \rfloor\le |I|\le \ell$, the inequalities $\eta^{(1)}_L(I)\leq (\rho+\varepsilon/4)\ell$ and $\eta^{(2)}_L(I)\geq (\rho+3\varepsilon/4)\ell$ hold (see the beginning of \S\ref{subsec:re} for notation). By~\eqref{pe:density} and a union bound over the choices of such intervals $I$, we get that 
\begin{equation}\label{eq:L2LE1}
\Q_L(G_1^c)\leq 40\ell L(1+{\nu}) \exp(-\Cr{densitydev} \varepsilon^2\ell/32).
\end{equation}
At time $L$ and on $G_1^c$, extend $\Q_L$ to the time interval $[L,2L]$ by letting $\eta^{(1)}$ and  $\eta^{(2)}$ follow their dynamics $P^{\eta^{(1)}_L}$ and $P^{\eta^{(2)}_L} $ independently of each other (using \eqref{pe:markov}).

Third, we continue the construction of the coupling on $G_1$ by applying \ref{pe:couplings} with $H=10L(1+\nu)$, $\eta_0=\eta^{(2)}_L(\cdot -2t)$ and $\eta'_0=\eta^{(1)}_L$, checking the assumptions using \eqref{eq:L2LdefL0-2} (note indeed that we have $1+\vert \log^3(\nu t)\vert\leq (\nu t)^{1/8}$ since $\nu t>10^{100}$, and then that $\Cr{SEPcoupling}\varepsilon^{-2}(\nu t)^{1/8}\leq \nu t^{1/4}$). This implies that, given $\eta^{(1)}_L$ and $\eta^{(2)}_L$, on $G_1$, there exists an extension of the coupling $\Q_L$ on the time interval $[L,L+t]$ such that, defining
\begin{equation}
G_2=\big\{\forall x\in [-8L-6{\nu}L , 8L+6{\nu}L ], \, \eta^{(1)}_{L+t}(x)\leq \eta^{(2)}_{L+t}(x-2t)  \big\},
\end{equation}
we have  
\begin{equation}\label{eq:L2LE2}
\Q_L(G_2^c)\leq 10\Cr{SEPcoupling2}t L(1+{\nu}) \exp\big( \textstyle-\Cr{SEPcoupling2}^{-1}\frac{{\nu}}{{\nu} +1}\varepsilon^2t^{1/4}\big).
\end{equation} 
At time $L+t$ and on $G_2^c$, we extend $\Q_L$ on the time interval $[L+t,2L]$ by letting $\eta^{(1)}$ and  $\eta^{(2)}$ follow their dynamics $P^{\eta^{(1)}_{L+t}}$ and $P^{\eta^{(2)}_{L+t}} $ independently of each other.

Fourth, we continue the construction of the coupling on $G_2$ by applying \ref{pe:drift} with $\eta_0=\eta^{(2)}_{L+t}(\cdot -2t)$ and $\eta'_0=\eta^{(1)}_{L+t}$, $H=8L+6{\nu}L$, $k=1$ and $t$ in \ref{pe:drift} equal to $L-t$. This implies that, given $\eta^{(1)}_{L+t}$ and $\eta^{(2)}_{L+t}$, on $G_2$, there exists an extension of the coupling $\Q_L$ on the time interval $[L+t,2L]$ such that, defining
\begin{equation}
G_3=\big\{\forall x\in [-8L , 8L ],\, s\in[L+t,2L], \, \eta^{(1)}_{s}(x)\leq \eta^{(2)}_{s}(x-2t)  \big\},
\end{equation}
we have
\begin{equation}\label{eq:L2LE3}
\Q_L(G_3^c)\leq 20\exp(-{\nu}({L-t})/{4}).
\end{equation} 
Finally, note that $G_1\cap G_2\cap G_3  \subseteq   G$, so that \eqref{eq:coup11} is a straightforward consequence of ~\eqref{eq:L2LE1},~\eqref{eq:L2LE2} and~\eqref{eq:L2LE3} provided that $L_0$ is chosen large enough, depending only on $\nu$ (as well as $\Cr{SEPcoupling2}$,  $\Cr{densitydev}$ and $\Cr{SEPcoupling}$). Note in particular that with our choices of $f(L)$, $\varepsilon$ above \eqref{eq:L2LdefL0-2} and $\ell (\geq c f(L)^{1/4})$, we have that $\min (L-t,\varepsilon^2t^{1/4},\varepsilon^2\ell)\geq f(L)^{1/20}$. All in all \eqref{eq:coup11} follows. The case $X^{(1)}\sim  \mathbf{P}^{\rho,2L}$ and $X^{(2)}\sim \mathbf{P}^{\rho + \varepsilon,L}$ can be treated in the same way, by means of an obvious analogue of \eqref{eq:coup11}. The remainder of the coupling (once the environments are coupled) remains the same.
\end{proof}

We are now ready to prove Proposition~\ref{Prop:vLapprox}. The proof combines the law of large numbers for the speed together
with Lemma~\ref{Lem:L2L} applied inductively over increasing scales.
The rough strategy is as follows. We aim to compare the speed of the full-range model $v(\rho)$ to the speed of the finite-range model $v_L(\rho)$ for some possibly large, but finite $L$, and prove that these two are close. For $\delta<v_L(\rho)$, we know that the probability for the finite-range model to go slower than speed $\delta$ goes to $0$ on account of Lemma~\ref{Prop:vLexists}. Thus, in a large box of size $2^K L_0$, the $L_0$-range model will most likely be faster than $\delta$ as soon as $K$ is large enough. Lemma~\ref{Lem:L2L} is used over dyadic scales to control the discrepancies between the $2^k L_0$-range model and the $2^{k+1} L_0$-range model for all $k$ from $0$ to $K-1$, and to prove that they are small. It will be seen to imply that with high probability, in a box of size $2^K L_0$, the $2^K L_0$-range model will be faster than $\delta$. Now, we only need to observe that when observed in a box of size $2^K L_0$, the $2^K L_0$-range model is equivalent to the full-range model. As Lemma~\ref{Lem:L2L} already hints at, this is but a simplified picture and the actual argument entails additional complications. This is because each increase in the range (obtained by application of Lemma~\ref{Lem:L2L}) comes not only at the cost of slightly `losing speed,' but also requires a compensation in the form of a slight increase in the density $\rho$, and so the accumulation of these various effects have to be tracked and controlled jointly.

\begin{proof}[Proof of Proposition~\ref{Prop:vLapprox}]
We only show the first inequality in Proposition~\ref{Prop:vLapprox}, i.e.~for $L \geq C(\nu)$ and $\rho$ such that $[\rho -(\log L)^{-1},\rho +(\log L)^{-1}]\subseteq J$, abbreviating $\alpha_L= \frac{(\log L)^{100}}{L}$, one has
\begin{equation}\label{eq:vLapproxoneside}
v_{L}(\rho - (\log  L)^{-1})-\alpha_L\leq v(\rho).
\end{equation}
The first inequality in \eqref{eq:vLapprox} then follows for all $L \geq 3$ by suitably choosing the constant $\Cr{C:approx}$ since $v,v_L \in [-1,1]$. %
The second inequality of \eqref{eq:vLapprox} is obtained by straightforward adaptation of the arguments below, using the last sentence of Lemma~\ref{Lem:L2L}.

For $L \geq 1$, define $L_k= 2^kL$ for $k \geq 0$. As we now explain, the conclusion \eqref{eq:vLapproxoneside} holds as soon as for $L \geq C(\nu)$ and $\rho$ as above, we show that
\begin{equation}\label{newconc}
\mathbb{P}^{\rho}\big(({X_{L_K}}/ L_K) \leq v_{L}(\rho -( \log L)^{-1})-\alpha_L\big)\to 0\text{, as }K\to\infty.
\end{equation}
Indeed under the assumptions of Proposition~\ref{Prop:vLapprox}, the law of large numbers \eqref{eq:intro-LLN} holds, and therefore in particular, for all $\delta>v(\rho)$, we have that $\mathbb{P}^{\rho}\left({X_{L_K}}\leq \delta {L_K}\right)$ tends to $1$ as $K \to \infty$. Together with \eqref{newconc} this is readily seen to imply \eqref{eq:vLapproxoneside}.

We will prove \eqref{newconc} for $L \geq C(\nu)$, where the latter is chosen such that the conclusions of Lemma \ref{Lem:L2L} hold for $L$, and moreover such that
\begin{equation}\label{elodie} 
 (\log L)^{-9/4}+\frac{100}{(\log L)^{5/4}} \le\frac{1}{\log L} \text{, }\log L\ge 5 \text{ and } \textstyle \left(\frac{3}{2}\right)^{-k}
\big(k\frac{\log2}{\log L} + 1\big)^{99}\le 1\text{, for all } k\ge0.
\end{equation}
For such $L$ we define, for all integer $K,k \geq 0$, all $\rho>0$ and all $\delta\in \R$, recalling the finite-range annealed measures $\mathbb{P}^{\rho, L}$ from  \S\ref{subsec:finiterange},
\begin{equation}\label{def:prhok}
\begin{split}
&p_{\rho,K,\delta}^{(k)}=\P^{\rho,L_k}(X_{L_K}\leq \delta {L_K}), \\
&p_{\rho,K,\delta}^{(\infty)}=\P^{\rho}(X_{L_K}\leq \delta {L_K})
\end{split}
\end{equation}
(observe that the notation is consistent with \S\ref{subsec:finiterange}, i.e.~$\mathbb{P}^{\rho,\infty}= \mathbb{P}^{\rho}$).
In this language \eqref{newconc} requires that $p_{\rho,K,\delta}^{(\infty)}$ vanishes in the limit $K \to \infty$ for a certain value of $\delta$. We start by gathering a few properties of the quantities in \eqref{def:prhok}.
For all $\rho'\in J$, the following hold:
\begin{align}\label{propofp}
& \lim_{K\rightarrow +\infty}p_{\rho',K,\delta}^{(0)}=0\text{ for all }\delta< v_L(\rho') \text{ (by \eqref{eq:LLNvL}),}\\ \label{propofp2}
& p_{\rho',K,\delta}^{(\infty)}=p_{\rho',K,\delta}^{(k)}\text{, for all }\delta\in[-1,1],\ K\ge0 \text{ and }k\ge K,\\
\label{propofp3}
& p_{\rho',K,\delta}^{(\infty)}\text{ and } p_{\rho',K,\delta}^{(k)}\text{ are non-increasing in $\rho'$ and non-decreasing in $\delta$.}
\end{align}
As explained atop the start of the proof, owing to the form of Lemma~\ref{Lem:L2L} we will need to simultaneously sprinkle the density and the speed we consider in order to be able to compare the range-$L_{k+1}$ model to the range-$L_{k}$ model. To this effect, let
\begin{equation} \label{eq:rho_k}
\rho_0=\rho-(\log L)^{-1}\text{ and } \rho_{k+1}=\rho_{k}+{(\log L_k)^{-9/4}}\text{, for all }k\ge0,
\end{equation}
as well as
\begin{equation}\label{def:delta0delta}
\delta_0=v_{L}(\rho_0)-{L^{-2}}\text{ and } \delta_{k+1}=\delta_{k}- {\log^{99}(L_k)}/L_k\text{, for all }k\ge0.
\end{equation}
A straightforward computation, bounding the sum below for $k\ge1$ by the integral $\int_0^\infty dx/(x\log 2 +\log L)^{9/4}$, yields that
\begin{equation}\label{ilpleut}
\lim_k \rho_k = \rho_0 + \sum_{k\ge0} (\log L_k)^{-9/4}\le \rho_0 + (\log L)^{-9/4}+\frac{100}{(\log L)^{5/4}} \stackrel{\eqref{elodie}}{\le} \rho_0 +\frac{1}{\log L} \stackrel{\eqref{eq:rho_k}}{\leq} \rho.
\end{equation}
Another straightforward computation yields that
\begin{multline}\label{ilpleutencore}\sum_{k\ge0}\frac{(\log L_k)^{99}}{L_k} = 
\sum_{k\ge0}\left(\frac{3}{2}\right)^{-k}
\left(k\frac{\log2}{\log L} + 1\right)^{99} \times  \frac{(\log L)^{99}}{L}\left(\frac34\right)^k\\
\stackrel{\eqref{elodie}}{ \le}  \frac{(\log L)^{99}}{L}\sum_{k\ge0} \left(\frac34\right)^k 
\stackrel{\eqref{elodie}}{ \le} \frac{(\log L)^{100}}{L}-\frac{1}{L^2}.
\end{multline}
In particular, since $(\rho_k)$ is increasing in $k$, \eqref{ilpleut} implies that for all $K\ge0$, we have $\rho\geq \rho_K$, and \eqref{ilpleutencore} yields in view of \eqref{def:delta0delta} that $\delta_K >v_{L}(\rho_0)-\alpha_L$, with $\alpha_L (=L^{-1}{(\log L)^{100}})$ as above \eqref{eq:vLapproxoneside}.
Using this, it follows that, for all $K\ge 0$,
\begin{multline}\label{eq:prhoHdelta-split}
p^{(\infty)}_{\rho,K,v_{L}(\rho_0)-\alpha_L}\stackrel{\eqref{propofp2}}{=}p_{\rho,K,v_{L}(\rho_0)-\alpha_L}^{(K)} \\\stackrel{\eqref{propofp3}}{\le} p_{\rho_K,K,\delta_K}^{(K)}     = p_{\rho_0,K,\delta_0}^{(0)}+\sum_{0 \leq k < K}\left(p_{\rho_{k+1},K,\delta_{k+1}}^{(k+1)} - p_{\rho_{k},K,\delta_{k}}^{(k)}\right).
\end{multline}
As the left-hand side in \eqref{eq:prhoHdelta-split} is precisely equal to the probability appearing in \eqref{newconc}, it is enough to argue that the right-hand side of \eqref{eq:prhoHdelta-split} tends to $0$ as $K \to \infty$ in order to conclude the proof.
By recalling that $\delta_0<v_L(\rho_0)$ from \eqref{def:delta0delta} and using \eqref{propofp}, we see that $\lim_{K\to\infty} p_{\rho_0,K,\delta_0}^{(0)}=0$, which takes care of the first term on the right of \eqref{eq:prhoHdelta-split}. 

We now aim to show that the sum over $k$ in \eqref{eq:prhoHdelta-split} vanishes in the limit $K \to \infty$, which will conclude the proof. Lemma~\ref{Lem:L2L}
now comes into play. Indeed recalling the definition \eqref{def:prhok}, the difference for fixed value of $k$ involves walks with range $L_k$ and $L_{k+1}$, and Lemma~\ref{Lem:L2L} supplies a coupling allowing good control on the negative part of this difference (when expressed under the coupling). Specifically, for a given $K \geq 1$ and
$0\le k \le K-1$, let $X^{(1)}\sim \mathbb{P}^{\rho_{k+1},L_{k+1}}$ and $X^{(2)}\sim \mathbb{P}^{\rho_k,L_k}$. Note that for $i\in\{1,2\}$, one has the rewrite
$$X^{(i)}_{L_K}=\sum_{0 \leq \ell< 2^{K-k-1}} \left( X^{(i)}_{(\ell+1)L_{k+1}}-X^{(i)}_{\ell L_{k+1}}  \right).
$$
Therefore, due to the regenerative structure of the finite-range model, explicated in \eqref{eq:regenstructure}, it follows that, for $i\in\{1,2\}$, under $\mathbb{P}^{\rho_{k+2-i},L_{k+2-i}}$,
\begin{equation}
X^{(i)}_{L_K}\stackrel{\text{law}}{=} \sum_{0 \leq \ell< 2^{K-k-1}} X^{(i,\ell)}_{L_{k+1}},
\end{equation}
 where $X^{(i,\ell)}_{L_{k+1}}$, $\ell\ge0$, is a collection of independent copies of $X^{(i)}_{L_{k+1}}$ under $\mathbb{P}^{\rho_{k+2-i},L_{{k+2-i}}}$.\\
 Now recall the coupling measure $\mathbb{Q}_L$ provided by Lemma \ref{Lem:L2L} with $f(L)=(\log L)^{90}$ and  let us denote by $\mathbb{Q}$ the product measure induced by this couplingfor the choices $L=L_k$ and $\varepsilon= \rho_{k+1}-\rho_k$ in Lemma \ref{Lem:L2L}, so that $\mathbb{Q}$ supports the i.i.d.~family of pairs $(X^{(1,\ell)}_{L_{k+1}},X^{(2,\ell)}_{L_{k+1}})$, $\ell\ge0$, each sampled under $\mathbb{Q}_{L_k}$. In particular, under $\mathbb{Q}$, for all $\ell\ge0$, $X^{(1,\ell)}_{L_{k+1}}$ and $X^{(2,\ell)}_{L_{k+1}}$ have law $\mathbb{P}^{\rho_{k+1},L_{k+1}}$ and $\mathbb{P}^{\rho_{k},L_{k}}$, respectively. Now, one can write, for any $K\ge1$ and $0\le k\le K-1$, with the sum over $\ell$ ranging over $0 \leq \ell< 2^{K-k-1}$ below, that
\begin{multline*}
  p_{\rho_{k+1},K,\delta_{k+1}}^{(k+1)} - p_{\rho_{k},K,\delta_{k}}^{(k)}
\stackrel{\eqref{def:prhok}}{=} \mathbb{Q}\Big(\sum_{\ell} X^{(1,\ell)}_{L_{k+1}}\le L_K \delta_{k+1}\Big) -  \mathbb{Q}\Big(\sum_{\ell} X^{(2,\ell)}_{L_{k+1}}\le L_K \delta_{k}\Big)\\
\le \mathbb{Q}\Big(\sum_{\ell} X^{(1,\ell)}_{L_{k+1}}\le L_K \delta_{k+1},\ \sum_{\ell} X^{(2,\ell)}_{L_{k+1}}> L_K \delta_{k}\Big)
\stackrel{\eqref{def:delta0delta}}{\le} \mathbb{Q}\Big(\sum_{\ell} (X^{(1,\ell)}_{L_{k+1}}- X^{(2,\ell)}_{L_{k+1}})\le -2^{K-k}\log^{99}(L_k)\Big).\end{multline*}
Now, using Chebyshev's inequality together with Lemma \ref{Lem:L2L} (recall that $f(L)=\log^{90}L$), it follows that for $K\ge1$ and $0\le k\le K-1$,
\begin{equation}\begin{split} \label{eq:cheby-final}
 p_{\rho_{k+1},K,\delta_{k+1}}^{(k+1)} - p_{\rho_{k},K,\delta_{k}}^{(k)}
&\le \frac{2^{K-k}(\log L_k)^{190}}{\big( 2^{K-k}(\log L_k)^{99}-2^{K-k}(\log L_k)^{90}  \big)^2}\\
&\le \frac{7}{2^{K-k}(k+5)^8}\le 2^{-\tfrac{K}{2}}+224K^{-8},
\end{split}
\end{equation}
where the second line is obtained by considering the cases when $k< K/2$ or $k\ge K/2$ separately, together with straightforward computations. The bound \eqref{eq:cheby-final} implies in turn that
$$\sum_{k=0}^{K-1}\left(p_{\rho_{k+1},K,\delta_{k+1}}^{(k+1)} - p_{\rho_{k},K,\delta_{k}}^{(k)}\right)\le K2^{-\tfrac{K}{2}}+224K^{-7} \stackrel{K\to\infty}{\longrightarrow} 0,
$$
which concludes the proof.
\end{proof}

\section{Quantitative monotonicity for the finite-range model}\label{sec:initialspeed}
The goal of this section is to prove Proposition~\ref{Prop:initialspeed}. For this purpose, recall that $J$ is an open interval, and fix $\rho\in J$ and $\epsilon>0$ such that $(\rho+\epsilon)\in J$. Morever, in view of~\eqref{eq:vLmonotonic}, we can assume that $\epsilon<1/100$. The dependence of quantities on $\rho$ and $\epsilon$ will be explicit in our notation. As explained above Proposition~\ref{Prop:initialspeed}, even if we are dealing with the finite-range model, the current question is about estimating the expectation in \eqref{eq:defvL}, which is actually equivalent to working on the full-range model. Thus, we retain much of the difficulty, including the fact that the environment mixes slowly. The upshot is that the speed gain to be achieved is quantified, and rather small, cf.~the right-hand side of \eqref{eq:vLincrease}.

The general idea of the proof is as follows. First, recall that we want to prove that at time $L$, the expected position of $X^{\rho+\epsilon}\sim \mathbb{P}^{\rho+\epsilon}$ (where $\sim$ denotes equality in law in the sequel) is larger than the expected position of $X^\rho\sim \mathbb{P}^{\rho}$ by $3(\log L)^{100}$. The main conceptual input is to couple $X^\rho$ and $X^{\rho+\epsilon}$ in such a way that, after a well-chosen time $T \ll L$, we create a positive discrepancy between them with a not-so-small probability, and that this discrepancy is negative with negligible probability, allowing us to control its expectation. We will also couple these walks in order to make sure that the environment seen from $X^\rho_T$ is dominated by the environment seen from $X^{\rho+\epsilon}_T$, even if these walks are not in the same position, and we further aim for these environments to be `typical'. The last two items will allow us to repeat the coupling argument several times in a row and obtain a sizeable gap between $X^\rho_L$ and $X^{\rho+\epsilon}_L$. In more quantitative terms, we choose below $T=5(\log L)^{1000}$ and create an expected discrepancy of $\exp(-(\log L)^{1/20})$ at time $T$. Repeating this procedure $L/T$ times provides us with an expected discrepancy at time $L$ larger than $3(\log L)^{100}$ (and in fact larger than $L^{1+o(1)}$). We refer to the second part of Section~\ref{sec:proofsketch} for a more extensive discussion of how the expected gap size comes about.

We split the proof of Proposition~\ref{Prop:initialspeed} into two parts. The main part (Section~\ref{subsec:onegap}) consists of constructing a coupling along the above lines. In Section~\ref{subsec:chaining}, we prove Proposition~\ref{Prop:initialspeed}.

\subsection{The trajectories $Y^\pm$}\label{subsec:onegap}

In this section, we define two discrete-time processes $Y_t^{-}$ and $Y^+_t$, $t \in [0,T]$, for some time horizon $T$, see \eqref{eq:ellTdef} below, which are functions of two deterministic environments $\eta^-$ and $\eta^{+}$ and an array $U=(U_w)_{w\in \bL}$  (see Section~\ref{sec:prelims} for notation) of numbers in $[0,1]$. This construction will lead to a deterministic estimate of the difference $Y^+_t-Y^-_t$, stated in Lemma~\ref{L:determ-speed}. In the next section, see Lemma~\ref{L:Q coupling-alternative}, we will prove that there exists a measure $\Q$ on $(\eta^\pm, U)$ such that under $\Q$, $Y^-$ dominates stochastically the law of a random walk $X^{\rho}\sim\mathbb{P}^{\rho}$ (see below \eqref{eq:RW_an} for notation), $Y^+$ is stochastically dominated by the law of a random walk $X^{\rho+\epsilon} \sim \mathbb{P}^{\rho+\epsilon}$, and we have a lower bound on $\mathbb{E}^\Q[Y^+_{T}-Y^-_{T}]$, thus yielding a lower bound on $\mathbb{E}^{\rho+\epsilon}[X^{\rho+\epsilon}] -\mathbb{E}^{\rho}[X^{\rho}] $.

\bigskip

The construction of the two processes $Y^{\pm}$ will depend on whether some events are realized for $\eta^\pm$ and $U$. We will denote $E_1, E_2,\ldots$ these events, which will occur (or not) successively in time. We introduce the convenient notation
\begin{equation} \label{eq:E_i-j}
E_{i-j}\stackrel{\text{def.}}{=}E_i\cap E_{i+1}\cap  \ldots\cap E_j, \text{ for all $j>i\geq 1$,}
\end{equation}
and write $E_{i-j}^c$ for the complement of $E_{i-j}$. We refer to Figure~\ref{f:One_step_gain} (which is a refined version of Figure~\ref{f:couplingmaster_intro}) for visual aid for the following construction of  $Y^{\pm}$, and to the discussion in Section~\ref{sec:proofsketch} for intuition. Let us give a brief outlook on what follows. The outcome of the construction will depend on whether a sequence of events $E_1$-$E_9$ happens or not: when all of these events occur, which we call a {\em success}, then $Y_T^+$ and $Y^-_T$  have a positive discrepancy and we retain a good control on the environments at time $T$, namely $\eta^+$ dominates $\eta^-$ (or, more precisely, $\eta^-$ shifted by  $Y_T^+-Y^-_T$). When at least one of these nine events does not happen, this can either result in a {\em neutral} event, where the trajectories end up at the same position and we still have control on the environments, or it could result in a {\em bad} event, where we can have $Y_T^+<Y^-_T$ (but we will show in Lemma~\ref{L:Q coupling-alternative} that we keep control on the environments with overwhelming probability, owing to the \textit{parachute coupling} evoked at the end of Section~\ref{sec:proofsketch}). Whenever we observe a neutral or a bad event, we will exit the construction and define the trajectories $Y^\pm$ at once from the time of observation all the way up to time $T$. Finally, we note that the walks $Y^+$ and $Y^-$ are actually not Markovian. This is however not a problem as we use them later to have bounds on the expected displacements of the actual walks $X^{\rho}$ and $X^{\rho+\varepsilon}$.

\begin{figure}[]
  \center
\includegraphics[scale=0.80]{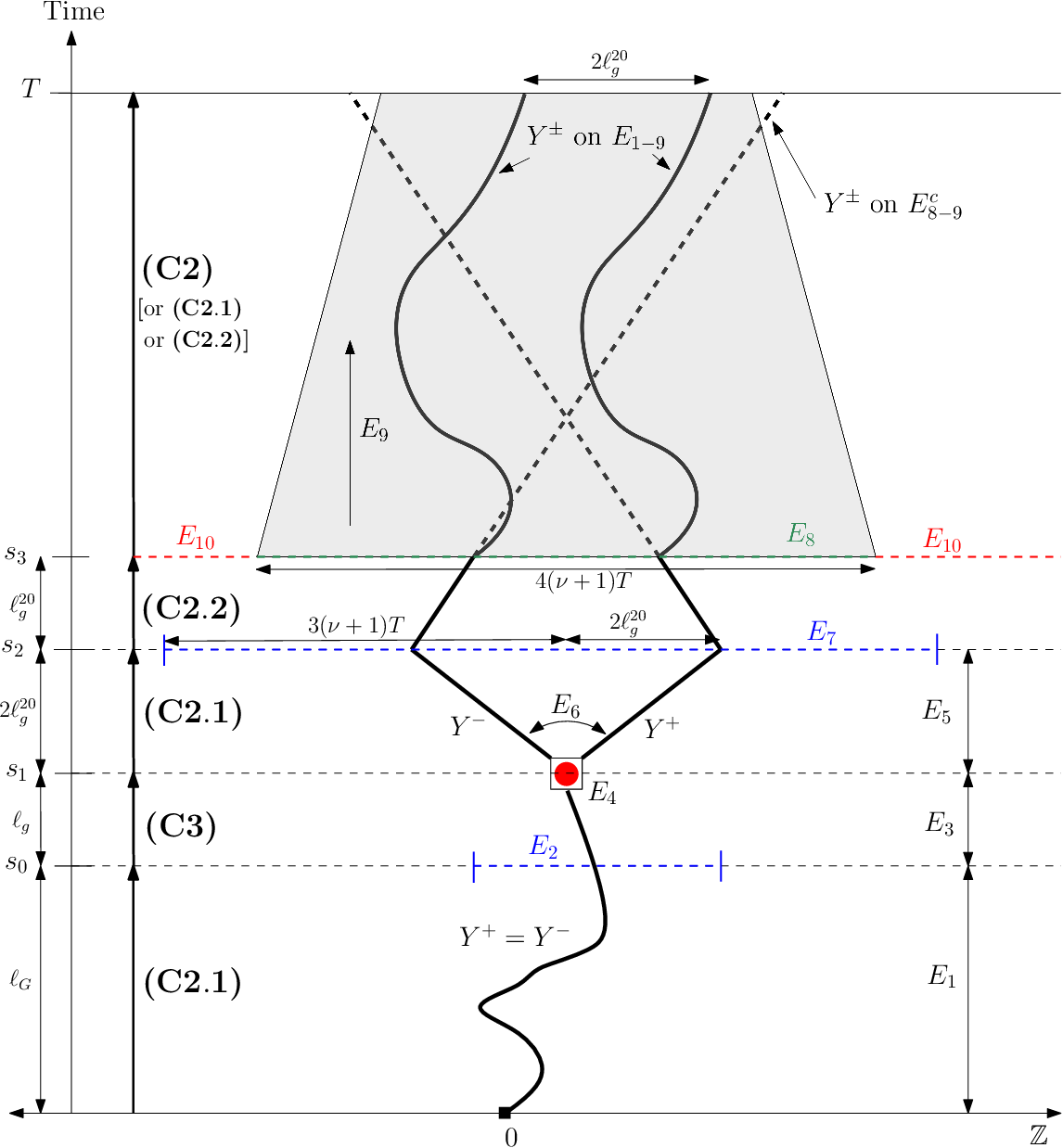}
  \caption{Trajectories of $Y^{\pm}$ during a time-interval of length $T$. For readability, we only picture the scenario when the events $E_i$ are all realized (except during $[s_3,T]$). The picture is not up to scale (all linear stretches on the trajectories of $Y^{\pm}$ have slope 1, those of the grey trapezoid have slope $2\nu$). The red disk represents a \textit{sprinkler} (extra particle of $\eta^+\setminus \eta^-$). The arrows on the left indicate which of the coupling conditions \ref{pe:densitychange}-\ref{pe:nacelle} from Section~\ref{subsec:C} is being used to perform the coupling in the given time interval. This is at the core of the proof of Lemma~\ref{L:Q coupling-alternative}. The various couplings involved in the time interval $[s_3,T]$ meet the necessity of having to restore domination of environments (possibly with a relative spatial shift) on a \emph{full} space interval of length $M$ comparable to $L$ with high probability, so as to be able to repeat proceedings in the next step.}
  \label{f:One_step_gain}
\end{figure}

We now proceed to make this precise. For $L\geq 3$, define 
\begin{equation}\label{eq:ellTdef}
\text{$\ell_G=\lfloor (\log L)^{1000} \rfloor $, $\ell_g = \lfloor (\log L)^{1/1000}\rfloor$ and $T=5\ell_G$,}
\end{equation}
as well as
\begin{equation}\label{eq:s_0}
s_0= \ell_G, \, s_1= s_0+\ell_g, \,  s_2= s_1+2\ell_g^{20}\text{ and }s_3=s_2+\ell_g^{20}.
\end{equation}
Let $M\geq 10(\nu+1)L$, which will parametrize the spatial length of a space-time box in which the entire construction takes place, and let $Y_{0}^{\pm}=0$. The definition of $Y^\pm$ depends on the values of $M$ and $L$, but we choose not to emphasize it in the notation. Given processes 
\begin{equation}\label{e:omega-coup}
 \gamma = (\eta^{+}, \eta^-,  U)
 \end{equation}
 on a state space $\Omega$ such that the first two coordinates take values in $ ( \mathbb Z_+ )^{[0,T] \times \mathbb Z} $ and the last one in $[0,1]^{\bL}$, we will define the events $E_i=E_i(\gamma)$ below measurably in $\gamma$ and similarly $Y^{\pm} =(Y^{\pm}_t(\gamma))_{ 0 \leq t \leq T} $ with values in the set of discrete-time trajectories starting at $0$. Unlike the sample paths of  $X^{\rho}, X^{\rho+\varepsilon}$, the trajectories of $Y^{\pm}$ may perform jumps that are not to nearest neighbors.   Probability will not enter the picture until Lemma~\ref{L:Q coupling-alternative} below, see in particular \eqref{e:coup-marg}, which specifies the law of $\gamma$. We now properly define the aforementioned three scenarii of success, neutral events and bad events, which will be mutually exclusive, and we specify $Y^{\pm}$ in all cases.
 
 \bigskip

Below $s$ and $t$ always denote integer times. Define the event (see Section~\ref{subsec:re} regarding the notation $\succcurlyeq$) 
\begin{align}\label{eq:E1deff}
E_{1}=&\{\forall s\in [0,s_0], \, \eta^+_s\vert_{[-M+2{\nu}s_0, M-2{\nu}s_0]}\succcurlyeq\eta^-_s\vert_{[-M+2{\nu}s_0, M-2{\nu}s_0)]}  \}.
\end{align}
The event $E_1$ guarantees the domination of the environment $\eta^-$ by $\eta^+$ and is necessary for a success, while $E_1^c$ will be part of the bad event.
On $E_{1}$, for all $1\leq t\leq s_0$, let recursively
\begin{equation}\label{eq:Ypmrecursive}
Y^-_t=Y^+_t=Y^-_{t-1}+A(\eta^-_{t-1}(Y^-_{t-1}), U_{(Y^-_{t-1},t-1)})
\end{equation}
where $A$ was defined in~\eqref{def:A}, so that we let the walks evolve together up to time $s_0$. On the bad event $E_1^c$, we exit the construction by defining
\begin{equation}\label{badexit1}
Y^-_t=t \text{ and } Y^+_t=-t\text{, for all } 1\leq t\leq T.
\end{equation}
Next, define the events
\begin{align}  \label{eq:E2deff}
E_{2}=&\{ \eta^+_{s_0}([Y^-_{s_0}, Y^-_{s_0}+\ell_g])>\eta^-_{s_0}([Y^-_{s_0}, Y^-_{s_0}+\ell_g])  \}\\ \nonumber
&\cap  \{\eta^-_{s_0}([Y^-_{s_0}-3\ell_g+1, Y^-_{s_0}+3\ell_g]) \leq (\rho+1)\cdot 6\ell_g\},\\[0.3em] \label{eq:E3deff}
E_{3}= & \{\forall s\in [s_0, s_1], \,\eta^+_{s}\vert_{[-M+(4\nu+1) s_1, M-(4\nu+1) s_1]} \succcurlyeq \eta^-_{s}\vert_{[-M+(4\nu+1) s_1, M-(4\nu+1) s_1]} \}.
\end{align}
The event $E_3$ again ensures the domination of $\eta^-$ by $\eta^+$ on a suitable spatial interval while $E_2$ creates favourable conditions at time $s_0$ to possibly see a sprinkler at time $s_1$.
On the event $E_{1-3}$ (recall \eqref{eq:E_i-j} for notation), for all $s_0+1\leq t\leq s_1$, let recursively
\begin{equation}\label{eq:Ypmrecursive1}
Y^-_t=Y^+_t=Y^-_{t-1}+A(\eta^-_{t-1}(Y^-_{t-1}), U_{(Y^-_{t-1},t-1)}),
\end{equation}
so that the walks, from time $s_0$, continue to evolve together up to time $s_1$ and, in particular, we have that
\begin{equation}\label{Ypmats1}
\text{on } E_{1-3}, \quad Y^-_{s_1}=Y^+_{s_1}.
\end{equation}
In order to deal with the case where $E_2$ or $E_3$ fail, we will distinguish two mutually exclusive cases, that will later contribute to an overall neutral and bad event, respectively; cf.~\eqref{def:softfailure} and \eqref{def:catastrophe}. To this end, we introduce
\begin{equation}\label{eq:E2bisdeff}
E_{2,\textnormal{bis}}=\{\forall s\in [s_0,T], \, \eta^+_s\vert_{[-M+2{\nu}T, M-2{\nu}T]}\succcurlyeq\eta^-_s\vert_{[-M+2{\nu}T, M-2{\nu}T]}  \}.
\end{equation}
On $E_1\cap E_2^c\cap E_{2,\textnormal{bis}}$, which will be a neutral event, we exit the construction by letting the walk evolve together up to time $T$, that is,  we define
\begin{equation}\label{eq:Ypmrecursive2}
Y^-_t=Y^+_t=Y^-_{t-1}+A(\eta^-_{t-1}(Y^-_{t-1}), U_{(Y^-_{t-1},t-1)})\text{, for all } s_0+1\le t\le T.
\end{equation}
On the bad event $\big(E_1\cap E_2^c\cap E_{2,\textnormal{bis}}^c\big)\cup\left(E_{1-2}\cap E_3^c\right)$, we exit the construction by defining
\begin{equation}\label{badexit2}
Y^-_t=t \text{ and } Y^+_t=-t\text{, for all } s_0+1\leq t\leq T.
\end{equation}
In the above, $Y^+$ and $Y^-$ may take a non-nearest neighbour jump at time $s_0$, which is fine for our purpose. Next, define
\begin{align}\label{eq:E4deff}
E_4=& \{\eta^+_{s_1}(Y_{s_1}^{ + \black})\geq 1,\eta^-_{s_1}(Y_{s_1}^{ - \black})=0\},\\ \label{eq:E5deff}
E_5=&\{ \forall s\in [s_1,s_2], \,\eta^+_{s}\vert_{[-M+(6\nu+1) s_2, M-(6\nu+1) s_2]} \succcurlyeq \eta^-_{s}\vert_{[-M+(6\nu+1) s_2, M-(6\nu+1) s_2]} \},
\end{align}
where $E_4$ states that the walkers  see a sprinkler at time $s_1$ and $E_5$ guarantees domination of the environments from time $s_1$ to $s_2$.  We will first define what happens on the neutral and the bad events. For this purpose, define
\begin{equation}\label{eq:E4bisdeff}
E_{4,\textnormal{bis}}=\{\forall s\in [s_1,T], \, \eta^+_s\vert_{[-M+(4{\nu}+1) T, M-(4{\nu}+1) T]}\succcurlyeq\eta^-_s\vert_{[-M+(4{\nu}+1) T, M-(4{\nu}+1) T]}  \}.
\end{equation}
Recall that on $E_{1-3}$, we have defined $Y^\pm$ up to time $s_1$. On the neutral event $E_{1-3}\cap E_{4-5}^c\cap E_{4,\textnormal{bis}}$, we define
\begin{equation}\label{eq:Ypmrecursive3}
Y^-_t=Y^+_t=Y^-_{t-1}+A(\eta^-_{t-1}(Y^-_{t-1}), U_{(Y^-_{t-1},t-1)})\text{, for all } s_1+1\le t\le T.
\end{equation}
On the bad event $E_{1-3}\cap E_{4-5}^c\cap E_{4,\textnormal{bis}}^c$, we exit the construction by defining
\begin{equation}\label{badexit3}
Y^-_t=t \text{ and } Y^+_t=-t\text{, for all } s_1+1\leq t\leq T.
\end{equation}
On the event $E_{1-5}$, because of the sprinkler, the walkers have a chance to split apart hence, for $t\in [s_1+1,s_2]$, we let
\begin{equation}\label{eq:YpmE1-5}
Y^-_t=Y^-_{t-1} +A(\eta^-_{t-1}(Y^-_{t-1}), U_{ ( Y^-_{t-1} ,t-1)   }   ) \text{ and } Y^+_t=Y^+_{t-1} +A(\eta^+_{t-1}(Y^+_{t-1}), U_{ ( Y^+_{t-1} ,t-1)   }   ).
\end{equation}
Above, $Y^+$ and $Y^-$ evolve on top of their respective environments $\eta^+$ and $\eta^-$ from time $s_1+1$ to time $s_2$. To be able to continue the construction from time $s_2$ to $s_3$, we define two events $E_6$ and $E_7$ concerning the environments from times $s_1$ to $s_2$. If we are on $E_{1- 6\black}$, between times $s_1+1$ and $s_2$ we will require that $Y^+$ and $Y^-$ drift away, regardless of the states of $\eta^+$ and $\eta^-$, using only the information provided by $U$. For this purpose, define
\begin{equation}
\label{eq:E6deff}
E_6= \bigcap_{s_1\leq t\leq s_2-1} E_6^t,
\end{equation}
where we set 
\begin{equation}
\label{eq:E6tdeff}
\begin{split}
&E^{s_1}_6=\big\{U_{(Y^-_{s_1},s_1)} \in (p_{\circ}, p_{\bullet})\big\} \text{ (recall that $p_\bullet >p_\circ$) and}\\
&E^t_6=\big\{ U_{ (  Y_{s_1}^- -(t-s_1) \black  ,t)} >  p_{\bullet}  \big\}\cap \big\{  U_{ (  Y^-_{s_1}+ t-s_1 \black ,t)}< p_{\circ}  \big\}\text{ for } s_1+1\leq t \leq s_2-1.
\end{split}
\end{equation}
On $E_4\cap E_6^{s_1}$, $Y^-$ steps to the left and $Y^+$ steps to the right from their common position, thus creating a gap at time $s_1+1$\black. Then for all $s_1<t< s_2$, as long as $E_4$ and $E_6^{s_1}$ to $E_6^{t-1}$ happen, $Y^-_t$ and $Y^+_t$ are at distinct positions and $E_6^t$ allows $Y^-$ to take one more step to the left and $Y^+$ one more step to the right. Hence, using what we have constructed so far on $E_{1-5}$, we have that
\begin{equation}\label{e:Ypm at s_1}
\text{on } E_{1-6}, \quad Y^+_{s_2}=Y^-_{s_2}+2(s_2-s_1),
\end{equation}
and we still have the domination of $\eta^-$ by $\eta^+$ at time $s_2$, cf.~\eqref{eq:E5deff}.

Since $Y^+$ and $Y^-$ are no longer at the same position, we are going to momentarily allow to lose this domination in order to recreate it at time $s_3$ but in a suitably shifted manner, namely, achieve that $\eta^+(Y^+_{s_3}+\cdot)\vert_I \succcurlyeq \eta^-(Y^-_{s_3}+\cdot) \vert_I$ for a suitable interval $I$, see~\eqref{eq:E8deff}. To do so, we first need favourable conditions at time $s_2$ encapsulated by the event
\begin{equation}
\label{eq:E7deff}
E_7=\left\{ \begin{array}{c} \text{for all intervals } I\subseteq [Y^-_{s_1} - 3(\nu+1)T,\, Y^-_{s_1} + 3(\nu+1)T]\text{ with }\\
\lfloor \ell_g^2 /2\rfloor \leq \vert I\vert \leq \ell_g^2,  \eta^+_{s_2}(I)\geq (\rho+3\epsilon/4) \vert I \vert\text{ and }\eta^-_{s_2}(I)\leq (\rho+\epsilon/4) \vert I \vert \end{array}
\right\}
\end{equation}
The above requires good empirical densities at time $s_2$ on an interval of length of order $T$ centred around the common position of the walkers $Y^\pm$ at time $s_1$. On $E_{1-7}$, we do not precisely control the position of the walkers from time $s_2$ to $s_3$ and define, for all $s_2< t\le s_3$,
\begin{equation}
\label{eq:Ypms2s3}
Y_{t}^- -Y_{t-1}^- = 1, \quad Y_{t}^+ - Y_{t-1}^+ = -1,
\end{equation}
which corresponds to the worst case scenario assuming nearest-neighbour jumps. In particular, using \eqref{eq:Ypms2s3}, \eqref{e:Ypm at s_1} and \eqref{eq:s_0}, we have that
\begin{equation}\label{eq:YpmE1-7}
\text{on } E_{1-7}, \quad Y^+_{s_3}-Y^-_{s_3}=2\ell_g^{20}.
\end{equation}
We now need to consider the case where $E_{6-7}$ fails, and we will again distinguish two types of failure. For this purpose, define
\begin{equation}\label{eq:E6bisdeff}
E_{6,\textnormal{bis}}=\{ \forall s\in [s_2,T], \,\eta^+_{s}\vert_{[-M+(6\nu+1) T, M-(6\nu+1) T]} \succcurlyeq \eta^-_{s}\vert_{[-M+(6\nu+1) T, M-(6\nu+1) T]} \},
\end{equation}
and, on the neutral event $E_{1-5}\cap E_{6,7}^c\cap E_{6,\textnormal{bis}}$, we merge $Y^+$ with $Y^-$ at time $s_2+1$ and then let them walk together, by defining
\begin{align}
\label{eq:Y+movedtoY-s2}
Y_{s_2+1}^+&=Y_{s_2+1}^-= Y_{s_2}^- + A\big(\eta_{s_2}^{-}(Y^-_{s_2}), U_{(Y^-_{s_2},s_2)}\big), \text{ and }\\
\label{eq:Ypmrecursive4}
 Y^-_t&=Y^+_t=Y^-_{t-1}+A(\eta^-_{t-1}(Y^-_{t-1}), U_{(Y^-_{t-1},t-1)})\text{, for all } s_2+2\le t\le T
\end{align}
(above one line would be sufficient but we single out \eqref{eq:Y+movedtoY-s2} because the merging will typically occasion a jump for $Y^+$). \black  On the remaining bad event $E_{1-5}\cap E_{6-7}^c\cap E_{6,\textnormal{bis}}^c$, we exit the construction in the now usual way by defining
\begin{equation}\label{badexit4}
Y^-_t=t \text{ and } Y^+_t=-t\text{, for all } s_2+1\leq t\leq T.
\end{equation}
It remains to define $Y^\pm$ from time $s_3$ to $T$ on the event $E_{1-7}$. The good event will require
\begin{equation}\label{eq:E8deff}
E_{8}\stackrel{\text{def.}}{=}\{\eta^+_{s_3}(\cdot+\ell_g^{20})\vert_{[Y^-_{s_1}-2(\nu+1)T,Y^-_{s_1}+2(\nu+1)T]}\succcurlyeq \eta^-_{s_3}(\cdot-\ell_g^{20})\vert_{[Y^-_{s_1}-2(\nu+1)T,Y^-_{s_1}+2(\nu+1)T]} \},
\end{equation}
that is, we want that the environment $\eta^+$ seen from $Y^+_{s_3}$ covers $\eta^-$ seen from $Y^-_{s_3}$ on an interval of length of order $T$, which will enable us to let the walkers move in parallel (i.e.~taking the same steps at the same time), even if they are at different positions. Moreover, we want this domination to persist from time $s_3$ to time $T$, hence we require
\begin{equation}\label{eq:E9deff}
E_{9}\stackrel{\text{def.}}{=}\{\forall s\in [s_3,T],\,\eta^+_{s}(\cdot+\ell_g^{20})\vert_{[Y^-_{s_1}-2 T,Y^-_{s_1}+2T]}\succcurlyeq \eta^-_{s}(\cdot-\ell_g^{20})\vert_{[Y^-_{s_1}-2 T,Y^-_{s_1}+2T]} \}.
\end{equation}
On the good event $E_{1-9}$, we let, for $t\in [s_3+1,T]$,
\begin{equation}\label{eq:YpmE1-9}
Y^-_t=Y^-_{t-1}+A(\eta^-_{t-1}(Y^-_{t-1}),U_{(Y^-_{t-1},t-1)})\text{ and }
Y^+_t=Y^+_{t-1}+A(\eta^+_{t-1}(Y^+_{t-1}),U_{(Y^+_{t-1}-2\ell_g^{20},t-1)}).
\end{equation}
Note that above, we choose to shift spatially the collection $(U_w)$ by $2\ell_g^{20}$ spatially for $Y^+$, which corresponds to the difference between $Y^-$ and $Y^+$. Therefore, both walks are using the same $U_w$ to determine their next step. Using~\eqref{eq:YpmE1-9},~\eqref{eq:YpmE1-7} and~\eqref{eq:E9deff}, one thus proves recursively that
\begin{equation}\label{def:success0}
\text{on } E_{1-9}, \, Y^+_t-Y^-_t=2\ell_g^{20} \text{, for all }s_3\le t\le T.
\end{equation}
Finally, on the bad event $E_{1-7}^c\cap E_{8-9}^c$, we finish the construction by defining
\begin{equation}\label{badexit5}
Y^-_t=Y^-_{t-1}+1,\, Y^+_t=Y^+_{t-1}-1\text{, for all } s_3+1\leq t\leq T-1,\text{ and } Y^-_T=T,Y^+_T=-T.
\end{equation}
This ends the definition of the trajectories $Y^+$ and $Y^-$. Let us emphasize once more that these trajectories are not nearest-neighbour and not Markovian w.r.t.~the canonical filtration associated to $(\eta^+,\eta^-,U)$, but that this will not prevent us from obtaining the desired bounds.

Below, we proceed to define our three key events and summarise in Lemma~\ref{L:determ-speed} some of the important (deterministic) properties we will use.

\begin{itemize}
\item  {\bf Scenario I: Good event.} We define
\begin{equation} \label{def:success}
E_{\textnormal{good}}\stackrel{\text{def.}}{=} E_{1-9}.
\end{equation} 
Notice that by~\eqref{def:success0}, on the event $E_{\textnormal{good}}$ we have that
\begin{equation}\label{success}
Y^+_T-Y^-_T=2\ell_g^{20}>0,
\end{equation}
and also that
\begin{equation}\label{success2}
\forall s\in [0,s_2],\,\eta^+_{s}\vert_{[-T,T]}\succcurlyeq \eta^-_{s}\vert_{[-T,T]}, 
\end{equation}
which follows by~\eqref{eq:E5deff},~\eqref{eq:E3deff} and~\eqref{eq:E1deff}, provided that $M-(6\nu+1)s_2\ge T$. Finally,~\eqref{eq:E9deff} and the fact that on $E_{\textnormal{good}}$, $\vert Y^-_{s_1}\vert \leq s_1$ (by~\eqref{eq:Ypmrecursive} and~\eqref{eq:Ypmrecursive1}) imply that
\begin{equation}\label{success3}
\forall s\in [s_3,T],\,\eta^+_{s}(\cdot+2\ell_g^{20})\vert_{[-T-2\ell_g^{20},T-2\ell_g^{20}]}\succcurlyeq \eta^-_{s}\vert_{[-T,T]}.
\end{equation}
\item {\bf Scenario II: Neutral event.} We define
\begin{equation}\label{def:softfailure}
E_{\textnormal{neutral}}\stackrel{\text{def.}}{=} \left(E_1\cap E_2^c\cap E_{2,\textnormal{bis}}\right) \cup   \left(E_{1-3}\cap E_{4-5}^c\cap E_{4,\textnormal{bis}}\right) 
 \cup  \left(E_{1-5}\cap E_{6-7}^c\cap E_{6,\textnormal{bis}}\right).
\end{equation} 
Note that on $E_{\textnormal{neutral}}$, \eqref{eq:Y+movedtoY-s2}-\eqref{eq:Ypmrecursive4},~\eqref{eq:Ypmrecursive3} and~\eqref{eq:Ypmrecursive2} imply that  
\begin{equation}\label{bruitacote}
\begin{split}
&Y^+_T=Y^-_T
\end{split}
\end{equation}
and~\eqref{eq:Ypmrecursive4},~\eqref{eq:Y+movedtoY-s2},~\eqref{eq:YpmE1-5},~\eqref{eq:Ypmrecursive3},~\eqref{eq:Ypmrecursive2},~\eqref{eq:Ypmrecursive1} and~\eqref{eq:Ypmrecursive}
imply that 
\begin{equation}\label{bruitacote2}
Y^-_t=Y^-_{t-1}+A(\eta^-_{t-1}(Y^-_{t-1}), U_{(Y^-_{t-1},t-1)})\text{, for all } 1\le t\le T.
\end{equation}
From~\eqref{eq:E6bisdeff},~\eqref{eq:E4bisdeff},~\eqref{eq:E5deff},~\eqref{eq:E2bisdeff},~\eqref{eq:E3deff} and~\eqref{eq:E1deff}, we also have 
\begin{equation}\label{bruitacote3}
\forall s\in [0,T], \,\eta^+_{s}\vert_{[-T, T]} \succcurlyeq \eta^-_{s}\vert_{[-T,T]},
\end{equation}
provided that $M-(6\nu+1)T\ge T$.

\item{\bf Scenario III: Bad event.} We finally define the event
\begin{equation}\label{def:catastrophe}
\begin{split}
E_{\textnormal{bad}}& \stackrel{\text{def.}}{=}  \left( E_{\textnormal{good}} \cup E_{\textnormal{neutral}}\right)^c=E_1^c \cup \left(E_1\cap E_2^c\cap E_{2,\textnormal{bis}}^c\right) \cup \left(E_{1-2}\cap E_3^c\right) \cup \\
& \left(E_{1-3}\cap E_{4-5}^c\cap E_{4,\textnormal{bis}}^c\right)  \cup  \left(E_{1-5}\cap E_{6-7}^c\cap E_{6,\textnormal{bis}}^c\right) \cup \left(E_{1-7}\cap E_{8-9}^c \right).
\end{split}
\end{equation}
In particular,~\eqref{badexit5},~\eqref{badexit4},~\eqref{badexit3},~\eqref{badexit2}, and~\eqref{badexit1} yield  
\begin{equation} \label{catastrophe}
Y^+_T = Y^-_T-2T=-T \text{ on the event } E_{\textnormal{bad}}.
\end{equation}
\end{itemize}

 The next deterministic lemma is a restatement of~\eqref{success},~\eqref{bruitacote} and~\eqref{catastrophe}.
 
\begin{Lem} \label{L:determ-speed}
The event $E_{\textnormal{good}}$, $E_{\textnormal{neutral}}$ and $E_{\textnormal{bad}}$ defined in \eqref{def:success0}, \eqref{def:softfailure}
 and \eqref{def:catastrophe}, respectively, form a partition of $\Omega$  such that
\begin{align}
{Y}_{T}^+&= Y_{T}^-+2\ell_g^{20} \text{ on }E_{\textnormal{good}}; \label{eq:Egood}\\
{Y}^{+}_{T} &=Y_{T}^- \text{ on }E_{\textnormal{neutral}}; \label{eq:Eneutral}\\  
{Y}_T^+ &= Y^{-}_{T} -2T  \text{ on }E_{\textnormal{bad}}.\label{eq:Ebad}
\end{align}
\end{Lem}

  \subsection{The coupling $\mathbb{Q}$}\label{sec:Q}
We now aim to compare $Y^\pm$ to $X^{\rho}$ and $X^{\rho+\varepsilon}$, and to integrate over the dynamics of the environments when started from a typical initial configuration. We will derive from this a bound  on the expected discrepancy between $Y^+$ and $Y^-$, and thus on the one between $X^{\rho}$ and $X^{\rho+\varepsilon}$. The main result of this section is Lemma~\ref{L:Q coupling-alternative}, which entails a coupling $\mathbb{Q}$ with these features. Lemma~\ref{L:Q coupling-alternative} is the key ingredient in the proof of Proposition~\ref{Prop:initialspeed}, which appears in the next subsection.

We will require the environments $\eta^+$ and $\eta^-$ to have a typical initial configuration under $\mathbf{P}^{\rho+\epsilon}$ and $\mathbf{P}^{\rho}$ in the following sense. Recall that $\rho\in J$ and $\epsilon\in (0,1/100)$ have been fixed at the start of this section, and that $(\rho+\epsilon)\in J$. For $M,L\in \mathbb{N}$, we say that $(\eta^+_0,\eta^-_0)\in \Sigma^2$ (recall that $\Sigma= (\Z_+)^\Z$ from Section~\ref{subsec:re}) is \emph{$(M,L)$-balanced} if all of the following occur:
\begin{enumerate}[label=(\roman*)]
\item \label{e:balanced domination} $\eta^+_0(x)\geq \eta^-_0(x)$ for all $x\in [-M,M]$,
\item \label{e:balanced eta +} $\eta^+_0([x,x+\lfloor(\log L)^2\rfloor -1])\geq (\rho +\tfrac{99\epsilon}{100}) \lfloor(\log L)^2\rfloor $ for all $x\in [-M,M-\lfloor(\log L)^2\rfloor +1]$, 
\item\label{e:balanced eta -}  $\eta^-_0([x,x+\lfloor(\log L)^2\rfloor -1])\leq (\rho +\tfrac{\epsilon}{100}) \lfloor(\log L)^2\rfloor $ for all $x\in [-M,M-\lfloor(\log L)^2\rfloor +1]$.
\end{enumerate}

\begin{Lem}\label{L:Q coupling-alternative}
There exists $L_2=L_2(\rho,\epsilon, \nu \black) 
\geq 1$ such that for all $L\ge L_2$, the following holds.  For all $M\in [10(\nu+1)L,20(\nu+1)L]$ and for every $(M,L)$-balanced choice of $(\eta^+_0,\eta^-_0)$, there exists a coupling $\Q= \Q_{(\eta^+_0,\eta^-_0)}$ of $(\eta^+, \eta^-, U)$ with the following properties:
\begin{align} 
\label{e:coup-marg}
&\text{\begin{minipage}{0.8\textwidth} 
$\eta^-\sim \mathbf{P}^{\eta^-_0}, \  \eta^+\sim \mathbf{P}^{\eta^+_0}$, and $\left(U_w\right)_{w\in\mathbb{L}}$ are i.i.d.~uniform variables on $[0,1]$; 
\end{minipage}}\\[0.3em]
\label{e:coup-domination} 
&\text{\begin{minipage}{0.85\textwidth} 
There exist $X^{\pm}=X^{\pm}(\eta^{\pm},U)$ such that $X^{\pm} \sim \mathbb{P}^{\eta^{\pm}_0}$ and $\mathbb{Q}$-a.s., the inequalities $ X^{-}_T \leq Y^-_T$ and $Y^{+}_T \leq X^+_T$ hold, \black with $Y^{\pm}= Y^{\pm} (\eta^+,\eta^-,U)$ as in Section \ref{subsec:onegap};
\end{minipage}}\\[0.3em]
&\label{e:coup-speed}\mathbb{E}^{\mathbb{Q}}\left[{Y}^+_T - Y^-_T\right] \geq \exp\big(- \left(\log L\right)^{1/20}\big);\\[0.3em] 
\begin{split} &  \Q\left( E_{\textnormal{restart}}   \right) \geq 1- L^{-100}, \text{ where }\\
&E_{\textnormal{restart}}\stackrel{\textnormal{def.}}{=} \{ \forall x\in [-(M-(8\nu+3)T),M-(8\nu+3)T],\,       \eta^+_T(x +{Y}^+_T) \geq   \eta^-_T(x +Y^-_T)   \}.
\end{split} \label{e:coup-restart balanced}
\end{align}
\end{Lem}

\begin{proof} 
Let $L_2\geq 1$ to be chosen later, and $L\ge L_2$. Fix $M\in [10(\nu+1)L, 20(\nu +1)L]$, and let $(\eta_0^+, \eta_0^-)$ be $(M, L)$-balanced. 
We construct the coupling $\mathbb{Q}$ below. We will then define $X^-$ and $X^+$ such that 
 \eqref{e:coup-domination} is satisfied. To prove the main estimate \eqref{e:coup-speed}, we will control the probabilities of the events $E_{\textnormal{good}}$, $E_{\textnormal{neutral}}$ and $E_{\textnormal{bad}}$ emerging from the construction of $Y^{\pm}$, and use Lemma~\ref{L:determ-speed}. Finally, proving \eqref{e:coup-restart balanced} will require to bound the probability of losing the synchronisation of $\eta^+$ and $\eta^-$ by time $T$.
 
We split the proof into five parts: the construction of $\mathbb{Q}$ and the proofs of \eqref{e:coup-marg}-\eqref{e:coup-restart balanced}.\\

\noindent\textbf{Part I: Construction of $\Q$.}\\

 We will denote the natural filtration generated by the triplet $(\eta^+,\eta^-,U)$ by
\begin{equation}\label{e:Ft-def}
\mathcal{F}_t=\sigma\{(\eta^{+}_{t'})_{0\leq t'\leq t},(\eta^{-}_{t'})_{0\leq t'\leq t}, (U_{w})_{w\in \bL, \pi_2(w)\leq t-1} \}\text{, for } t\ge0.
\end{equation}
Under $\mathbb{Q}$, we first let $U=(U_w)_{w\in\bL}$ be a collection of i.i.d.~uniform random variables on $[0,1]$. We now define the coupling of $\eta^-$ and $\eta^+$ under $\mathbb{Q}$ in such a way that given $\mathcal{F}_t$, the evolution of $(\eta^\pm_s)_{t\le s\le t+1}$ has the right marginal $\mathbf{P}^{\eta^\pm_{t}}$ (up to time one), cf.~\eqref{pe:markov}. Moreover, throughout the construction of $\mathbb{Q}$, i.e.~from~\eqref{eq:etapmE1} to~\eqref{eq:etapmE17E8c10c} below, we will in fact ensure that for all integer $t \in [0,T]$, under $\mathbb{Q}$,
\begin{equation}\label{eq:couplingsindepU}
\text{$(U_w: \pi_2(w)=t)$ is independent from $\mathcal{F}_t$}.
\end{equation}
  We now proceed to specify $\eta^\pm$ under $\mathbb{Q}$, and refer again to Figure~\ref{f:One_step_gain} for visual aid. On the time-interval $[0,s_0]$, we
\begin{equation}\label{eq:etapmE1}
\text{couple $(\eta^+,\eta^-)$ as $(\eta,\eta')$ in~\ref{pe:drift} with $\eta_0=\eta^+_0$, $\eta'_0=\eta^-_0$, $H=M$, $t=s_0$ and $k=1$;}
\end{equation}
in particular, the domination $\eta_0 \vert_{[-H, H]}\succcurlyeq \eta_0' \vert_{[-H, H]}$ required by~\ref{pe:drift} is ensured by the fact that $(\eta^+_0,\eta^-_0)$ is $(M,L)$-balanced; see item (i) in the corresponding definition above Lemma~\ref{L:Q coupling-alternative}.
Since $E_1 \in \mathcal{F}_{s_0}$ by~\eqref{eq:E1deff} and \eqref{e:Ft-def}, we can observe at time $s_0$ whether $E_1$ occurred. Conditionally on $\mathcal{F}_{s_0}$ and on $E_{1}^c$, during $[s_0,T]$, we
\begin{equation}\label{eq:etapmE1c}
\text{let $\eta^-$ and $\eta^+$ evolve independently  with respective marginals $\bP^{\eta^-_{s_0}} $ and $\bP^{\eta^+_{s_0}} $.}
\end{equation}
Recalling the definition \eqref{eq:E2deff}, we have that $E_2\in\mathcal{F}_{s_0}$. 
Conditionally on $\mathcal{F}_{s_0}$ and on $E_{1}\cap E_{2}^c$,  we couple $(\eta^+,\eta^-)$ during the interval $[s_0,T]$ as 
\begin{equation}\label{eq:etapmE1E2c}
\text{$(\eta,\eta')$ in~\ref{pe:drift} with $\eta_0=\eta^+_{s_0}$, $\eta'_0=\eta^-_{s_0}$, $H=M-2\nu s_0$, $t=T-s_0$ and $k=1$.}
\end{equation}
On $E_{1-2}$, note that $\vert Y^-_{s_0}\vert \leq s_0 <s_1$. Hence, as we now explain, during the time-interval $[s_0, s_1]$, conditionally on $\mathcal{F}_{s_0}$ and on $E_{1-2}$, we can 
\begin{equation}\label{eq:etapmnacellecoupling}
\begin{split}
&\text{ apply the coupling of~\ref{pe:nacelle} to $\eta_{ t}(\cdot)=\eta^+_{s_0+ t \black}(\cdot+Y^-_{s_0})$ and $\eta_{ t }'(\cdot)=\eta^-_{ s_0+ t \black}(\cdot+Y^-_{s_0})$, $ t \geq0$,}
\\
&\text{ with $\ell=\ell_g$, $H=M-(2\nu+1) s_1$, $k=\lfloor s_1/\ell_g\rfloor$ and $x= \ell \text{ mod } 2$. }
\end{split}
\end{equation}
Note indeed that by~\eqref{eq:s_0} and~\eqref{eq:ellTdef}, $k\geq \ell_g> 48\nu^{-1}(\nu +\log (40) - p_{\circ}(1-p_{\bullet})\nu/2)$ for $L$ large enough and $H= M-(2\nu+1)  s_1\geq 2\nu s_1\geq  2\nu k\ell_g = 2 \nu \ell k$ for all $L\geq 3$, as required in~\ref{pe:nacelle}. Together with the definition of $E_2$ in~\eqref{eq:E2deff}, this allows us to apply the coupling of~\ref{pe:nacelle}.

So far we have specified $\eta^{\pm}$ for all of $[0,T]$ on $E_{1-2}^c$ and up to time $s_1$ on $E_{1-2}$. As we now briefly elaborate, it is also plain from the construction above that the independence property postulated in \eqref{eq:couplingsindepU} holds for all $t \leq s_1$. This is a trivial matter for $t<s_0$ in view of \eqref{eq:etapmE1}, which does not involve $U$ at all. For $s_0 \leq t \leq s_1$, the only dependence on $U$ arises through $Y^-_{s_0}$ via $E_2$ and \eqref{eq:etapmnacellecoupling}. However, on account of \eqref{eq:Ypmrecursive}, \eqref{badexit1}, \eqref{eq:Ypmrecursive1}, \eqref{eq:Ypmrecursive2} and \eqref{badexit2}, $Y^-_{s_0}$ only relies on variables in $U$ with time label at most $s_0-1$, whence the claim. In the sequel (more precisely, up to \eqref{eq:etapmE17E8c10c}), considerations along similar lines allow to extend \eqref{eq:couplingsindepU} to larger times $t$. These will not be made explicit.

Returning to the construction of $\mathbb{Q}$, it remains to specify $\eta^{\pm}$
after time $s_1$ on the event $E_{1-2}$. Recall that by \eqref{eq:E3deff} and \eqref{eq:E4deff}, both $E_3$ and $E_{4}$ are in $\mathcal{F}_{s_1}$. Over the time interval $[s_1, T]$, conditionally on $\mathcal{F}_{s_1}$ and on $E_{1-2}\cap E_{3}^c$, we choose to
\begin{equation}\label{eq:etapmE12E3c}
\text{let $\eta^-$ and $\eta^+$ evolve independently  with respective marginals $\bP^{\eta^-_{s_1}} $ and $\bP^{\eta^+_{s_1}} $.}
\end{equation}
On the other hand, conditionally on $\mathcal{F}_{s_1}$ and on $E_{1-3}\cap E_{4}^c$, during $[s_1,T]$, we
\begin{equation}\label{eq:etapmE13E4c}
\begin{split}
& \quad \text{couple }(\eta^+,\eta^-) \text{ as }(\eta, \eta')\text{ in~\ref{pe:drift}},\\
\text{with }& H=M-(4\nu+1) s_1,\quad t=t=T-s_1,\quad k= 1 .
\end{split}
\end{equation}
Conditionally on $\mathcal{F}_{s_1}$ and on $E_{1-4}$, during the time-interval $[s_1, s_2]$, we
\begin{equation}\label{eq:etapmE14}
\begin{split}
& \quad \text{couple }(\eta^+,\eta^-) \text{ as }(\eta, \eta')\text{ in~\ref{pe:drift}},\\
\text{with }& H=M-(4\nu+1) s_1,\quad t=s_2-s_1,\quad k= \lfloor s_1/(s_2-s_1)\rfloor. 
\end{split}
\end{equation}
Note that $k \geq 1$ in \eqref{eq:etapmE14} by \eqref{eq:s_0}  and since $L \geq L_2$ by assumption (upon possibly enlarging $L_2$). Moreover, note that the conditions on the environments at time $s_1$ needed for  \ref{pe:drift} to apply in both \eqref{eq:etapmE13E4c} and \eqref{eq:etapmE14} are met owing to the occurrence of $E_3$, see \eqref{eq:E3deff}.

We still have to specify $\eta^{\pm}$ from time $s_2$ onwards on the event $E_{1-4}$. By definition \eqref{eq:E5deff}, we have that $E_5\in\mathcal{F}_{s_2}$. Hence, conditionally on $\mathcal{F}_{s_2}$ and on $E_{1-4}\cap E_{5}^c$, during $[s_2, T]$, we 
\begin{equation}\label{eq:etapmE14E5c}
\text{let $\eta^-$ and $\eta^+$ evolve independently  with respective marginals $\bP^{\eta^-_{s_2}} $ and $\bP^{\eta^+_{s_2}} $.}
\end{equation}
Using the definitions \eqref{eq:E6deff} and \eqref{eq:E7deff}, we have that both $E_6$ and $E_7$ are in $\mathcal{F}_{s_2}$. During $[s_2,T]$, conditionally on $\mathcal{F}_{s_2}$ and on $E_{1-5}\cap E_{6-7}^c$, observing that the defining features of $E_5$ (see \eqref{eq:E5deff}) allow to apply \ref{pe:drift}, we 
\begin{equation}\label{eq:etapmE15E67c}
\text{couple $(\eta^+,\eta^-)$ as $(\eta, \eta')$ in~\ref{pe:drift} with $H=M-(6\nu+1) s_{2 \black}$, $t=T-s_2$ and $k=1$}.
\end{equation}
Conditionally on $\mathcal{F}_{s_2}$ and on $E_{1-7}$, during $[s_2,s_3]$, as we now explain, we
\begin{equation}\label{eq:etapmE17}
\begin{split}
&\text{couple $(\eta^+(\cdot+Y^-_{s_2}+\ell_g^{20}),\eta^-(\cdot+Y^-_{s_2}-\ell_g^{20} ))$ as $(\eta, \eta')$ in~\ref{pe:couplings}}
\\
&\text{with $H=3(\nu+1)T$, $t=s_3-s_2$ and $(\rho,\varepsilon)=(\rho,\epsilon)$}.
\end{split}
\end{equation}
Indeed, one checks that the conditions needed for~\ref{pe:couplings} to apply are all satisfied on the event $E_7$ whenever $L_2 (\rho,\epsilon,\nu)$ is large enough. First, by choice of $L_2$ and for all $L \geq L_2$, one readily ensures (recall \eqref{eq:ellTdef}-\eqref{eq:s_0}) that all of $3(\nu+1)T > 4\nu(s_3-s_2)$, $\nu^8(s_3-s_2)>1$ and $\nu(s_3-s_2)^{1/4}>\Cr{SEPcoupling}\epsilon^{-2}(1+\vert\log^3(\nu(s_3-s_2)) \vert)$ hold. Moreover, the conditions on the empirical densities of $\eta_0$ and $\eta_0'$ appearing above \eqref{eq:doublecoupleSEP} hold by definition of $E_{7}$ at~\eqref{eq:E7deff}, since $\lfloor (s_3-s_2)^{1/4}\rfloor \geq \ell_g^2$ by~\eqref{eq:s_0}. Note that the relevant intervals $I$ in the context of \ref{pe:couplings}, which have length $\lfloor \ell/2 \rfloor \leq |I| \leq \ell$ with $\ell= \lfloor (s_3-s_2)^{1/4} \rfloor$, may in practice be much larger than those appearing in the definition of $E_{7}$, but they can be paved by disjoint contiguous intervals as entering $E_7$. The required controls on the corresponding empirical densities $\eta_0(I)$ and $\eta_0'(I)$ are thus inherited from those defining $E_7$. Similar considerations also apply below (whenever either of \ref{pe:couplings} or \ref{pe:compatible} are used).

 It remains to specify $\eta^{\pm}$ on the event $E_{1-7}$ for the time interval $[s_3,T]$. \black We now define an additional event whose realisation will allow us to couple $\eta^\pm$ on the whole window of width order $M$ at time $T$, regardless of whether the coupling at~\eqref{eq:etapmE17} succeeds or not; here, by ``succeed'' we mean that the (high-probability) event appearing in \eqref{eq:doublecoupleSEP} is realized. The following will play a key role when establishing~\eqref{e:coup-restart balanced}. We set
\begin{equation}\label{eq:E10deff}
E_{10}\stackrel{\text{def.}}{=}\left\{\begin{array}{c} \forall I \subseteq [-M+2\nu s_3,M-2\nu s_3]\text{ with } \lfloor \lfloor (\log L)^2\rfloor/2\rfloor \leq \vert I \vert \leq \lfloor (\log L)^2\rfloor, \\
\eta^-_{s_3}(I)\leq (\rho+\epsilon/4)\vert I\vert\text{ and }\eta^+_{s_3}( I )\geq (\rho+3\epsilon/4)\vert I \vert \end{array}\right\}.
\end{equation}
The event $E_{10}$ above and $E_8$ defined in \eqref{eq:E8deff} are both $\mathcal{F}_{s_3}$-measurable. Conditionally on $\mathcal{F}_{s_3}$ and on $E_{1-8}\cap E_{10}$, during $[s_3,T]$, we couple
\begin{equation}\label{eq:etapmE1810}
\begin{split}
&(\eta^+(\cdot +Y^-_{s_1}+\ell_g^{20} ),\eta^-(\cdot +Y^-_{s_1}-\ell_g^{20} ))\text{ as in~\ref{pe:compatible} with } H_1=2(\nu+1)T,
\\
& H_2=M-2\nu s_3-s_1,\ t=T-s_3,\ \ell=\ell_G^{1/200} \text{ and }(\rho,\varepsilon)=(\rho,\epsilon).
\end{split}
\end{equation}
To this effect, we verify that for $L$ large enough, by~\eqref{eq:s_0},~\eqref{eq:ellTdef} and since $M \geq 10(\nu+1)L$, we indeed have that
$\min\{H_1, H_2-H_1-1\}> 10{\nu}t>4\nu\ell^{100}>\Cr{compatible}$, $\nu\ell>\Cr{compatible}\epsilon^{-2}(1+\vert \log^3(\nu \ell^4)\vert)$ and $\ell>80\nu\epsilon^{-1}+\nu^{-2} .$
Also, the domination on $[-H_1,H_1]$ and the empirical density condition are respectively guaranteed by definition of $E_8$ at~\eqref{eq:E8deff} and $E_{10}$ at~\eqref{eq:E10deff}.

If $E_8$ does not occur (hence we temporarily lose the domination of $\eta^-$ by $\eta^+$) but $E_{10}$ does, we proceed to what was referred to as the \textit{parachute coupling} in Section~\ref{sec:proofsketch}, in order to recover this domination by time $T$. Precisely, conditionally on $\mathcal{F}_{s_3}$ and on $E_{1-7}\cap E_{8}^c\cap E_{10}$, we
\begin{equation}\label{eq:etapmE17E8cPARACHUTE}
\begin{split}
 &\text{couple $(\eta^+(\cdot -T),\eta^-(\cdot +T))$ as $(\eta, \eta')$ in~\ref{pe:couplings}}
 \\
 &\text{with $H=M-T-2\nu s_3$, $t=T-s_3$ and $(\rho,\varepsilon)=(\rho,\epsilon)$}.
 \end{split}
\end{equation}
One checks indeed that the conditions of~\ref{pe:couplings} hold, since for $L$ large enough (recalling~\eqref{eq:s_0} and~\eqref{eq:ellTdef}) we have that
$H>4\nu t$,$\nu^8t>1$, $\nu t^{1/4} >\Cr{SEPcoupling}\epsilon^{-2}(1+\vert \log^{3}(\nu t) \vert)$ and $t^{1/4}>\lfloor (\log L)^2 \rfloor$, so that on $E_{10}$ the condition on the empirical density holds. 

The remaining cases are straightforward to specify. Conditionally on $\mathcal{F}_{s_3}$, on $E_{1-8}\cap E_{10}^c$, we
\begin{equation}\label{eq:etapmE18E10c}
\begin{split}
&\text{couple $(\eta^+(\cdot +Y^-_{s_1}+\ell_g^{20} ),\eta^-(\cdot +Y^-_{s_1}-\ell_g^{20} ))$ as in~\ref{pe:drift} }
\\
&\text{with $H=2(\nu+1)T-\ell_g^{20}$, $t=T-s_3$, and $k=1$.}
\end{split}
\end{equation}
Finally, conditionally on $\mathcal{F}_{s_3}$ and on $E_{1-7}\cap E_8^c\cap E_{10}^c$, during $[s_3,T]$, we
\begin{equation}\label{eq:etapmE17E8c10c}
\text{let $\eta^-$ and $\eta^+$ evolve independently  with respective marginals $\bP^{\eta^-_{s_3}} $ and $\bP^{\eta^+_{s_3}} $.}
\end{equation}
We have now fully defined the measure $\mathbb{Q}$ and with it the triplet $(\eta^+,\eta^-,U)$ (in the time interval $[0,T]$). The task is now to verify that with these choices, all of \eqref{e:coup-marg}-\eqref{e:coup-restart balanced} hold. For concreteness we extend all three processes $(\eta^+,\eta^-,U)$ independently at times $t >T$, using the Markov property at time $T$ for $\eta^{\pm}$. These extensions will de facto play no role because all of \eqref{e:coup-domination}-\eqref{e:coup-restart balanced} only concern matters up to time $T.$

\bigskip

\noindent
\textbf{Part II: Proof of  \eqref{e:coup-marg}.}\\ 

The fact that $U$ has the desired law is immediate, see below \eqref{e:Ft-def}. We proceed to show that $\eta^-\sim \mathbf{P}^{\eta_0^-}$. Since the construction of $\eta^-$ only consists of successive couplings at times $0,s_0,s_1,s_2$ and $s_3$ where the marginals have the desired distributions (recall~\ref{pe:compatible},~\ref{pe:drift},~\ref{pe:couplings} and~\ref{pe:nacelle}), by virtue of the Markov property~\eqref{pe:markov}, it is enough to check that the events deciding which coupling to apply at time $s_i$ are $\mathcal{F}_{s_i}$-measurable for $i\in \{0,1,2,3\}$, which we already did at~\eqref{eq:etapmE1c}-\eqref{eq:etapmE17},~\eqref{eq:etapmE1810},~\eqref{eq:etapmE17E8cPARACHUTE},~\eqref{eq:etapmE18E10c} 
and~\eqref{eq:etapmE17E8c10c}. Therefore, $\eta^-\sim \mathbf{P}^{\eta_0^-}$. In the same way we deduce that $\eta^+\sim \mathbf{P}^{\eta_0^+}$. 

\bigskip

\noindent
\textbf{Part III: Proof of \eqref{e:coup-domination}.}\\

We now use $(\eta^+,\eta^-,U)$ to construct explicit functions $X^+=X^+(\eta^+,U)$ and $X^-= X^-(\eta^-,U)$ with the correct marginal laws $X^{\pm}\sim \mathbb{P}^{\eta_0^{\pm}}$. The case of $X^-$ is easily dispensed with: we define  $X^-$ as in~\eqref{def:A} and~\eqref{eq:defX} with $(\eta^-,U)$ instead of $(\eta,U)$. Since we have already established that $\eta^-\sim \mathbf{P}^{\eta_0^-}$ as part of \eqref{e:coup-marg}, it follows using \eqref{eq:couplingsindepU} and Lemma~\ref{lem:Xlawredef} that $X^-\sim \mathbb{P}^{\eta_0^-}$.

As for $X^+$, we also define it via~\eqref{def:A} and~\eqref{eq:defX} using the construction specified around \eqref{eq:arrows-variation}, replacing $(\eta,U)$ by $(\eta^+,U^+)$ where $U^+$ is defined as follows: for all $w\in\mathbb{L}$ with $\pi_2(w)\le s_3-1$, $U^+_w=U_w$, and for all $w\in\mathbb{L}$ with $\pi_2(w)\ge s_3$, 
\begin{equation}\label{eq:Ushifted}
U^+_w=
\begin{cases} 
U_{w-(2\ell_g^{20},0)}\text{ on }E_{1-8}, \\
U_{w}\text{ on }E_{1-8}^c.
 \end{cases}
\end{equation}
Recalling that $E_{1-8}\in\mathcal{F}_{s_3}$, and hence is independent of $(U_w; w\in\mathbb{L}, \pi_2(w)\ge s_3)$, it follows that $X^+\sim \mathbb{P}^{\eta_0^+}$ again by combining the established fact that $\eta^+\sim \mathbf{P}^{\eta_0^+}$ and \eqref{eq:couplingsindepU}. The rationale behind \eqref{eq:Ushifted} will become clear momentarily.

We show that $X^{\pm}$ as defined above satisfy $Y^-_T\geq X^-_T$ and $Y^+_T\geq X^+_T$ $\mathbb{Q}$-a.s.~To this end, note first that on $E_{\textnormal{bad}}$ and by \eqref{catastrophe}, $Y^-_T=T$ and $Y^+_T=-T$, hence $\mathbb{Q}$-a.s.~on $E_{\textnormal{bad}}$ we have using the trivial bounds $X^{\pm}_T \in [-T,T]$ that $X^-_T\le Y^-_T$ and $Y^+_T\le X^+_T$.

Second, on $E_{\textnormal{neutral}}$ and by \eqref{bruitacote2}, under $\mathbb{Q}$, the process $(Y^-_t:  0\le t\le T)$ simply follows the environment $(\eta^-,U)$ as per~\eqref{def:A} and~\eqref{eq:defX}, and so does $(X_t^-: 0\le t\le T)$ by above definition of $X^-$. Hence $Y^-_T=X^-_T$. Moreover, since $E_{\textnormal{neutral}}\subset E_{1-8}^c$ by~\eqref{def:softfailure}, the process $(X^+_t: 0\le t\le T)$ follows the environment $(\eta^+,U)$ by~\eqref{eq:Ushifted}. By Lemma~\ref{L:monotone} applied with $(X,\widetilde{X})=(Y^-,X^+)$, $(\eta,\tilde\eta)=(\eta^-,\eta^+)$ and $K=[-T,T]\times [0,T]$, recalling~\eqref{bruitacote3}, we get that $X^+_T\geq Y^-_T$. Using~\eqref{bruitacote}, we finally obtain $X^+_T\geq Y^+_T=Y^-_T=X^-_T$ $\mathbb{Q}$-a.s.~on $E_{\textnormal{neutral}}$.

Third, on $E_{\textnormal{good}}\subset E_{1-6}$, the process $(Y^-_t:  0\le t\le s_2)$ follows the environment $(\eta^-,U)$ by~\eqref{eq:Ypmrecursive},~\eqref{eq:Ypmrecursive1} and~\eqref{eq:YpmE1-5}, as is the case of $X^-$, so that $X^-_{s_2}=Y^-_{s_2}$. Then by~\eqref{eq:Ypms2s3} we have that $Y^-_{s_3}=Y^-_{s_2}+(s_3-s_2)$ which guarantees that $X^-_{s_3}\le Y^-_{s_3}$. From time $s_3$ to time $T$, $X^-$ and $Y^-$ follow the same environment $(\eta^-,U)$ by~\eqref{eq:YpmE1-9} and thus, by Lemma \ref{L:monotone} with $(X,\widetilde{X})=(Y^-,X^-)$, $(\eta,\tilde\eta)=(\eta^-,\eta^-)$ and $K=\mathbb{Z}\times [s_3,T]$, we have that $X^-_T\le Y^-_T$. 

Similarly for $X^+$ and $Y^+$, on $E_{\textnormal{good}}\subset E_{1-3}$, during $[0,s_1]$ the process $X^+$ follows the environment $(\eta^+,U)$ while $Y^+$ follows $(\eta^-,U)$ by~\eqref{eq:Ypmrecursive},~\eqref{eq:Ypmrecursive1}. By~\eqref{success2} we can apply Lemma~\ref{L:monotone} with $(X,\widetilde{X})=(Y^+,X^+)$, $(\eta,\tilde\eta)=(\eta^-,\eta^+)$ and $K=[-T,T]\times [0,s_1]$ to deduce that $Y^+_{s_1}\leq X^+_{s_1}$.  Then during $[s_1,s_2]$, both $Y^+$ and $X^+$ follow $(\eta^+,U)$ (recall~\eqref{eq:YpmE1-5}, and that $E_{\textnormal{good}}\subset E_{1-5}$). Thus Lemma~\ref{L:monotone} with $(X,\widetilde{X})=(Y^+,X^+)$, $(\eta,\tilde\eta)=(\eta^+,\eta^+)$ and $K=\mathbb{Z}\times [s_1,s_2]$ ensures that $Y^+_{s_2}\leq X^+_{s_2}$ $\mathbb{Q}$-a.s.~on $E_{\textnormal{good}}$. Next,~\eqref{eq:Ypms2s3}, which holds on $E_{1-7}\supset E_{\textnormal{good}}$, implies that $Y^+_{s_3}\leq X^+_{s_3}$, given that $X^+$ only takes nearest-neighbour steps. Finally, since $E_{\textnormal{good}}\subset E_{1-9}$, from time $s_3$ to $T$, both $Y^+$ and $X^+$ follow the environment $(\eta^+,U^+)$, where $U^+$ is as in \eqref{eq:Ushifted} (recall~\eqref{eq:YpmE1-9}; this explains the choice in \eqref{eq:Ushifted}). Applying Lemma~\ref{L:monotone} with $(X,\widetilde{X})=(Y^+,X^+)$, $(\eta,\tilde\eta)=(\eta^+,\eta^+)$ and $K=\mathbb{Z}\times [s_3,T]$ therefore yields $Y^+_{T}\leq X^+_{T}$ on $E_{\textnormal{good}}$. This concludes the proof of~\eqref{e:coup-domination}.
\\

\noindent

\textbf{Part IV: Proof of~\eqref{e:coup-speed}.} \\

By Lemma~\ref{L:determ-speed}, we have that 
\begin{equation}
\mathbb{E}^\Q[Y^+_T-Y^-_T]=2\ell_g^{20} \Q(E_{\textnormal{good}})-2T\Q(E_{\textnormal{bad}}).
\end{equation}
Recalling the definitions of $E_{\textnormal{good}}$ and $E_{\textnormal{bad}}$ from \eqref{def:success} and \eqref{def:catastrophe}, as well as~\eqref{eq:ellTdef} and~\eqref{eq:s_0}, it is thus enough to show that
\begin{equation}\label{eq:QEgood}
(\Q(E_{\textnormal{good}})=) \  \Q(E_{1-9})\geq \exp(-\ell_g^{22})
\end{equation}
and
\begin{multline}\label{eq:QEbad}
S\stackrel{\text{def.}}{=}\Q(E_1^c) +\Q(E_1\cap E_2^c\cap E_{2,\textnormal{bis}}^c ) +\Q(E_{1-2}\cap E_3^c)
+\Q(E_{1-3}\cap E_{4,5}^c\cap E_{4,\textnormal{bis}}^c ) \\ +\Q(E_{1-5}\cap E_{6-7}^c\cap E_{6,\textnormal{bis}}^c) +\Q(E_{1-7}\cap E_{8-9}^c ) \leq \Q(E_{1-9})/(10T).
\end{multline}
\textbf{Proof of~\eqref{eq:QEgood}}. Recall $E_1$ from~\eqref{eq:E1deff}. By~\eqref{eq:etapmE1} and~\eqref{eq:SEPdriftdeviations}, which is in force, we get
\begin{equation}\label{eq:QE1c}
\mathbb{Q}(E_1^c)\leq 20 \black \exp (-\nu s_0/4) \leq \exp(-\ell_g^{10^5}),
\end{equation}
the last inequality being true for $L$ large enough by~\eqref{eq:s_0} and~\eqref{eq:ellTdef}. 
Next, recalling~\eqref{eq:E2deff}, that $4s_0>s_0+3\ell_g$ by~\eqref{eq:s_0} and that $\epsilon<1/100$, and using that $\vert Y^-_{s_0}\vert \leq s_0$, we note that
\begin{equation}
\begin{split}
E_2^c\subseteq &\bigcup_{\substack{ x\in [-4s_0+1,4s_0],\\
\lfloor\ell_g/2\rfloor\leq \ell'\leq \ell_g}}
\{\eta^-([x,x+\ell'-1])\leq (\rho+\epsilon/2)\ell' <\black \eta^+([x,x+\ell'-1])\leq  (\rho+2\epsilon)\ell'\}^c.
\end{split}
\end{equation}
We note in passing that we cannot locate precisely $Y^-_{s_0}$ without (possibly heavily) conditioning the evolution of $\eta^-$ on $[0,s_0]$. By~\ref{pe:densitychange} applied with either $(\eta_0,\rho,\varepsilon)=(\eta^-_0,\rho ,\epsilon/20)$ or $(\eta_0,\rho,\varepsilon)=(\eta^-_0,\rho+\epsilon ,\epsilon/20)$, both times with $(\ell, H,t)= (\lfloor\log ^2 L\rfloor,(4\nu+8)s_0,s_0)$ and $\ell'$ ranging from $\lfloor\ell_g/2\rfloor$ to $\ell_g$ we have by a union bound on $x$ and $\ell'$: 
\begin{equation}\label{eq:QE2c}
\mathbb{Q}(E_2^c)\leq (2\times 8s_0\times \ell_g) 4 (4\nu+8)s_0\exp (-\Cr{densitystableexpo}\epsilon^2 \lfloor\ell_g/2\rfloor/400)\leq 1/4,
\end{equation}
the last inequality holding for $L$ large enough by~\eqref{eq:ellTdef} and~\eqref{eq:s_0}. Note that we can indeed apply~\ref{pe:densitychange} since $(\eta^+_0,\eta^-_0)$ is $(M,L)$-balanced (see items (ii) and (iii) above Lemma~\ref{L:Q coupling-alternative}), and $M>(4\nu+8)s_0> 4\nu s_0 >\Cr{densitystable}\epsilon^{-2}\lfloor (\log L)^2\rfloor^2(1+\vert \log^3(\nu s_0)\vert)$ and $\ell_g\leq \sqrt{s_0}$ for $L$ large enough.

Next, we aim to derive a suitable deterministic lower bound on $\mathbb{Q}(E_{3-4}\,\vert \mathcal{F}_{s_0})$, on the event $E_{1-2}$, which is $\mathcal{F}_{s_0}$-measurable. We aim to apply \ref{pe:nacelle}, cf.~\eqref{eq:etapmnacellecoupling}, and dealing with $E_3$ is straightforward, see \eqref{eq:penacelleNEW2}, but $E_4$ requires a small amount of work, cf.~\eqref{eq:penacelleNEW} and \eqref{eq:E4deff}. To this effect, we first observe that under $\mathbb{Q}$, with $\eta, \eta'$ as defined in \eqref{eq:etapmnacellecoupling}, one has the inclusions
\begin{multline} \label{eq:incl-newC3}
E_{1-4} \stackrel{\eqref{Ypmats1}, \eqref{eq:E4deff}}\supset \{\eta^+_{s_1}(Y_{s_1}^{ - })\geq 1,\eta^-_{s_1}(Y_{s_1}^{ - })=0\} \cap  E_{1-3} \\ \supset \{\eta_{\ell}(x)\geq 1,\eta'_{\ell}(x)=0\} \cap \{Y_{s_1}^-- Y_{s_0}^-=x \} \cap  E_{1-3},
\end{multline}
where $\ell=\ell_g=s_1-s_0$ and $x= \ell \text{ mod } 2$ on account of~\eqref{eq:etapmnacellecoupling} and \eqref{eq:s_0}. Now recalling that the evolution of $Y^{-}$ on the time interval $[s_0,s_1]$ follows \eqref{eq:Ypmrecursive1} on $E_{1-3}$, and in view of \eqref{def:A}, one readily deduces that the event  $\{Y_{s_1}^-- Y_{s_0}^-=x \}$ is implied by the event $ E_{1-3} \cap F$, where $F$ refers to the joint occurrence of
$\{ U_{(Y^-_{s_0},s_0+ n)} < p_\circ\}$ for even integer $n$ satisfying $0 \leq n \leq \ell-1$ and $\{ U_{(Y^-_{s_0},s_0+ n)} > p_{\bullet}\}$ for odd integer $n$ satisfying $0 \leq n \leq \ell-1$.
(Observe indeed that, if $\ell$ is odd, whence $x=1$, there will be one more step to the right than to the left in the resulting trajectory for $Y^+$). Feeding this into \eqref{eq:incl-newC3}, applying a union bound, using \eqref{eq:penacelleNEW2} and \eqref{eq:penacelleNEW} (which are in force on account of \eqref{eq:etapmnacellecoupling}), it follows that on $E_{1-2}$,
\begin{equation*}
\mathbb{Q}(E_{3-4}\,\vert \mathcal{F}_{s_0})\geq
\mathbb{Q}(\{ \eta_{\ell}(x)\geq 1,\eta'_{\ell}(x)=0\} \cap F \,\vert \mathcal{F}_{s_0}) -\delta' \geq 2\delta(p_{\circ}(1-p_{\bullet}))^{\ell} -\delta' \geq \delta'
\end{equation*}
(see above \ref{pe:nacelle} regarding $\delta$ and $\delta'$);
in the penultimate step above, we have also used that, conditionally on $\mathcal{F}_0$, the events $\{\eta_{\ell}(x)\geq 1,\eta'_{\ell}(x)=0\}$ and $F$ are independent. Combining this with~\eqref{eq:QE1c} and~\eqref{eq:QE2c}, we deduce that for $L$ large enough, using \eqref{eq:ellTdef},
\begin{equation}\label{eq:QE14}
\mathbb{Q}(E_{1-4})\geq 3\delta'/4-  20 \black e^{-\ell_g^{10^5}} \geq \delta'/2.
\end{equation}
\black

Next, recalling the definition \eqref{eq:E5deff} of $E_5$ and that $E_{1-4}\in \mathcal{F}_{s_1}$, by  \eqref{eq:etapmE14} and the bound given in \ref{pe:drift},  we have that, $\mathbb{Q}$-a.s.~on $E_{1-4}$,
\begin{equation}\label{eq:QE5}
\mathbb{Q}(E_5^c\,\vert \mathcal{F}_{s_1})\leq 20\exp(- \lfloor s_1/(s_2-s_1)\rfloor (s_2-s_1)\tfrac{\nu}{4} )\leq  \exp(- \tfrac{\nu}{4} \ell_G )\le \exp\big(- \tfrac{\nu}{8} \ell_g^{10^6} \big),
\end{equation}
for all $L$ large enough by~\eqref{eq:s_0}. 

We will return to $E_6$ momentarily and first consider $E_7$. To compute the probability of $E_7^c$, we take a union bound over $\ell'$ such that $\lfloor \ell_g^2/2\rfloor \leq \ell'\leq \ell_g^2$, use the definition \eqref{eq:E7deff} and the fact that $(\eta_0^-,\eta_0^+)$ is $(M,L)$-balanced in order to apply \ref{pe:densitychange} twice, once for $\eta^-$ and once for $\eta^+$, choosing in both cases $(\ell,H,t)=(\lfloor (\log L)^2\rfloor, 38(\nu+1)\ell_G, s_2)$ and, for $\eta^-$, with $(\eta_0,\rho, \varepsilon)=(\eta^-_0,\rho ,\epsilon/100 )$ and, for $\eta^+$, with $(\eta_0,\rho, \varepsilon)=(\eta^+_0,\rho+\epsilon ,\epsilon/100 )$. Note that for $L$ large enough by~\eqref{eq:s_0} and~\eqref{eq:ellTdef}, we have indeed $M>H>4\nu t>\Cr{densitystable}\ell^2\varepsilon^{-2}(1+\vert \log^3(\nu t)\vert)$ and $\ell'\leq \ell_g^2\leq \sqrt{t}$, as required in~\ref{pe:densitychange}. We thus obtain that
\begin{equation}\label{eq:QE7c}
\mathbb{Q}(E_7^c)\leq 400(\nu+1)\ell_g^2\ell_G \exp(-\Cr{densitystableexpo}\epsilon^2\lfloor \tfrac{\ell_g^2}{2}\rfloor/10^4)\le \exp\big(  -\tfrac{\Cr{densitystableexpo}\epsilon^2}{10^5} \ell_g^2 \big),
\end{equation}
where the last inequality holds for $L$ large enough (depending on $\nu$, $\rho$ and $\epsilon$). 
Putting \eqref{eq:QE14}, \eqref{eq:QE5} and \eqref{eq:QE7c} together and taking $L$ (hence $\ell_g$) large enough, it follows that 
\begin{equation}\label{dontknow}
\Q(E_{1-5}\cap E_7) \geq \Q(E_{1-5})-\Q(E_7^c) \geq \delta'/4. 
\end{equation}

 Concerning $E_6$, we first observe that under $\mathbb{Q}$, the field $(U_{w}: s_1 \leq \pi_2(w) \leq s_2)$ is independent from the $\sigma$-algebra generated by $\mathcal{F}_{s_1}$ and $(\eta^\pm_t)_{s_1\leq t\leq s_2}$; this can be seen by direct inspection of the coupling construction until time $s_2$, paying particular attention (with regards to the evolution of $\eta^{\pm}$ during $[s_1,s_2]$) to \eqref{eq:etapmE12E3c}, \eqref{eq:etapmE13E4c}, \eqref{eq:etapmE14} (and \eqref{eq:etapmE1c}), which all involve only either i) trivial couplings or ii) couplings relying on \ref{pe:drift}. In particular these couplings do not involve $U$ at all. Using the previous observation and recalling \eqref{eq:E6deff}-\eqref{eq:E6tdeff}, it follows that $\mathbb{Q}$-a.s., 
\begin{equation}\label{eq:QE6}
\Q\left(\left. E_6\right| \mathcal{F}_{s_1}, (\eta^\pm_t)_{s_1\leq t\leq s_2} \right)=(p_\bullet -p_\circ)\left( p_\circ (1-p_\bullet)\right)^{s_2-s_1-1}\ge (p_\bullet -p_\circ)\left( p_\circ (1-p_\bullet)\right)^{2\ell_g^{20}}.
\end{equation}
Next, observe that $E_{1-5}$ and $E_7$ are measurable with respect to the $\sigma$
-algebra generated by $\mathcal{F}_{s_1}$ and $(\eta^\pm_t)_{s_1\leq t\leq s_2}$ owing to~\eqref{eq:E1deff},~\eqref{eq:E2deff},
~\eqref{eq:E3deff},~\eqref{eq:E4deff},
~\eqref{eq:E5deff} and~\eqref{eq:E7deff}. Combining this with~\eqref{dontknow}, and recalling the value of $\delta'$ from \ref{pe:nacelle}, we obtain that 
\begin{equation}\label{eq:QE17}
\begin{split}
\Q(E_{1-7})
&\ge \left( p_\circ (1-p_\bullet)\right)^{2\ell_g^{20}} \delta'/4 \geq  \left( p_\circ (1-p_\bullet)\right)^{3\ell_g^{20}}\ge \exp(-\ell_g^{21}),
\end{split}
\end{equation}
where, in the last two inequalities, we took $L$ (hence $\ell_g$) large enough. \black

Recalling $E_8$ from \eqref{eq:E8deff} along with the coupling defined in \eqref{eq:etapmE17} and using that $E_{1-7}\in\mathcal{F}_{s_2}$, we can apply \ref{pe:couplings}, and obtain that, $\Q$-a.s.~on $E_{1-7}$, for large enough $L$ (owing to~\eqref{eq:s_0} and~\eqref{eq:ellTdef}),
\begin{equation*}
\Q\left( \left. E_8^c\right| \mathcal{F}_{s_2}\right)\leq  \Cr{SEPcoupling2}\times 3(\nu+1)T(s_3-s_2)\exp\big(-(\Cr{SEPcoupling2}(1+\nu^{-1}))^{-1} \epsilon^2 (s_3-s_2)^{1/4}\big)\leq \exp\left(-\ell_g^4\right),
\end{equation*}
which implies that
\begin{equation}\label{eq:QE8better}
\Q\left( \left. E_8^c\right| E_{1-7}\right)\leq  \exp\left(-\ell_g^4\right).
\end{equation}
Next, we control the probability of $E_{10}$ defined at \eqref{eq:E10deff}. To do so, we take a union bound over $\ell'\in [\lfloor\lfloor (\log L)^2\rfloor/2\rfloor,\lfloor (\log L)^2\rfloor]$ and, using that $(\eta_0^-,\eta_0^+)$ is $(M,L)$-balanced, we apply \ref{pe:densitychange} twice for any fixed $\ell'$ in this interval: once with $(\eta_0, \rho, \varepsilon)=(\eta^-_0,\rho,\epsilon/20 )$, once with $(\eta_0, \rho, \varepsilon)=(\eta^+_0,\rho+\epsilon,\epsilon/20 )$ and both times with $(\ell,H,t)=(\lfloor (\log L)^2\rfloor, M, s_3)$. Note once again that for $L$ large enough by~\eqref{eq:s_0} and~\eqref{eq:ellTdef}, we have $M>H>4\nu t>\Cr{densitystable}\ell^2\varepsilon^{-2}(1+\vert \log^3(\nu t)\vert)$ and $\ell'\leq \ell_g^2\leq \sqrt{t}$, as required for~\ref{pe:densitychange} to apply. This yields, applying a union bound over the values of $\ell'$, that for large enough $L$ (recalling that $M\leq 20(\nu+1)L$),
\begin{equation}\label{eq:QE10}
\Q(E_{10}^c) \leq 8M\lfloor (\log L)^2\rfloor\exp\left(-\Cr{densitystableexpo} \epsilon^2 \lfloor \log^2 L\rfloor /400\right)\leq \exp(-\ell_g^{1500}).
\end{equation} 
Finally, recalling the coupling~\eqref{eq:etapmE1810} (together with~\eqref{eq:E8deff},~\eqref{eq:E9deff} and~\eqref{eq:E10deff}), we have by~\ref{pe:compatible}, more precisely by~\eqref{eq:compatible1}, for large enough $L$,
\begin{equation}\label{eq:QE9}
\Q(E_9^c\,\vert\, E_{1-8}\cap E_{10})\leq 20T\exp(-\nu(T-s_3)/4)\leq \exp(-\ell_g^{10^5}).
\end{equation}
Putting together~\eqref{eq:QE17},~\eqref{eq:QE8better},~\eqref{eq:QE10} and~\eqref{eq:QE9}, we obtain~\eqref{eq:QEgood}, as desired. 
\\

\noindent
\textbf{Proof of~\eqref{eq:QEbad}.} We bound individually the six terms comprising $S$ on the left-hand side of~\eqref{eq:QEbad}. The first of these is already controlled by~\eqref{eq:QE1c}. Recall the coupling defined in~\eqref{eq:etapmE1E2c} together with the definitions of $E_1$, $E_2$ and $E_{2,\textnormal{bis}}$ in~\eqref{eq:E1deff},~\eqref{eq:E2deff} and~\eqref{eq:E2bisdeff}. Observe that by~\eqref{eq:etapmE1E2c} and~\ref{pe:drift}, we have $\Q$-a.s.~on $E_1\cap E_2^c$ that $\Q( E_{2,\textnormal{bis}}^c| \mathcal{F}_{s_0})\leq 20\exp( -\nu (T-s_0)/4 ).$ Hence for large enough $L$, by~\eqref{eq:s_0} and~\eqref{eq:ellTdef}, this yields
\begin{equation}\label{eq:QE1E2cE2bisc}
\Q(E_1\cap E_2^c\cap E_{2,\textnormal{bis}}^c)\leq \exp(-\ell_g^{10^5}). 
\end{equation}
Recall now the coupling~\eqref{eq:etapmnacellecoupling} and the definition~\eqref{eq:E3deff} of $E_3$. By~\eqref{eq:penacelleNEW2} which is in force, we have
\begin{equation}\label{eq:QE12E3c}
\Q(E_{1-2}\cap E_3^c)\leq 20\exp\left(-\nu \ell_g \lfloor s_1/\ell_g \rfloor /4 \right)\leq \exp(-\ell_g^{10^5}),
\end{equation}
for $L$ large enough due to~\eqref{eq:s_0} and~\eqref{eq:ellTdef}. Next, remark that 
\begin{equation}\label{eq:QE13E45cE4biscdecompo}
E_{1-3}\cap E_{4-5}^c\cap E_{4,\textnormal{bis}}^c \subseteq \{E_{1-4}\cap E_5^c \} \cup \{E_{1-3}\cap E_4^c\cap E_{4,\textnormal{bis}}^c \}.
\end{equation}
On one hand, by~\eqref{eq:etapmE14} and~\ref{pe:drift} (recalling~\eqref{eq:E5deff}) we have
\begin{equation}\label{eq:QE14E5c}
\Q(E_{1-4}\cap E_5^c)\leq 20\exp\left(-\nu (s_2-s_1)\lfloor s_1/(s_2-s_1)\rfloor/4\right). 
\end{equation}
On the other hand, by~\eqref{eq:etapmE13E4c} and~\ref{pe:drift} (recalling~\eqref{eq:E4bisdeff}), we have
\begin{equation}\label{eq:QE13E4cE4bisc}
\Q(E_{1-3}\cap E_4^c\cap E_{4,\textnormal{bis}}^c )\leq 20\exp\left(-\nu (T-s_1)/4\right)
\end{equation}
Together,~\eqref{eq:QE13E45cE4biscdecompo},~\eqref{eq:QE14E5c} and~\eqref{eq:QE13E4cE4bisc} yield that
\begin{equation}\label{eq:QE13E45cE4bisc}
\Q(E_{1-3}\cap E_{4-5}^c\cap E_{4,\textnormal{bis}}^c)\leq \exp(-\ell_g^{10^5}), 
\end{equation}
for $L$ large enough, again via~\eqref{eq:s_0} and~\eqref{eq:ellTdef}. Similarly, by~\eqref{eq:etapmE15E67c},~\ref{pe:drift} and~\eqref{eq:E6bisdeff}, for $L$ large enough we have that
\begin{equation}\label{eq:QE15E67cE6bisc}
\Q(E_{1-5}\cap E_{6-7}^c\cap E_{6,\textnormal{bis}}^c)\leq \exp(-\ell_g^{10^5}). 
\end{equation}

In view of~\eqref{eq:QEbad}, putting together~\eqref{eq:QE1c},~\eqref{eq:QE1E2cE2bisc},~\eqref{eq:QE12E3c},~\eqref{eq:QE13E45cE4bisc} and~\eqref{eq:QE15E67cE6bisc} yields that
\begin{equation}
S\leq \Q(E_{1-7}\cap E_{8-9}^c)+5\exp(-\ell_g^{10^5}).
\end{equation}
By~\eqref{eq:QEgood}, which has been already established, we know that $5\exp(-\ell_g^{10^5})\leq \Q(E_{1-9})/(20T)$ for $L$ large enough by \eqref{eq:ellTdef}. Hence, in order to conclude the proof, it is enough to show that
$\Q(E_{1-7}\cap E_{8-9}^c)\leq \Q(E_{1-9})/(20T)$,  or equivalently that
\begin{equation}\label{eq:QE89cond17}
\Q(E_{8-9}\,\vert\, E_{1-7})\geq 1-({20T})^{-1}.
\end{equation}
Indeed,
\begin{multline*}
\Q(E_{8-9}\,\vert\, E_{1-7})= \Q(E_8\,\vert \, E_{1-7})\Q(E_9\,\vert \, E_{1-8})=\Q(E_8\,\vert \, E_{1-7})\Q( E_{1-9}\,\vert\, E_{1-8}\cap E_{10} )\frac{\Q( E_{1-8}\cap E_{10} )}{\Q( E_{1-8})}
\\
\geq \Q(E_8\,\vert \, E_{1-7})\Q( E_{1-9}\,\vert\, E_{1-8}\cap E_{10} )\left(1- \frac{\Q(E_{10}^c)}{\Q( E_{1-9})} \right) \\ \geq \left(1-\exp(-\ell_g^4)\right)\left(1-\exp(-\ell_g^{10^5})\right)\left(1-\exp(-\ell_g^{1500})\exp(\ell_g^{22})\right), 
\end{multline*}
by virtue of~\eqref{eq:QE8better},~\eqref{eq:QE9},~\eqref{eq:QE10} and~\eqref{eq:QEgood} in the last line. For large enough  $L$ (recalling~\eqref{eq:s_0} and~\eqref{eq:ellTdef}), this readily yields~\eqref{eq:QE89cond17} and concludes the proof of~\eqref{eq:QEbad}.
\\

\noindent
\textbf{Part V: Proof of~\eqref{e:coup-restart balanced}.} \\

Recall the event $E_{\textnormal{restart}}$ from \eqref{e:coup-restart balanced}. The proof of \eqref{e:coup-restart balanced} is relatively straightforward at this point, we just need to keep careful track of those events in the above construction that can force us out of the event $E_{\textnormal{restart}}$; this is key to get the rapid decay in \eqref{e:coup-restart balanced}. \black First note that $E_{\textnormal{neutral}}\subseteq E_{2,\textnormal{bis}}\cup E_{4,\textnormal{bis}}\cup E_{6,\textnormal{bis}}$ by~\eqref{def:softfailure}, so that by~\eqref{eq:E2bisdeff},~\eqref{eq:E4bisdeff},~\eqref{eq:E6bisdeff}  and \eqref{eq:Eneutral} (and using that $|Y_T^{\pm}| \leq T$ along with \eqref{eq:ellTdef}), \black $E_{\textnormal{restart}}$ holds on $E_{\textnormal{neutral}}$ for large enough $L$ (tacitly assumed in the sequel), \black hence $E_{\textnormal{restart}}^c$ can only happen on $E_{\textnormal{good}}\cup E_{\textnormal{bad}}$. As we now explain, looking at the decomposition of $E_{\textnormal{good}}$ and $E_{\textnormal{bad}}$ at~\eqref{def:success} and~\eqref{def:catastrophe}, and inspecting closely the construction $\eta^\pm$ (and $Y^{\pm}$), especially in Part I of the proof, starting with the paragraph of \eqref{eq:E10deff} until \eqref{eq:etapmE17E8c10c}\black, one notices that $E_{\textnormal{restart}}^c$ can in fact only happen:
\begin{itemize}
\item[(i)] on all but the last event (i.e.~all but $E_{1-7}\cap E_{8-9}^c$) defining $E_{\textnormal{bad}}$ at~\eqref{def:catastrophe}, or
\item[(ii)] on $E_{1-7}\cap E_{10}^c$ (see \eqref{eq:E10deff}), or
\item[(iii)] on $E_{1-7}\cap E_8^c\cap  E_{10}$ if the coupling~\ref{pe:couplings} applied at~\eqref{eq:etapmE17E8cPARACHUTE} fails, i.e.~if the event
\begin{equation}
E_{11}\stackrel{\text{def.}}{=} \big\{\eta^+_{T \black}(\cdot -T)\vert_{[-M+T+4\nu T,M-T-4\nu T ]}\succcurlyeq \eta^-_{ T \black}(\cdot +T)\vert_{[-M+T+4\nu T,M-T-4\nu T ]} \big\}
\end{equation}
(cf.~\eqref{eq:doublecoupleSEP}) does not occur, or
\item[(iv)] on $E_{1-8}\cap E_9^c\cap E_{10}$, or
\item[(v)] on $E_{1-10}$ if the coupling~\ref{pe:compatible} (more precisely~\eqref{eq:compatible2}) applied at~\eqref{eq:etapmE1810} fails, i.e.~if the event
\begin{equation}
E_{12}\stackrel{\text{def.}}{=} \big\{\eta^+_{ T \black}(\cdot +Y^-_{s_1}+\ell_g^{20})\vert_{I}\succcurlyeq \eta^-_{ T \black}(\cdot +Y^-_{s_1}-\ell_g^{20})\vert_I \big\}
\end{equation}
with $I= [-M+T+6\nu T,M-T-6\nu T ]$ does not occur.
\end{itemize}
We now detail how the cases~(i)-(v) arise. First note that, since $ E_{\textnormal{restart}}\subset (E_{\textnormal{bad}}\cup E_{\textnormal{good}})$ as established above, and since $E_{\textnormal{good}}= E_{1-9}$ by definition (see \eqref{def:success}), after discarding item~(i) from the above list it only remains to investigate matters on the event $E_{1-7}$, and the cases considered in items (ii)-(v) indeed form a partition of this event, save for the additional specifications (``if the coupling...'') in items (iii) and (v), which we now discuss. For item (iii), note indeed that if $E_{1-7}\cap E_8^c\cap  E_{10} \subseteq E_{\textnormal{bad}}$ holds (see \eqref{def:catastrophe}) then ${Y}^+_T=Y^-_T-2T$ by Lemma~\ref{L:determ-speed}, and that since $\vert Y^-_T\vert \leq T$, if in addition $E_{11}$ holds then so does $E_{\textnormal{restart}}$ in view of \eqref{e:coup-restart balanced}. Similarly, regarding item (v), we have by \eqref{def:success} that $E_{1-10}\subseteq E_{\textnormal{good}}$ so that if $E_{1-10}$ holds, $\widetilde{Y}^+_T=Y^-_T+2\ell_g^{20}$ by Lemma~\ref{L:determ-speed}. Since $\vert Y^-_T-(Y^-_{s_1}-\ell_g^{20})\vert \leq T-s_1+\ell_g^{20}\leq T$, if in addition $E_{12}$ occurs then $E_{\textnormal{restart}}$ occurs as well.

\black
Combining items~(i)-(v) above and recalling \eqref{def:catastrophe} in the context of item~(i), by a union bound we have that 
\begin{multline*}
\Q(E_{\textnormal{restart}}^c)\leq  
\Q(E_1^c) +\Q(E_1\cap E_2^c\cap E_{2,\textnormal{bis}}^c )  \\+\Q(E_{1-2}\cap E_3^c) +\Q(E_{1-3}\cap E_{4,5}^c\cap E_{4,\textnormal{bis}}^c )
 +\Q(E_{1-5}\cap E_{6-7}^c\cap E_{6,\textnormal{bis}}^c)   + \Q(E_{10}^c)\\
+\Q(E_{1-8}\cap E_9^c\cap E_{10})+\Q(E_{1-7}\cap E_8^c\cap E_{10}\cap E_{11}^c)+\Q(E_{1-8}\cap E_{10}\cap E_{12}^c).
\end{multline*}
By~\eqref{eq:QE1c},~\eqref{eq:QE10},~\eqref{eq:QE9},~\eqref{eq:QE1E2cE2bisc},~\eqref{eq:QE12E3c},
~\eqref{eq:QE13E45cE4bisc} and~\eqref{eq:QE15E67cE6bisc}, for large enough $L$ we obtain that
\begin{equation}\label{eq:Erestartcdecompo}
\Q(E_{\textnormal{restart}}^c)\leq 7\exp(-\ell_g^{1500})+Q(E_{1-7} \cap E_{10}\cap E_8^c\cap E_{11}^c)+\Q(E_{1-8}\cap E_{10}\cap E_{12}^c).
\end{equation}
As to the last two terms in \eqref{eq:Erestartcdecompo}, using the coupling defined in \eqref{eq:etapmE17E8cPARACHUTE} and~\ref{pe:couplings}, we get 
\begin{multline}\label{eq:Erestartccoupling1}
Q( E_{11}^c\,\vert\,E_{1-7}\cap E_8^c\cap E_{10})\leq \Cr{SEPcoupling2} TM\exp\big(-(\Cr{SEPcoupling2}(1+\nu^{-1}))^{-1}\epsilon^2(T-s_3)^{1/4} \big)\\
\le C_3 2\ell_g^{10^6} \cdot 20(\nu+1) L \cdot \exp\big(-(\Cr{SEPcoupling2}(1+\nu^{-1}))^{-1}\epsilon^2\ell_g^{2\cdot 10^5} \big) \le \exp\big(-\ell_g^{10^5} \big),
\end{multline}
where we used \eqref{eq:ellTdef}-\eqref{eq:s_0}, the fact that $M\le 20(\nu +1) L$ and took $L$ large enough.
Similarly, using the coupling defined in~\eqref{eq:etapmE1810} and~\ref{pe:compatible} (more precisely~\eqref{eq:compatible2}), we obtain that
\begin{equation}\label{eq:Erestartccoupling2}
Q( E_{12}^c\,\vert\,E_{1-8}\cap E_{10})\leq 5\Cr{SEPcoupling2} \ell_G^{1/20} M\exp\big(- (\Cr{SEPcoupling2}(1+\nu^{-1}))^{-1}\epsilon^2 \ell_G^{1/100}\big)\le \exp(-\ell_g^{2\cdot 10^3}).
\end{equation}
Finally, putting \eqref{eq:Erestartcdecompo},~\eqref{eq:Erestartccoupling1} and~\eqref{eq:Erestartccoupling2} together, and using \eqref{eq:ellTdef} with $L$ large enough, leads to $\Q(E_{\textnormal{restart}}^c)\leq L^{-100}$.
This concludes the proof of~\eqref{e:coup-restart balanced} and thus of Lemma~\ref{L:Q coupling-alternative}, taking $L_2$ large enough so that~\eqref{eq:etapmnacellecoupling}-\eqref{eq:Erestartccoupling2} (which represent a finite number of constraints) hold.
\end{proof}

\begin{Rk}\label{Rk:couplingprecise}
In the coupling $\mathbb{Q}$ constructed in Lemma~\ref{L:Q coupling-alternative}, $U$ is a priori not independent from $(\eta^+,\eta^-)$; for instance~a synchronous evolution of $\eta^+(\cdot+2\ell_g^{20})$ and $\eta^-$ after $s_3$ could indicate that $E_6$, which depends on $U$ during $[s_1,s_2]$, has happened.
\end{Rk}

\subsection{Proof of Proposition~\ref{Prop:initialspeed}}

\label{subsec:chaining}
Let $L\geq 1$ be an integer satisfying $L\ge L_2(\rho,\epsilon,\nu)$, where $L_2$ is given by Lemma~\ref{L:Q coupling-alternative}. Let $k:=\lfloor L/T\rfloor$, with $T$ as defined in~\eqref{eq:ellTdef}. We start with a brief overview of the proof. To deduce \eqref{eq:vLincrease}, we will couple two walks $\widehat{X}^+\sim \mathbb{P}^{\rho+\epsilon}$ and $\widehat{X}^-\sim \mathbb{P}^{\rho}$ on the time interval $[0,iT]$, $1\leq i \leq k$, recursively in $i$. The processes $\widehat{X}^{\pm}$ will be specified in terms of  associated environments $\hat{\eta}^+\sim \mathbf{P}^{\rho+\epsilon}$, $\hat{\eta}^-\sim \mathbf{P}^{\rho}$, and an i.i.d.~array $\widehat{U}=(\widehat{U}_{w})_{w\in \mathbb{L}}$ using Lemma~\ref{L:Q coupling-alternative} repeatedly, cf.~\eqref{e:coup-marg}-\eqref{e:coup-domination}. We will denote $\widehat{\Q}$ the associated coupling measure defined below, which will also comprise associated auxiliary walks $\widehat{Y}^\pm$ that will be defined using the construction of $Y^\pm$ in Section~\ref{subsec:onegap}, iterated over $i$ and allow to keep control on the gap between $\widehat{X}^{+}$ and $\widehat{X}^-$ via a combination of \eqref{e:coup-domination}-\eqref{e:coup-speed}. The very possibility of iteration is guaranteed by the high-probability event $E_{\textnormal{restart}}$ in \eqref{e:coup-restart balanced}.

We now proceed to make the above precise.  For later reference we set $M_i=20(\nu+1) L-(8\nu+3)i$, for $i\in\{0,\dots,k-1\}$. We further recall the filtration $\mathcal{F}_t$ defined in \eqref{e:Ft-def} and write $\widehat{\mathcal{F}}_t$ below when adding hats to all processes involved. Finally, for all $i\in\{0,\dots,k-1\}$, let
\begin{equation} \label{eq:def-Bi}
B_i\stackrel{\text{def.}}{=}\bigcap_{j=0}^i \left\{  ( \hat\eta^+_{jT}(\cdot +\widehat{Y}^{+}_{jT}),\hat\eta^-_{jT}( \cdot + \widehat{Y}^{-}_{jT})  \text{ is }(M_j,L)\text{-balanced}   \right\}
\end{equation}
(see items (i)-(iii) above Lemma~\ref{L:Q coupling-alternative} for notation).

 By successive extensions of $\widehat{\Q}$, we will construct a coupling such that the following hold for all $L \geq L_2$ (as supplied by Lemma~\ref{L:Q coupling-alternative}) and $i\in\{0,\dots, k\}$:

\begin{itemize}
\item[($\widehat{\Q}$-1)] The processes $(\hat\eta^+_t, \hat\eta^-_t)_{0\leq t\leq iT}$, $(\widehat{U}_w)_{w\in \mathbb{L}, \pi_2(w)\le iT-1}$ (absent when $i=0$) and $(\widehat{X}^{\pm}_t)_{0\le t\le iT}$ are defined under $\widehat{\Q}$ with the correct marginal laws. That is, $\hat \eta^+_0\sim \mu_{\rho+\epsilon}$, $\hat \eta^-_0\sim \mu_\rho$ and $(\hat\eta^{\pm}_t)_{0< t\leq iT}$ has the same law as (the restriction to $[0,iT]$) of $\eta^\pm$ under $\mathbf{P}^{\hat \eta_0^\pm}$. Moreover, $(\widehat{U}_w)_{w\in \mathbb{L}, \pi_2(w)\le iT-1}$ are i.i.d.~uniform variables on $[0,1]$, and given $\hat\eta_0^{\pm}$, $(\widehat{X}^{\pm}_t)_{0\le t\le iT}$ is $\widehat{\mathcal{F}}_{iT}$-measurable and has the law of (the restriction to $[0,iT]$ of) $X^{\pm}$ under $\mathbb{P}^{\hat\eta_0^{\pm}}$.
\item[($\widehat{\Q}$-2)] The processes $(\widehat{Y}^{\pm}_t)_{0\le t\le iT}$ are $\widehat{\mathcal{F}}_{iT}$-measurable and $\widehat{X}^-_{iT}\le \widehat{Y}^-_{iT}$ and $\widehat{X}^+_{iT}\ge \widehat{Y}^+_{iT}$ hold $\widehat{\Q}$-a.s.
\item[($\widehat{\Q}$-3)] $\widehat{Y}^{+}_0= \widehat{Y}^{-}_0=0$, and if $i \geq1$, with $B_i$ as in \eqref{eq:def-Bi},
\begin{equation} \label{eq:ind-speedincrease}
\mathbb{E}^{\widehat{\Q}}\left[  \left. \big(\widehat{Y}^{+}_{iT}- \widehat{Y}^{+}_{(i-1)T}\big)-\big(\widehat{Y}^{-}_{iT}- \widehat{Y}^{-}_{(i-1)T} \big)\right| \widehat{\mathcal{F}}_{(i-1)T}  \right] 1_{B_{i-1}} \geq \exp(-(\log L)^{1/20}).
\end{equation}
\item[($\widehat{\Q}$-4)] $\widehat{\Q}(B_0^c)=0$ and if $i \geq 1$, $\widehat{\Q}(B_i^c \vert B_{i-1}) \leq L^{-100}$.
\end{itemize}

For $i=0,$ we simply couple under $\widehat{\Q}$ two configurations $\hat \eta^+_0\sim \mu_{\rho+\epsilon}$ and $\hat \eta^-_0\sim \mu_\rho$ such that a.s.~$ \hat \eta^+_0(x)\geq \hat \eta^-_0(x)$ for all $x\in \mathbb{Z}$, which we can do with probability one by~\eqref{pe:monotonicity}. We set $\widehat{X}^+_0=\widehat{X}^-_0=\widehat{Y}^{+}_0=\widehat{Y}^{-}_0=0$. Thus ($\widehat{\Q}$-1) and ($\widehat{\Q}$-2) are satisfied, and ($\widehat{\Q}$-3) is trivial. Finally $\widehat{\Q}(B_0^c)=0$ since 
$ \hat \eta^{\pm}_0$ are in particular $(M_0,L)$-balanced, whence ($\widehat{\Q}$-4) holds.

Assume by induction that for some $i\in\{0,\dots, k-1\}$, we have constructed a coupling $\widehat{\Q}$ with the above properties. We now proceed to extend $\widehat{\Q}$ so as to have ($\widehat{\Q}$-1)-($\widehat{\Q}$-4) with $(i+1)$ in place of $i$.
We first specify matters on  the event $B_i^c$. Conditionally on $\widehat{\mathcal{F}}_{iT}$, if $B_i^c$ occurs, we let $\hat\eta^+_t$ and $\hat\eta^-_t$ evolve independently for $iT< t\le (i+1)T$ according to $ \mathbf{P}^{\hat\eta^+_{iT}}$ and $\mathbf{P}^{\hat\eta^-_{iT}}$ respectively, and independently of this, we choose $\widehat{U}_w$ as uniform random variables on $[0,1]$ in an i.i.d.~manner, for $w$ such that $iT\le \pi_2(w)\le (i+1)T -1$. On $B_i^c$, we further let $\widehat{Y}^{+}_{iT+t}=\widehat{Y}^{+}_{iT}-t$ and $\widehat{Y}^{-}_{iT+t}=\widehat{Y}^{-}_{iT}+t$, for all $0 < t\le T$ and $(X^\pm_t)_{iT\leq t\leq (i+1)T}$ evolve as in~\eqref{def:A} and~\eqref{eq:defX} with $(\eta^\pm,U)$ instead of $(\eta,U)$. With these choices it is clear that the inequalities in ($\widehat{\Q}$-2) hold on $B_i^c$, since for instance $\widehat{Y}^+_{(i+1)T}\leq \widehat{Y}^+_{iT}-T \leq X^+_{iT}-T\leq X^+_{(i+1)T} $ $\widehat{\Q}$-a.s., using the induction hypothesis and the fact that increments of $X^+$ are  bounded from below by $-1$. The inequality $\widehat{X}^-_{(i+1)T}\le \widehat{Y}^-_{(i+1)T}$ is derived similarly.

We now turn to the case that $B_i$ occurs, which brings into play Lemma~\ref{L:Q coupling-alternative}.
Conditionally on $\widehat{\mathcal{F}}_{iT}$ and on the event $B_i$, we couple $(x,t)\mapsto \hat\eta^+_{iT+t}(x+\widehat{Y}^+_{iT})$ and $(x,t)\mapsto \hat\eta^-_{iT+t}(x+\widehat{Y}^-_{iT})$ for $x\in\mathbb{Z}$ and $t\in[0,iT]$, as well as $( \widehat{U}_{w+(0,iT)}: w \in \mathbb{L}, \,  \pi_2(w)\le T -1)$ following the coupling of $(\eta^+,\eta^-, U)$ provided by Lemma \ref{L:Q coupling-alternative}, with the choice $M=M_j$. The requirement of $(M,L)$-balancedness of the initial condition needed for Lemma~\ref{L:Q coupling-alternative} to apply is precisely provided by $B_i$, cf.~\eqref{eq:def-Bi}.

Combining~\eqref{e:coup-marg}, the Markov property~\eqref{pe:markov} applied at time $iT$, and in view of the choices made on $B_i^c$, it readily follows that the processes
 $(\hat\eta^+_t, \hat\eta^-_t)_{0\leq t\leq (i+1)T}$, $(\widehat{U}_w)_{w\in \mathbb{L}, \pi_2(w)\le (i+1)T-1}$ thereby defined have the marginal laws prescribed in ($\widehat{\Q}$-1). Moreover, by   above application of Lemma~\ref{L:Q coupling-alternative}, the processes $X^{\pm}$ and $Y^{\pm}$ satisfying all of \eqref{e:coup-domination}-\eqref{e:coup-restart balanced} are declared. Thus, setting 
 \begin{equation}\label{eq:Xpm-induction}
 \begin{split}
  &\widehat{X}^{\pm}_{iT+t}=\widehat{X}^{\pm}_{iT}+X^{\pm}_t,\quad \widehat{Y}^{\pm}_{iT+t}=\widehat{Y}^{\pm}_{iT}+Y^{\pm}_t, \quad \text{for all $0 \leq t \leq T$},
 \end{split}
 \end{equation}
 it readily follows, combining~\eqref{eq:Xpm-induction} and the induction assumption on the law of 
 $(\hat\eta^{\pm}_t)_{0\leq t\leq iT}$, $(\widehat{X}^{\pm}_t)_{0\le t\le iT}$, combined with the Markov property \eqref{pe:markov} and that of the quenched law, that $(\widehat{X}^{\pm}_t)_{0\le t\le (i+1)T}$ declared by \eqref{eq:Xpm-induction}
 has the desired marginal law, thus completing the verification of 
 ($\widehat{\Q}$-1) with $(i+1)$ in place of $i$. Next, we show ($\widehat{\Q}$-3) and ($\widehat{\Q}$-4), before returning to ($\widehat{\Q}$-2). Since $\widehat{Y}^{\pm}_{(i+1)T}-\widehat{Y}^{\pm}_{iT} = Y^{\pm}_T$ by \eqref{eq:Xpm-induction}, the inequality~\eqref{eq:ind-speedincrease} with $(i+1)$ in place of $i$ is an immediate consequence of \eqref{e:coup-speed}. Hence  
($\widehat{\Q}$-3) holds. Finally, by construction of the coupling extension on the event $B_i$, which uses Lemma~\ref{L:Q coupling-alternative}, and in view of \eqref{eq:Xpm-induction} and
 \eqref{eq:def-Bi}, the failure of $B_{i+1}$ on the event $B_i$ amounts to the failure of $E_{\textnormal{restart}}$ in \eqref{e:coup-restart balanced}, from which ($\widehat{\Q}$-4) 
 follows with $(i+1)$ in place of $i$.

 It remains to show that ($\widehat{\Q}$-2) holds with $(i+1)$ in place of $i$. To this effect, we introduce two auxiliary processes $(\widehat{Z}^{\pm}_t)_{0 \leq t \leq (i+1)T}$, defined as
 \begin{equation}\label{eq:Z-induction}
 \widehat{Z}^{\pm}_t = \begin{cases}
 \widehat{Y}^{\pm}_{t}, & \text{ if } 0 \leq t \leq iT,\\
 \widehat{Y}^{\pm}_{iT} + X_t^{\pm} , & \text{ if } 0 < t \leq T.
 \end{cases}
 \end{equation}
 Combining the induction assumption ($\widehat{\Q}$-2), the definition of $ \widehat{Y}^{\pm}_t$ and $\widehat{Z}^{\pm}_t$ in \eqref{eq:Xpm-induction} and \eqref{eq:Z-induction} (the latter implying in particular that $\widehat{Z}^{\pm}_{iT}= \widehat{Y}^{\pm}_{iT}$), one readily deduces from property \eqref{e:coup-domination} the $\widehat{\mathbb{Q}}$-almost sure inequalities 
 \begin{equation}\label{eq:YZ}
  \widehat{Z}^{-}_{(i+1)T} \leq \widehat{Y}^{-}_{(i+1)T}  \text{ and } \widehat{Y}^{+}_{(i+1)T} \leq \widehat{Z}^{+}_{(i+1)T}.
  \end{equation} To deduce from this the analogous inequalities with $\widehat{X}$ in place of $\widehat{Z}$, we apply Lemma~\ref{L:monotone}, with $(X, \widetilde{X})= (\widehat{Z}^+, \widehat{X}^+)$, $\eta= \tilde{\eta}= \eta^+$ (whence \eqref{domineta} plainly holds), $\pi_1(w')= \widehat{Z}^{+}_{iT}= \widehat{Y}^{+}_{iT}$, $\pi_1(w)= \widehat{X}^+_{iT}$, $\pi_2(w)=\pi_2(w')$ and $K= \mathbb{Z} \times[iT, (i+1)T]$ to deduce that $\widehat{Z}^{+}_{(i+1)T} \leq \widehat{X}^{+}_{(i+1)T}$ holds $\widehat{\mathbb{Q}}$-a.s. Note that the condition
  $\pi_1(w')\le \pi_1(w)$ necessary for Lemma~\ref{L:monotone} to apply is in force by induction hypothesis in ($\widehat{\Q}$-2).  Together with \eqref{eq:YZ}, this yields one of the desired inequalities in ($\widehat{\Q}$-2) with $i+1$ in place of $i$. The other one is obtained in a similar way using Lemma~\ref{L:monotone}. This completes the proof of the induction step.
  \black

We now use the coupling $\widehat{\mathbb{Q}}$, which satisfies 
 ($\widehat{\Q}$-1)-($\widehat{\Q}$-4) for all $0 \leq i \leq k $ (and $L \geq L_2$), to complete the proof of \eqref{eq:vLincrease}.
To this end, we first extend the laws of $\widehat{X}_t^{\pm}$ to all $t >kT$ using the Markov property, by sampling $\widehat{X}^{\pm}_{kT+ \cdot}$ independently conditionally on $\widehat{\mathcal{F}}_{kT}.$ In particular, recalling that $k=\lfloor L/T\rfloor$, this implies that $\widehat{X}_L^{\pm}$ is declared under $\widehat{\mathbb{Q}}$. We thus proceed to derive a suitable lower bound on $\mathbb{E}^{\widehat{\Q}}[\widehat{X}_L^{+}-\widehat{X}_L^{-}]$, which is well-defined, and from which 
\eqref{eq:vLincrease} will follow. 

Using that $|\widehat{X}_{n+1}^{\pm} - \widehat{X}_n^{\pm}| \leq 1$ for any $n \geq 0$, we obtain (with $k=\lfloor L/T\rfloor$) that
\begin{multline}\label{haaa}
\mathbb{E}^{\widehat{\Q}}[\widehat{X}_L^{+}-\widehat{X}_L^{-}]  + 2T \geq \mathbb{E}^{\widehat{\Q}}[\widehat{X}_L^{+}-\widehat{X}_L^{-}] + 2(L-kT) \ge \mathbb{E}^{\widehat{\Q}}\big[   \widehat{X}^+_{kT}-\widehat{X}^-_{kT}   \big] \\
\stackrel{\textnormal{($\widehat{\Q}$-2)}}{\ge} \mathbb{E}^{\widehat{\Q}}\big[  \widehat{Y}^+_{kT}-\widehat{Y}^-_{kT} \big]
 {\ge} \sum_{1\leq i \leq k} \mathbb{E}^{\widehat{\Q}}\Big[ \big(\widehat{Y}^{+}_{iT}- \widehat{Y}^{+}_{(i-1)T}\big)-\big(\widehat{Y}^{-}_{iT}- \widehat{Y}^{-}_{(i-1)T} \big)   \Big]\\
 \stackrel{\textnormal{($\widehat{\Q}$-3)}}{\ge} ke^{-(\log L)^{1/20}} - 2T \sum_{1 \leq i \leq k}\widehat{\Q}(B_{i-1}^c) \geq \frac{L}{2T}e^{-(\log L)^{1/20}} - 2L\widehat{\Q}(B_{k-1}^c),
\end{multline}
 for large enough $L$, where, in the second inequality of the second line, we have used that $\widehat{Y}^{\pm}_{0}=0$, see ($\widehat{\Q}$-3), and in the last line, we have first used that the difference of increments is deterministically bounded from below
 by $-2T$ (see \eqref{eq:Xpm-induction} and Lemma~\ref{L:determ-speed}), and for the last inequality that the events $B_i^c$ are increasing in $i$, cf.~\eqref{eq:def-Bi}. We also used in various places that $L-T\le kT\le L$.
 
It remains to suitably estimate the probability $\widehat{\Q}(B_{k-1}^c)$ appearing in the last line of \eqref{haaa}, for which we use ($\widehat{\Q}$-4) and a straightforward induction argument to bound
\begin{equation*}
\widehat{\Q}(B_{k-1}^c)\leq \widehat{\Q}(B_{k-1}^c \vert B_{k-2}) + 
\widehat{\Q}(B_{k-2}^c) \stackrel{\text{($\widehat{\Q}$-4)}}{\leq} L^{-100} + \widehat{\Q}(B_{k-2}^c) \leq \dots \leq   (k-1)L^{-100}\leq L^{-99},
\end{equation*}
using also in the penultimate step that $\widehat{\Q}(B_{0}^c)=0$. Feeding this into \eqref{haaa} yields that
\begin{equation}\label{eq:speedincreasealmostlinear}
\begin{split}
\mathbb{E}^{\widehat{\Q}}[\widehat{X}_L^{+}-\widehat{X}_L^{-}]& \ge \frac{L}{2T}e^{-(\log L)^{1/20}}  - 2L^{-98} -2T \ge 3\Cr{C:approx} (\log L)^{100},
\end{split}
\end{equation}
as soon as $L$ is large enough (recall $T$ from \eqref{eq:ellTdef}). Dividing by $L$ and applying \eqref{eq:defvL} whilst observing that $\widehat{X}_L^{+}$ has the same law under $\widehat{\Q}$ as $X_L$ under $\P^{\rho+\epsilon, L}$ and  $\widehat{X}_L^{-}$ has the same law under $\widehat{\Q}$ as $X_L$ under $\P^{\rho, L}$, \eqref{eq:speedincreasealmostlinear} implies that
\begin{equation}
 v_L(\rho+\delta)-v_L(\rho)\geq {3\Cr{C:approx}L^{-1} (\log L)^{100} },
\end{equation}
which concludes the proof of Proposition~\ref{Prop:initialspeed}. \hfill $\qedsymbol$

\begin{Rk}\label{Rk:initialstepstrong}
As mentioned at~\eqref{eq:sizeablemargin}, we prove in fact a much stronger statement than Proposition~\ref{Prop:initialspeed}, owing to~\eqref{eq:speedincreasealmostlinear}, namely that 
\begin{equation}\label{eq:speedincreasealmostlinear2}
 v_L(\rho+\delta)-v_L(\rho)\geq L^{o(1)}
\end{equation}
where the $o(1)$ denotes a negative quantity that goes to 0 as $L\rightarrow\infty$. This is however insufficient to imply directly Theorem~\ref{T:generic}, and the renormalisation in Section~\ref{sec:approx} is still essential to improve~\eqref{eq:speedincreasealmostlinear2} to a right-hand side bounded away from 0 as $L\rightarrow \infty$. 
\end{Rk}

\begin{Rk}[Necessity for couplings with quenched initial condition]\label{Rk:quenchedconditions}
We explain here the main reason why we need quenched conditions for our couplings. In short, this is due to the lack of an invariant measure (or reasonable proxy thereof) from the point of view of the walk. 

More precisely, abbreviating $t=s_2-s_1$, the only a-priori lower bound we have for the probability that $X^\rho$ and $X^{\rho+\varepsilon}$ drift away linearly from each other during $[s_1,s_2]$ (corresponding to $\mathbb{Q}(E_6)$ at~\eqref{eq:QE6}) is $\exp(-Ct )$ for some large constant $C$ by uniform ellipticity -- we are in fact precisely trying to derive a better bound in this section. 

But this gap is necessary to create a difference between $X^{\rho+\epsilon}-X^\rho$ on a time interval of length $T$. Hence, to accrue a significant gain in expectation between $X^{\rho+\epsilon}-X^\rho$, we need to repeat this at least $\exp(Ct)$ times. Thus, during that time, $X^\rho$ and $X^{\rho+\epsilon}$ could straddle an interval of width at least $\exp(C t)$. The main issue is that we have no a priori information on their local environment (which would not be the case if we had access to an invariant measure and could estimate the speed of convergence to it). 
Hence if the coupling~\ref{pe:couplings}, that we use between $s_1$ and $s_2$ on a interval much narrower than $\exp(C t)$, was only valid under the annealed product Bernoulli initial condition (which is a priori \emph{not} what the walk sees), we could resort to the annealed-to-quenched trick (via Markov's inequality) and a union bound over $ \exp(Ct)$ intervals to control the probability that the coupling fails. However, the failure probability at~\ref{pe:couplings} is $\exp(-C't^{1/4}) \gg \exp(-Ct)$, hence we cannot obtain any non-trivial bound this way. Note that this does not depend on the choice of $t$. Furthermore due to the diffusivity of the environment particles and large deviation considerations, it seems unlikely that one could improve the bound of~\ref{pe:couplings} beyond $\exp(-Ct^{1/2})$.

This is why we resort to some quenched control, cf.~items (i)-(iii) above Lemma~\ref{L:Q coupling-alternative}, and also \eqref{eq:def-Bi}. The empirical density was the most accessible and relevant statistic (in particular if the environment is conservative, as is the case of SEP). For similar reasons, we had to establish~\ref{pe:densitychange} in a quenched setting, to ensure that whatever the distribution of the environment around the walker at time $iT$ for some $i\geq 1$ (as long as it is balanced), there is still a uniformly low probability not to have the required empirical density at time $s_1$ (see~\eqref{eq:QE7c}) to perform the coupling of~\ref{pe:couplings}. Of course the necessity for quenched couplings encapsulated in \ref{pe:densitychange} and in particular \ref{pe:compatible}, means that we have (more) work to do in order to verify this in specific instances, as we do for SEP in the next section.
\end{Rk}

\section{Exclusion process and couplings}\label{sec:SEP}

We start by giving in~\S\ref{subsec:SEP-basic} a formal definition of the main environment $\eta$ of interest in this article, the symmetric simple exclusion process (SEP), and first check that it fits the setup of~\S\ref{subsec:re}, in particular, that the basic properties \eqref{pe:markov}--\eqref{pe:density} listed in~\S\ref{subsec:re} hold. The main result of this section, proved in \S\ref{subsec:C-SEP}, is to show that the SEP satisfies the conditions~\ref{pe:densitychange}-\ref{pe:nacelle} stated in \S\ref{subsec:C}; see Proposition~\ref{P:SEP-C} below. This implies that our main result, Theorem~\ref{T:generic}, applies in this case; cf.~also Theorem~\ref{T:main1} and its proof in \S\ref{subsec:main}. We refer to Appendix~\ref{sec:PCRW} for another environment $\eta$ of interest which fits this framework.

\subsection{Definition of SEP and basic properties} \label{subsec:SEP-basic}

We fix a parameter $\nu>0$, which will be constant throughout this section and often implicit in our notation. The (rate $\nu$) symmetric simple exclusion process (SEP) is the Markov
process on the state space $ \{0,1\}^{\mathbb{Z}}$ (tacitly viewed as a subset of $\Sigma$, cf.~\S\ref{subsec:re}) with (pre-)generator
\begin{align}\label{eq:gen-SEP}
L f(\eta)=\sum_{x,y\in\mathbb{Z}: |x-y|=1} \1_{\{\eta(x)=1,\eta(y)=0\}}\frac{\nu}{2}\left(f(\eta_{xy})-f(\eta)\right),
\end{align}
for $\eta\in\{0,1\}^{\mathbb{Z}}$ and $f$ in the domain of $L$, where $\eta_{xy}$ is the configuration obtained from $\eta$ by exchanging the states of $x$ and $y$, i.e.~such that $\eta_{xy}(x)=\eta(y)$, $\eta_{xy}(y)=\eta(x)$ and $\eta_{xy}(z)=\eta(z)$ for all $z\in\mathbb{Z}\setminus\{x,y\}$; see \cite[Chap.~VIII]{MR0776231}. We denote $\mathbf{P}_{\textrm{SEP}}^{\eta_0}$ its canonical law with initial configuration $\eta_0$ and drop the subscript SEP whenever there is no risk of confusion. In words, \eqref{eq:gen-SEP} entails that the vertices $x$ such that $\eta(x)=1$, which can be seen as the locations of particles evolve like continuous-time symmetric simple random walks on $\mathbb{Z}$ with rate $\nu$ that obey the exclusion rule; that is, particles are only allowed to jump onto empty locations.

It will often be useful to consider the interchange process on $\mathbb Z$, with generator $\widehat{L}$ defined as in \eqref{eq:gen-SEP} but omitting the exclusion constraint $\{\eta(x)=1,\eta(y)=0\}$, which interchanges the state of neighbors $x$ and $y$ independently at rate $\nu/2$. We will use the following specific construction of this process. 
Let $E=\{ \{x,x+1\}: x \in \mathbb{Z}\}$ denote the set of edges on $\mathbb{Z}$ and $\widehat{\mathbf{P}}$ be a probability governing independent Poisson counting processes $\mathcal{P}_e$ of intensity $\nu/2$ on $\mathbb{R}_+$ attached to every edge $e \in E$. For any given ${\eta_0}\in \{0,1\}^{\mathbb Z}$, one defines $(\eta_t)_{t \geq 0}$ under $\widehat{\mathbf{P}}$ by exchanging the states of $\eta$ at $x$ and $y$ every time the `clock rings' for $\mathcal{P}_e$, where $e=\{x,y\}$. This is well-defined up to a set of measure zero. Then for every $\eta_0 \in \{0,1\}^{\mathbb{Z}}$,
\begin{equation}\label{eq:SEP-basic-coup}
\text{$(\eta_t)_{t \geq 0}$ has the same law under $\widehat{\mathbf{P}}$ and $\mathbf{P}_{\textrm{SEP}}^{\eta_0}$}.
\end{equation}
 This follows upon observing that the states of neighboring sites suffering the exclusion constraint can also be exchanged. For our purpose, these two processes are equivalent, but note that they differ when one distinguishes the particles of the system (for instance studying the motion of a tagged particle).

A useful feature of this alternative description is the following. A \textit{particle trajectory} of the interchange process is obtained by following the trajectory of a state $x \in \mathbb{Z}$ such that $\eta_0(x)=1$ (a particle) under
$\widehat{\mathbf{P}}$. We won't define this formally but roughly speaking, if $e$ and $e'$ are the two edges incident on $x$, one waits until the minimum of the first arrival times of these two processes (which is an exponential variable with parameter $\nu$) and jumps across the corresponding edge. Then one repeats this procedure. In particular, it immediately follows that
\begin{equation}\label{eq:SEPSRW}
\text{\parbox{12.5cm}{for each $x$ such that $\eta_0(x)=1$, the particle trajectory of $x$ under $\widehat{\mathbf{P}}$ follows the law of a continuous time simple random walk with jump rate $\nu$.}}
\end{equation}
Recalling the properties \eqref{pe:markov}--\eqref{pe:density} from~\S\ref{subsec:re}, we first record the following fact.

\begin{Lem}
\label{L:SEP-P}
With
\begin{equation}
\label{eq:SEP-mu-rho}
\mu_{\rho} = \big((1- \rho) \delta_0 + \rho \delta_1 \big)^{\otimes \mathbb{Z}}, \quad \rho \in J \stackrel{\textnormal{def.}}{=} (0,1),
\end{equation}
the measures $(\mathbf{P}^{\eta_0}: \eta_0 \in \{0,1\}^{\mathbb{Z}})$ with $\mathbf{P}^{\eta_0}=\mathbf{P}_{\textnormal{SEP}}^{\eta_0}$ and $(\mu_{\rho}: \rho \in J)$ satisfy all of  \eqref{pe:markov}--\eqref{pe:density}.
\end{Lem}
\begin{proof} Property \eqref{pe:markov} is classical, see \cite[Chap.~I, Thm. 3.9, p.27]{MR0776231} along with Example 3.1(d), p.21 of the same reference. So is \eqref{pe:stationary}, i.e.~the stationarity of the measure $\mu_{\rho}$ in \eqref{eq:SEP-mu-rho}, see \cite[Chap.~VIII, Thm.~1.12, p.369]{MR0776231}. The required coupling (for two given initial configurations $\eta_0'  \preccurlyeq \eta_0$) needed to verify the quenched monotonicity asserted in \eqref{pe:monotonicity},i) is simply obtained by realizing the process $ \eta= (\eta_t)_{t \geq 0}$ under the auxiliary measure $\widehat{\mathbf{P}}^{\eta_0'}$ and $\widehat{\mathbf{P}}^{\eta_0}$ using the same Poisson processes ${(\mathcal{P}_e)}$. In particular this measure yields a coupling over all possible initial distributions, including $\eta_0$ and $\eta_0'$, and this coupling is seen to preserve the partial order $\eta_0'  \preccurlyeq \eta_0$ for all $t>0$. The monotonicity in \eqref{pe:monotonicity},ii) is classical. Lastly,  upon observing that $\eta_0[0,\ell-1]$ is a binomial random variable with parameters $\ell$ and $\rho$ under $\mathbf{P}_{\rho}$, property \eqref{pe:density} is obtained by combining \eqref{eq:distriBinupper} and \eqref{eq:distriBinlower}, which are well-known large deviation estimates.
\end{proof}

In anticipation of \S\ref{subsec:C-SEP}, we now collect a simple lemma to bound the linear deviations of an SEP particle, which we will routinely use in our couplings below.

\begin{Lem}\label{Lem:driftsingleparticle}
Let $Z=(Z_t)_{t\geq 0}$ denote a simple random walk on $\mathbb{Z}$ with jump rate $\nu$, starting from $0$ at time $0$, with law denoted by $P$. For all $t>0$ and $k,a\in \N$:
\begin{multline}\label{eq:driftsingleparticle}
P\Big(\max_{0\leq s\leq t}\vert Z_s\vert \geq 2k\nu t+a\Big) \\
\leq P(Z\text{ makes more than $2k\nu t+a$ jumps during }[0,t])\leq e^{-(2k\nu t+a)/8}.
\end{multline}
\end{Lem}
\begin{proof}
The first inequality is immediate since $Z$ only performs nearest-neighbor jumps. As for the second one, remark that the number $N$ of jumps performed by $Z$ during $[0,t]$ is a Poisson random variable with parameter $\nu t$.  Hence by~\eqref{eq:distriPoissonupper} applied with $\lambda =\nu t$ and $x=(2k-1)\nu t +a\geq (2k\nu t+a)/2> 0$, we obtain that  
\begin{equation}\label{eq:SEPdrift1particle}
P(N\geq 2k\nu t+a)\leq \exp \left( -\frac{((2k-1)\nu t+a)^2}{2(2k\nu t+a)} \right)\leq \exp \left (-(2k\nu t+a)/8\right),
\end{equation}
and the conclusion follows. 
\end{proof}

\subsection{Conditions \ref{pe:densitychange}--\ref{pe:nacelle} for SEP} \label{subsec:C-SEP}

We now proceed to verify that the conditions introduced in \S\ref{subsec:C} all hold for the exclusion process introduced in \S\ref{subsec:SEP-basic}, as summarized in the next proposition. Its proof occupies the bulk of this section. These properties (above all, \ref{pe:densitychange} and \ref{pe:compatible}) are of independent interest.

\begin{Prop}
\label{P:SEP-C} 
For $(\mathbf{P}^{\eta_0}: \eta_0 \in \{0,1\}^{\mathbb{Z}_+})$ with $\mathbf{P}^{\eta_0} =\mathbf{P}^{\eta_0}_{\textnormal{SEP}}$, $ J = (0,1)$ (cf.~\eqref{eq:SEP-mu-rho}), and with $\nu$ as appearing in~\eqref{eq:gen-SEP}, all of \ref{pe:densitychange},~\ref{pe:compatible},~\ref{pe:drift},~\ref{pe:couplings} and~\ref{pe:nacelle} hold.
\end{Prop}

\begin{proof}
This follows directly by combining Lemmas \ref{Lem:SEPdensity},~\ref{Lem:SEPdriftdeviations},~\ref{Lem:doublecoupleSEP},~\ref{Lem:compatible} and~\ref{lem:nacelle} below.
\end{proof} 

We now proceed to investigate each of the relevant conditions separately. Throughout the remainder of this section, we work implicitly under the assumptions of Proposition~\ref{P:SEP-C}. In particular, in stating that some property $P$ `holds for SEP,' we mean precisely that $P$ is verified for the choice $(\mathbf{P}^{\eta_0}: \eta_0 \in \{0,1\}^{\mathbb{Z}_+})$ with $\mathbf{P}^{\eta_0} =\mathbf{P}^{\eta_0}_{\textnormal{SEP}}$ for all $\rho \in J = (0,1)$, and with $\nu$ the rate parameter underlying the construction of SEP.

\begin{Lem}\label{Lem:SEPdensity}
The condition~\ref{pe:densitychange} holds for \textnormal{SEP}, with no restriction on the choice of $\ell'\geq 1$. 
\end{Lem}

We give a brief overview of the proof. A key idea is to exploit the fact that SEP particles, although not independent, are in fact `more regularly' spread out than a bunch of independent random walks (starting from the same initial positions), which is due to the inherent negative association of the SEP. This fact is implicit in the bound \eqref{eq:densitySEPvsSRW} below, which is borrowed from \cite{HS15} (itself inspired from \cite{zbMATH06309113}), the proof of which uses in a crucial way an inequality due to Liggett, see \cite[Chap.~VIII, Prop.~1.7, p.~366]{MR0776231}, encapsulating this property. With this observation, it is enough to argue that after time $t \gg \ell^2$ (where $\ell$ is the precision mesh of the empirical density of the initial configuration, as in \ref{pe:densitychange}), random walks have diffused enough to forget their initial positions and average their density, which follows from classical heat kernel estimates.

\begin{proof}
We focus on proving the upper inequality (i.e.~when each $\eta_0(I)\leq (\rho+\varepsilon)\ell$), and comment where necessary on the minor adjustments needed to derive the other inequality in the course of the proof. 
Let $\rho,\varepsilon \in (0,1)$, $H,\ell,t\geq 1$ and $\eta_0$ be such that the conditions of~\ref{pe:densitychange} hold. We will in fact show \ref{pe:densitychange} for an arbitrary value of $\ell'\geq 1$ although the restriction to $\ell' \leq \ell$ is sufficient for later purposes (note that the statement is empty if $\ell'>2(H-2\nu t)$). Thus, let $\ell'\geq 1$ and $\mathcal{I}$ denote the set of intervals of length $\ell'$ included in $[-H+2\nu t, H-2\nu t]$, and define
\begin{equation}\label{eq:bar-eta}
\bar{\eta}_t(I)= \sum_{x\in I}\mathbf{P}^{\eta_0}(\eta_t(x)=1), \quad I  \in \mathcal{I},
\end{equation}
the average number of occupied sites of $I$ after time $t \geq 0$. Note that $\bar{\eta}_0(I)=\eta_0(I)$. 
By Lemma~2.3 of~\cite{HS15} (see also Lemma 5.4 of \cite{zbMATH06309113}) one knows that for all $ t \geq 0$ and suitable $\Cr{densitystableexpo}\in (0,\infty)$,
\begin{equation}\label{eq:densitySEPvsSRW}
\mathbf{P}^{\eta_0}\left(\eta_t(I)\geq \bar{\eta}_t(I)+\varepsilon\ell'\right)\leq \exp(-\Cr{densitystableexpo}\varepsilon^2\ell')
\end{equation}
(in fact the bound \eqref{eq:densitySEPvsSRW} holds for any initial configuration $\eta_0$). Thus, if 
\begin{equation}\label{eq:densitySEPofSRW}
 \max_{I\in \mathcal{I}}\bar{\eta}_t(I) \leq (\rho+2 \varepsilon)\ell',
\end{equation}
under our assumptions on $\eta_0$ and $t$,
then~\eqref{eq:densitySEPvsSRW}, a union bound over $\mathcal{I}$, and the upper bound on the maximum in \eqref{eq:densitySEPofSRW} yield that
\begin{equation}
\mathbf{P}^{\eta_0}\big(\max_{I\subseteq\mathcal{I}}\eta_t(I)>(\rho+3\varepsilon)\ell'\big)\leq 2H  \exp(-\Cr{densitystableexpo}\varepsilon^2\ell').
\end{equation}
A companion inequality to~\eqref{eq:densitySEPvsSRW} can be deduced in a similar way using the lower bound on the minimum in \eqref{eq:densitySEPofSRW} and exploiting symmetry, i.e.~rewriting $\{ \eta_t(I)< (\rho-3\varepsilon) \ell'\} = \{ \xi_t(I)> ((1-\rho)+3\varepsilon) \ell'\}$ where $\xi_t=1-\eta_t$, while observing that $(\xi_t)_{t \geq0}$ has law $\mathbf{P}^{\xi_0}$, cf.~\eqref{eq:gen-SEP}. The conclusion \ref{pe:densitychange} then follows. 

Therefore, we are left with showing~\eqref{eq:densitySEPofSRW}. In view of \eqref{eq:bar-eta}, it is enough to prove that for every $x\in [-H+2\nu t, H- 2\nu t]$, and under our assumptions on $\eta_0$ and $t$,
\begin{equation}\label{eq:densitySEPonesite}
\mathbf{P}^{\eta_0}(\eta_t(x)=1) \leq \rho +2\varepsilon. 
\end{equation}
Let $(Z_s)_{s\geq 0}$ denote a continuous-time symmetric simple random walk on $\mathbb{Z}$ with jump rate $\nu$, defined under an auxiliary probability $P$, and let 
$p_s(x,y)={P}(Z_s=y\,\vert Z_0=x )$, for $x,y\in \mathbb{Z}$ and $s\geq 0$ denote its transition probabilities, which are symmetric in $x$ and $y$. Writing $\{ \eta_t(x)=1\}$ as the disjoint (due to the exclusion constraint) union over $y \in \mathbb{Z}$ of the event that $\eta_0(y)=1$ and the particle starting at $y$ is located at $x$ at time $t$, it follows using \eqref{eq:SEPSRW} that for all $x\in \mathbb{Z}$,
\begin{equation}\label{eq:SEPheatkernel}
\mathbf{P}^{\eta_0}(\eta_t(x)=1)= \sum_{y\in \mathbb{Z}}p_t(y,x) \eta_0(y)= \sum_{y\in S}p_t(x,y),
\end{equation}
where $S:=\{y\in \mathbb{Z}: \eta_0(y)=1\}$ and we used reversibility in the last step. Using the rewrite \eqref{eq:SEPheatkernel}, we first show the upper bound in \eqref{eq:densitySEPonesite}. Let us define $c_t=\Cl{C:c-t}\sqrt{\nu t\log(\nu t)}$, where the constant $\Cr{C:c-t}>0$ will be chosen below. Recalling that $x\in  [-H+2\nu t, H-2\nu t]$, we cover the sites of $[x-c_t ,x+c_t] \subseteq [-H,H]$ using intervals $I_1, \ldots , I_q$ all contained in this interval and each containing $\ell$ sites, with $q=\lceil (2c_t +1)/\ell\rceil$. We may assume that all but the last interval $I_q$ are disjoint. For later reference, we note that $\vert I_r\cap S\vert \leq (\varrho +\varepsilon)\ell$ for all $1\leq r\leq q$ by assumption on $\eta_0$. By~\eqref{eq:SEPheatkernel}, we have that
\begin{equation}\label{eq:SEPdualityupperbound}
\mathbf{P}^{\eta_0}(\eta_t(x)=1)\leq \sum_{y\in\mathbb{Z}\setminus [x-c_t,x+c_t]   }p_t(x,y)+\sum_{r=1}^q\sum_{y\in S\cap I_r}p_t(x,y)
\end{equation}
We will look at the two terms on the right-hand side separately. Using \eqref{eq:distriPoissonupper}), one knows that with $N$ denoting the number of jumps of a continuous time random walk with jump rate $\nu$ until time $t>0$, one has $\mathbb{P}(A)\geq 1-Ce^{-\nu t/C}$ with $A \coloneqq \{ 2\nu t/3\leq N\leq 4\nu t/3\}$. Combining this with  the deviation estimate \eqref{eq:distriBinupper} yields that
\begin{multline}\label{violette-lechat0}
\sum_{y\in\mathbb{Z}\setminus [x-c_t,x+c_t]   }p_t(x,y)=\mathbb{P}(A^c)+2\sum_{k=\lceil 2\nu t/3\rceil }^{\lfloor4\nu t/3\rfloor} \mathbb{P}(N=k) \sum_{y'\ge c_t }\mathbb{P}\left(  \text{Bin}(k,\tfrac12)=\tfrac{y'+k}{2}  \right)\\
\le Ce^{-\nu t/C}+2\sum_{k=\lceil 2\nu t/3\rceil }^{\lfloor 4\nu t/3\rfloor} \mathbb{P}(N=k) \exp\left( -\frac{c_t^2}{ 16 \nu t}\right)\\
\le Ce^{-\nu t/C}+2 \exp\left( -\frac{\Cr{C:c-t}^2\log(\nu t)}{16}\right)
\le \frac{1}{\nu t}\le \frac{\varepsilon}{3},
\end{multline}
where we choose $\Cr{C:c-t}\ge \sqrt{11}$, use  that $\nu t\ge \Cr{densitystable}$, thus choosing  $\Cr{densitystable}$ large enough, and use the conditions from \ref{pe:densitychange} for the last inequality.

Next, we want to deal with the points in the interval $[x-c_t,x+c_t]$. Recall that  a continuous-time simple random walk with rate $\nu$ at time $t$ has the same law as a rate $1$ continuous-time simple random walk at time $\nu t$. Hence, by  \cite[Theorem 2.5.6]{zbMATH05707092},  for all $x,y,z$ such that $y,z\in [x-c_t,x+c_t]$ and $|z-y|\le \ell$,  using that $c_t<\nu t/2$ for $\Cr{densitystable}$ large enough, with $C$ changing from line to line,
\begin{multline}\label{violette-lechat}
\frac{p_t(x,y)}{p_t(x,z)}  \le \exp\left(        \frac{|x-z|^2-|x-y|^2}{2\nu t}+C\left(\frac{1}{\sqrt{\nu t}}+\frac{|x-y|^3+|x-z|^3}{(\nu t)^2} \right)        \right)  \\
\le\exp\bigg(      c  \frac{\ell \sqrt{\log(\nu t)}}{\sqrt{\nu t}}+C\frac{1+\log^{3/2}(\nu t)}{\sqrt{\nu t}}           \bigg) \le \exp\left(  C\frac{\varepsilon}{\sqrt{\Cr{densitystable}}}        \right)\le 1+\frac{\varepsilon}{3},
\end{multline}
where we used that $4\nu t>\Cr{densitystable} \ell^2\varepsilon^{-2} \log^3(\nu t)$ and $\ell\ge1$, from  \ref{pe:densitychange}, and chose $\Cr{densitystable}$ large enough. Using \cite[Theorem 2.5.6]{zbMATH05707092} again, we also have that, as soon as $\Cr{densitystable}$ is large enough, using the conditions from \ref{pe:densitychange},
\begin{equation} \label{violette-lechat2}
p_t(x,z)\le \frac{1}{\sqrt{\nu t}}\le \frac{\varepsilon}{3\ell}.
\end{equation}
Now, recalling that the intervals $I_r$ are disjoint for $1\le r\le q-1$ and that they all have cardinality $\ell$, using \eqref{violette-lechat} and \eqref{violette-lechat2}, we have that
\begin{multline}\label{violette-lechat3}
\sum_{r=1}^q\sum_{y\in S\cap I_r}p_t(x,y)
\le \sum_{r=1}^{q-1}\sum_{y\in S\cap I_r}p_t(x,y) + \sum_{y\in S\cap I_q}p_t(x,y)
\\ \le \left(1+\frac{\varepsilon}{3}\right)\sum_{r=1}^{q-1}\sum_{y\in S\cap I_r}\min_{z\in I_r}p_t(x,z) + \sum_{y\in S\cap I_q}\frac{\varepsilon}{3\ell}  \le \left(1+\frac{\varepsilon}{3}\right)(\rho+\varepsilon)\sum_{r=1}^{q-1}|I_r|\min_{z\in I_r}p_t(x,z) + \frac{\varepsilon |I_q|}{3\ell} \\
\le\left(1+\frac{\varepsilon}{3}\right)(\rho+\varepsilon)\sum_{r=1}^{q-1}\sum_{y\in S\cap I_r}p_t(x,y) + \frac{\varepsilon}{3}
\le \left(1+\frac{\varepsilon}{3}\right)(\rho+\varepsilon) + \frac{\varepsilon}{3}.
\end{multline}
Putting together \eqref{eq:SEPdualityupperbound}, \eqref{violette-lechat0} and \eqref{violette-lechat3}, and using that $\rho+\varepsilon\le 1$, we obtain
\begin{equation}
\mathbf{P}^{\eta_0}(\eta_t(x)=1)\leq \frac{\varepsilon}{3}+\left(1+\frac{\varepsilon}{3}\right)(\rho+\varepsilon) + \frac{\varepsilon}{3}\le \rho +2\varepsilon.
\end{equation}
This yields the upper bound in \eqref{eq:densitySEPonesite}, and the upper inequality in~\ref{pe:densitychange} follows.

For the lower bound, which requires to change~\eqref{eq:densitySEPonesite} to $\rho-2\varepsilon \leq  \mathbf{P}^{\eta_0}(\eta_t(x)=1)$, one uses the estimate $\mathbf{P}^{\eta_0}(\eta_t(x)=1)\geq \sum_{r=1}^{q-1}\sum_{y\in S\cap I_r}p_t(x,y)$ instead of~\eqref{eq:SEPdualityupperbound} and proceeds similarly as in \eqref{violette-lechat3}. \end{proof}

\begin{Lem}\label{Lem:SEPdriftdeviations}
The condition~\ref{pe:drift} holds for \textnormal{SEP}.
\end{Lem}

\begin{proof} 
Recall the construction in \eqref{eq:SEP-basic-coup} of the SEP using $\widehat{\mathbf{P}}$. Setting $\mathbb{Q}=\widehat{\mathbf{P}}$, this yields a natural coupling of $\eta$ and $\eta'$ with marginal laws $\mathbf{P}^{\eta_0}$ and $\mathbf{P}^{\eta'_0}$, respectively, for any choice of initial distribution $\eta_0$ and $\eta'_0$. In words, the coupling $\mathbb{Q}$ identifies $\eta$ and $\eta'$ as interchange processes and uses the same Poisson processes $(\mathcal{P}_e)$ on the edges of $\mathbb{Z}$. Now if $\eta_0 \vert_{[-H, H]}\succcurlyeq \eta_0' \vert_{[-H, H]}$, then under $\mathbb{Q}$, we claim that
\begin{equation}\label{eq:SEPdriftevent}
\big\{\exists x\in [-H+2k\nu t ,H-2k\nu t ], s\in [0,t]  , \, \eta_s(x)<\eta'_s(x) \big\}
\subseteq \bigcup_{y:\,\eta'_0(y)=1, \vert y\vert >H} {E}(y),
\end{equation}
where ${E}(y)$ is the event that the particle at $y$ enters the interval $[-H+2k\nu t ,H-2k\nu t ]$ before time $(s\leq)t$. 
By~\eqref{eq:driftsingleparticle} with $a=\vert y\vert -H$, we have
\begin{equation}\label{eq:SEPdrift1particlebis}
\mathbb{Q}({E}(y))\leq \exp \left (-(2k\nu t+\vert y\vert -H)/8\right).
\end{equation}
Applying a union bound to \eqref{eq:SEPdriftevent} and feeding \eqref{eq:SEPdrift1particlebis} thus yields that the probability of the complement of the event appearing on the left-hand side of \eqref{eq:SEPdriftdeviations} is bounded from above by
\begin{equation*}
2 \sum_{y=H}^{+\infty}\mathbb{Q}({E}(y)) \leq 2 \sum_{y'=0}^{+\infty}\exp (-y'/8-k\nu t/4) \leq 20\exp(-k\nu t/4).
\end{equation*}
\end{proof}

\begin{Rk}[Locality in~\ref{pe:drift}]
\label{R:C-2.1-local} For later reference, we record the following \textit{locality} property of the coupling $\mathbb Q$ constructed in the course of proving Lemma~\ref{Lem:SEPdriftdeviations}. Let $E_H \subset E$ denote the edges having both endpoints in $[-H,H]$. Then the above argument continues to work for \textit{any} specification of clock processes $\mathcal{P}_e$, $\mathcal{P}_e'$ for $e \notin E_H$ for $\eta$ and $\eta'$, respectively, so long as $(\mathcal{P}_e)_{e \in E}$ and $(\mathcal{P}_e')_{e \in E}$ end up having the correct law. This observation will be important when several couplings are `concatenated,' as in the proof of Lemma~\ref{Lem:compatible} below (see also Figure~\ref{f:Coupling2global_Exclusion}).
\end{Rk}

\begin{Lem}\label{Lem:doublecoupleSEP}
The condition~\ref{pe:couplings} holds for \textnormal{SEP}. 
\end{Lem}

We first give a brief overview of the argument. Lemma~\ref{Lem:doublecoupleSEP} corresponds to a quenched version of~\cite[Lemma 3.2]{BaldassoTeixeira18} by Baldasso and Teixeira, in which $\eta_0$ and $\eta'_0$ are sampled under $\mathbf{P}^{\rho+\varepsilon}$ and $\mathbf{P}^\rho$ respectively. We mostly follow their argument. Since fixing the initial environments induces some changes, and because we will use a variation of this proof to show~\ref{pe:compatible} for the SEP (Lemma~\ref{Lem:compatible}), we detail below the coupling of $\eta$ and $\eta'$, seen as interchange processes. In doing so we also clarify an essential aspect of this coupling; see in particular \eqref{eq:eta-coup}-\eqref{eq:eta'-coup} below.

 In a nutshell, the idea is to pair injectively each particle of $\eta'$ with one of $\eta$ at a relatively small distance (of order $t^{1/4}$), during a relatively long time (of order $t^{3/4}$), so that they perform independent random walks until they meet (which has a large probability to happen since $t^{3/4}\gg(t^{1/4})^2$), after which they coalesce, i.e.~follow the same evolution. To make such a matching possible, we ensure that with large probability, $\eta$ has more particles than $\eta'$ on each interval of length roughly $ t^{1/4}$. Within the present quenched framework, this property is now obtained by means of~\ref{pe:densitychange}, which has already been proved; see Lemma~\ref{Lem:SEPdensity}. Since the probability that at least one particle of $\eta'$ does not get paired, i.e.~does not meet its match, is relatively high (polynomially small in $t$), we repeat this coupling; see also Figure~\ref{f:Coupling22_Exclusion}.

\begin{figure}[]
  \center
\includegraphics[scale=0.8]{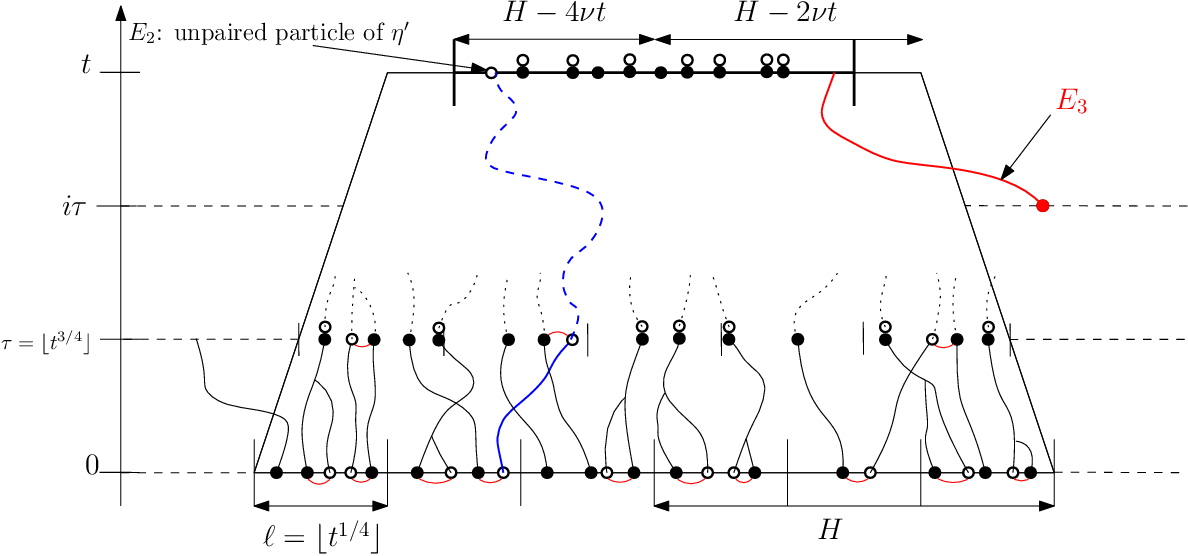}
  \caption{Coupling performed in Lemma~\ref{Lem:doublecoupleSEP}. The $\bullet$ are particles of $\eta$, the $\circ$ are particles of $\eta'$. The red links between pairs of particles represent matchings. On the (bad) event $B_2$, the particle with blue trajectory, in spite of having found a (nearby) match at all stages $i$, remains unpaired at time $t$.}
  \label{f:Coupling22_Exclusion}
\end{figure}

\begin{proof}[Proof of Lemma~\ref{Lem:doublecoupleSEP}] 
Let $\rho \in J$, $\varepsilon \in (0,1)$, and  $H,t\geq 1$ and $\eta_0,\eta'_0$ satisfy the conditions appearing in~\ref{pe:couplings}. Recall that $$\text{$\ell=\lfloor t^{1/4}\rfloor$ and let $\tau=\ell^3$}.$$ 
We abbreviate $I_s= [-H+2\nu s, H-2\nu s]$ for $s \geq 0$ in the sequel.

Let $\mathbb{Q}= \widehat{\mathbf{P}} \otimes \widehat{\mathbf{P}}'$, where $\widehat{\mathbf{P}}'$ is a copy of $\widehat{\mathbf{P}}$, with $ \widehat{\mathbf{P}}$, $ \widehat{\mathbf{P}}'$ governing the independent processes $(\mathcal{P}_e)$, $(\mathcal{P}_e')$, respectively, cf.~above \eqref{eq:SEP-basic-coup}. We will define a coupling of $(\eta,\eta')$ under $\mathbb{Q}$ (to be precise, a suitable extension of $\mathbb{Q}$ carrying additional independent randomness), inductively in $i$ over the time interval $(i\tau, (i+1)\tau]$, for $0 \leq i < \ell$. 
Suppose $(\eta_{[0, i\tau]}, \eta_{[0, i\tau]}' )$ have been declared under $\mathbb{Q}$ for some $0 \leq i < \ell$, with the correct marginal laws for $\eta_{[0, i\tau]}$ and $\eta_{[0, i\tau]}'$ (the case $i=0$ of this induction assumption holds trivially). We start by controlling the empirical densities of $\eta$ and $\eta'$ at time $i\tau$, which are already defined under $\mathbb{Q}$. The original proof of~\cite{BaldassoTeixeira18} uses the stationarity of $\mathbf{P}^\rho$ and $\mathbf{P}^{\rho+\varepsilon}$ and~\eqref{pe:density}, while we resort to~\ref{pe:densitychange} (cf.~\eqref{eq:doublecoupledensitychange} below). Let 
\begin{equation}\label{eq:G_1,i}
G_{1,i}\stackrel{\text{def.}}{=}\left\{
\begin{aligned}
&\text{for all $I$ with $I \subset I_{i\tau} $ of length $\lfloor \ell/2 \rfloor \leq |I| \leq \ell$,}
\\
&\text{one has that $\eta_{i\tau}(I)\geq (\rho+\varepsilon/2)|I| \geq\eta'_{i\tau}(I)$ } 
\end{aligned}
\right\}.
\end{equation}
Observe that $G_{1,0}$ is automatically satisfied by assumption on $\eta_0$ and $\eta_0'$.

We proceed to define $( \eta_{i \tau+t})_{0 < t \leq \tau}$ and $( \eta'_{i \tau+t})_{0 < t \leq \tau}$ under $\mathbb{Q}$.
If $G_{1,i}$ does not occur, we sample $(\eta'_{t+ i \tau})_{0 < t  \leq \tau}$ (and in fact for all $t>0$) in a manner as in \eqref{eq:SEP-basic-coup} using the processes $(\mathcal{P}_e')$ after time $\tau_i$, and similarly for $(\eta_{t+ i \tau})_{0 < t  \leq \tau}$ using $(\mathcal{P}_e)$ instead. Thus in this case $\eta'$ evolves independently from $\eta$ from time $\tau_i$ on.

If on the other hand $G_{1,i}$ occurs, we proceed as follows. 
Let $\text{Pai}_s \subset [-H,H]$ denote the set 
\begin{equation}\label{eq:paired-i}
\text{Pai}_s= \{ x \in I_{s}: \, \eta_{s}(x)= \eta_{s}'(x)=1  \},
\end{equation}
so that $\text{Pai}_{i \tau}$ is measurable relative to $(\eta_{i \tau}, \eta_{i \tau}')$. We refer to $\text{Pai}_s$ as the set of \textit{paired} particles (at time $s$). Let $\Pi_s = \{ x \in I_{s}: \eta_{s }(x)=1\}$ and $\Pi_s' = \{ x \in I_{s}: \eta_{s }'(x)=1\}$. Observe that $\text{Pai}_{s}  \subset \Pi_{s}'$. Our goal is to reduce the size of their difference as $s= i \tau$ for $i=1,2,\dots$ and eventually achieve equality when $i=\ell$.

To this effect, we first define a matching, i.e.~an injective map $\psi_i: \Pi_{i\tau}' \to \Pi_{i\tau} $, still measurable relative to $(\eta_{i \tau}, \eta_{i \tau}')$, as follows. The map $\psi_i$ acts as identity map on $\text{Pai}_{i \tau}$, a subset of both $\Pi_{i\tau}'$ and $\Pi_{i\tau}$, cf.~\eqref{eq:paired-i}. For each $x\in \Pi_{i\tau}' \setminus \text{Pai}_{i \tau} $, $\psi_i(x)$ is a point in $\Pi_{i\tau} \setminus \text{Pai}_{i \tau} $ at distance at most $\ell$ from $x$. As we now briefly explain, owing to the occurrence of  $G_{1,i}$, this can be achieved in such a way that $\psi_i$ is injective. To see this, first note that one can write $I_{i\tau}$ as disjoint union of intervals of length $|I|$ ranging  in $\lfloor \ell/2 \rfloor \leq |I| \leq \ell$, as follows. One covers $I_{i\tau}$ with contiguous intervals of length $\ell$ starting at one boundary, leaving a remaining interval $I_r$ at the other boundary of length less than $\ell$. If $I_r$ has length at least $\lfloor \ell/2 \rfloor$, one simply adds it, else unless $I_r$ is empty one cuts the penultimate interval into two halves of length at least $\lfloor \ell/2 \rfloor$ each and merges $I_r$ with the last of them. By construction any of the disjoint intervals $I$ thereby obtained has length $\lfloor \ell/2 \rfloor \leq |I| \leq \ell$ as required, and thus on the event $G_{1,i}$, see \eqref{eq:G_1,i}, one knows that $\eta_{i\tau}(I) \geq \eta'_{i\tau}(I)$. Since the $I$'s are disjoint and their union is $I_{i\tau}$, it follows that we can pair injectively each particle of $\Pi_{i\tau}' \setminus \text{Pai}_{i \tau} $ with a particle of $\Pi_{i\tau} \setminus \text{Pai}_{i \tau} $ within the same interval. We now fix any such matching $\psi_i$ and call any two particles $ (x,\psi_i(x)) \in \Pi_{i\tau}' \times \Pi_{i\tau}  $ \textit{matched}. We note that $|x-\psi_i(x)| \leq \ell$ by construction.
 
 The evolution for $( \eta_{i \tau+s})_{0 < s \leq \tau}$ and $( \eta'_{i \tau+s})_{0 < s \leq \tau}$ under $\mathbb{Q}$ (and on $G_{1,i}$) is now prescribed as follows. Both $( \eta_{i \tau+s})_{0 < s \leq \tau}$ and $( \eta'_{i \tau+s})_{0 < s \leq \tau}$ will be realized as interchange processes as in \eqref{eq:SEP-basic-coup}, thus it is sufficient to specify the relevant (Poisson) clock processes attached to each edge of $\mathbb{Z}$.
Let $E_H$ denote the set of edges of $\mathbb{Z}$ having both endpoints in $[-H,H]$. For $e \notin E_H$, $( \eta_{i \tau+s})_{0 < s \leq \tau}$ and $( \eta'_{i \tau+s})_{0 < s \leq \tau}$ simply use the clocks of $\mathcal{P}_e\circ \theta_{i\tau}$ and $\mathcal{P}_e'\circ \theta_{i\tau}$, respectively, where $\theta_s$ denotes the canonical time-shift of the process by $s$. It remains to specify the clock processes for $e \in E_H$.
Let $\widehat{\mathcal{P}}_e= \mathcal{P}_e+ \mathcal{P}_e' $. All clock processes attached to edges $e \in E_H$ will be defined via suitable thinning of $\widehat{\mathcal{P}}_e$.

First, one orders chronologically all arrivals for the processes $ \widehat{\mathcal{P}}_e\circ \theta_{i\tau} = ( \widehat{\mathcal{P}}_e(s+ i \tau))_{s \geq 0}$ as $e \in E_H$ varies  (there are countably many such times and they are a.s.~different so this is well-defined on a set of full measure). Let $\sigma_0=0$ and $\sigma_1,\sigma_2$ etc.~denote the chronologically ordered times thereby obtained. By suitable extension, $\mathbb{Q}$ is assumed to carry a family $\{ X_n: \, n \geq 0 \}$ of i.i.d.~Bernoulli variables with  $\Q(X_{n}=1)= 1- \Q(X_{n}=0)= \frac12$. We regard $X_n$ as the label attached to $\sigma_n$. Let $\mathcal{R}_e$ be the thinned process obtained from $\widehat{\mathcal{P}}_e$ by only retaining arrivals with label $X_n=1$. Then,
 \begin{equation}
 \label{eq:eta-coup}
 \text{$\eta_{\cdot + i \tau}$ uses the clock process $\mathcal{R}_e$ for each $e \in E_H$ (and $\mathcal{P}_e$ for each $e \notin E_H$).}
 \end{equation}
By elementary properties of Poisson processes and applying \eqref{eq:SEP-basic-coup}, it follows that $(\eta_{s+ i \tau})_{0< s \leq \tau}$ has the correct conditional law given $\eta_{i \tau}$; indeed by construction to each edge $e $ of $\Z$ one has associated independent Poisson processes having the correct intensity. 

The definition of $( \eta'_{i \tau+s})_{0 < s \leq \tau}$ is analogous to \eqref{eq:eta-coup}, and the clock process $\mathcal{R}_e'$ for $e \in E_H$ underlying the definition of $( \eta'_{i \tau+s})_{0 < s \leq \tau}$ is specified as follows. For a particle $x \in \Pi_{i\tau}'$ (resp.~$\Pi_{i\tau}$), let $\gamma_{i; \cdot}'(x)$ (resp.~$\gamma_{i;\cdot}(x)$) denote its evolution
under $ \eta'_{i \tau+\cdot}$ (resp.~$ \eta_{i \tau+\cdot}$). Proceeding chronologically starting at $n=1$, one chooses whether the clock $\sigma_n$ is retained or not according to the following rule. With $e=\{x,y\} \in E_H$ denoting the edge of $\sigma_n$,
\begin{equation}
 \label{eq:eta'-coup}
 \begin{split}
& \text{if for some $z \in \Pi_{i\tau}'$, $\gamma'_{i;\sigma_{n-1}}(z)= \gamma_{i;\sigma_{n-1}}(\psi_i(z))  \in \{x,y\}$,}\\
 &\text{then $\sigma_n$ is retained iff $X_n=1$, otherwise iff $X_n=0$;}
\end{split}
 \end{equation}
 here we think of right-continuous trajectories so $\gamma'_{i;\sigma_{n-1}}(z)$ is the position of the particle $z$ after the $(n-1)$-th jump; in fact one could replace each occurrence of $\sigma_{n-1}$ in \eqref{eq:eta'-coup} by an arbitrary time $s$ with $\sigma_{n-1} \leq  s< \sigma_n$, since there is no jump between those times. In words, at time $\sigma_{n-1}$, one inspects if at least one endpoint of $e$ contains (the evolution to time $s$ of) two matched particles, in which case the clock $\sigma_n$ is retained if it has label $1$ only. If no endpoint of $e$ contains matched  particles the clock is retained if it has label $0$. The process $( \eta'_{i \tau+s})_{0 < s \leq \tau}$ then simply uses the clocks $\mathcal{R}_e'$ on edges $e \in E_H$ that are retained according to \eqref{eq:eta'-coup} and such that $\sigma_n < \tau$. We will now argue that
\begin{equation}
\label{eq:eta'-law}
\text{given $\eta'_{i \tau}$, the process $( \eta'_{i \tau+s})_{0<s\leq \tau}$ has law $\mathbf{P}^{\eta'_{i \tau}}$ under $\mathbb{Q}$.}
\end{equation}
To see this, one simply notes using a straightforward induction argument that the conditional law of $\xi_n = 1\{\text{the $n$-th arrival in $(\widehat{\mathcal{P}})_{e \in E_H}$ is retained}\}$  given $\xi_1,\dots \xi_{n-1}$, $\sigma_1,\dots,\sigma_{n}$, $X_1,\dots, X_{n-1}$ and $( \eta_{i \tau+t},\eta'_{i \tau+s} )_{0 < s \leq \sigma_{n-1}}$ is that of a Bernoulli-$\frac12$ random variable. From this and the thinning property for Poisson processes it readily follows that $\mathcal{R}'$ has the right law.

Overall we have now defined a coupling of $(\eta,\eta')$ until time $\ell \tau \leq t$. In case $\ell \tau < t$ we simply use the same process $(\mathcal{P}_e)$ to define the evolution of both $\eta$ and $\eta'$ in the remaining time interval $(\ell\tau, t]$. The Markov property~\eqref{pe:markov},~\eqref{eq:eta-coup} and~\eqref{eq:eta'-law} ensure that $\eta,\eta'$ indeed have the desired marginals during $[0,t]$.
In view of \eqref{eq:paired-i}, this immediately yields that
\begin{equation}\label{eq:doublecoupleinclusion-first}
 \big\{ \eta_t \vert _{[-H+4{\nu} t, H-4{\nu} t]}\succcurlyeq \eta_t' \vert _{[-H+4{\nu} t, H- 4{\nu} t]}\big\}^c \subset \big\{ (\Pi_{t}'\setminus \text{Pai}_{t}) \cap [-H+4{\nu} t, H- 4{\nu} t] \neq \emptyset \big\}.
\end{equation}
The key of the above construction is that the latter event forces one of three possible unlikely scenarios. Namely, as we explain below, one has that
\begin{equation}\label{eq:doublecoupleinclusion}
\big\{ (\Pi_{t}'\setminus \text{Pai}_{t}) \cap [-H+4{\nu} t, H- 4{\nu} t] \neq \emptyset \big\} \subseteq  B_1 \cup B_2 \cup B_3, 
\end{equation}
where $B_1=\bigcup_{i=1}^{\ell-1} G_{1,i}^c$,
\begin{equation}\label{eq:B_2,3}
\begin{split}
B_2&=\bigcup_{s=0}^{ t-1 }\left\{
\begin{array}{c}\text{one particle of $\eta'_{s}(\mathbb{Z}\setminus I_s)$ ends up}\\\text{in $[-H+4 \nu t,  H-4\nu t]$ at time $t$ } \end{array}
\right\},\\
B_3&= B_1^c \cap  \left\{
\begin{array}{c}\text{$\exists x \in \Pi_0'$ s.t.~$\gamma_{i\tau}'(x) \in I_{i\tau}$ and }\\\text{$\inf_{ s \in [0, \tau]} Z_s^i(x)>0$ for all $0 \leq i < \ell$} \end{array}
\right\};
\end{split}
\end{equation} 
here $\gamma_{\cdot}'(x)$ refers to the evolution of particle $x$ under $\eta'$ and with $x_i= \gamma_{i\tau}'(x)$, one sets $Z_s^i(x)= |\gamma_{i;s}'(x_i)- \gamma_{i;s}(\psi_i(x_i))|$. In words $Z_s^i$ follows the evolution of the difference between $x_i$, which is not paired at time $i\tau$ since $Z_0^i(x) \neq 0$, and its match $\psi_i(x_i)$, which is well-defined on the event $B_1^c$. Thus $B_3$ refers to the event that some particle $ x \in \Pi_0'$ is found in $I_{i\tau}$ at time $i\tau$ for all $i$ and never meets its match during the time interval $(i\tau, (i+1)\tau]$.

We now explain \eqref{eq:doublecoupleinclusion}. To this effect we first observe that \eqref{eq:eta-coup} and \eqref{eq:eta'-coup} ensure that two matched particles at some stage $i$ follow the same evolution once they meet (and thus belong to $\text{Pai}_s$ for all later times $s$) as long as they stay in $I_s$. Therefore, on the event $B_1^c \cap B_2^c$, on which (due to occurrence of $B_2^c$) no unpaired $\eta'$-particle at time $t$ can arise by drifting in from the side, meaning that such a particle cannot be seen in $\eta'_{s}(\mathbb{Z}\setminus I_s)$ at any time $0 \leq s <t$,
the set $ (\Pi_{t}'\setminus \text{Pai}_{t}) \cap [-H+4{\nu} t, H- 4{\nu} t]$ being non-empty requires at least one particle from $\Pi_0'$ to never meet its match at any of the stages $ 1 \leq i \leq \ell$ (matching happens at all stages due to occurrence of $B_1^c$). That is,~$B_3$ occurs, and  \eqref{eq:doublecoupleinclusion} follows.

To finish the proof, we now bound the (bad) events appearing in \eqref{eq:doublecoupleinclusion} separately. In view of \eqref{eq:G_1,i}, we apply~\ref{pe:densitychange} (which now holds on account of Lemma~\ref{Lem:SEPdensity}) to $\eta'$ (resp. $\eta$) at time $ i\tau$, with $(\rho+\varepsilon/8, \varepsilon/8)$ (resp. $(\rho+7\varepsilon/8, \varepsilon/8)$) instead of $(\rho,\varepsilon)$,  with $(H,\ell, i\tau )$ instead of $(H,\ell,t)$ and with $\ell'$ ranging from $\lfloor \ell/2\rfloor$ to $\ell$, which fulfils the conditions of \ref{pe:densitychange} if $\Cr{SEPcoupling}$ is large enough so that $H> 4\nu t $ and
\[
\min_{1\leq i \leq \ell^4-1} 4\nu i\tau \geq 4\nu \tau >  \Cr{densitystable}\ell^2\varepsilon^{-2}(1+\vert \log^3(\nu t)\vert)\geq \max_{1\leq i \leq \ell^4-1} \Cr{densitystable}\ell^2\varepsilon^{-2}(1+\vert \log^3(\nu i\tau)\vert) .
\]
Recalling that $\ell=\lfloor t^{1/4}\rfloor$ and summing over the possible values of $\ell'$, this gives that
\begin{equation}\label{eq:doublecoupledensitychange}
\mathbb{Q}\big(B_1\big)\leq \sum_{i=1}^{\ell-1} \mathbb{Q}\left(G_{1,i}^c\right)\leq 4t^{1/2}H\exp\left( -\Cr{densitystableexpo}\varepsilon^{2} 2^{-7} (\ell-1)\right).
\end{equation}
We deal with $\mathbb{Q}(B_2)$ by applying~\ref{pe:drift}, which is in force on account of Lemma~\ref{Lem:SEPdriftdeviations}.
Noticing that the event indexed by $s$ entering the definition of $B_2$ in \eqref{eq:B_2,3} implies that at least one particle of $\eta'_{s}(\mathbb{Z}\setminus I_s)$ ends up in $[-H+2\nu s+2\nu t, H- 2\nu s-2\nu t ]$ before time $s+ t $, we get using \eqref{eq:SEPdriftdeviations} with $k=1$, $\eta_0=\eta_s \equiv 0$, and $I_s$ playing the role of $[-H,H]$ that
\begin{equation}\label{eq:doublecoupleE3}
\mathbb{Q}(B_2)
\leq  \sum_{0 \leq s < t}\mathbb{Q}\left(
\begin{array}{c}
\text{a particle of $\eta'_{s}(\mathbb{Z}\setminus I_s)$ ends up in}\\ 
\text{$[-H+2\nu(s+ t), H- 2\nu(s+t) ]$}\\
\text{before time $s +t$}
\end{array}
\right)
\leq 20t\exp(-\nu t/4).
\end{equation}
Finally, owing to our coupling in \eqref{eq:eta-coup}, \eqref{eq:eta'-coup}, conditionally on $(\eta_{u}, \eta'_{u})$, $u \leq i\tau$, the process $Z_s^i(x)$ is a one-dimensional  continuous-time random walk with rate $2\nu$ started at a point in $[0,\ell]$ (owing to the separation of $x_i$ and its match $\psi_i(x_i)$) with absorption at $0$. If $(Z_s)_{s\geq 0}$ is such a random walk and $\mathbb{P}_k$ denotes the probability for this walk starting at $k\in [0,\ell]$, we have by invariance by translation and the reflection principle:
\begin{equation}
\begin{split}
\max_{0\leq k\leq \ell}\mathbb{P}_k(\min_{0\leq s\leq \tau}Z_s>0)&= \mathbb{P}_\ell(\min_{0\leq s\leq \tau}Z_s>0)= \mathbb{P}_0(\min_{0\leq s\leq \tau}Z_s>-\ell)
\\
&\leq 2\mathbb{P}_0(\max_{0\leq s\leq \tau}\vert Z_s\vert<\ell)\leq 2\mathbb{P}(Z_\tau \in [-\ell,\ell]).
\end{split}
\end{equation}
Using again~\cite[Theorem 2.5.6]{zbMATH05707092} and taking  $\Cr{SEPcoupling}$ large enough (so that in particular $2\nu\tau=t \geq 2\vert x \vert$ for any $x\in [-\ell,\ell]$),  we have thus for some universal constant $C>0$ (changing from one expression to the next): 
\begin{equation}
\max_{0\leq k\leq \ell}\mathbb{P}_k(\min_{0\leq s\leq \tau}Z_s>0)\leq \frac{2\ell +1}{\sqrt{4\pi \nu\tau}}\exp(C(\nu^{-1}\tau^{-1}+\ell^3\nu^{-2}\tau^{-2}))\leq \frac{C\ell}{\sqrt{\nu\tau}}\leq C\nu^{-1/2}t^{-1/8}. 
\end{equation}
Recalling that $t>\nu^{-8}$ and taking a union bound over $x \in \Pi_0'$, the previous estimate applied with the Markov property for $(\eta,\eta')$ yields that, as long as $\Cr{SEPcoupling}$ is large enough:
\begin{equation}\label{eq:doublecoupleE2}
\mathbb{Q}(B_3)\leq 2H(Ct^{-1/16})^{\ell}.
\end{equation}
Putting together~\eqref{eq:doublecoupleinclusion-first},~\eqref{eq:doublecoupleinclusion},~\eqref{eq:doublecoupledensitychange},~\eqref{eq:doublecoupleE3} and~\eqref{eq:doublecoupleE2}, we obtain that 
\begin{multline*}
\mathbb{Q}( \eta'_t \vert _{[-H+4{\nu} t, H-4{\nu} t]}\preccurlyeq \eta_t \vert _{[-H+4{\nu} t, H- 4{\nu} t]})
\\
\geq 1 -8t^{1/2}H\exp\left( -\Cr{densitystableexpo}\varepsilon^{2}2^{-7}(\ell-1) \right)  -20t\exp(- \nu t/4)-2H(Ct^{-1/16})^{\ell},
\end{multline*}
which is larger than $1- \Cr{SEPcoupling2} tH\exp(-\Cr{SEPcoupling2}^{-1}\frac{\nu}{\nu +1}\varepsilon^2t^{1/4} )$ as required by \eqref{eq:doublecoupleSEP}
provided $\Cr{SEPcoupling}$ and $\Cr{SEPcoupling2}$ are chosen large enough. 
\end{proof}
\begin{Rk}[Locality in \ref{pe:couplings}]
\label{R:C-2.2-local} Similarly as in Remark~\ref{R:C-2.1-local}, which exhibits an analogous property for the coupling inherent to \ref{pe:drift}, the coupling $\mathbb{Q}$ yielding property~\ref{pe:couplings} constructed in the proof of Lemma~\ref{Lem:doublecoupleSEP} can be performed for \textit{any} specification of clock processes $(\mathcal{P}_e, \mathcal{P}_e')_{e \notin E_H}$ used to define $\eta,\eta'$, so long as the marginal laws of $(\mathcal{P}_e)_{e \in E}$ and $(\mathcal{P}_e')_{e \in E}$, are that of independent Poisson processes of intensity $\nu/2$. This can be seen by inspection of the proof: the only `non-trivial' joint distribution concerns $(\mathcal{P}_e, \mathcal{P}_e')_{e \in E_H}$, which are obtained by suitable thinning from $(\widehat{\mathcal{P}}_e)_{e\in E_H}$, see in particular the discussion around \eqref{eq:eta-coup} and \eqref{eq:eta'-coup}. 
\end{Rk}

Combining the couplings supplied by Lemmas~\ref{Lem:SEPdriftdeviations} and~\ref{Lem:doublecoupleSEP} multiple times, which will be permitted owing to Remarks~\ref{R:C-2.1-local} and~\ref{R:C-2.2-local}, yields the following result.

\begin{Lem}\label{Lem:compatible}
The condition~\ref{pe:compatible} holds for \textnormal{SEP}. 
\end{Lem}

\begin{figure}[]
  \center
\includegraphics[scale=0.8]{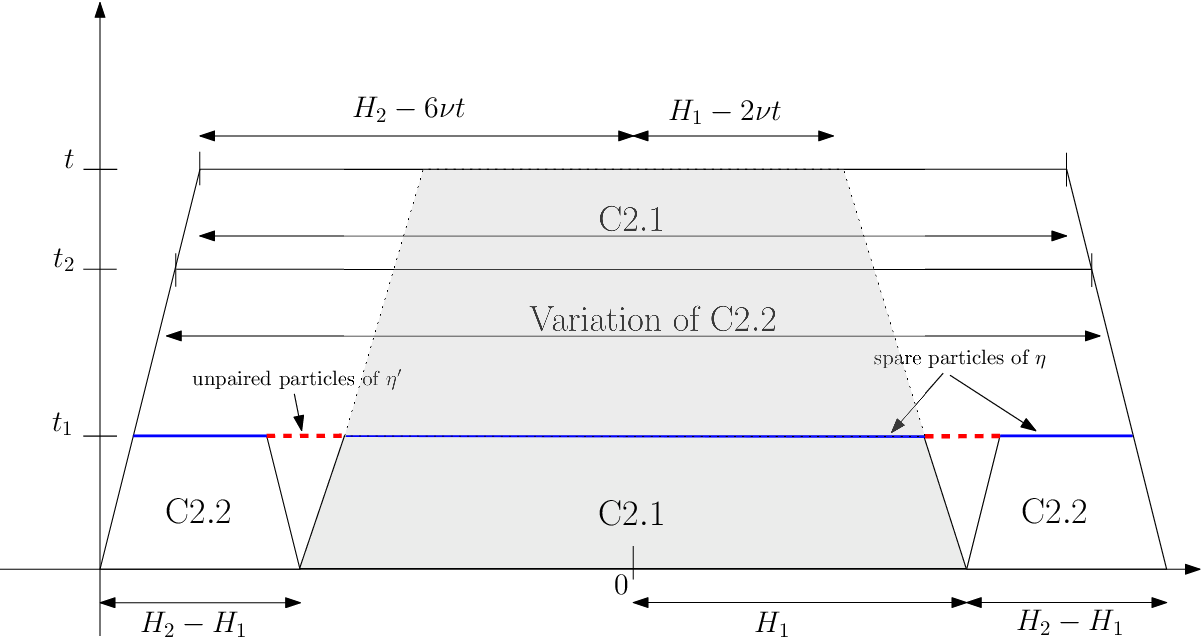}
  \caption{Couplings used in the proof of Lemma~\ref{Lem:compatible}. If the couplings of Step 1 (operating with time horizon $t_1$) are successful, then the only unpaired particles of $\eta'$ at time $t_1$ are in the two red dashed segments, and there are enough spare particles of $\eta$ in the blue segments to cover these particles of $\eta'$ before time $t_2$ (Step 2, using coupling `Variation of~\ref{pe:couplings}').
}
  \label{f:Coupling2global_Exclusion}
\end{figure}

\begin{proof}
Let $\rho,\varepsilon\in (0,1)$, $H_1,H_2,t,\ell\geq 1$ and $\eta_0,\eta'_0 \in \{0,1\}^{\mathbb{Z}}$ be such that the assumptions of~\ref{pe:compatible} hold. In particular, note that these entail that $H_2 > H_1$. Define 
\begin{equation}\label{eq:compatibledeft1}
t_1= \ell^{4}, \quad t_2= t_1+\ell^{19}. 
\end{equation}
We proceed in two steps and refer to Figure~\ref{f:Coupling2global_Exclusion} for visual aid. In the first step, we apply simultaneously the couplings of Lemma~\ref{Lem:SEPdriftdeviations} on $[-H_1,H_1]$ and of Lemma~\ref{Lem:doublecoupleSEP} on $[-H_2,H_2]\setminus [-H_1,H_1]$, in the time-interval $[0,t_1]$ (in doing so we shall explain how this preserves the marginals of $\eta$ and $\eta'$). As a result, we get that $\eta_{t_1}(x)\geq \eta'_{t_1}(x)$ for all $x\in [-H_2+2 \nu t_1 , H_2-2 \nu t_1]$, except possibly around two intervals around $-H_1$ and $H_1$, of width $O( \nu t_1)$.

In the second step, we couple the particles of $\eta'$ on these intervals with "additional" particles of $\eta_{t_1}\setminus \eta'_{t_1}$ on $[-H_2,H_2]$ (using that the empirical density of $\eta'$ is slightly larger than that of $\eta$ on $[-H_2,H_2]$ by~\ref{pe:densitychange}), in a manner similar to that used in the proof of Lemma~\ref{Lem:doublecoupleSEP}. This ensures that with large enough probability, all these particles of $\eta'$ get covered by particles of $\eta$ within time $t_2-t_1$, without affecting the coupling of the previous step on account of the Markov property. Finally, during the time interval $[t_2,t]$ we use again the natural coupling of Lemma~\ref{Lem:SEPdriftdeviations}, and conclude by showing~\eqref{eq:compatible1} and~\eqref{eq:compatible2}. We now proceed to make this precise.
\\
\\
\textbf{Step 1:} we construct a coupling $\mathbb{Q}_1$ of the evolutions of $\eta$ and $\eta'$ during the time-interval $[0,t_1]$ such that if we define the (good) events
\begin{equation*}
\begin{split}
&G_1=\{\forall x\in [-H_1+2  \nu t_1, H_1-2  \nu t_1],\,\forall s\in [0,t_1]:  \eta_s(x)\geq \eta'_s(x) \}, \\
&G_2=\{\forall x \in [-H_2+2 \nu t_1, -H_1-2 \nu t_1-1]\cup [H_1+2 \nu t_1+1,  H_2-2 \nu t_1] :  \eta_{t_1}(x)\geq \eta'_{t_1}(x) \},
\end{split}
\end{equation*}
 then 
\begin{equation}\label{eq:compatiblecoupling1}
\mathbb{Q}_1(G_1\cap G_2)\geq 1 - 3\Cr{SEPcoupling2} t_1H_2\exp\left(- \Cr{SEPcoupling2}^{-1}\textstyle\frac{\nu}{\nu+1}\varepsilon^2\ell\right). 
\end{equation}

\noindent
The coupling $\mathbb{Q}_1$ is defined as follows. Let $(\mathcal{P}_e)_{e \in E}$ be a family of i.i.d.~Poisson processes on $\mathbb{R}_+$ of intensity $\nu/2$. These processes are used to describe the exchange times for $\eta$ (seen as an interchange process as in \eqref{eq:SEP-basic-coup}), during the time-interval $[0,t_1]$. We now define the exchange times  $(\mathcal{P}_e')_{e \in E}$ to be used for $\eta'$ during $[0,t_1]$ as follows. For an edge $e$ having at least one endpoint outside $[-H_2,H_2]$ or at least one endpoint inside $[-H_1, H_1]$, set $\mathcal{P}_e'= \mathcal{P}_e$. It remains to specify $\mathcal{P}_e'$ for $e$ with both endpoints in $[-H_2,H_2]$ but outside $[-H_1,H_1]$. This set splits into two disjoint intervals $E_{\pm}$, which are both dealt with separately and in exactly the same manner. Thus restricting our attention to $E_+$, one couples $(\mathcal{P}_e')_{e \in E_+}$ and $(\mathcal{P}_e)_{e \in E_+}$ in exactly the same manner as in the proof of Lemma~\ref{Lem:doublecoupleSEP}, with the interval $I_+$, defined as the set of all endpoints of edges in $E_+$, playing the role of $[-H,H]$. The fact that Lemma~\ref{Lem:doublecoupleSEP} applies even though the processes  $\mathcal{P}_e$ and $\mathcal{P}_e'$ have been specified for certain edges $e \notin E_+$ is owed to Remark~\ref{R:C-2.2-local}. Define $I_-$ similarly and perform the same coupling of $(\mathcal{P}_e')_{e \in E_-}$ and $(\mathcal{P}_e)_{e \in E_-}$.

Since the sets of edges $E_-$, $E_+$ and $E \setminus (E_- \cup E_+)$  are disjoint, it readily follows that $(\mathcal{P}_e)_{e \in E}$ is an i.i.d.~family of Poisson processes on $\mathbb{R}_+$ of intensity $\nu/2$. This ensures that under $\mathbb{Q}_1$, $\eta\sim \mathbf{P}^{\eta_0}$ and  $\eta'\sim \mathbf{P}^{\eta'_0}$. 

Now, by our assumptions on $\rho, \varepsilon, H,\ell,t_1,\eta_0$ and $\eta'_0$, the above construction of $\mathbb{Q}_1$ together with Remark~\ref{R:C-2.1-local} ensure that Lemma~\ref{Lem:SEPdriftdeviations} applies on the interval $[-H_1, H_1]$ with $t=t_1$ and $k=1$ (recall to this effect that the relevant coupling for which \ref{pe:drift} is shown to hold is simply $\mathbb{Q}=\widehat{\mathbf{P}}$, see above \eqref{eq:SEPdriftevent}, and that this coupling is also local in the sense of Remark~\ref{R:C-2.1-local}). This yields that 
\begin{equation}\label{eq:compatiblecoupling1.1}
\mathbb{Q}_1(G_1) \geq 1 -20 \exp (-  {\nu}t_1/4).
\end{equation}
Second, our assumptions (taking $\Cr{compatible}>\Cr{SEPcoupling}$) and the above construction of $\mathbb{Q}$ also allow us to apply Lemma~\ref{Lem:doublecoupleSEP} on $E_-$ and on $E_+$ instead of $[-H,H]$, with the same values of $\rho, \varepsilon$ and $\ell$, and with $t=t_1$ (in particular, $\nu^8t_1=(\nu^2\ell)^4>1$). We obtain that
\begin{equation}\label{eq:compatiblecoupling1.2}
\mathbb{Q}_1(G_2) \geq 1 - 2\Cr{SEPcoupling2} t_1H_2\exp\left(- \Cr{SEPcoupling2}^{-1}\textstyle\frac{\nu}{\nu+1}\varepsilon^2\ell\right).
\end{equation}
Combining~\eqref{eq:compatiblecoupling1.1} and~\eqref{eq:compatiblecoupling1.2} yields~\eqref{eq:compatiblecoupling1}, since $20\exp (-  {\nu}t_1/4)\leq \Cr{SEPcoupling2} t_1H_2\exp(- \Cr{SEPcoupling2}^{-1}\frac{\nu}{\nu+1}\varepsilon^2\ell) $ if $\Cr{compatible}$ is chosen large enough. This concludes Step 1.
\\

\noindent
\textbf{Step 2:}  We now extend $\mathbb{Q}_1$ to a coupling $\mathbb{Q}_2$ up to time $t$, using a slight variation of the coupling in the proof of Lemma~\ref{Lem:doublecoupleSEP} in the time-interval $[t_1, t_2]$. 
For definiteness of $\mathbb{Q}_2$, on the complement of  $G_1\cap G_2$, use $(\mathcal{P}_e)_{e \in E}$ for the exchange times of both $\eta$ and $\eta'$ during the time $[t_1,t]$. Focusing now on the case where $G_1 \cap G_2$ occurs, the aim of the step is to show that we can have
\begin{equation}\label{eq:compatiblecoupling2}
\mathbb{Q}_2(\eta'_{t_2}\vert_{[-H+4\nu t_2, H-4\nu t_2]}\preccurlyeq \eta_{t_2}\vert_{[-H+4\nu t_2, H-4\nu t_2]}\,\vert\, G_1\cap G_2)\geq 1 -  H_2\exp\left( - \textstyle \frac{\nu}{\nu+1}\varepsilon^2\ell^4\right). 
\end{equation}
Fix any realization of $\eta_{t_1}, \eta'_{t_1}$ such that $G_1\cap G_2$ holds.
Define $\text{Pai}_0$ (cf.~\eqref{eq:paired-i}) as the set of vertices in $I_0 \stackrel{\text{def.}}{=}[-H_2+2 \nu t_1, H_2-2 \nu t_1]$ containing both a particle of $\eta_{t_1}$ and of $\eta'_{t_1}$, and think of these particles as being paired. Let $U_0$ (resp.~$U'_0$) be the set of unpaired particles of $\eta_{t_1}$ (resp.~$\eta'_{t_1}$) in $I_0$. By definition of $G_1$ and $G_2$, the particles of $U'_0$ must be in $[-H_1-2 \nu t_1, -H_1+2 \nu t_1]\cup [H_1-2 \nu t_1, H_1+2 \nu t_1]$ when the event $G_1\cap G_2$ occurs; cf.~Figure~\ref{f:Coupling2global_Exclusion}. Hence, on $G_1\cap G_2$ we have that 
\begin{equation}\label{eq:bound-U-0'}
\vert U'_0\vert \leq 8 \nu t_1.
\end{equation}
Denote $B_{1,0}^c$ the event that on every interval of length between $\frac{\ell^5}2$ and $\ell^5$ included in $I_0$, $\eta_{t_1}$ has at least $\varepsilon\ell^5/10\geq 8 \nu t_1$ unmatched particles (recall that $\ell>80\nu/\varepsilon$ by assumption). On $B_{1,0}$, use $(\mathcal{P}_e)_{e \in E}$ as exchange time process to define both $\eta$ and $\eta'$ during $[t_1,t]$ (cf.~the discussion leading to \eqref{eq:SEP-basic-coup}). 

Henceforth, assume that $B_{1,0}^c \cap G_1 \cap G_2$ occurs. Following the line of argument in the paragraph after \eqref{eq:paired-i}, one matches injectively each particle of $U_0'$ with a particle of $U_0$ at distance at most $\ell^5$. In the present context this is possible owing to \eqref{eq:bound-U-0'} and occurrence of $B_{1,0}^c$. Then one couples the evolutions of $\eta$ and $\eta'$, first during $[t_1, t_1+\ell^{15}]$ as done in the proof of Lemma~\ref{Lem:doublecoupleSEP} during the time interval $[0, \tau]$. Then iteratively at times $t_1+i\ell^{15}$ for $1\leq i \leq \ell^4-1$, one performs on $ I_i= [-H_2+2 \nu(t_1+i\ell^{15}), H_2-2 \nu(t_1+i\ell^{15})]$ the same coupling at times $i \lfloor t^{3/4}\rfloor$ on the event $B_{1,i}^c$ that for any interval $I\subseteq I_i$ of length ranging between $\frac{\ell^5}2$ and $\ell^5$, one has $\eta_{t_1+i\ell^{15}}(I)\geq \eta'_{t_1+i\ell^{15}}(I)$.  On $B_{1,i}$, use $(\mathcal{P}_e)_{e \in E}$ for the exchange times of both $\eta$ and $\eta'$ during the time interval $[t_1+i\ell^{15},t]$.

Overall, this yields a coupling of $(\eta_{\cdot \wedge t_2}, \eta_{\cdot \wedge t_2}')$ with the correct marginal law (as in the proof of Lemma~\ref{Lem:doublecoupleSEP}). Finally
one extends this coupling during $[t_2, t]$ on the event $G_1 \cap G_2 \cap \bigcap_{i=0}^{\ell^4-1}B_{1,i}^c$ by using the same exchange times for $\eta$ and $\eta'$. It follows that $\eta\sim \mathbf{P}^{\eta_0}$ and $\eta'\sim \mathbf{P}^{\eta'_0}$ under $\mathbb{Q}_2$.

In much the same way as in \eqref{eq:doublecoupleinclusion-first}, it follows from the above construction that the complement of the event on the left-hand side of \eqref{eq:compatiblecoupling2} implies  that at least one particle of $\eta'$ in the interval $[-H_2+2 \nu t_2,  H_2-2 \nu t_2]$ is unpaired at time $ t_2$, which in turn (cf.~\eqref{eq:doublecoupleinclusion}-\eqref{eq:B_2,3}) implies the occurrence of 
$$
 \left( \bigcup_{i=0}^{\ell^4-1}B_{1,i}\right) \cup B_2 \cup B_3
$$
where
\begin{align*}
&B_2= \bigcup_{s=0}^{\ell^{19}-1}\left\{
\begin{array}{c}
\text{a particle of $\eta'_{t_1+ s}(\mathbb{Z}\setminus [-H_2+2 \nu(t_1+s),  H_2-2 \nu(t_1+s)])$} 
\\
 \text{ ends up in $\eta'_{t_2}([-H_2+4 \nu t_2,  H_2-4 \nu t_2])$}
\end{array} 
\right\}\\[0.3em]
&B_3=\left\{ 
\begin{aligned}
&\text{at time $ t_2$, one particle from $\eta'_{t_1}([-H_2+2 \nu t_1,  H_2-2 \nu t_1])$}
\\
&\text{ has been in $I_i$ for all $0\leq i < \ell^4$ and remains unpaired}
\end{aligned}
\right\}.
\end{align*}
We now mimic~\eqref{eq:doublecoupledensitychange},~\eqref{eq:doublecoupleE3} and~\eqref{eq:doublecoupleE2} to handle $\sum_{i=0}^{\ell^4-1}\mathbb{Q}_2(B_{1,i})$, $\mathbb{Q}_2(B_2)$ and $\mathbb{Q}_2(B_3) $  respectively. In detail, for the first term, we apply~\ref{pe:densitychange} with $(H,t)=(H_2,i\ell^{15})$ for $1\leq i \leq \ell^4-1$, and $\ell'$ ranging from $\frac{\ell^5}{2}$ to $\ell^5$, noting that 
\begin{equation}
H_2>4\nu i\ell^{15}\geq 4\nu\ell^{15}\geq \Cr{densitystable}\ell^{19/2}\varepsilon^{-2}(1+\vert \log^3(\nu\ell^{19})\vert)\geq \Cr{densitystable}(i\ell^{15})^{1/2}\varepsilon^{-2}(1+\vert \log^3(\nu i\ell^{15})\vert).
\end{equation}
For the third inequality, remark that $\nu \ell> \Cr{densitystable}\varepsilon^{-2}(1+\vert\log^3(\nu \ell^4) \vert)$ (taking $\Cr{compatible}>\Cr{densitystable}$), hence it is enough to show that $\ell^{9/2}(1+\vert \log^3(\nu \ell^4)\vert)\geq 1+\vert \log^3(\nu \ell^{19})\vert$. But $1+\vert \log^3(\nu \ell^{19})\vert \leq 1+4\vert \log^3(\nu \ell^4)\vert+60\log \ell$ (since $(a+b)^3\leq 4a^3+4b^3$ for $a,b\geq 0$). Taking $\Cr{compatible}$ large enough so that $\nu\ell^{100}$ and thus $\ell$ is large enough (recall that $\ell>80\nu/\varepsilon>\nu$), we have $ 1+4\vert \log^3(\nu \ell^4)\vert+60\log \ell\leq 1+4\vert \log^3(\nu \ell^4)\vert+ \ell\leq \ell^{9/2}(1+\vert \log^3(\nu \ell^4)\vert)$ as desired. 
\\
For the second term $\mathbb{Q}(B_2)$, we use~\ref{pe:drift}, and for the third term $\mathbb{Q}(B_3)$, we note that a continuous-time random walk with rate $2\nu$ started in $[0,\ell^5]$ will hit $0$ before time $\ell^{15}$ with probability at least $1-C\ell^5/(\nu\ell^{15})^{1/2}\geq 1-C\ell^{-2}$ (note that $\nu \ell\geq \Cr{compatible}>1$ if we take  $\Cr{compatible}>1$). 
We obtain that
\begin{align*}
&\mathbb{Q}_2(\forall x\in [-H_2+4 \nu t_2,  H_2-4 \nu t_2],\, \eta_{ t_2}(x) \geq\eta'_{ t_2}(x)  )\\
&\qquad\qquad\ge  1-\sum_{i=0}^{\ell^4-1}\mathbb{Q}_2(B_{1,i}) - \mathbb{Q}_2(B_2)-\mathbb{Q}_2(B_3)    \\
&\qquad\qquad\geq 1- 4\ell^9 H_2 \exp\left(- \Cr{densitystableexpo}\varepsilon^{2}2^{-7}(\ell^5-1) \right) - 20 \ell^{19}\exp (- \nu \ell^{19} /4) - 2H_2 (C\ell^{-2} )^{\ell^4} \\
&\qquad\qquad\geq 1 - H_2\exp\left( \textstyle-\frac{\nu}{\nu+1}\varepsilon^2\ell^4\right),
\end{align*}
if $\Cr{compatible}$ in \ref{pe:compatible} is chosen large enough. This yields~\eqref{eq:compatiblecoupling2}.
\\

Let $\mathbb{Q} \stackrel{\text{def.}}{=}\mathbb{Q}_2$. It remains to establish~\eqref{eq:compatible1} and~\eqref{eq:compatible2}.
Owing to the way the coupling $\mathbb{Q}$ is defined during the intervals $[0,t_1],[t_1,t_2]$ and $[t_2,t]$, $\mathbb{Q}$ has the following property: in $[-H_1,H_1]$, any particle of $\eta'$ that is covered at some time $s<t$ by a particle of $\eta$ will be covered by this particle until time $t$, or until it leaves $[-H_1,H_1]$. Moreover, by assumption every particle of $\eta'_0([-H_1,H_1])$ is covered by a particle of $\eta_0$.  Therefore,
\begin{multline}\label{eq:Q-2-main}
\{ \forall s\in [0,t], \, \eta_s\vert_{[-H_1+4{\nu}t, H_1-4{\nu}t]}\succcurlyeq \eta'_s\vert_{[-H_1+4{\nu}t, H_1-4{\nu}t]} \}^{ c}
\\
\subseteq \bigcup_{s=1}^t \{\text{a particle of $\eta'_s(\Z\setminus [-H_1,H_1])$ 
enters $[-H_1+4{\nu}t,H_1-4{\nu}t ]$ before or at time $t$} \}. 
\end{multline}
By Lemma~\ref{Lem:driftsingleparticle} applied with $(t,k,a)=(t-s,2,4\nu t+x)$ for every $s\leq t$ and $x\geq 0$ if $\eta'_s(\pm (H_1+x))=1$, a union bound over all particles appearing in the event on the right-hand side of \eqref{eq:Q-2-main} leads to the bound on the $\Q$-probability of the left-hand side by $ 2t\sum_{x\geq 0}\exp(-(4\nu t+x) /8)\leq 20t\exp (-\nu t/4),$ and \eqref{eq:compatible1} follows.

As for~\eqref{eq:compatible2}, first note that by~\eqref{eq:compatiblecoupling1} and~\eqref{eq:compatiblecoupling2}, and for $\Cr{compatible}$ large enough, we have that 
\begin{equation}
\mathbb{Q}(\eta'_{ t_2}\vert_{[-H_2+4\nu t_2, H_2-4\nu t_2]}\preccurlyeq \eta_{ t_2}\vert_{[-H_2+4\nu t_2, H_2-4\nu t_2]})\geq 1 - 4\Cr{SEPcoupling2} t_1 H_2\exp\left(-\textstyle \Cr{SEPcoupling2}^{-1}\frac{\nu}{\nu+1}\varepsilon^2\ell\right). 
\end{equation}
Since we use, during $[t_2,t]$, the same natural coupling as in the proof of Lemma~\ref{Lem:SEPdriftdeviations}, letting $B_4$ denote the event that a particle of $\eta'_{t_2}(\mathbb{Z}\setminus [-H_2+4\nu t_2, H_2-4\nu t_2])$ ends up in $[-H_2+6\nu t, H_2-6\nu t]\subseteq [ -H_2+4\nu t_2+2\nu t, H_2-4\nu t_2-2\nu t]$ in time at most $t$, we obtain, provided $\Cr{compatible}$ is large enough and abbreviating $\xi= \Cr{SEPcoupling2}^{-1}\frac{\nu}{\nu+1}\varepsilon^2\ell$, that {\color{black}
\begin{multline*}
\mathbb{Q}(\eta'_{ t}\vert_{[-H_2+6\nu t, H_2-6\nu t]}\preccurlyeq \eta_{ t}\vert_{[-H_2+6\nu t, H_2-6\nu t]})\\
\geq 1 - 4\Cr{SEPcoupling2} t_1 H_2e^{-\xi}-\mathbb{Q}(B_4)
\geq  1 - 4\Cr{SEPcoupling2} t_1 H_2e^{-\xi}-20e^{-\nu t/4}
\geq  1 - 5\Cr{SEPcoupling2} t_1 H_2e^{-\xi}.
\end{multline*}
This shows~\eqref{eq:compatible2} (recalling that $t_1=\ell^4$ by~\eqref{eq:compatibledeft1}), and concludes the proof.
}
\end{proof}

The final result which feeds into the proof of Proposition~\ref{P:SEP-C} is the following.

\begin{Lem}\label{lem:nacelle}
The condition~\ref{pe:nacelle} holds for \textnormal{SEP}. 
\end{Lem}

\begin{proof}
Let $\ell\geq 1$, $\rho\in (0,1)$ and $\eta_0, \eta'_0$ be such that the conditions of~\ref{pe:nacelle} hold. In particular, these imply the existence of $x_0\in [0, \ell]$ such that $\eta_0(x_0)=1$ and $\eta'_0(x_0)=0$. Let $\mathbb{Q} = \widehat{\mathbf{P}}$ as above \eqref{eq:SEP-basic-coup}, by which $(\eta,\eta')$ are coupled as interchange processes using the same exchange times $(\mathcal{P}_e)_{e\in E}$.  As observed in \eqref{eq:SEPSRW}, under $\mathbb{Q}$ the particle of $\eta_0$ starting at $x_0$ moves like a simple random walk $Z=(Z_t)_{t\geq 0}$ with jump rate $\nu$, and we have $\eta'_t(Z_t)=0$ for all $t\geq 0$ by construction of $\widehat{\mathbf{P}}$.  Hence, $\mathbb{Q}(\eta_{\ell}( x )>0  , \,    \eta'_{\ell}( x )=0) \geq \mathbb{Q}(Z_\ell= x) $ and therefore, in order to deduce \eqref{eq:penacelleNEW} it is enough to argue that
\begin{equation}\label{eq:nacelleE1}
\mathbb{Q}(Z_\ell= x)\geq 
 \big(\tfrac{\nu}{2e^\nu}  \big)^{6(\rho+1)\ell}, \quad x=0,1.
\end{equation}
Indeed, let $Z^L_\ell$ (resp. $Z^R_\ell$) denote the number of jumps of $Z$ to the left (resp.~right) during the time interval $[0,\ell]$. These two variables are independent and distributed as $\text{Poi}(\nu \ell/2)$, which entails that, for $x=0,1$, and $x_0 \geq x$,
\begin{multline*}
\mathbb{Q}(Z_\ell= x)\geq \mathbb{Q}(Z^L_\ell=x_0-x)\mathbb{Q}(Z^R_\ell=0)\geq \left(\frac{\ell \nu}{2}  \right)^{x_0-x}  \frac{e^{-\ell \nu}}{(x_0 -x)!} 
\geq \left(\frac{ \nu }{2} \right)^{x_0-x} e^{-\ell \nu}
\\
\geq \left(\frac{ \nu }{2e^\nu} \right)^{x_0-x}e^{-(\ell-(x_0-x))\nu},
\end{multline*}
using that $(x_0-x)! \leq x_0^{x_0-x}\leq \ell^{x_0-x}$  in the third step.
Since $0\leq x_0-x\leq \ell$ and $\nu\leq 2e^\nu$,~\eqref{eq:nacelleE1} follows for all $x_0 \geq x$ (and $x=0,1$), and it is easy to see that the bound remains true in the remaining case, i.e.~when $x_0=0 = 1-x$ (now forcing $Z^L_\ell=0$ and $Z^R_\ell=1$ instead, which by symmetry of $Z^L_\ell$ and $Z^R_\ell$ yields the same bound as when $x_0=x+1$). Overall this yields~\eqref{eq:penacelleNEW}. 
Finally, since the marginal of the coupling $\mathbb{Q}$ for $(\eta,\eta')$ is that of Lemma~\ref{Lem:SEPdriftdeviations}, we get the first inequality of~\eqref{eq:penacelleNEW2} from~\ref{pe:drift} with $t=\ell$. The second inequality follows from our condition on $k$ (using that $\rho \leq 1$).
\end{proof}

\appendix

\section{Concentration estimates}\label{sec:appendix2}

We collect here a few classical facts on concentration of Poisson and binomial distributions, that are repeatedly use to control the probability that the environment could have an abnormal empirical density. The PCRW case corresponds to Poisson distributions, and the SSEP to binomial distributions).
\begin{Lem}\label{Lem:PoissonBinomDomination}
Let $\lambda,x >0$, and $X\sim \text{Poisson}(\lambda)$. Then 
\begin{equation}\label{eq:distriPoissonupper}
\mathbb{P}(X\geq \lambda+x)\leq \exp \left(-\frac{x^2}{2(\lambda+x)}\right).
\end{equation}
If $x\in [0,\lambda]$, then 
\begin{equation}\label{eq:distriPoissonlower}
\mathbb{P}(X\leq \lambda-x)\leq \exp \left(-\frac{x^2}{2(\lambda+x)}\right).
\end{equation}
Let $m\in\mathbb{N}, p\in(0,1)$ and $X\sim \text{Bin}(m,p)$. Then for all $q \in (0,1-p)$, we have
\begin{equation}\label{eq:distriBinupper}
\mathbb{P}(X\geq (p+q)m)\leq \exp (-mq^2/3)
\end{equation}
and for all $q\in (0,p)$:
\begin{equation}\label{eq:distriBinlower}
\mathbb{P}(X\leq (p-q)m)\leq \exp (-mq^2/2).
\end{equation}
\end{Lem}

\begin{proof}
We get (\ref{eq:distriBinupper}) (resp. (\ref{eq:distriBinlower})) from Theorem 4.4 (resp. Theorem~4.5) of~\cite{Mitzen} with $\mu=mp$ and $\delta=q/p\geq q$ in both cases.  We now turn to (\ref{eq:distriPoissonupper}).
By a classical application (and optimization) of Chernoff's bound (see Theorem 5.4 of~\cite{Mitzen}), one shows that 
\[
\mathbb{P}(X\geq \lambda+x)\leq \frac{e^{-\lambda}(e\lambda)^{\lambda+x}}{(\lambda+x)^{\lambda+x}}=\exp\left(-\lambda - (\lambda+x)(\log(\lambda +x)-\log\lambda\,-1)\right).
\]
Hence for (\ref{eq:distriPoissonupper}), it remains to show that 
\[
\lambda + (\lambda+x)(\log(\lambda +x)-\log\lambda\,-1)\geq \frac{x^2}{2(\lambda+x)}.
\]
Setting $u=1+x/\lambda$ and multiplying both sides by $u/\lambda$ this amounts to show that the map $f:[1,\infty) \to \mathbb{R}$ defined by 
\[
f(u)= u+u^2(\log u\, -1) - \frac{(u-1)^2}{2}
\]
remains non-negative. Indeed, $f(1)=0$ and $f'(u)=1+2u(\log u\, -1) +u-(u-1)=2(1+u(\log u\, -1))$. One checks that $\log u\, -1\geq -1/u$ for all $u\geq 1$ (with equality iff $u=1$) so that $f'$ is nonnegative on $[1,\infty)$, and this concludes the proof of (\ref{eq:distriPoissonupper}). One proves (\ref{eq:distriPoissonlower}) via the same method.
\end{proof}

\section{The PCRW environment} \label{sec:PCRW}
In this Appendix, we consider another environment, the Poisson Cloud of Random Walks (PCRW) previously considered in \cite{HHSST15, HKT20} among others; we refer to the introduction for a more complete list of references. For this environment, the parameter $\rho \in (0,\infty)$ governs the intensity of walks entering the picture. We show below, see Lemma~\ref{L:PCRWP} and Proposition~\ref{P:PCRW-C}, that the PCRW also fits the setup of \S\ref{subsec:re} and  
 satisfies~\ref{pe:densitychange}-\ref{pe:nacelle}. Hence, this environment yields another example to which the conclusions of our main result, Theorem~\ref{T:generic}, apply. The proof is virtually the same as that of Theorem~\ref{T:main1} given in \S\ref{subsec:main}, upon noting that the relevant law of large numbers \eqref{eq:LLNagain} in this context was also shown in \cite{HKT20}, yielding the existence of the speed $v(\rho)$ at all but at most two values $\rho=\rho_{\pm}$ (as for SEP). 
The PCRW is defined below using random walks evolving in discrete time, following the practice of \cite{HKT20} and other previous works, but the following results could easily be adapted to random walks in continuous time (with exponential holding times of mean one).

\subsection{Definition of the PCRW}\label{subsec:PCRWdef}

The PCRW is a stochastic process $\eta= (\eta_t(x);\, x \in \mathbb{Z}, t \in \mathbb{R_+})$ with state space $\Sigma=\mathbb{Z}^{\mathbb{Z}_+}$ defined as follows: for any given initial configuration $\eta_0\in \Sigma$ and every $x\in \mathbb{Z}$, place $\eta_0(x)$ particles at $x$. Then, let all the particles follow independent discrete-time lazy simple random walks, i.e.~at each integer time, any given particle stays put with probability $1/2$, or jumps to its left or right neighbour with the same probability $1/4$. 
For $t\geq 0$ and $x\in \mathbb{Z}$, let $\eta_t(x)$ be the number of particles located at $x$. We denote by $\bP^{\eta_0}_{\text{PCRW}}$ the canonical law of this environment with initial state $\eta_0$, and frequently abbreviate $\bP^{\eta_0}_{}=\bP^{\eta_0}_{\text{PCRW}}$ below.

\subsection{Properties of the PCRW}
We proceed to show that the PCRW environment has the desired features, i.e.~that the properties \eqref{pe:markov}-\eqref{pe:density} listed in \S\ref{subsec:re} as well as the conditions~\ref{pe:densitychange}-\ref{pe:compatible} appearing in \S\ref{subsec:C} all hold. The parameter $\rho$ indexing the stationary measures will naturally vary in $(0,\infty)$. 
We will in practice always consider a bounded open interval $J= {\color{black}(K^{-1}},K)$ for arbitrary $K > 1$ below. The constants $\Cr{densitydev}, \ldots, \Cr{SEPcoupling}$ appearing as part of the conditions we aim to verify will henceforth be allowed to tacitly depend on $K$. Note that this is inconsequential for the purposes of deriving monotonicity of $v(\cdot)$ on $(0,\infty)$ (cf.~\eqref{eq:monotonic}) since this is a local property: to check monotonicity at $\rho$ one simply picks $K$ large enough such that $\rho \in {\color{black}(K^{-1}},K) =J$.

For the remainder of this appendix, let $K>1$ be arbitrary and $J\stackrel{\text{def.}}{=} (K^{-1},K)$. We start by verifying properties~\eqref{pe:markov}-\eqref{pe:density}. 
\begin{Lem}
\label{L:PCRWP}
With
\begin{equation}
\label{eq:PCRW-mu-rho}
\mu_{\rho} = \emph{Poi}(\rho)^{\otimes \mathbb{Z}}, \quad \rho \in (0,+\infty),
\end{equation}
the measures $(\mathbf{P}^{\eta_0}: \eta_0 \in \{0,1\}^{\mathbb{Z}})$ with $\mathbf{P}^{\eta_0}=\mathbf{P}_{\textnormal{PCRW}}^{\eta_0}$ and $(\mu_{\rho}: \rho \in J)$ satisfy all of  \eqref{pe:markov}-\eqref{pe:density}.
\end{Lem}

\begin{proof}
Fix $\rho >0$. 
Property~\eqref{pe:markov} is classical, and follows readily from the time-homogeneity, translation invariance and axial symmetry of the lazy simple random walk. 
Property~\eqref{pe:stationary} is also standard: if $\eta_0\sim\mu_{\rho}$, by suitably thinning the Poisson process one can realize $\eta_0$ by decomposing $\eta_0(x)=l(x)+c(x)+r(x)$ for all $x\in \mathbb{Z}$, where $l(x)$ (resp.~$c(x)$, $r(x)$) is the number of particles starting from $x$ that make their first move to the left (resp.~stay put, and make their first move to the right), with $l(x),r(x)\sim \text{Poi}(\rho/4)$, $c(x)\sim \text{Poi}(\rho/2)$ and the family of variables $(l(x),c(x),r(x))_{x\in \mathbb{Z}}$ is independent. From this one infers that $\eta_1\sim\mu_{\rho}$, as $\eta_1(x)=l(x+1)+c(x)+r(x-1) \sim \text{Poi}(\rho)$ for all $x\in \mathbb{Z}$, and the variables $\eta_1(x)$ are independent as $x$ varies. 

As for property~\eqref{pe:monotonicity}, it holds with the following natural coupling (which straightforwardly yields the correct marginal laws for $\eta$ and $\eta'$): if $\eta'_0(x)\leq \eta_0(x)$ for all $x\in\mathbb{Z}$, then one matches injectively each particle of $\eta'_0$ to a particle of $\eta_0$ located at the same position. The coupling imposes that matched particles follow the same trajectory, and  the remaining particles of $\eta_0$ (if any) follow independent lazy simple random walks, independently of the matched particles. 
Finally, Property~\eqref{pe:density} is a consequence of the fact that under $\mathbf{P}^\rho$, for every finite interval $I$ and every time $t \geq0$, $\eta_t(I)\sim \text{Poi}(\rho\vert I\vert)$, and combining with the tail estimates~\eqref{eq:distriPoissonupper}-\eqref{eq:distriPoissonlower}.
\end{proof}

We now establish the conditions~\ref{pe:densitychange}-\ref{pe:nacelle}.
\begin{Prop}
\label{P:PCRW-C} 
For $(\mathbf{P}^{\eta_0}: \eta_0 \in \{0,1\}^{\mathbb{Z}_+})$ with $\mathbf{P}^{\eta_0} =\mathbf{P}^{\eta_0}_{\textnormal{PCRW}}$, $\rho \in J$ and with $\nu=1$, all of~\ref{pe:densitychange},~\ref{pe:compatible},~\ref{pe:drift},~\ref{pe:couplings} and~\ref{pe:nacelle} hold.
\end{Prop}

The proof of Proposition~\ref{P:PCRW-C} is given in \S\ref{sec:PCRW-C} below. We start with a coupling result (which for instance readily implies~\ref{pe:densitychange} as shall be seen), similar in spirit to Lemma~B.3 of~\cite{HHSST15}, stating that the evolution of the PCRW with a sufficiently regular deterministic initial condition can be approximated by a product of independent Poisson  variables. This is of independent interest (and lurks in the background of various more elaborate coupling constructions employed in \S\ref{sec:PCRW-C}).

\begin{Prop}\label{Lem:PoissonapproxPCRW}
There exist positive and finite constants $\Cl{c:diffusive}$, $\Cl[c]{EGrestesmallerthanepsilon}$ and $\Cl[c]{coupling-1}$ such that the following holds. Let  $\rho\in (0,K)$, $\varepsilon\in (0,(K-\rho)  \wedge  \rho \wedge 1 )$ and $H,\ell,t\in \mathbb{N}$ be such that $\Cr{c:diffusive}\ell^2< t < H/2$ and $(\rho +\varepsilon) {(t^{-1}{\log t})^{1/2}} \ell< \Cr{EGrestesmallerthanepsilon}\varepsilon$. There exists a coupling $\Q$ of $(\eta^{\rho-\varepsilon} , \eta_t, \eta^{\rho+\varepsilon}) $ with $\eta^{\rho\pm\varepsilon}\sim \mu_{\rho \pm\varepsilon}$ (and $\eta_t$ sampled under $\mathbf{P}^{\eta_0}$) such that, if  $\eta_0\in \Sigma$ is such that for any interval $I\subseteq [0,H]$ with $\vert I\vert=\ell$, 
\begin{equation}\label{eq:cond-dense}
\textstyle
(\rho - \frac{\varepsilon}{2}) \ell \leq
\eta_0(I) \quad (\text{resp.~} \eta_0(I) \leq (\rho + \frac{\varepsilon}{2}) \ell),
\end{equation}
then with $G=\{\eta^{\rho-\varepsilon}\vert_{[t, H-t]} \preccurlyeq  \eta_{t}\vert_{[t, H-t]}  \}$ (resp.~$G=\{ \eta_{t}\vert_{[t, H-t]} \preccurlyeq  \eta^{\rho+\varepsilon}\vert_{[t, H-t]}\}$),
 \begin{equation}
 \label{eq:coupling-1}
 \begin{split}
 &\Q(G) \geq 1 - H \exp\big(-\Cr{coupling-1} (\rho+\varepsilon)^{-1} \varepsilon^2\sqrt{t}\big).
 \end{split}
 \end{equation}
  Moreover, the coupling $\mathbb{Q}$ is local in the sense that $(\eta^{\rho-\varepsilon} , \eta_t, \eta^{\rho+\varepsilon}) \vert_{[t,H-t]}$ depends on the initial condition $\eta_0$ through $\eta_0(x)$, $x \in [0,H]$, alone. 
\end{Prop}

We now prepare the ground for the proof of Proposition~\ref{Lem:PoissonapproxPCRW}. Let us abbreviate $I_t=[t,H-t]$. We use the framework of soft local times from Appendix~A of~\cite{HHSST15} (the latter following Section~4 of~\cite{PopovTeixeira}), which we extend to fit our needs. We define a coupling $\mathbb{Q}$ as follows.
Let $\Lambda$ (defined under $\mathbb{Q}$) be a Poisson point process on $\Z\otimes \mathbb{R}_+$ with intensity $1\otimes \lambda$ where $1$ stands for the counting measure and $\lambda$ is the Lebesgue measure on $\mathbb{R}$. For each $z\in I_t$, set 
\begin{equation}\label{eq:PCRW-coup101}
\eta^{\rho\pm\varepsilon}\stackrel{\text{def.}}{=}\Lambda(\{z\}\times (0,\rho\pm\varepsilon]).
\end{equation}
 For $z\in \Z\setminus I_t$, let independently $\eta^{\rho\pm\varepsilon}(z)\sim \text{Poi}({\rho\pm\varepsilon}) $. This indeed yields the correct marginal distributions $\mu_{\rho\pm\varepsilon}$ in view of \eqref{eq:PCRW-mu-rho}.

As for $\eta_t$, given any initial configuration $\eta_0 \in \Sigma$ let $(x_i)_{i\geq 1}$ denote an arbitrary ordering of the positions of the (finitely many) particles of $\eta_0([0,H])$ (counted with multiplicity, hence the sequence $(x_i)_{i\geq 1}$ is not necessarily injective). Define for $i\geq 1$ and $z\in [-t,H+t]$ $g_i(z):=q_t(x_i,z)$ with $(q_n)_{n \in \mathbb{N}}$ denoting the discrete-time heat kernel for the lazy simple random walk. Let $\xi_1:=\sup\{t\geq 0 :  \bigcup_{z\in \Z}\Lambda(\{z\}\times (0,tg_1(z)])=0\}$ and for $i\geq 2$, define recursively 
\begin{equation}\label{eq:softlocaltimesxii}
\xi_i\stackrel{\text{def.}}{=}\sup\big\{t\geq 0: \textstyle \bigcup_{z\in \Z}\Lambda\big(\{z\}\times (0,\xi_1g_1(z)+\ldots + \xi_{i-1}g_{i-1}(z)+tg_i(z)]\big)=i-1\big\}, 
\end{equation}
see Figure~5 of~\cite{HHSST15}. 
Note that since each $g_i$ has a finite support (included in $[y_i-t,y_i+t]$), the $\xi_i$'s are well-defined. In fact by Propositions~A.1-A.2 of~\cite{HHSST15}, the variables are i.i.d.~Exp(1). For all $z \in [-t,H+t]$, we define the soft local time
\begin{equation}\label{eq:couple-pf1}
G_{\eta_0} (z)\stackrel{\text{def.}}{=} \sum_{i\geq 1} \xi_i q_t(x_i,z), 
\end{equation}
with $\xi_i$ as in \eqref{eq:softlocaltimesxii}. With these definitions, it follows, denoting 
\begin{equation}\label{eq:PCRW-coup102}
h(z)\stackrel{\text{def.}}{=}\Lambda(\{z\}\times (0,G_{\eta_0}(z))),
\end{equation}
 that the family $(h(z))_{z\in \Z}$ is distributed as $\tilde{\eta}_t$, defined as the restriction of $\eta_t$ under $\mathbf{P}^{\eta_0}$ restricted to the particles of $\eta_{0}([0,H])$. To see this, note that $\mathbb{Q}$ is such that, if $u$ is the $\Z$-coordinate of the particle of $\Lambda$ seen when determining $\xi_i$ at~\eqref{eq:softlocaltimesxii}, then the particle $x_i$ of $\eta_0$ moves to $u$ by time $t$. Remark indeed that by~\eqref{eq:softlocaltimesxii} and the spatial Markov property for Poisson point processes, the choice of $u$ is proportional to $g_i(\cdot)=q_t(x_i,\cdot)$ and independent of what happened in the first $i-1$ steps. In particular, $(h(z))_{z\in I_t}$ is distributed as $\eta_t\vert_{I_t}$, since all the particles of $\eta_t\vert_{I_t}$ perform at most one step per unit of time, and must have been in $[0,H]$ at time $0$. Finally, independently of all this, let all particles of $\eta_0(\Z\setminus [0,H])$ follow independent lazy random walks. 
 
 Overall $\mathbb{Q}$ indeed defines a coupling of 
$(\eta^{\rho-\varepsilon} , \eta_t, \eta^{\rho+\varepsilon}) $ with the required marginal law, and the desired locality (see below \eqref{eq:coupling-1}) follows immediately from the previous construction. Moreover, \eqref{eq:PCRW-coup101} and \eqref{eq:PCRW-coup102} imply that under $\mathbb{Q}$, for all $t \in \mathbb{N}$  and $\eta_0\in \Sigma$, 
\begin{equation}\label{eq:softlocaltimescoupling}
\begin{split}
&\left\{ \eta^{\rho-\varepsilon}\vert_{I_t} \preccurlyeq  \eta_{t}\vert_{I_t}\right\} \subseteq\{\forall z\in I_t:\, \rho-\varepsilon \leq G_{ \eta_0 }(z) \},\\[0.3em]
&\left\{  \eta_{t}\vert_{I_t} \preccurlyeq  \eta^{\rho+\varepsilon}\vert_{I_t}\right\} \subseteq\{\forall z\in I_t:\, G_{ \eta_0 }(z)\leq \rho+\varepsilon \}.
\end{split}
\end{equation}
It is now clear from \eqref{eq:softlocaltimescoupling} that the desired high-probability domination in \eqref{eq:coupling-1} hinges on a suitable control of the soft local time $G_{\eta_0}$ defined by \eqref{eq:couple-pf1}. To this effect we first isolate the following first moment estimate.

\begin{Lem}
\label{L:SLT-exp}
Under the assumptions of Proposition~\ref{Lem:PoissonapproxPCRW}, there exists $\Cl{EGheatkernel}$ such that for all $z \in I_t$, 
\begin{equation}
\label{eq:G-exp}
\textstyle
  (\rho - \frac{\varepsilon}{2}) \Big(1- \Cr{EGheatkernel}{\color{black} {\ell} \sqrt{\frac{\log t}{t}}}\Big) \leq\mathbb E^{\Q}[G_{\eta_0} (z)],  \quad \left(\text{resp.~}  \mathbb E^{\Q}[G_{\eta_0} (z)] \leq   (\rho + \frac{\varepsilon}{2}) \Big(1+ \Cr{EGheatkernel} {\color{black} {\ell} \sqrt{\frac{\log t}{t}}}\Big)\right).
\end{equation}
\end{Lem}

\begin{proof}
Let $c_t=\sqrt{t \log t}$ and fix $z\in [-H+t,H-t]$.
Cover $[z-\lfloor c_t\rfloor,z+\lfloor c_t\rfloor]$ by a family $(I_i)_{1\leq i\leq \lceil (2\lfloor c_t\rfloor+1)/\ell\rceil}$ of intervals of length $\ell$, all disjoint except possibly $I_1$ and $I_2$. We focus on the upper bound on $\mathbb E^{\Q}[G_{\eta_0} (z)]$ in \eqref{eq:G-exp} (the lower bound is derived in a similar fashion). 
By assumption in \eqref{eq:cond-dense}, we have that $\eta_0(I_i) \leq (\rho + \frac{\varepsilon}{2}) \ell$ for all $i$, hence, recalling that $\xi_i$ in \eqref{eq:softlocaltimesxii} has unit mean, 
\begin{multline}
\label{eq:G-exp-2}
 E^{\Q}[G_{\eta_0} (z)] \stackrel{\eqref{eq:couple-pf1}}{=} \sum_{i\geq 1} q_t(x_i,z) \leq  \sum_{i=1}^{\lceil (2\lfloor c_t\rfloor+1)/\ell\rceil} (\rho + \frac{\varepsilon}{2}) |I_i| \max_{x \in I_i} q_t(z,x) +2\sum_{x\ge z+c_t-1} q_t(z,x).
\end{multline}
We start by dealing with the last term above. By Azuma's inequality, we have that
\begin{equation}\label{malodo}
\sum_{x\ge z+c_t-1} q_t(z,x)\le \exp\left( -\frac{(c_t-1)^2}{2t} \right)\le \exp\bigg( -\frac{\log t}{2}+\sqrt{\frac{\log t}{t}} \bigg)\le \frac{C(\rho+\tfrac{\varepsilon}{2})}{\sqrt{t}},
\end{equation}
for some  constant $C>0$, depending on $K$. 
Let us now handle the first term in the right-hand side of \eqref{eq:G-exp-2}. We start by noting that
\begin{multline} \label{malalagorge}
\sum_{i=1}^{\lceil (2\lfloor c_t\rfloor+1)/\ell\rceil} (\rho + \tfrac{\varepsilon}{2}) |I_i| \max_{x \in I_i} q_t(z,x) \\\le (\rho + \tfrac{\varepsilon}{2}) \sum_{i=1}^{\lceil (2\lfloor c_t\rfloor+1)/\ell\rceil} \sum_{y \in I_i} q_t(z,y) +(\rho + \tfrac{\varepsilon}{2})\sum_{i=1}^{\lceil (2\lfloor c_t\rfloor+1)/\ell\rceil} \sum_{y \in I_i} \max_{x \in I_i} (q_t(z,x)-q_t(z,y))
\end{multline}
By \cite[Proposition 2.4.4]{zbMATH05707092}, we have that
\begin{equation}\label{eq:heatkernelPCRW2}
\max_{y\in  \Z}q_t(z,y) \leq {C}t^{-1/2}.
\end{equation}
Recalling that at most $I_1$ and $I_2$ may overlap, the above implies that
\begin{equation}\label{eq:heatkernelPCRWdiagonal}
\sum_{i=1}^{\lceil (2\lfloor c_t\rfloor+1)/\ell\rceil} \sum_{y \in I_i} q_t(z,y) \leq 1+\sum_{y\in I_1}q_t(z,y)\leq 1+C\frac{\ell}{\sqrt{t}}.
\end{equation}
It remains to deal with the last term in \eqref{malalagorge}. By standard heat kernel estimates (see for instance \cite[Proposition 2.5.3]{zbMATH05707092} and  \cite[Corollary 2.5.4]{zbMATH05707092}) and a computation similar to \eqref{violette-lechat}, combined with a large deviation estimate on the number $N_t$ of non-zero steps performed by the lazy random walk up to time $t$ (using Azuma's inequality for instance),  for all $x,y\in [z-c_t,z+c_t]$  with $|x-y|\le \ell$ and  first assuming that both $|x-z|$ and $|y-z|$ are even, leaving $C$ be a universal constant changing from line to line, we have that
\begin{equation}\label{eq:heatkernelPCRW}
\begin{split}
q_t(z,x) &\leq \sum_{n=\lfloor t/4-c_t\rfloor}^{\lceil t/4+c_t \rceil}\mathbb{P}(N_t=2n) q_{2n}(z,y)\times \exp\left(  C\frac{\ell c_t}{t} \right) + \mathbb{P}(|N_t-\tfrac{t}{2}|>2c_t)\\
&\leq q_{t}(z,y)\times \bigg(1+C\ell\sqrt{\frac{ \log t}{t}} \bigg) + 2\exp(-2\log t),
\end{split}
\end{equation}
where we have to choose  $\Cr{c:diffusive}$ (and hence $t$) large enough in the assumptions of Proposition~\ref{Lem:PoissonapproxPCRW}
The case where both $|x-z|$ and $|y-z|$ are odd is treated similarly, considering $2n+1$ instead of $2n$. If $|x-z|$ is even and $|y-z|$ is odd, note that for all $n$ such that $\lfloor t/4-c_t\rfloor\le n\le \lceil t/4+c_t \rceil$,\[
\mathbb{P}(N_t=2n)=\frac{2n+1}{t-2n}\cdot\mathbb{P}(N_t=2n+1)\le \bigg(1+C\sqrt{\frac{\log t}{t}}\bigg)\cdot\mathbb{P}(N_t=2n+1),
\]
where a corresponding lower bound holds in order to treat the case where $|x-z|$ is odd and $|y-z|$ is even.
Hence the result in \eqref{eq:heatkernelPCRW} holds regardless of the parity of $|x-z|$ and $|y-z|$. 
Using \eqref{eq:heatkernelPCRW}, we obtain that
\begin{multline}\label{eq:S1-S2}
\sum_{i=1}^{\lceil (2\lfloor c_t\rfloor+1)/\ell\rceil} \sum_{y \in I_i} \max_{x \in I_i} (q_t(z,x)-q_t(z,y))
\leq  \sum_{i=1}^{\lceil (2\lfloor c_t\rfloor+1)/\ell\rceil} \sum_{y \in I_i} C\ell \sqrt{\frac{\log t}{t}} q_t(z,y) + \frac{1}{t}\\
\le C\ell \sqrt{\frac{\log t}{t}}+\sum_{y\in I_1}q_t(z,y) \le C\ell \sqrt{\frac{\log t}{t}},
\end{multline}
where we used~\eqref{eq:heatkernelPCRW2} and the fact that $1\le\ell^2<t/\Cr{c:diffusive}$, chose  $\Cr{c:diffusive}$ (and hence $t$) large enough, and let the value of $C$ change from one line to the next.
Substituting~\eqref{eq:heatkernelPCRWdiagonal} and \eqref{eq:S1-S2} into \eqref{malalagorge} and feeding the resulting estimate together with \eqref{malodo}, into \eqref{eq:G-exp-2} yields that
$$
E^{\Q}[G_{\eta_0} (z)]  \leq\Big(\rho + \frac{\varepsilon}{2}\Big) \bigg( 1+ C\frac{\ell}{\sqrt{t}}  +C\ell \sqrt{\frac{\log t}{t}} +\frac{C}{\sqrt{t}} \bigg)\le \Big(\rho + \frac{\varepsilon}{2}\Big) \bigg( 1  +C\ell \sqrt{\frac{\log t}{t}}  \bigg),
$$
and the conclusion follows.
\end{proof}

We are now ready to give the short proof of Proposition~\ref{Lem:PoissonapproxPCRW}, which  combines the above ingredients.

\begin{proof}[Proof of Proposition~\ref{Lem:PoissonapproxPCRW}]
We use the coupling $\mathbb{Q}$ defined atop Lemma~\ref{L:SLT-exp} and show that $G_{\eta_0}$ concentrates in order to exploit~\eqref{eq:softlocaltimescoupling}. To this end, first note that for all $\theta < \frac12 \min_x q_t(x,z)^{-1}$,
\begin{equation}\label{eq:couple-pf2}
\mathbb E^{\Q}[ e^{\theta G_{\eta_0} (z)}]= \prod_{i \geq 1} \frac1{1-\theta q_t(x_i,z)}.
\end{equation}
Observe that only a finite number of factors may differ from $1$ (which requires $q_t(x_i,z)>0$, hence $\vert x_i-z\vert \leq t$). 
 By~\eqref{eq:softlocaltimescoupling}, the probability $\Q(G^c)$ with $G=\{ \eta_{t}\vert_{[t, H-t]} \preccurlyeq  \eta^{\rho+\varepsilon}\vert_{[t, H-t]}\}$ is thus bounded from above by 
\begin{multline}\label{eq:couple-pf3}
\Q[\exists z\in I_t : \, G_{\eta_0} (z) > \rho + \varepsilon ] \leq H \sup_{z \in I_t }\Q[ G_{\eta_0} (z) > \rho + \varepsilon ] \\ \leq  H  \sup_z \exp\big\{- \textstyle \theta \big(\rho + \varepsilon - \sum_{i\geq 1 } q_t(x_i,z) \big) + \sum_{i\geq 1 } \theta^2q_t(x_i,z)^2\big\}
\leq H \exp\big\{\textstyle -\frac{(\rho+ \varepsilon - \sum q)^2}{4\sum q^2} \big\},
\end{multline}
using \eqref{eq:couple-pf2}, the exponential Markov inequality and the inequality $\log(1-x)\ge-x - x^2$ for $|x|< \frac12$ in the second step and optimizing over $\theta$ in the third, and abbreviating $\sum q^{\alpha}= \sum_{i\geq 1}q_t(x_i,z)^{\alpha}$. Using~\eqref{eq:heatkernelPCRW2} and then \eqref{eq:G-exp}, we have
\begin{multline}
 \sum_{i\geq 1}q_t(x_i,z)^{2}\leq  \frac{C}{\sqrt t}   \sum_{i\geq 1}q_t(x_i,z) =  \frac{C}{\sqrt t} E^{\Q}[G_{\eta_0} (z)] \\
 \le \frac{C}{\sqrt t} \bigg(K+1+ \Cr{EGheatkernel} (\rho+\varepsilon){ {\ell} \sqrt{\frac{\log t}{t}}}\bigg)\le    \frac{C}{\sqrt t} \left((K+1) +    {\Cr{EGheatkernel}} \Cr{EGrestesmallerthanepsilon}\varepsilon\right)\le  \frac{C}{K\sqrt t},
 \end{multline}
 where we used the  assumptions of Proposition~\ref{Lem:PoissonapproxPCRW}, chose $\Cr{EGrestesmallerthanepsilon}$  small enough (depending on $\Cr{EGheatkernel}$), and let the value of $C$ change in the last inequality (depending on $K\ge 1$).
 Moreover, 
 \begin{equation}
 K+1\ge \rho+ \varepsilon - \sum_{i\geq 1}q_t(x_i,z)\ge \frac{\varepsilon}{2}-{\Cr{EGheatkernel}}(\rho+\varepsilon){ {\ell} \sqrt{\frac{\log t}{t}}}\ge      \frac{\varepsilon}{2}-{\Cr{EGheatkernel}} \Cr{EGrestesmallerthanepsilon}\varepsilon\ge \frac{\varepsilon}{4},
 \end{equation}
 provided that $\Cr{EGrestesmallerthanepsilon}$ is small enough (depending on $\Cr{EGheatkernel}$).
Substituting the two displays above into \eqref{eq:couple-pf3} provides the asserted upper bound on $\mathbb{Q}(G^c)$ in \eqref{eq:coupling-1}. For the other choice of $G$ in \eqref{eq:coupling-1}, we bound, using that $\log (1+x)\geq x-x^2/2$ for all $x\ge 0$,
\begin{multline}\label{eq:couple-pf4}
\Q[\exists z\in I_t : \, G_{\eta_0} (z) < \rho - \varepsilon ] \leq H\sup_{z \in I_t }\Q[ G_{\eta_0} (z) < \rho - \varepsilon ] \\ \leq  H  \sup_z \exp\big\{- \textstyle \theta \big(\varepsilon -\rho + \sum_{i\geq 1}q_t(x_i,z) \big) + \frac{1}{2}\sum_{i\geq 1}\theta^2q_t(x_i,z)\big\}
\leq H \exp\big\{\textstyle -\frac{(\varepsilon -\rho+ \sum q)^2}{2\sum q^2} \big\},
\end{multline}
optimizing again over $\theta$ in the last step. We conclude in the same way as for the upper bound.
\end{proof}

\subsection{Proof of Proposition~\ref{P:PCRW-C}} \label{sec:PCRW-C} The proof of Proposition~\ref{P:PCRW-C} follows immediately by combining Lemmas~\ref{Lem:C1PCRW}-\ref{L:nacelle-PCRW} below, each of which focuses on one specific property among~\ref{pe:densitychange},~\ref{pe:compatible}, \ref{pe:drift}, \ref{pe:couplings} and~\ref{pe:nacelle}, which are proved in this order. Recall that $J=(K^{-1},K)$ for some $K >1$ and that constants may implicitly depend on $K$.
\begin{Lem}\label{Lem:C1PCRW}
Condition~\ref{pe:densitychange} (with $\nu=1$) holds for PCRW.
\end{Lem}

\begin{proof}
 Let $\rho,\varepsilon, H,\ell,t$ and $\eta_0$ be such that the conditions of~\ref{pe:densitychange} hold.  It is straightforward to check that these imply the conditions for applying Proposition~\ref{Lem:PoissonapproxPCRW} with $(2H,2\varepsilon)$ instead of $(H,\varepsilon)$, provided that $\Cr{densitystable}$ is large enough w.r.t.~$\Cr{c:diffusive}$,~$\Cr{EGrestesmallerthanepsilon}$  and $K$. Hence, we can now use Proposition~\ref{Lem:PoissonapproxPCRW} with $(2H,2\varepsilon)$ to show~\ref{pe:densitychange}.

Let $1 \leq \ell' \leq \ell $. By~\eqref{eq:coupling-1} (with an appropriate coupling $\mathbb{Q}$ under which $\eta\sim \mathbf{P}^{\eta_0}$ and $\eta^{\rho\pm \varepsilon}\sim \mu_{\rho\pm \varepsilon}$, and translating $[0,2H]$ to $[-H,H]$ by means of~\eqref{pe:markov}), we have that
\begin{equation}\label{eq:poissonapproxPCRWdensity}
\mathbf{P}^{\eta_0}\left(\begin{array}{c} \text{for all $I' \subset [-H+2 t,H-2 t]$} \\ \text{of length $\ell'$: $ \pm( \eta_t(I') -\rho \ell' ) \leq 3\varepsilon\ell'$} \end{array} \right)\geq 1 -  2 \exp\left(-\Cr{coupling-1} (\rho+\varepsilon)^{-1} \varepsilon^2\sqrt{t}\right)-p_{\pm }
\end{equation}
where
\begin{equation*}
p_\pm\stackrel{\text{def.}}{=} \mu_{\rho\pm \varepsilon}\left(\begin{array}{c} \text{there exists an interval $I'$ of length $\ell'$ included in} \\ \text{$[-H+2 t,H-2 t]$ so that: $ | \eta^{\rho\pm\varepsilon}(I') -(\rho\pm\varepsilon) \ell' | \geq 2\varepsilon\ell'$} \end{array} \right).
\end{equation*}
By~\eqref{pe:density} (which holds on account of Lemma~\ref{L:PCRWP}) and a union bound over all intervals $I'\subseteq [-H+2 t,H-2 t]$ of length $\ell'$, we have that
$p_\pm\leq 2H \exp(-\Cr{densitydev}\varepsilon^2\ell').$  Combining this and~\eqref{eq:poissonapproxPCRWdensity}, noting that $\Cr{coupling-1}(\rho+\varepsilon)^{-1}\sqrt{t}\geq\Cr{densitystableexpo} \ell \geq \Cr{densitystableexpo}\ell'$ if we choose $\Cr{densitystableexpo}$ small enough w.r.t.~$K$ and $\Cr{coupling-1}$, yields~\ref{pe:densitychange}. 
\end{proof}

\begin{Lem}\label{Lem:driftPCRW}
For every $H,t\geq 0$, and all $\eta_0,\eta'_0 \in \Sigma$ such that $\eta_0 \vert_{[0, H]}\succcurlyeq \eta_0' \vert_{[0, H]}$, there exists a coupling $\Q$ of $\eta,\eta'$ with respective marginals $\bP^{\eta_0}$ and $\bP^{\eta'_0}$ such that
\begin{equation}\label{eq:driftPCRW}
\Q\big(\forall s\in [0,t], \, \eta_s\vert_{[t, H-t]}\succcurlyeq\eta'_s\vert_{[t, H-t]} \big)=1.
\end{equation}
Therefore, condition~\ref{pe:drift} holds for PCRW with $\nu=1$. 
\end{Lem}

\begin{proof}
Clearly,~\eqref{eq:driftPCRW} implies~\ref{pe:drift}, up to changing $H$ to $2H$ and translating $[0,2H]$ to $[-H,H]$ (using \eqref{pe:markov}, as established in Lemma~\ref{L:PCRWP}). We now show~\eqref{eq:driftPCRW}.

Let $H,t\geq 0$ and $\eta_0,\eta'_0\in \Sigma$ be as above. 
Couple $\eta$ and $\eta'$ by matching injectively each particle of $\eta'_0(x)$ to a particle of $\eta_0(x)$, for all $x\in [0,H]$, and by imposing that matched particles follow the same trajectory (and by letting all other particles follow independent lazy random walks). 
\\
Since particles can make at most one move (to a neighbouring position) per unit of time due to the discrete-time nature of the walks, no particle of $\eta'_0$ outside of $[0,H]$ can land in $[t,H-t]$ before or at time $t$. Thus the event in~\eqref{eq:driftPCRW} holds with probability 1.
\end{proof}

\begin{Lem}\label{Lem:doublecouple}
Let $\rho\in (K^{-1},K)$, $\varepsilon\in (0,(K-\rho) \wedge \rho \wedge 1)$, and $H,\ell,t\in \mathbb{N}$ be such that $\Cr{c:diffusive}\ell^2<t<H/2$ and $(\rho +\frac32\varepsilon) {(t^{-1}{\log t})^{1/2}} \ell< \Cr{EGrestesmallerthanepsilon}\varepsilon/4$. Let $\eta_0,\eta'_0\in \Sigma$ be such that for every interval $I\subseteq [0,H ]$ of length $\ell$, we have $\eta_0(I)\geq (\rho+3\varepsilon/4)\ell$ and $\eta'_0(I)\leq (\rho+\varepsilon/4)\ell$. Then there exists a coupling $\mathbb{Q}$ of $\eta$ and $\eta'$ such that 
\begin{equation}\label{eq:doublecouple}
\mathbb{Q}( \eta'_t \vert _{[t, H-t]}\preccurlyeq \eta_t \vert _{[t, H-t]})\geq 1- 4H \exp\big(-\Cr{coupling-1}(\rho+\varepsilon)^{-1}\varepsilon^2\sqrt{t}/4\big).
\end{equation}
Consequently,~\ref{pe:couplings} with $\nu=1$ holds for PCRW. Moreover, $\mathbb{Q}$ is local in that $(\eta'_t, \eta_t) \vert_{[t,H-t]}$ depends on the initial conditions $(\eta_0,\eta_0')$ through $\eta_0(x), \eta_0'(x)$, $x \in [0,H]$, alone.
\end{Lem}

\begin{proof}
We first show how~\eqref{eq:doublecouple} implies~\ref{pe:couplings}. Let $\rho,\varepsilon,H,\ell,t,\eta_0$ and $\eta'_0$ satisfy the assumptions of~\ref{pe:couplings} (in particular, $\ell=\lfloor t^{1/4}\rfloor$). Then they also satisfy the assumptions of Lemma~\ref{Lem:doublecouple} (with $2H$ instead of $H$), upon taking $\Cr{SEPcoupling}$ large enough in~\ref{pe:couplings}. By~\eqref{eq:doublecouple} applied to $[-H,H]$ instead of $[0,2H]$ (again using translation invariance, see~\eqref{pe:markov}, established in Lemma~\ref{L:PCRWP}),~\eqref{eq:doublecoupleSEP} holds with $\Cr{SEPcoupling2}=\max(4,\Cr{coupling-1}^{-1}K/2)$ since $\rho+\varepsilon\leq K$.

We now proceed to the proof of~\eqref{eq:doublecouple}. Let $\rho,\varepsilon,H,\ell,t,\eta_0$ and $\eta'_0$ satisfy the assumptions of Lemma~\ref{Lem:doublecouple}. Then we can apply Proposition~\ref{Lem:PoissonapproxPCRW} to $\eta_0$ with $(\rho+\varepsilon,\varepsilon/2)$ instead of $(\rho,\varepsilon)$, and the same values of $H,\ell,t$. Similarly, we can apply it to $\eta'_0$ with $(\rho-\varepsilon,\varepsilon/2)$ instead of $(\rho,\varepsilon)$. This entails the existence of two couplings $\mathbb{Q}^1$ of $(\eta_t,\eta^{\rho+\varepsilon/2})$ and $\mathbb{Q}^2$ of $(\eta^{\rho+\varepsilon/2},\eta'_t)$, where $\eta^{\rho+\varepsilon/2}\sim \mu_{{\rho+\varepsilon/2}}$, such that, abbreviating $I_t=[t,H-t]$,
 \begin{equation}\label{eq:q1q2}
 \begin{split}
 &\Q^1\big(\eta'_{t}\vert_{I_t} \preccurlyeq  \eta^{\rho+\varepsilon/2}   \vert_{I_t} \big) \wedge \Q^2\big(  \eta^{\rho+\varepsilon/2}\vert_{I_t}\preccurlyeq \eta_{t}\vert_{I_t}\big)  \geq 1 - H \exp\big(-\textstyle\frac{\Cr{coupling-1}}{4} (\rho+\varepsilon/2)^{-1} \varepsilon^2\sqrt{t}\big).
 \end{split}
 \end{equation}
Applying \cite[Lemma 2.4]{RI-III} with $(X,Y)=(\eta_t,\eta^{\rho+\varepsilon/2}) $ and $(Y',Z)= (\eta^{\rho+\varepsilon/2},\eta'_t) $, one can `chain' $\mathbb{Q}^1$ and $\mathbb{Q}^2$, i.e.~one obtains a coupling $\mathbb{Q}$ of $(\eta_{t}\vert_{I_t},\eta^{\rho+\varepsilon/2}\vert_{I_t}, \eta'_{t}\vert_{I_t})$ such that the pair $(\eta_{t}\vert_{I_t},\eta^{\rho+\varepsilon/2}\vert_{I_t})$ has the same (marginal) law as under $\mathbb{Q}^1$ and $(\eta^{\rho+\varepsilon/2}\vert_{I_t}, \eta'_{t}\vert_{I_t})$ has the same law as under $\mathbb{Q}^2$; explicitly, a possible choice is
  \begin{multline*}
 \Q\big(\eta_{t}\vert_{I_t}=\mu_A, \eta'_{t}\vert_{I_t}=\mu_B, \eta^{\rho+\varepsilon/2}\vert_{I_t}=\mu_C\big)   \\[0.5em]
  =\Q^1\big(\eta_{t}\vert_{I_t}=\mu_A\,\big|\, \eta^{\rho+\varepsilon/2}\vert_{I_t}=\mu_C\big) \cdot \Q^2\big(\eta'_{t}\vert_{I_t}=\mu_B\,\big|\, \eta^{\rho+\varepsilon/2}\vert_{I_t}=\mu_C\big) \cdot \Q^1\big(\eta^{\rho+\varepsilon/2}\vert_{I_t}=\mu_C\big),
 \end{multline*}
with $\mu_A, \mu_B, \mu_C$  ranging over point measures on $I_t$. On account of \cite[Remark 2.5,2)]{RI-III} applied with the choices $\varepsilon_1=1- \Q^1(\eta'_{t}\vert_{I_t} \preccurlyeq  \eta^{\rho+\varepsilon/2}   \vert_{I_t} )$ and $\varepsilon_2=1-  \Q^2(  \eta^{\rho+\varepsilon/2}\vert_{I_t}\preccurlyeq \eta_{t}\vert_{I_t})$, $\mathbb{Q}$ has the property that 
$\mathbb{Q}(\eta'_{t}\vert_{I_t}\preccurlyeq \eta_{t}\vert_{I_t}) \geq 1 -\varepsilon_1-\varepsilon_2$. In view of \eqref{eq:q1q2}, \eqref{eq:doublecouple} follows. The asserted locality of $\mathbb{Q}$ is inherited from $\mathbb{Q}^i$, $i=1,2$, due to Proposition~\ref{Lem:PoissonapproxPCRW} (used to define $\mathbb{Q}^i$). 
\end{proof}

\begin{Lem}
Condition~\ref{pe:compatible} with $\nu=1$ holds for PCRW. 
\end{Lem}

\begin{proof}
Let $\rho\in (K^{-1},K), \varepsilon\in (0,1)$, $H_1,H_2,t,\ell\in \N$, and $\eta_0,\eta'_0\in \Sigma$ be such that the conditions of~\ref{pe:compatible} hold. We proceed by a two-step coupling similar to the one in Lemma~\ref{Lem:compatible} and first give a short overview of both steps; cf.~also Fig.~\ref{f:Coupling2global_Exclusion}. In the first step, we couple $\eta$ and $\eta'$ during the time interval $[0, t_1]$, with $t_1:=\ell^4$, using the coupling of Lemma~\ref{Lem:doublecouple} on $[-H_2,H_2]\setminus [-H_1,H_1]$ and the coupling given in Lemma~\ref{Lem:driftPCRW}, making sure that these couplings can be simultaneously performed on disjoint intervals. As a result, we get that $\eta_{t_1}(x)\geq \eta'_{t_1}(x)$ for all $x\in [-H_2+t_1 , H_2- t_1]$, except possibly within two intervals around $-H_1$ and $H_1$, of width $O(t_1)$. 

In the second step, during the time interval $[t_1,t]$, we couple the particles of $\eta'$ on these intervals with "additional" particles of $\eta_{t_1}\setminus \eta'_{t_1}$ on $[-H_2,H_2]$ (using that the empirical density of $\eta'$ is slightly larger than that of $\eta$ on $[-H_2,H_2]$ by~\ref{pe:densitychange}), using Lemma~\ref{Lem:doublecouple}. This ensures that with large enough probability, all these particles of $\eta'$ get covered by particles of $\eta$ within time $t-t_1$, without affecting the coupling of the previous step by the Markov property.
\\
\\
\textbf{Step 1:} choosing $\Cr{compatible}$ large enough (in a manner depending on~$K, \Cr{c:diffusive}$ and $\Cr{EGrestesmallerthanepsilon}$), as we now briefly explain, the conditions of Lemma~\ref{Lem:doublecouple} hold for $\eta_0,\eta'_0$, with $(H,\ell,t)=(H_2-H_1-1,\ell,t_1)$ up to translating $[0,H]$ in either of the intervals $[-H_2,-H_1-1]$ or $[H_1+1,H_2]$. Indeed, the choice $t_1=\ell^4$ and the assumptions in \ref{pe:compatible} yield that 
$$(\rho+3\varepsilon/2)(t_1^{-1}\log t_1)^{1/2}\ell\le(K+2)\ell^{-1/2}\leq (K+2){\varepsilon}/{\sqrt{\Cr{compatible}}}\leq \Cr{EGrestesmallerthanepsilon}\varepsilon/4,$$
where we choose $\Cr{compatible}$ large enough depending on $\Cr{EGrestesmallerthanepsilon}$ and $K$.

During the time interval $[0,t_1]$, we apply  Lemma~\ref{Lem:doublecouple}, which we now know is in force, \textit{simultaneously} on $[-H_2,-H_1-1]$ and $[H_1+1,H_2]$. This is possible owing to the locality property stated as part of Lemma~\ref{Lem:doublecouple} (see below \eqref{eq:doublecouple}), since the couplings involved rely independently on the particles of $\eta_0([-H_2,-H_1-1])$ and $\eta'_0([-H_2,-H_1-1])$, and those of $\eta_0([H_1+1,H_2])$ and $\eta'_0([H_1+1,H_2])$ respectively. Moreover, by suitable extension of this coupling we can also couple, during the interval $[0,t_1]$, the particles of $\eta_0([-H_1,H_1])$ and $\eta'_0([-H_1,H_1])$ in the following way: match injectively each particle of $\eta'_0(x)$ to a particle of $\eta_0(x)$, for all $x\in [-H_1,H_1]$, and impose that matched particles follow the same trajectory (note that this is precisely the coupling underlying the statement of Lemma~\ref{Lem:driftPCRW}). 

From the construction in the previous paragraph, the four groups of particles $\eta_0(\Z\setminus [-H_2,H_2])$, $\eta_0([-H_2,-H_1-1]) $, $\eta_0([-H_1,H_1]) $ and $\eta_0([H_1+1,H_2]) $ evolve independently under $\Q$, and within each group the particles themselves follow independent lazy simple random walks. Hence under $\Q$, during $[0,t_1]$, we have indeed $\eta\sim \bP^{\eta_0}$, and by a similar argument that $\eta'\sim \bP^{\eta'_0}$. 

Now consider the events $\cE_1=\{\forall x\in [-H_1+t_1,H_1-t_1],\, \forall s\in [0,t_1],\,\eta_s(x)\geq \eta'_s(x) \}$ and $\cE_2= \{\forall x\in [-H_2+t_1,-H_1-1-t_1]\cup [H_1+1+t_1,H_2-t_1],\,\eta_{t_1}(x)\geq \eta'_{t_1}(x)  \}$ declared under $\Q$. By Lemmas~\ref{Lem:driftPCRW} and~\ref{Lem:doublecouple}, respectively, we have that
\begin{equation}\label{eq:PCRWC2step1E1E2}
\Q(\cE_1)=1\text{ and }\Q(\cE_2)\geq 1 - 8 (H_2-H_1)\exp\big(-\Cr{coupling-1}(\rho+\varepsilon)^{-1}\varepsilon^2{\color{black}\ell^2}/4\big).
\end{equation}
Then, define (still under $\Q$) two further events
\begin{align*}
&\cE_3= \left\{
\text{for all intervals }I\subseteq  [-H_2+2t_1,H_2-2t_1] \text{ of length $\ell^2$}: \,\eta'_{t_1}(I)\leq  (\rho+2\varepsilon/5) \ell^2\right\}
\\[0.5em]
&\cE_4=\big\{ \text{for all intervals }I\subseteq [-H_2+2t_1,H_2-2t_1]\text{ of length $\ell^2$}: \,\eta_{t_1}(I)\geq (\rho+3\varepsilon/5)\ell^2 \big\}.
\end{align*}
We apply condition~\ref{pe:densitychange},  which holds by Lemma~\ref{Lem:C1PCRW} to $\eta'$ with $(\rho,\varepsilon,\ell,\ell',H,t)=(\rho+\varepsilon/5,\varepsilon/20,\ell,\ell^2,2H_2,t_1)$, and to $\eta$ with $(\rho,\varepsilon,\ell,\ell',H,t)=(\rho+3\varepsilon/4,\varepsilon/20,\ell,\ell^2,2H_2,t_1)$. A straightforward computation proves that the necessary conditions are implied by the assumptions in \ref{pe:compatible}. This yields
\begin{equation}\label{eq:PCRWC2step1E3E4}
\Q(\cE_3 \cap \cE_4) \geq 1-16H_2\exp(-\Cr{densitystableexpo}\varepsilon^2\ell^2/400).
\end{equation}

\medskip
\noindent
\textbf{Step 2:} since $\cE_1$ holds with full $\Q$-measure, we can, as in Lemma~\ref{Lem:driftPCRW}, pair injectively each particle of $\eta'_{t_1}([-H_1+t_1,H_1-t_1])$ to one of $\eta_{t_1}$ on the same site, and impose by suitable extension of $\Q$ that paired particles follow the same trajectory during $[t_1,t]$, all pairs being together independent. We obtain in this way that
\begin{equation}\label{eq:PCRWcentraltrapezoid}
\Q\big(\forall s\in [0,t], \, \eta_s\vert_{[-H_1+t, H_1-t]}\succcurlyeq \eta'_s\vert_{[-H_1+t, H_1-t]} \big)=1.
\end{equation}
To complete the construction of $\Q$ it remains to describe the trajectory of the particles of $\eta'_{t_1}(\Z\setminus [-H_1+t_1,H_1-t_1])$ and $\eta_{t_1}(\Z\setminus [-H_1+t_1,H_1-t_1])$, and those of $\eta_{t_1}([-H_1+t_1,H_1-t_1])$ that were not paired. On  $(\cE_1 \cap \cE_2\cap \cE_3 \cap \cE_4 )^c$, let all these particles follow independent lazy simple random walks during $[t_1,t]$, independently from the particles paired at~\eqref{eq:PCRWcentraltrapezoid}. On  $(\cE_1 \cap \cE_2\cap \cE_3 \cap \cE_4 )\subseteq (\cE_1\cap \cE_2)$, note that $\eta_t(x)\geq \eta'_t(x)$ for all $x\in [-H_2+t_1, H_2-t_1]\setminus I_1$, where
$$
I_1=[-H_1-t_1,-H_1+t_1-1] \cup [H_1-t_1+1,H_1+t_1]. 
$$
We pair injectively each particle of $\eta'_{t_1}( [-H_2+t_1, H_2-t_1]\setminus I_1)$ to one of $\eta_{t_1}$ on the same site, and impose by suitable extension of $\Q$ that paired particles follow the same trajectory during $[t_1,t]$, all pairs being together independent.

For convenience, we introduce
\begin{equation}\label{eq:PCRWetatilde}
\begin{split}
&\widetilde{\eta}_{t_1}(x)=
        \eta_{t_1}(x)-\eta'_{t_1}(x){\1}_{\{x \in [-H_2+t_1, H_2-t_1]\setminus I_1 \}}
    \\
  &\widetilde{\eta}_{t_1}'(x)=\eta'_{t_1}(x){\1}_{\{x \in I_1 \}},
    \end{split}
\end{equation}
which denote the number of particles of $\eta_{t_1}$ and $\eta'_{t_1}$ at position $x\in \mathbb{Z}$ whose trajectory has not yet been described. It thus remains to cover the particles of $\widetilde{\eta}'$ by those of $\widetilde{\eta}$ by time $t$. Let us first explain the reasoning to establish this covering. Note that the two intervals making up $I_1$ escape to our couplings during $[0,t_1]$, so that we can only guarantee that the empirical density of $\widetilde{\eta}_{t_1}'$ is lower than $\rho+2\varepsilon/5$, see $\cE_3$. Fortunately, $\cE_3$ and $\cE_4$ ensure that the empirical density of $\widetilde{\eta}_{t_1}$ is at least $\varepsilon/5$ on $[-H_2+t_1, H_2-t_1]\setminus I_1$, which is much wider than $I_1$. We thus apply Lemma~\ref{Lem:doublecouple} with a mesh much larger than $\vert I_1\vert =4t_1$ in order to 'dilute' the particles of $\widetilde{\eta}'_{t_1}$. By making this coupling independent of the other particles of $\eta,\eta'$ previously paired, we will ensure that both $\eta$ and $\eta'$ have the correct PCRW marginals. 
\\
We now formalise this coupling. By choosing $\Cr{compatible}$ large enough, one can check, via a straightforward computation, that the assumptions in~\ref{pe:compatible} imply the necessary conditions to apply Lemma~\ref{Lem:doublecouple} for $\eta$ and $\eta'$ of  on $[-H_2+2t_1,H_2-2t_1]$ with $(H,\ell,t,\rho,\varepsilon )= (2H_2-4t_1,\ell^5,t-t_1,\varepsilon/40,\varepsilon/50)$. Hence by Lemma~\ref{Lem:doublecouple}, we can extend $\Q$ such that 
\begin{equation}\label{eq:PCRWC2coupling}
\begin{split}
&\text{the trajectories of the particles of $\widetilde{\eta}_{t_1+\cdot}$ and}
\\
&\text{$\widetilde{\eta}'_{t_1+\cdot}$ are independent of those paired at~\eqref{eq:PCRWcentraltrapezoid},}
\end{split}
\end{equation}
 and such that (using that $\varepsilon < 1$), on the event $\cE_1 \cap \cE_2\cap \cE_3 \cap \cE_4$,
\begin{multline}\label{eq:PCRWC2couplingproba}
\Q\big(\widetilde{\eta}_{t}\vert_{[-H_2+2t,H_2-2t]}\succcurlyeq \widetilde{\eta}'_{t}\big\vert_{[-H_2+2t,H_2-2t]} \, | \, (\eta_s, \eta_s' )_{ s \in [0, t_1]} \big) \black \\
 \geq 1- 8H_2 \exp\left(-\Cr{coupling-1}\varepsilon^2\sqrt{t-t_1}/10^4\right).\black
\end{multline}
Let us check that Step 2 yields the marginals $\eta \sim \bP^{\eta_1}$ and $\eta \sim \bP^{\eta'_1}$ during $[t_1,t]$ (we have already seen in Step 1 that  $\eta \sim \bP^{\eta_0}$ and $\eta \sim \bP^{\eta'_0}$ during $[0,t_1]$ so that the Markov property~\eqref{pe:markov} will ensure that  $\eta \sim \bP^{\eta_0}$ and $\eta \sim \bP^{\eta'_0}$ during $[0,t]$). Remark that the events $\cE_i$, $1\leq i\leq 4$ are measurable w.r.t.~the evolution of $\eta$ and $\eta'$ until time $t_1$. On $(\cE_1 \cap \cE_2\cap \cE_3 \cap \cE_4 )^c$, by~\eqref{eq:PCRWcentraltrapezoid} and the paragraph below~\eqref{eq:PCRWetatilde}, it is clear that all particles of $\eta_{t_1}$ follow independent lazy simple random walks, and that the same is true for $\eta'_{t_1}$. On $\cE_1 \cap \cE_2\cap \cE_3 \cap \cE_4$,~\eqref{eq:PCRWC2coupling} and Lemma~\ref{Lem:doublecouple} ensure that this is also the case. Therefore, we have indeed that $\eta \sim \bP^{\eta_1}$ and $\eta \sim \bP^{\eta'_1}$.

Finally, we explain how to derive~\ref{pe:compatible} from our construction. Note that~\eqref{eq:compatible1} immediately follows from~\eqref{eq:PCRWcentraltrapezoid} (with full $\Q$-probability). As for~\eqref{eq:compatible2}, we have
\begin{multline}\label{eq:final-C-2PRCW}
Q \stackrel{\text{def.}}{=} \Q({\eta}_{t}\vert_{[-H_2+6t,H_2-6t]}\succcurlyeq {\eta}'_{t}\vert_{[-H_2+6t,H_2-6t]}) \\ \geq \Q(\widetilde{\eta}_{t}\vert_{[-H_2+2t,H_2-2t]}\succcurlyeq \widetilde{\eta}'_{t}\vert_{[-H_2+2t,H_2-2t]})
\end{multline}
by~\eqref{eq:PCRWcentraltrapezoid} and~\eqref{eq:PCRWetatilde}. Thus, by combining~\eqref{eq:PCRWC2step1E1E2},~\eqref{eq:PCRWC2step1E3E4},~\eqref{eq:PCRWC2couplingproba} and \eqref{eq:final-C-2PRCW}, we get
\begin{multline*}
Q \geq 1- \Q(\cE_2^c\cup \cE_3^c\cup \cE_4^c)-8H_2 \exp\left(-\Cr{coupling-1}\varepsilon^2\sqrt{t-t_1}/10000\right)
\\\geq 1-32H_2\exp\left( -\Cr{densitystableexpo}\varepsilon^2\ell/64 \right)
\geq 1-5\Cr{SEPcoupling2}\ell^4H_2\exp\big( -{\varepsilon^2\ell}/{(2\Cr{SEPcoupling2})} \big), 
\end{multline*} 
choosing $\Cr{compatible}$ and $\Cr{SEPcoupling2}$ large enough (w.r.t.~$K$, $\Cr{densitystableexpo}$ and $\Cr{coupling-1}$). 
This yields~\eqref{eq:compatible2} and concludes the proof,  since \eqref{eq:compatible1} is implied by Lemma \ref{Lem:driftPCRW}.
\end{proof}

\begin{Lem}\label{L:nacelle-PCRW}
Condition~\ref{pe:nacelle} (with $\nu=1$) holds for PCRW, the constraint on $k$ with the pre-factor $48$ now replaced by $24(K+1)$.
\end{Lem}

\begin{Rk}
The modification of the pre-factor appearing in the constraint on $k$ is inconsequential for our arguments (and consistent with SEP where one can afford to choose $K=1$ since $J=(0,1)$). Indeed~\ref{pe:nacelle} is only used at~\eqref{eq:etapmnacellecoupling}, where this modified condition on $k$ clearly holds for $L$ large enough ($K$ being fixed). 
\end{Rk}

\begin{proof} We adapt the proof of Lemma~\ref{lem:nacelle}. 
Let $H,\ell,k\geq 1$, $\rho\in (0,K), \varepsilon\in (0,K-\rho)$ and $\eta_0, \eta'_0$ be such that the conditions of~\ref{pe:nacelle} hold. We can thus pair injectively each particle of $\eta'_0([-H,H])$ to one particle of $\eta_0([-H,H])$ located at the same position. Moreover, y assumption there is at least one particle of $\eta_0([0,\ell])$ that is not paired. For $s\geq 0$, denote $Z_s$ the position of this particle at time $s$. 

Let $\Q$ be a coupling of $\eta$ and $\eta'$ during $[0,\ell]$ such that paired particles perform the same lazy simple random walk (independently from all other pairs), and all other particles of $\eta_0$ and $\eta'_0$ follow independent lazy simple random walks (which yields the marginals  $\eta\sim \bP^{\eta_0}$ and $\eta'\sim \bP^{\eta'_0}$). As in Lemma~\ref{Lem:driftPCRW}, we get that 
\begin{equation}
\Q\big(\forall s\in [0,\ell], \, \eta_s\vert_{[-H+\ell, H-\ell]}\succcurlyeq\eta'_s\vert_{[-H+\ell, H-\ell]} \big)=1
\end{equation}
and~\eqref{eq:penacelleNEW2} follows. 

It remains to show~\eqref{eq:penacelleNEW}. Assume for convenience that $\ell$ is even (the case $\ell$ odd being treated in essentially the same way). Similarly as in the proof of Lemma~\ref{lem:nacelle}, we get that
\begin{equation}\label{eq:PCRWnacelleproba}
\mathbb{Q}(\eta_{\ell}( x )>0  , \,    \eta'_{\ell}( x )=0)\geq \mathbb{Q}( Z_\ell=x , \, \eta'_\ell(x)= 0),
\end{equation}
for $x=0,1$. Note that by our construction, the two events $\{Z_\ell=x \}$ and $ \{ \eta'_\ell(x)= 0\}$ are independent. Clearly, we have that for both $x=0,1$,
\begin{equation}\label{eq:PCRWnacelleproba1}
\Q(Z_\ell=x)\geq 4^{-\ell}.
\end{equation}
Note that there is no parity issue in \eqref{eq:PCRWnacelleproba1} because $Z$ performs a {lazy} random walk. 
Moreover, since $\eta'_0([-3\ell+1, 3\ell])\leq 6(\rho+1)\ell$ by assumption and since no particle outside $[-3\ell+1, 3\ell]$ can reach $0$ by time $\ell$, it follows that at time $\ell-1$, there are at most $6(\rho+1)\ell$ particles of $\eta'_0$ in $[-1,1]$, and each of them has probability at least $1/2$ not to be at $x \in \{0,1\}$ at time $\ell$. Hence
\begin{equation}\label{eq:PCRWnacelleproba2}
\Q( \eta'_\ell(x)= 0)\geq 2^{-6(\rho +1)\ell}.
\end{equation}
Putting together~\eqref{eq:PCRWnacelleproba},~\eqref{eq:PCRWnacelleproba1} and~\eqref{eq:PCRWnacelleproba2}, we get that the probability on the left of \eqref{eq:PCRWnacelleproba} is bonded from below by $(2e)^{-6(\rho+1)\ell}$, whence \eqref{eq:penacelleNEW}.
\end{proof}

\bibliography{bibliographie}
\bibliographystyle{abbrv}

\end{document}